\documentclass[12pt,english]{smfart}

\usepackage[all,cmtip]{xy}
\usepackage{todonotes}
\usepackage{yfonts}
  
  \usepackage{bbm}

\usepackage{amsmath,amssymb,amscd}

\usepackage{babel}
\usepackage{amstext}
\usepackage{amsmath}
\usepackage{amsfonts}
\usepackage{latexsym}
\usepackage{ifthen}
\usepackage{xypic}
\usepackage[all,cmtip]{xy}
\xyoption{all}
\usepackage{hyperref}
\usepackage{tikz-cd}
\usepackage{graphicx}
\usepackage{enumitem}
\setlist[enumerate]{label=(\thethm.\arabic*), before={\setcounter{enumi}{\value{equation}}}, after={\setcounter{equation}{\value{enumi}}}}

\usepackage{thmtools,bm}
\usepackage{cleveref} 
\usepackage[T1]{fontenc}
\usepackage{newtxtext}
\usepackage{newtxmath}
\usepackage{textcomp}
\theoremstyle{plain}
\newlist{thmlist}{enumerate}{1}
\setlist[thmlist]{wide = 0pt, labelwidth = 2em, labelsep*=0em, itemindent = 0pt, leftmargin = \dimexpr\labelwidth + \labelsep\relax, noitemsep,topsep = 1ex, font=\normalfont, label=(\roman*), ref=\thethm.(\roman{thmlisti})}

\usepackage[text={155mm,251mm},centering, marginparwidth=75pt]{geometry} 
 \usepackage{smfhyperref}
\usepackage[hyperpageref]{backref}

\newcommand\numberthis{\addtocounter{equation}{1}\tag{\theequation}}

\addtotheorempostheadhook[thm]{\Crefalias{thmlisti}{thm}}

\addtotheorempostheadhook[lemme]{\Crefalias{thmlisti}{lemme}}

\renewcommand{\P}{\mathbb{P}}
\newcommand{\R}{\mathbb{R}}
\newcommand{\CC}{\mathbb{C}}
\newcommand{\Q}{\mathbb{Q}}
\newcommand{\DD}{\mathbb D}
\newcommand{\Z}{\mathbb{Z}}

\renewcommand{\d}{\partial}

\newcommand{\db}{\bar{\partial}}
\newcommand{\ddbar}{\partial\bar{\partial}}

\newcommand{\wh}{\widehat}
\newcommand{\wt}{\widetilde}
\newcommand{\cX}{\mathcal{X}}
\newcommand{\cF}{\mathcal{F}}
\newcommand{\Lie}{\mathcal{L}}
\newcommand{\bH}{\mathbb{H}}
\newcommand{\cC}{\mathcal{C}}
\newcommand{\cA}{\mathcal{A}}
\newcommand{\cI}{\mathcal{I}}
 
\newcommand{\CR}{\mathcal{R}}

\renewcommand{\O}{\mathcal{O}}
\renewcommand{\cD}{\mathcal{D}}

\newcommand{\ep}{\varepsilon}
\renewcommand{\epsilon}{\varepsilon}
\newcommand{\im}{\mathrm{Im} \,}
\newcommand{\rel}{\mathrm{Re} \,}

\renewcommand{\ker}{\mathrm{Ker} \,}

\newcommand{\ol}{\overline}

\renewcommand{\leq}{\leqslant}
\renewcommand{\geq}{\geqslant}

\newcommand{\Id}{\mathrm{Id}}
\newcommand{\Jac}{\mathrm{Jac}}
\newcommand{\Res}{\mathrm{Res} \,}
\newcommand{\dom}{\mathrm{Dom}}

\newcommand{\IM}{\mathrm{Im}}

\newcommand{\D}{D}
\newcommand{\cE}{\mathcal{E}}

\newcommand{\Supp}{\mathrm {Supp}}

\newcommand{\Hom}{\mathrm{Hom}_B}

\newcommand{\res }{\mathrm{Res }}

\newcommand{\dbar}{\bar \partial}

\newtheorem{thm}{Theorem}[section]
\newtheorem{lemme}[thm]{Lemma}
\newtheorem{claim}[thm]{Claim}
\newtheorem{proposition}[thm]{Proposition}

\newtheorem{conjecture}[thm]{Conjecture}
\newtheorem{convent}[thm]{Convention}
\newtheorem{defn}[thm]{Definition}
\newtheorem{cor}[thm]{Corollary}
\theoremstyle{remark}

\newtheorem{remark}[thm]{Remark}

\newtheorem{property}[thm]{Property}

\numberwithin{equation}{thm}

\title[Hodge theory for local systems]{Hodge theory for local systems and cohomological support loci}
\author[J. Cao]{Junyan Cao}
\address{Laboratoire de Mathématiques J.A. Dieudonné UMR 7351 CNRS, Université Côte d'Azur Parc Valrose 06108, Nice, France} 
\email{junyan.cao@unice.fr}
\urladdr{https://sites.google.com/site/junyancao} 

\author[Y. Deng]{Ya Deng}
\address{CNRS,  
	Institut de Math\'ematiques de Jussieu-Paris Rive Gauche,
	Sorbonne Universit\'e, Campus Pierre et Marie Curie,
	4 place Jussieu, 75252 Paris Cedex 05, France}
\email{ya.deng@math.cnrs.fr, deng@imj-prg.fr}    
\urladdr{https://ydeng.perso.math.cnrs.fr}

\author[C. Hacon]{Christopher D. Hacon}
\address{Department of Mathematics, University of Utah, Salt Lake City, UT 84112,
	USA}
\email{hacon@math.utah.edu}
\urladdr{https://www.math.utah.edu/~hacon/}

\author[M. Paun]{Mihai Paun}
\address{Universit\"at Bayreuth, Mathematisches Institut, Lehrstuhl Mathematik VIII, Universit\"atsstrasse 30,
	D-95447, Bayreuth, Germany}
	\email{mihai.paun@uni-bayreuth.de}
	\urladdr{https://www.mathe8.uni-bayreuth.de/de/team/prof-paun/index.php}

 \thanks{The first and the second authors are partially supported by the ANR grant Karmapolis (ANR-21-CE40-0010). The third author was partially supported by  NSF research grant   DMS-2301374 and
by a grant from the Simons Foundation SFI-MPS-MOV-00006719-07. The fourth author gratefully acknowledge support from \emph{Deutsche Forschungsgemeinschaft}}
\begin{document}

\maketitle

\bigskip

\noindent 
In this article, we pursue two main objectives. The first is to demonstrate that the fundamental results of Green–Lazarsfeld \cite{GL87,GL91} and Budur–Wang \cite{BW15,BW20} can be established in a unified framework, using various versions of the 
$\ddbar$-lemma. Our second—and primary—goal is to develop the technical tools necessary to achieve this. We begin by highlighting some of our key contributions, and conclude this introductory section with a discussion of additional applications, which will be explored in a companion piece to the present work.
\smallskip

Let \(X\) be a compact Kähler manifold endowed with a holomorphic line bundle \((L,\dbar_L)\), and let \(\alpha\) be a \((0,1)\)-form whose complex conjugate is holomorphic. For each \(t \in (\CC,0)\), consider the twisted \(\dbar\)-operator
\begin{equation}\label{intr1}
	\dbar_t := \dbar_L + t\alpha
\end{equation}
acting on the space of smooth \(L\)-valued \((p,q)\)-forms on $X$.  
By basic Hodge theory, this operator satisfies the integrability condition \(\dbar_t^2 = 0\).  
 A natural problem is to study how the finite-dimensional spaces
\begin{equation}\label{intr2}
	\ker(\dbar_t) \big/ \im(\dbar_t)
\end{equation}
vary with \(t\). General results (see \cite{Kod86}) ensure that the dimension of the spaces in \eqref{intr2} defines an upper semicontinuous function of \(t\).

\smallskip

As starting point for the type of results we obtain in this article, we re-interpret next a beautiful and important 
statement due to Green-Lazarsfeld, \cite{GL87}. Let $u\in \ker \dbar$ be a holomorphic $(n, 0)$-form on $X$, that is to say, a section of the canonical bundle. We say that $u$ \emph{extends to order one in the direction $\alpha$} if the wedge product $u\wedge \alpha$ is $\dbar$-exact, i.e. we can solve the equation
\begin{equation}\label{intr3}
	\dbar v= u\wedge \alpha. \end{equation} 
\smallskip

\noindent Following \cite[\S 1]{GL87}, we have the following.
\begin{thm}
	Let $u$ be a holomorphic section of $K_X$, which extends to order one in the direction $\alpha$. Then there exists a smooth family $\displaystyle (u_t)_{t\in (\mathbb C, 0)}$ of $(n, 0)$-forms, such that 
	the following hold
	\begin{equation}\label{intr4}
		u_0:= u, \qquad u_t\in \ker\dbar_t,
	\end{equation} 
	so that $u$ can be extended to a holomorphic section for the operator $\dbar_t$. 
\end{thm}
\noindent Phrased in more colloquial terms, this shows that extension to order one implies extension to arbitrary order. 
Basically all the results we will discuss here have this flavor. We will present them in an increasing degree of complexity, but we first introduce some notations. 

Let $L\to X\times \Jac(X)$ be the holomorphic line bundle whose restriction to the fiber $X\times \gamma$ is the flat line bundle corresponding to the $(0,1)$-form $\gamma$. 

The higher rank analogue of this construction involves a hermitian flat vector bundle $E\to X$, together with a harmonic $(0,1)$-form $\xi$ with values in $End(E)$ such that $\xi\wedge \xi= 0$. Then the twisted $\dbar$ operator 
\begin{equation}\label{intr4}
	\dbar_E+ t\xi \end{equation}
is integrable, and induces a deformation $\cE\to X\times \mathbb C$ of $E$. 
\smallskip

\noindent In sections 1-4 we give a proof of the next well-known statements from the perspective we promote here.

\begin{thm}\cite{YTS}, \cite{Wang}, \cite{Bud09} Let $u$ be one of the following objects:
\begin{enumerate}
	\item [\rm (1)] a holomorphic section of $K_X+L|_{X\times \gamma}$,
	
	 \item [\rm (2)] a holomorphic section of $K_X+M+L|_{X\times \gamma}$, where $M$ is a line bundle on $X$ whose first Chern class contains the snc divisor $\sum a_i D_i$, with $0< a_i< 1$, or
	 
	 \item [\rm (3)] a $\dbar$-closed, $E$-valued $(n, p)$-form. 
\end{enumerate}	
We assume that $u$ extends to order one. Then it extends to any arbitrary order.	
\end{thm}

\noindent Before stating our next result, notice that the "extension to order one" hypothesis is the same as assuming that $u$ admits an extension to the first infinitesimal neighborhood of the fiber on which it is defined.
\medskip

\noindent A somehow weaker theorem holds in the pluricanonical case.
\begin{thm} [\protecting{\cite[Section 10]{HPS18}}] 
Let $u$ be a section of $kK_X+ L|_{X\times \gamma_0}$, and let
$\gamma: (\mathbb C, 0)\to \Jac(X)$ be a smooth path, such that $\gamma(0)= \gamma_0$. Assume further that the function
\[t\to h^0(X, kK_X+ L|_{X\times \gamma(t)})\]
is constant. Then $u$ deforms along the linear path $t\to \gamma(0)+ t\gamma'(0)$.	 
\end{thm}
\medskip

\noindent These results are established in Sections 1-4, ultimately by using appropriate versions of the $\ddbar$--lemma in Hodge theory. In Section 5 we consider the following problem.

\begin{conjecture}\label{invariance}
Let $p:\mathcal X\to \mathbb D$ be a smooth, proper family of Calabi-Yau manifolds and let $L\to \mathcal X$ be a line bundle, whose restriction to the central fiber of $p$ is globally generated. We assume that $h^{1,0}(X)=0$ where $X$ is the central fiber. Then the function
\[t\to h^0(\mathcal X_t, L|_{\mathcal X_t})\] 
is constant.
\end{conjecture} 
The main results in \cite{JCMP} are used to verify \Cref{invariance} in certain special cases. 
For example, in Theorem \ref{ext0}, we show that if there exists an integer $m\geq 2$ and an element $f \in H^0(X, mL)$ such that $\operatorname{Div}(f)$ is smooth and $f$ can be extended to $\mathcal X$, then Conjecture~\ref{invariance} holds. 

\medskip

\noindent We come now to the main part of our work, concerning the quasi-compact version of this type of theorems. Let $(X, D)$ be a pair consisting of a compact K\"ahler manifold $X$ and an snc divisor $D$. We denote by $M_{\rm B}(X_0)$ the space of rank one local systems on $X\setminus D$.
Then the following holds true.
\begin{thm}[\cite{BW20}]\label{introcomp}
	Consider the germ of a holomorphic disk $\gamma: (\CC, 0)\to M_{\rm B} (X_0)$, and fix $\gamma(0)= \tau$. Let $k$ be a positive integer so that we have $\displaystyle \dim  H^p (X\setminus D, \gamma ( t)) \geq k$ for every $|t| \ll 1$.  Then we have
	\begin{equation} 
		\dim H^p (X\setminus D, \tau + t\dot\alpha )\geq k
	\end{equation}
	for $ |t|\ll 1$, where $\dot\alpha$ is the infinitesimal deformation  $d\gamma(0)$. 
\end{thm}

\noindent In the statement above, we denote by \(H^\bullet(X\setminus D, \tau)\) the cohomology groups with coefficients in \(\tau\).  
Our proof of \Cref{introcomp} is self-contained and differs substantially from the original argument in \cite{BW20}, which relies on global considerations involving moduli spaces and the Riemann–Hilbert correspondence (with some ideas tracing back to \cite{Sim93}).  
We now briefly outline our main ideas of the proof of \Cref{introcomp}.

Let $M_{\rm DR}(X/D)$ be the space of logarithmic flat bundles, i.e. pairs $(L, \nabla)$ consisting of line bundles $L\to X$ together with a holomorphic connection   
\[\nabla: L\to L\otimes \Omega_X(\log D)\]
with log poles along $D$ such that $\nabla^2= 0$. We have a natural map
\begin{equation}\label{introsurjj}
\rho: M_{\rm DR}(X/D)\to M_{\rm B}(X\setminus D),
\end{equation}
which is surjective, by a well-known result due to Deligne, see \cite[II, 6.10]{Del}. As shown in the same reference, the correspondence between the 
two spaces in \eqref{introsurjj} goes much further: if we consider the complex 
\begin{equation}\label{introderham}
	0 \to \mathcal{O} (L) \to\mathcal{O} (L) \otimes \Omega_X^{1}(\log D) \to  \mathcal{O} (L) \otimes \Omega_X^{2}(\log D) \to \cdots \mathcal{O} (L) \otimes \Omega_X^{n}(\log D) \to 0 
	\end{equation}
induced by the connection $\nabla$, then we have $\displaystyle \mathbb H^p (X, L \otimes \Omega_X ^{\bullet }(\log D) )\simeq H^p (X_0, \tau)$,
provided that the residue of $(L, \nabla)$ on any component of $D$ is not a positive integer. Here $\tau$ is the image of $(L, \nabla)$ via the map $\rho$. 
\smallskip

\noindent The strategy we adopt for the proof of Theorem \ref{introcomp} consists in developing the necessary tools in order to construct a resolution of  \eqref{introderham}. Then we show that the cohomology of the resulting complex is computed by the kernel of a Laplace-type operator corresponding to $\nabla + \dbar$. After that, implementing the principle \emph{extension to order one implies extension to infinite order} becomes relatively easy, and it is based on a version of $\ddbar$-lemma we establish here. We next provide a few details about all of this.

Let $h$ be a metric on $L$, which is non-singular when restricted to
$X\setminus D$, with at most analytic singularities along $D$. If $D'_h$ is the $(1, 0)$-part of the corresponding Chern connection, then the difference
\[\theta_0:= \frac{1}{2}(\nabla- D'_h)\] 
is a $\partial$-closed $(1, 0)$-form with log poles on $D$. Recall that the metric $h$ is said to be \emph{harmonic} if the operator
\[D''_K:= \dbar+ \theta_0- \ol\theta_0\]
verifies the integrability condition $D''_K\circ D''_K= 0$. Its conjugate reads as 
\[D'_K:= D'_h+ \theta_0+ \ol\theta_0\]
so that $D_K:= D'_K+ D''_K$ equals $\nabla + \dbar$. 
\medskip

\noindent Next, given $\tau \in M_{\rm B}(\cX)$ we show that we can construct a logarithmic flat bundle $(L, \nabla)$ such that $\rho(L, \nabla)= \tau$ and such that the following hold:
\begin{itemize}

\item $L$ admits a harmonic metric $h$;

\item  the divisor $D$ can be decomposed as follows $D= \Delta_1+ \Delta_2+ \Delta_3$, such that we have:
\begin{enumerate}

\item[\rm (a) ] the residue $\res(\theta_0)$ of the Higgs field along the components of $\Delta_1+ \Delta_3$ is zero, and it is non-zero along each component of $\Delta_2$;

\item[\rm (b) ] the support of the curvature current $\sqrt{-1}\Theta(L, h)$ is contained in $D$. All its Lelong numbers are rational, and those corresponding to components of $\Delta_2+ \Delta_3$ belong to $]-1, 0[$. The Lelong numbers along the components of $\Delta_1$ are equal to zero.  
\end{enumerate}
\end{itemize}
\smallskip

\noindent The resolution of the complex \eqref{introderham} mentioned above involves $D_K$-closed forms of fixed total degree with values in $L$ and log poles on $D$. In order to analyze their properties, we endow $X_0:= X\setminus\Delta_2$ with a Kähler metric $\omega_D$ which has Poincaré singularities on $\Delta_2$ and conic singularities 
along $\Delta_3$, cf. Section \ref{S-qc}. 
Let $\xi$ be a log, $L$-twisted differential form. 

\noindent $\bullet$ In case $\Delta_1= 0$, we see that 
its $L^2$-norm $\displaystyle \int_{X_0}|\xi|^2_{\omega_D}e^{-\varphi_L}dV< \infty$ of $\xi$ is finite (due to property (b) above), and therefore it is conceivable that a version of the familiar Hodge theory for the (Higgs version of the) Laplace operator 
\begin{equation}\label{intro30} 
\Delta_K= [D_K, D_K^\star]
	\end{equation}
still holds. In Section \ref{s-tech-I} we show that this is indeed the case: our main contribution here is the a-priori inequality \S \ref{apriori}, which incorporates the 
Higgs field $\theta_0$. Based on it, we establish a Hodge decomposition for the space of $L^2$-forms, and we give a very precise description  
of the space of harmonic forms $\ker \Delta_K$.  
\smallskip

\noindent $\bullet$ In case $\Delta_1\neq 0$, the crucial $L^2$ integrability of forms $\xi$ as above fails to hold in general. There are at least two manners in which one can deal with this situation.
\begin{enumerate}

\item[\rm (i)] The form $\xi$ writes as $\displaystyle \xi= \sum_{p+q= k} \xi_{p, q}$, where $\xi_{p, q}$ is an $L$-valued form of type $(p, q)$ with log-poles along $D$. Consider the vector bundle $T_X\langle \Delta_1\rangle$ corresponding to the sheaf of logarithmic vector fields on $(X, D)$, endowed with a Hermitian metric $g_D$ which has the same type of singularities as $\omega_D$  (this is sometimes referred to as \emph{metric with cylindrical singularities}, cf. \cite{LZ14}). Then $\xi_{p, q}$ can be seen as $(0, q)$-form with values in 
$\Omega^p_X\langle \Delta_1\rangle\otimes L$ and log-poles along $\Delta_2+ \Delta_3$. It follows that we have 
$\displaystyle \int_{X_0}|\xi|^2_{g_D, \omega_D}e^{-\varphi_L}dV< \infty$. The differential $D_K$ still makes sense in this context, and so does its formal adjoint $D_K^\star$ (computed by using the metric data $g_D, \omega_D$ and $h$). We establish in Section \ref{HodgeDek} that the 
Hodge decomposition holds, again based on an a-priori inequality for $\Delta_K$. Notice that the usual commutation relations in Kähler geometry are no longer true, which causes a few additional complications.

\item[\rm (ii)] Our form $\xi$ can equally be interpreted as $L$-valued current on $X_0$, a 
point of view which goes back to the work of Noguchi in \cite{Nog95}.
 This turns out to be extremely useful: we show that any current admits a Hodge decomposition (in the same spirit as Kodaira-deRham results in \cite{KdeR50}), and as consequence, that the analogue of the $\ddbar$-lemma holds. For this to be done it is indispensable that $D_K^c:= D'_K- D''_K$ and $D_K^\star$ commute -- which is not the case in (i), because the curvature of $(T_X\langle \Delta_1\rangle, g_D)$ comes into play.  
\end{enumerate}
\smallskip

\noindent With all this at hand, the proof of  \Cref{introcomp} reduces to solving certain differential equations, cf. \Cref{sec:solve}. This is done in two steps: first we obtain solutions in the sense of currents, by using the results mentioned at (ii) above. Then we prove that among the solutions obtained we can find forms with log poles -- here the techniques discussed in (i) are important. 

\noindent It is understood that there are many details \emph{we sweep under the rug} in this short presentation of our article, hoping that the reader will enjoy discovering them in the coming pages.

\medskip

\noindent Anyway, we conclude the introduction by mentioning further applications. The results in (i) and (ii) can be seen as new analytic techniques for non-abelian Hodge theory in the quasi-compact setting. While introduced here by necessity, they are of independent interest and we expected them to have a broad range of applications, e.g. in the study of representation varieties and their Hodge  structures.
	
In fact, in a subsequent work \cite{CDP25} we intend to address the following two fundamental problems, concerning the deformation theory for representations of $\pi_1(X_0)$ into ${\rm GL}_N(\CC)$.    
	\begin{itemize}
		\item \emph{Analytic germs of representation varieties.}  
		
	Let $R_{\rm B}(X_0,{\rm GL}_N)$ denote the affine $\Q$-scheme of finite type representing the functor 
	\[
	A \longmapsto {\rm Hom}(\pi_1(X_0),{\rm GL}_N(A)),
	\] 
	for commutative $\CC$-algebras $A$. A central problem is to describe precisely the analytic germ of $R_{\rm B}(X_0,{\rm GL}_N)(\CC)$ at a semisimple representation $\varrho:\pi_1(X_0)\to {\rm GL}_N(\CC)$.  
	
	This circle of ideas originates in  work of Deligne--Goldman--Millson \cite{GM88}, who showed, in the compact case $D=\varnothing$, that the germ has quadratic singularities. Subsequently, Kapovich--Millson \cite[Theorem~1.13]{KM98} (see also Lef\`evre \cite{Lef19}) proved, under the additional assumption that the representation $\varrho$ has \emph{finite image}, that the analytic germ can be described as a quasi-homogeneous cone with generators of weights $1$ and $2$ and relations of weights $2$, $3$, and $4$, though without providing explicit equations. Their argument relies essentially on Sullivan’s theory of minimal models together with the foundational work of Morgan \cite{Mor78}.  
	
	More recently, Pridham \cite[Proposition 14.28]{Pri16}, building on the $L^2$-methods of Timmerscheidt \cite{Tim87}, obtained a refined description of the germ at $\varrho$ \emph{without} assuming finite image, but under the assumption that $\varrho$ has \emph{unitary} local monodromy around $D$. The analytic techniques developed in the present paper allow us to remove this restrictive unitarity condition and to derive explicit defining equations for the singularity of the germ in the most general quasi-projective setting.
	
	\item \emph{Extension of Eyssidieux--Simpson’s construction of mixed Hodge structures attached to the deformation theory of complex variations of Hodge structures on compact K\"ahler manifolds.}
	
	In the special case $D=\varnothing$, when $\varrho$ underlies a complex variation of Hodge structure, Eyssidieux--Simpson \cite{ES11}, based on the work  of Goldman--Millson, constructed a universal variation of mixed Hodge structure for the \emph{tautological representation}  
	\[
	\hat{\varrho}:\pi_1(X_0)\to {\rm GL}_N(\hat{\O}_T),
	\] 
	where $T$ is a certain formal germ contained in the formal germ ${\rm Spf}(\hat{\O}_\varrho)$ of $\varrho$. This result plays a key role in the proof of the \emph{linear Shafarevich conjecture} by Eyssidieux et al. \cite{EKPR12}.
	
	Building on the analytic methods introduced in the present work, we will extend Eyssidieux--Simpson’s theorem to the quasi-projective setting in our future work \cite{CDP25}.   
	\end{itemize}    

\tableofcontents

\section*{Notations and conventions}
\begin{itemize}
	\item Unless otherwise specified, throughout this paper, \( X \) denotes a compact K\"ahler manifold, \( D =\sum D_i \) is a simple normal crossing divisor on \( X \).

    \item We denote by $\Omega_X(\log D)$ the sheaf of holomorphic forms with logarithmic poles along $D$, and with $\Omega_X\langle D\rangle$
	the corresponding vector bundle. We will use the same type of notations for the duals, exterior/symmetric powers, etc.
    
	\item \( M_{\rm DR}(X/D) \) denotes the moduli space of pairs \( (L, \nabla) \), where \( L \) is a holomorphic line bundle on \( X \), and \( \nabla: L \to L \otimes \Omega_X(\log D) \) is a logarithmic integrable connection.
	
	\item A Hermitian metric on \( L|_{X\setminus D} \) is said to be \emph{harmonic} if it satisfies the conditions described in \Cref{def:harmonic}. Let \( D_K'' \), \( D_K' \), and \( D_K \) be the associated connections on \( L \) as defined in \Cref{def:connection}. Let \( \mathcal{D}_K \) be the connection incorporating the current defined in \eqref{end3}.
	
	\item For any given \( (L, \nabla) \in M_{\rm DR}(X/D) \) with $\rel (\Res{D_i} (L, \nabla)) \in ]-1,0]$ for every component $D_i$, a \emph{harmonic metric} often refers to the specific construction provided in \Cref{specialhar}. In this context, we decompose the divisor \( D = \Delta_1 + \Delta_2 + \Delta_3 \) as described therein.
 \item With this setup, let $\omega_D$ be a K\"ahler metric on $X \setminus \Delta_1$ with suitable asymptotic behavior as defined in \eqref{m1}.  
If we decompose \(\Delta_2 = \sum_{i \in I} D_i\) into prime components, then for every multi-index \(a = (a_i)_{i \in I} \in \mathbb N^I\) we define the volume form \(d\mu_a\) on \(X \setminus \Delta_1\) as in \eqref{eq:mu}.  
	
	\item  In the same setting, we denote by 
	\(\cC^\infty_{i,K}(X,L)\) and \(\cC^\infty_{i,K}(X,\Delta_1,L)\) for the sheaves on \(X\) introduced in \Cref{defend1}.  
	When the degree \(i\) is not relevant, we abbreviate them as \(\cC^\infty_{K}(X,L)\) and \(\cC^\infty_{K}(X,\Delta_1,L)\).  
	We denote their spaces of global sections by \(C^\infty_{i,K}(X,L)\) and \(C^\infty_{i,K}(X,\Delta_1,L)\), respectively.  
	In other words, the calligraphic symbol \(\cC\) always refers to a sheaf, whereas the plain \(C\) refers to its global sections.  
\end{itemize}

\section{Families of flat line bundles} 

\subsection{Preliminaries} Let $X$ be a compact K\"ahler manifold. By elementary Hodge theory, the image of the map
\begin{equation}\label{simp1}
\iota: H^1(X, \Z)\to H^1(X, \O)
\end{equation}
is a lattice. The quotient
\begin{equation}\label{simp2}
\Jac(X):= H^1(X, \O)/\IM(\iota)
\end{equation}
is a complex torus called the Jacobian variety of $X$. Via the exponential sequence, it is isomorphic to the subgroup of 
$H^1(X, \O^\star)$ corresponding to bundles with trivial Chern class. 
\smallskip

\noindent We consider the holomorphic bundle 
\begin{equation}\label{simp3} 
L\to X\times \Jac(X)
\end{equation} 
such that the restriction $\displaystyle L|_{X\times \{\gamma\}}$ is the flat bundle induced by the $(0,1)$-form $\gamma\in \Jac(X)$. Following \cite{YTS} we describe next the transition functions of $L$. As a consequence, we will see that this bundle admits a natural metric whose associated curvature form plays a crucial role in what follows.
\medskip

\noindent Let $\alpha_1,\dots,\alpha_q$ be a basis of holomorphic $(0,1)$-forms on $X$ and let $(U_i)_{i\in I}$ be a family of open subsets of $X$ such that any finite intersection of them if Stein and simply connected.
It follows that
the restriction of each $\alpha_j$ to any finite intersection 
$\displaystyle U_{i_1}\cap\dots\cap U_{i_N}$ is $\dbar$-exact. We then write 
\begin{equation}\label{simp4}
\alpha_j|_{U_p}= \ol{\partial f_{jp}}
\end{equation}
for some holomorphic function $f_{jp}$ defined on $U_p$. For each fixed $j$ the difference function $f_{jp}- f_{js}$ defined on $U_p\cap U_s$ is a constant which we denote by $c_{jps}$. 
The family of constants $\displaystyle \left(\exp({\ol c_{jps}})\right)_{p, s}$ are the transition function for the bundle $\displaystyle L|_{X\times \{\alpha_j\}}$.
\smallskip

\noindent Now an arbitrary element of $\Jac(X)$ is represented by a form $\sum t^j\alpha_j$, and it follows that all in all the transition functions for $L$ are given by 
\begin{equation}\label{simp5}
\Big(\exp \sum_{j=1}^{q_X} t^j\ol c_{jps}\Big)_{p, s}
\end{equation}
on the intersection $U_s\cap U_p$.
\smallskip

\noindent Moreover, it follows that the set of functions 
\begin{equation}\label{simp6}
\Big(\exp2\Re\big(\sum_{j=1}^{q_X} t^j\ol f_{jp}\big)\Big)_p
\end{equation}
define a Hermitian metric $h_L$ on the holomorphic bundle $L$. The $(1,0)$--part of the Chern connection writes as 
\begin{equation}\label{simp8}
D'\sigma|_{U_j}:= \partial \sigma_j+ \sum_{j=1}^{q_X} \ol f_{jp}dt^j\wedge \sigma_j+ \sum_{j=1}^{q_X} \ol t^j df_{jp}\wedge \sigma_j
\end{equation}
and thus the corresponding 
curvature form is written as
\begin{equation}\label{simp7}
\Theta(L, h_L)= -\sum dt^j\wedge \alpha_j- \sum \ol \alpha_j\wedge d\ol t^j.
\end{equation}
In particular, we see that $\sqrt{-1}\Theta(L, h_L)$ doesn't have any positivity properties. 
\bigskip

\subsection{An extension result}\label{ext}

\noindent For each positive number $m$ we consider the set 
\begin{equation}\label{simp9}
Z_m:= \{\alpha\in \Jac(X)/ h^0(X, K_X+ L_{\alpha})\geq m\}
\end{equation}
where for simplicity we denote by $L_{\alpha}$ the restriction $\displaystyle L|_{X\times \{\alpha\}}$.
\smallskip

\noindent We consider a germ of holomorphic disk $\gamma: (\CC, 0)\to \Jac(X)$ such that $\gamma(0)\in Z_m$. 
Let $\alpha_0, \dot \alpha_0$ the forms $\gamma(0)$ and $d\gamma(0)$, respectively, where we identify $\dot \alpha_0$ with its projection to 
$\Jac(X)$.
Then we prove the following statement.
\medskip

\begin{thm}\label{fs} Suppose that $u\in H^0(X, K_X+ L_{\alpha_0})$ extends to order one in the direction $\dot \alpha_0$.	
We consider the segment $l(z):= \alpha_0+ z\dot\alpha_0$ in $\Jac(X)$. Then there exists a positive 
real $\ep_0$ and $U(z)\in H^0(X, K_X+ L_{\alpha_0+z\dot \alpha _0})$ such that $U(0)=u$, where $z\in \CC$ such that $|z|\leq \ep_0$.\end{thm}
\medskip

\begin{proof} 
Let $u$ be a holomorphic section of the bundle $\displaystyle K_X+ L_{\alpha_0}$. We show next that we can extend $u$ in direction $\dot \alpha_0$ by proving that $u$ admits an extension to an infinitesimal neighborhood of the fiber $X:= X\times \{l(0)\}$ inside the family 
\begin{equation}\label{simp11}\cX:= \big(X\times \{l(z)\}\big)_{|z|\leq 1}.
\end{equation} 
If we can do that, the conclusion follows from \cite{ML}. 
\smallskip

\noindent In all cases treated in \cite{JCMP} the extension of a holomorphic section defined on the central fiber $X$ to the first infinitesimal neighborhood of $X$ in $\cX$ is far easier then the general case --since the problem we have to solve concerns the central fiber $X$ with its reduced structure. We will see that in the current context the obstruction to lift $u$ to an infinitesimal neighborhood of arbitrary order is virtually the same as the one corresponding to the first order.

\medskip

\subsubsection{Extension to the infinitesimal neighborhood of $1^{\rm st}$ order} \smallskip
We still use in what follows the notation $(L, h_L)$ for the holomorphic line bundle we have introduced above together with its Hermitian structure restricted to $\cX$. The curvature formula \eqref{simp7} implies that the equality 
\begin{equation}\label{simp12}
\Theta(L, h_L)= -dz\wedge \dot\alpha_0- \ol{\dot\alpha_0}\wedge d\ol z
\end{equation}
holds. 
\medskip

\noindent Let $\displaystyle \Xi:= \frac{\partial}{\partial z}$ be the holomorphic vector field giving the derivative in the base direction on $\cX$. We introduce the following operator
\begin{equation}\label{simp13}
\Lie_\Xi:= D'\circ \iota_\Xi
\end{equation}
acting on $L$-valued $(n+1, q)$ forms on the total space $\cX$, that is to say we first contract our form with the vector field $\Xi$ and then we apply the covariant derivative. 
\smallskip

\noindent The commutation formula (cf. \cite{JCMP} for the general case) shows that we have 
\begin{equation}\label{simp14}
\Lie_\Xi\circ \dbar= \dbar\circ\Lie_\Xi- \dot\alpha_0\wedge. 
\end{equation}
Note that in our case the equality \eqref{simp14} is greatly simplified by the fact that $\Xi$ is holomorphic, together with the shape of the curvature form in \eqref{simp12}. 
\smallskip

\noindent Let $U_0$ be a smooth extension of $u$ to the total space $\cX$; it is a $(n+1, 0)$ form
such that 
\begin{equation}\label{simp15}
\dbar U_0= z\Lambda_0
\end{equation}
for some $L$-valued $(n+1, 1)$ form $\Lambda_0$ on $\cX$. We apply the Lie derivative $\Lie_\Xi$ to 
\eqref{simp15} and restrict the result to the central fiber $X$. The RHS of \eqref{simp15} becomes simply $\displaystyle \lambda_0:= \frac{\Lambda_0}{dz}\Big|_X$; the formula \eqref{simp14} implies that we have 
\begin{equation}\label{simp16}
\dbar v- \dot\alpha_0\wedge u=  \lambda_0
\end{equation}
where $v$ is the restriction of $\Lie_\Xi(U_0)$ to $X$. 
\smallskip

\noindent We see that the obstruction form 
$\lambda_0$ is $\dbar$-exact if and only if this is the case $\dot\alpha_0\wedge u$. Recall that $\dot\alpha_0\wedge u$ is $\dbar$-exact by hypothesis, and therefore $u$ admits an extension across the first infinitesimal neighborhood of $X$ in 
$\cX$.



\medskip

\subsubsection{Extension to the infinitesimal neighborhood of $2^{\rm nd}$ order} We argue in a similar way, by applying the Lie derivative twice to the equality
\begin{equation}\label{simp19}
\dbar U_1= z^2\Lambda_1
\end{equation}
obtained at the previous step. After applying $\Lie_\Xi$ a first time the LHS is equal to
\begin{equation}\label{simp20}
\dbar \big(\Lie_\Xi(U_1)\big)- \dot \alpha_0\wedge U_1
\end{equation}
and a new derivation of \eqref{simp20} shows that the new obstruction form 
$\displaystyle\lambda_1:= \frac{\Lambda_1}{dz}\Big|_X$ is $\dbar$-exact precisely when 
\begin{equation}\label{simp21}
\dot \alpha_0\wedge \frac{\Lie_\Xi(U_1)}{dz}\Big|_X
\end{equation}
is $\dbar$-exact on $X$. In order to show this we argue as follows.\medskip

\noindent As consequence of Hodge decomposition on $X$ we have 
\begin{equation}\label{simp22}
\frac{\Lie_\Xi(U_1)}{dz}\Big|_X= v+ D'w
\end{equation}
where $v$ is holomorphic $\displaystyle L_{\alpha_0}$-valued $(n,0)$ form and $w$ is
a $(n-1, 0)$ form with values in the same bundle.
\smallskip

\noindent We already know that  
\begin{equation}\label{simp23}
\dot \alpha_0\wedge v\in \IM(\dbar)
\end{equation}
by hypothesis. 

\noindent For the other term, we note that the form 
\begin{equation}\label{simp25}
\dot \alpha_0\wedge D'w
\end{equation}
is $\dbar$-closed, as a consequence of the equality \eqref{simp22}. On the other hand, we have the equality 
\begin{equation}\label{simp24}
\dot \alpha_0\wedge D'w= -D'\big(\dot \alpha_0\wedge w\big)
\end{equation}
so the form \eqref{simp25} is $\dbar$-closed and $D'$-exact. Hodge theory shows that the form
$\displaystyle \dot \alpha_0\wedge D'w$ is $\dbar$-exact, so we are done with the second order.

\medskip

\begin{remark}\label{simp26} The arguments before show that if $\beta$ is any $\displaystyle L_{\alpha_0}$-valued $(n,0)$ form on the central fiber such that $\dbar (\dot\alpha_0\wedge \beta)= 0$, then 
$\dot\alpha_0\wedge \beta$ is automatically $\dbar$-exact.
\end{remark}

\medskip

\subsubsection{Higher order extensions} We argue by induction in order to extend $u$ to an arbitrary order.
Assume that we have 
\begin{equation}\label{simp27}
\dbar U_k= z^{k+1}\Lambda_k
\end{equation}
on $\cX$. We apply the Lie derivative to \eqref{simp27} and infer that we have
\begin{equation}\label{simp28}
\dbar \Lie_\Xi (U_k)- \dot \alpha_0\wedge U_k= (k+1)z^{k}\Lambda_k+ \O(z^{k+1}).
\end{equation}
Another derivative gives
\begin{equation}\label{simp29}
\dbar \Lie_\Xi^2 (U_k)- 2\dot \alpha_0\wedge \Lie_\Xi(U_k)= (k+1)kz^{k-1}\Lambda_k+ \O(z^{k}),
\end{equation}
so it is clear that after $k+1$ such operations we obtain
\begin{equation}\label{simp30}
\dbar \Lie_\Xi^{k+1} (U_k)- (k+1)\dot \alpha_0\wedge \Lie_\Xi^k(U_k)= (k+1)!\Lambda_k+ \O(z).
\end{equation}
\smallskip

\noindent By Remark \ref{simp26} the $\dbar$-closed form
\begin{equation}\label{simp31}
\dot \alpha_0\wedge \frac{\Lie_\Xi^k(U_k)}{dz}\Big|_X
\end{equation}
is $\dbar$-exact, and this is the end of the proof.
\end{proof}

\section{Higher rank version}
The local structure of moduli of hermitian flat bundles has been obtained by Nadel in \cite{Nad88}. The main statement for vector bundles is due to Wang \cite{Wang}: cohomology jump loci of hermitian flat vector bundles are defined by linear equations. In this section we shall give a direct proof for that. Actually it is sufficient to show that the following hold.
\medskip

\begin{thm}\label{HR}
Assume that $E$ is a hermitian flat vector bundle on a compact K\"ahler manifold $(X,\omega)$.	Assume that $\xi\in H^{0,1}(X, End(E))$ is a harmonic $(0,1)$-form  with value in $End(E)$ so that $\xi\wedge\xi=0$. Then $\dbar_E+z\xi$ defines the deformation of flat bundles of  $E$, which we denote by $\mathcal{E}\to X\times \mathbb{C}$. Assume that the  section $[u]\in H^p(X,K_X\otimes E)$ extends to the first order   of $\mathcal{E}$ at $X\times\{0\}$. Then it extends to arbitrary order.
\end{thm} 
\medskip

\begin{proof}
 We argue via a formal Taylor expansion: our aim is to solve this equation
 $$
 (\dbar_E+z\xi)U=0
 $$
 for 
$$
U=\sum_{i=0}^{\infty}z^iu_i$$
with $u_0=u$ and $ u_i=A^{n,p}(X,E)$ is the smooth $(n,p)$-form with value  in $L$. To this end, we show that the following equations
\begin{equation}\label{simp46}
\begin{cases}
\dbar_E(u_1)+\xi\wedge u_0=0\\
\dbar_E(u_2)+\xi\wedge u_1=0\\
\ldots\\
\dbar_E(u_{i+1})+\xi\wedge  u_i=0\\
\ldots
\end{cases} 
\end{equation}
admit a solution.
The condition that $u$ extends to the first order means that existence of $u_1\in A^{n,p}(E)$ so that the first $\dbar_E$-equation is solved. One can take $u_1=D_E'v_1$ with $v_1\in A^{n-1,p}(X,E)$ by the Hodge decomposition
$$
A^{n,p}(E)=\mathcal{H}^{n,p}(E)\oplus im(D_E').
$$
 Since $\xi\in \mathcal{H}^{0,1}(X,End(E))$, one has
$$
D_{End(E)}'(\xi)=\db_{End(E)}(\xi)=0.
$$ 
For $\xi \wedge u_1\in A^{n,p+1}(X,E)$,  one has 
$$
D'_E(\xi \wedge u_1)=D_{End(E)}'(\xi)\wedge u_1-\xi\wedge D_E' u_1=0 $$
and
$$
\db_E(\xi\wedge u_1)=\db_{End(E)}(\xi)\wedge u_1-\xi\wedge \xi \wedge u_0=0
$$
by $\xi\wedge \xi =0$. Moreover 
$$D_E' (\xi\wedge v_1)=D_{End(E)}'(\xi)\wedge v_1-\xi\wedge D_E' v_1=-\xi\wedge u_1.$$
By the $\partial\dbar$-lemma one concludes that $\xi\wedge u_1=-\dbar_E D'_E v_2$ for some $v_2\in A^{n-1,p}(E)$.  Write $u_2:=D'_E v_2$, which is our desired second order extension. The  argument for higher order extensions is exactly the same.
	\end{proof} 
	
\begin{remark}
Notice that in order to establish the result in \cite{Wang}, it is sufficient to prove the next statement (in which the hypothesis is stronger than the one in Theorem \ref{HR}).	
\end{remark}	

\begin{lemme}
	Let $\mathcal{E}\to X\times \mathbb{D}$ be a family of hermitian flat vector bundles with $E_z:=\mathcal{E}|_{X\times \{z\}}$. Write $E:=E_0$.  If for some $[u]\in H^p(X, K_X\otimes E)$, there are smooth sections $U(x,z)\in A^{n,p}(X, K_X\otimes E)$ parametrized by $z\in \mathbb{D}$ so that $U_z(x):=U(x,z)$ in $H^p(X, K_X\otimes E_z)$ with $U_0(x)=u(x)$. Then $u$ extends to first order.  
\end{lemme}
\begin{proof}
	We do a Taylor expansion $U(x,z)=\sum_{i=0}^{\infty}z^iu_i$ with $u_i\in A^{n,p}(X, K_X\otimes E)$.  We write
	$$
	\db_{E_z}=\db_E+ t \xi + O(t^2)
	$$ 
	By $	\db_{E_z} U_z(x)=0$ one has 
	$$
	\db_E u_1+\xi\wedge u_0=0.
	$$
	This implies the first order extension. 
	\end{proof}

\medskip

\section{Adjoint version}

\noindent In this section our target is the following result.
\medskip 

	\begin{thm}\label{twistedversion1}
		Let $X$ be a compact K\"ahler manifold, and let $M$ be a line bundle on $X$ so that $c_1(M)=[\sum_{i=0}^{\ell}a_iD_i]$ with $0<a_i<1$ and $D:= \sum_{i=0}^{\ell}a_iD_i$ is SNC.  We consider a germ of holomorphic disk $\gamma: (\CC, 0)\to \Jac(X)$ such that $\gamma(0)= \O_X$ and such that the any element of $H^0(X, K_X+M+ L_0)$ extends to order one in the direction $\dot\alpha_0\in H^{0,1}(X)$ (i.e. the infinitesimal deformation  $d\gamma(0)$).
		Then the function
		\begin{equation} 
			h^0(X, K_X+M+ L_{z\dot\alpha_0})
		\end{equation}
		is constant for $0\leq |z|\ll 1$. 
	\end{thm}
\begin{proof}
	It suffices to prove that any section $u\in H^0(X, K_X+M)$ can be extended \emph{smoothly} to sections in 	$H^0(X, K_X+M+ L_{z\dot\alpha_0})$ for $|z|$ small enough.
	For $L_{\gamma(\tau)}$, there is $\xi_i\in A^{0,1}(X)$ so that
	$$
	\db+\sum_{i=1}^{\infty}\tau^i\xi_i
	$$	
	defines the complex structure for $L_{\gamma(\tau)}$.  Since 
		$(\db+\sum_{i=1}^{\infty}\tau^i\xi_i)^2=0$, this means that 
		$$
		\db\xi_1=0
		$$
		which is the infinitesimal deformation of $L_{\gamma(\tau)}$ at $0$, namely $[\xi_1]=\dot\alpha_0$. By the assumption that 	$h^0(X, K_X+M+ L_{\gamma(\tau)})$ is constant, there is  $\tilde{U}_{\tau}\in H^0(X, K_X+M+ L_{\gamma(\tau)})$ which depends on $\tau$ smoothly  so that $\tilde{U}_{0}(x)=u$. Hence one can write
$$
\tilde{U}_\tau(x)=\sum_{i=0}^{\infty}\tau^i\tilde{u}_i$$
with $\tilde{u}_0=u$ and $ \tilde{u}_i=A^{n,0}(X,M)$ is the smooth $(n,0)$-form with value  in $M$, such that
$$
(\db_M+\sum_{i=1}^{\infty}\tau^i\xi_i)(\sum_{i=0}^{\infty}\tau^i\tilde{u}_i)=0.
$$
This implies that 
\begin{eqnarray}\label{1}
\db_M \tilde{u}_1+\xi_1 \wedge u=0
\end{eqnarray} 
We can choose $[\xi_1]\in H^{0,1}(X)$  the harmonic representative with respect to some K\"ahler metric $\omega_X$ on $X$.  In particular, 
\begin{eqnarray}\label{2}
\db \xi_1=\partial \xi_1=0
\end{eqnarray}

We next use a few results in $L^2$-cohomology. Consider the manifold $X^\circ:=X-D$ and equip it with the Poincar\'e type metric $\omega$. The restriction $M|_{X^\circ}$ is endowed with the smooth metric $h_M$ induced by $\sum_{i=0}^{\ell}a_iD_i$ so that its curvature is equal to zero. The 
equation to be solved is as follows
$$
(\db_M+\sum_{i=1}^{\infty}z^i\xi_1)(\sum_{i=0}^{\infty}z^iu_i)=0.
$$
		with $u_0=u|_{X^\circ} $ and $u_i\in L^{n,0}_{(2)}(X^\circ, M)$. By identifying the coefficients of $z^i$, this is equivalent with the following system,
		\begin{equation}\label{equation}
			\begin{cases}
				\dbar_M(u_1)+\xi_1\wedge u_0=0\\
				\dbar_M(u_2)+\xi_1\wedge u_1=0\\
				\ldots\\
				\dbar_M(u_{i+1})+\xi_1\wedge  u_i=0\\
				\ldots
			\end{cases} 
		\end{equation}
modulo convergence issues.

	 By the Hodge decomposition for $L^2$-spaces, one has
		$$
		L^{n,p}_{(2)}(X^\circ, M)=	\mathcal{H}^{n,p}(X^\circ, M)\oplus im D_M'
		$$ 
		for $p\geq 0$, 
	where $D_M'$ is the $(1,0)$-connection of $(M,h_M)$ over $X^\circ$.  Note that $\xi|_{X^\circ}$ is $L^2$-integrable with respect to the complete metric $\omega$ on $X^\circ$, and 
	$$
	\tilde{u}_1\in A^{n,0}(X,M)\subset 	L^{n,0}_{(2)}(X^\circ, M).
	$$
	Hence $\xi_1\wedge u \in L^{n,1}_{(2)}(X^\circ, M)$. One can take 
	$$
	u_1=D_M' v_1
	$$
	for some $v_1\in   L^{n-1,0}_{(2)}(X^\circ, M)$ with $u_1-\tilde{u}_1\in \mathcal{H}^{n,0}(X^\circ, M)$. Therefore,
by \eqref{1} $$
\db_M  {u}_1+\xi_1 \wedge u=0
$$
By the $L^2$-norm estimate in \cite[Theorem 3.3]{CP}, one has
$$
\lVert {u}_1\rVert\leq C \lVert \xi_1 \wedge u\rVert\leq C \lVert \xi_1\rVert\cdot  \lVert  u\rVert
$$
where $C$ is some universal constant. 
We can thus make a rescaling of $\xi_1$ a priori so that $C \lVert \xi_1\rVert\leq \frac{1}{2}$. 

For $\xi_1\wedge u_1\in  L^{n,1}_{(2)}(X^\circ, M)$,  by \eqref{2} one has
$$
\db_M(\xi_1\wedge u_1)=-\xi_1\wedge \xi_1 \wedge u_1=0
$$
$$
-D_M'(\xi_1\wedge v_1)=-\partial\xi_1\wedge  v_1+\xi_1 \wedge D_M'v_1=\xi_1\wedge u_1.
$$
Hence
$$
D_M'(\xi_1\wedge u_1)=0. 
$$
By the $L^2$ version of the $\partial\db$-lemma, there is 
$$
v_2\in  L^{n-1,0}_{(2)}(X^\circ, M)
$$
such that $D_M' v_2\in  L^{n,0}_{(2)}(X^\circ, M)$ and 
$$
\xi_1\wedge u_1=-\db_MD_M' v_2
$$
By \cite[Theorem 3.3]{CP} again, one has
$$
\lVert D_M' v_2\rVert \leq C \lVert \xi_1 \wedge u_1\rVert\leq \frac{1}{2}  \lVert  u_1\rVert\leq \frac{1}{4}\lVert u\rVert
$$
Write $u_2:=D_M' v_2$, which is our desired second order extension. 
We do this inductively to find $u_i\in L^{n,0}(X^\circ, M)$ so that
$$
\lVert u_i\rVert \leq   \frac{1}{2^i}\lVert u\rVert
$$
and the equations \eqref{equation} solves. Then the sum
$$
U_z:=\sum_{i=0}^{\infty}z^iu_i
 $$
 belongs to $L^{n,0}_{(2)}(X^\circ, M)$ for any $|z|\leq 1$ so that
 $$
 (\db_M+z\xi_1)(U_z)=0
 $$
 Note that $\db_M+z\xi_1$ defines the complex structure for the line bundle $M+L_{z\dot\alpha_0}$.  Hence
 $$
 U_z\in H^0(X^\circ, K_X+M+L_{z\dot\alpha_0}|_{X^\circ}) 
 $$
 which is $L^2$-integrable with respect to the Poincare metric and $h_M$. It thus extends to a holomorphic section  in $H^0(X, K_X+M+L_{z\dot\alpha_0}) $. 
	\end{proof}
	
	
	\section{Pluricanonical version}
	
	\noindent Let $L \to X \times \Jac(X)$ be the universal holomorphic line bundle from \eqref{simp3}. 
In this section, we consider the pluricanonical case. 
Our main object of study is the set
\begin{equation}\label{jump1}
	Z_{k,m} := \{ \alpha \in \Jac(X) \mid h^0(X, mK_X + L_\alpha) \geq k \},
\end{equation}
where $k \geq 1$ is an integer. 
We show that $Z_{k,m}$ has the same flatness properties as the analogous sets in the preceding sections. 
That is, we have the following:
\smallskip

\begin{thm}\label{jumppluri}
Let $X$ be a compact Kähler manifold. Then, for every $k, m \in \mathbb{N}$, the jumping locus
$Z_{k,m}$ defined in \eqref{jump1} 
is a finite union of translates of subtori in $\Jac(X)$.
\end{thm}

	\medskip
    
    We provide in what follows two proofs of \Cref{jumppluri}. The first approach relies on a decomposition theorem for direct images proved in \cite{HPS18}.
\smallskip

\begin{thm}\label{t-42}
    Let $X$ be a smooth projective variety, $s\in H^0(mK_X)$ and $v\in H^1(\mathcal O _X)$.
    If $s$ deforms to 1-st order in the direction of $v$, then it deforms to arbitrary order. In particular, the jumping locus $Z_{k, m}$ is a union of subtori in $\Jac (X)$.
\end{thm}
\begin{proof}
  Suppose first that $\kappa (X)=0$. Let $a:X\to A$ be the Albanese morphism, and identify $H^1(\mathcal O _X)=H^1(\mathcal O _A)$. In this case $V^0(mK_X):=Z_m\ne  \emptyset$ consists of isolated points for some $m>0$ and therefore $a_* (\omega _X^m)$ is a direct sum of torsion translates of unipotent vector bundles $a_* (\omega _X^m)=\oplus  U_i\otimes P_i$ (see \cite[Proposition 7.5]{HPS18}). Since $\kappa (X)=0$ and the $P_i$ are torsion, then $a_* (\omega _X^m)=  U\otimes P$ is indecomposable (see \cite[Corollary 4.3]{HPS18}). 
  By \cite{CP17} (see also \cite[Corollary 27.2]{HPS18}), any surjective homomorphism $U\to  P$ where $P\in {\rm Pic}^0(A)$ is split and hence $U\cong P$. The claim is now immediate because if $s\in H^0(mK_X\otimes P^\vee)$, then $a_*s=1\in H^0(\mathcal O_A)=H^0(a_*(mK_X\otimes P^\vee))$. But then $a_*(s\wedge v)=v\in H^1(\mathcal O _A)$ does not vanish.

  In the general case, consider the Iitaka fibration $f:X\to Y$ with general fiber $F$ of Kodaira dimension 0. We denote by $\phi:A_X\to A_Y$ the corresponding surjective map between the Albanese varieties of $X$ and $Y$. By \cite[Lemma 11.1]{HPS18}, and the kernel of ${\rm Pic}^0(X)\to {\rm Pic}^0(F)$ is a finite union of torsion translates of
  ${\rm Pic}^0(Y)$. 
  Suppose that $0\ne s\in H^0(mK_X\otimes P^\vee)$ and $s\wedge v=0$ for some $v\in H^1(\mathcal O _A)$, then $s|_F\wedge v|_F=0$. Since $s|_F\in H^0(mK_F\otimes P|_F^\vee)$, it follows that $v|_F=0$ and hence that $v\in H^1(\mathcal O _Y)$.
  However, by \cite[Theorem 11.2]{HPS18}, we know that $h^0(mK_X\otimes P^\vee\otimes Q)$ is constant for $Q\in {\rm Pic}^0(Y)$ and hence $s$ deforms to arbitrary order in the direction of $v$.
\end{proof}
\begin{remark}
    We remark that Theorem \ref{t-42} actually strengthens Theorem \ref{jumppluri} in 2 important ways. First of all, it's proof shows that all the components of $Z_m$ are torsion translates of subtori of ${\rm Jac}(X)$ and using similar arguments one can show that the same holds for $Z_{k,m}$. Secondly, the proof implies that if $S$ is the symmetric algebra generated by $H^1(\mathcal O _X)^\vee $ then locally around any point $P\in {\rm Jac}(X)$ the deformations of $H^0(mK_X+P)$ are formally determined by the derivative complex \[0\to S\otimes _{\mathbb C}H^0(mK_X+P)\to S\otimes _{\mathbb C}H^1(mK_X+P)\to \ldots \to S\otimes _{\mathbb C}H^n(mK_X)\to 0.\]
    Note that this is not a consequence of the second proof we give below (see Lemma \ref{L2}), because in that proof we make use of the metric $h_{m,L}$ and it is a priori not clear that if $s\in H^0(mK_X)$ deforms to first order then so does the corresponding section in $s'\in H^0(mK_X\otimes \mathcal I(h_{m,L}))$. For example, if $Z_m=\{\mathcal O _X\}$ and $s=\sigma ^m$ where $\sigma \in H^0(K_X)$, then $H^0(mK_X\otimes \mathcal I(h_{m,L}))=H^0(K_X)$ so that $\sigma '=\sigma$. In this case it is conceivable that $[\sigma ^m\wedge v]=0$ whilst $[\sigma \wedge v]\ne 0$. Theorem \ref{t-42} excludes these cases. 
\end{remark}
    
We now provide another proof of \Cref{jumppluri}, using the idea of the extension argument developed in the previous sections. 
To begin with, we write
\[
mK_X + L = K_X + (m-1)K_X + L,
\]
and our first task will be to construct a “natural” metric on the bundle
\begin{equation}
	L_m := (m-1)K_X + L.
\end{equation}
We then proceed as in the case $m=1$, that is, by showing that the system of equations of type \eqref{simp46} admits a solution.

	\medskip
	
	\noindent Let $\mathfrak I\subset \O_X$ be the ideal generated by the set of global sections of the 
	bundles $m K_X + L_\tau$, for all $\tau\in \Jac(X)$. In other words, the fiber of $\mathfrak I$ at an arbitrary point $x\in X$ is given by
	\begin{equation}
		\mathfrak I_x:= \{f\in \O_{X, x}/ \exists \tau\in \Jac(X), \exists s\in H^0(X, m K_X + L_\tau), s= f \hbox{ near } x\}.
	\end{equation} 
	By the Noetherian property of $\O_X$ there exists a finite number of sections 
	\begin{equation}
		s_i\in H^0(X, m K_X + L_{\alpha_i})
	\end{equation}
	generating $\mathfrak I$ locally near every point of $X$, where $i=1,\dots, N$. 
	\smallskip
	
	\noindent Now let $h_i= e^{-\varphi_i}$ be the restriction of the metric $h_L$ of the bundle $L$ to the fiber $X\times \{\alpha_i\}$. The local weights 
	\begin{equation}
		\psi_i:= \frac{m-1}{m}\left(\log|f_i|^2-\varphi_i\right)
	\end{equation}
	define a metric for the bundle $(m-1)K_X$, for every $i=1,\dots, N$. We introduce the metric 
	$h_{m, L}$ on the bundle $(m-1)K_X+ L$ via the local weights
	\begin{equation}\label{simp47}
		\varphi_{m, L}:= \log\big(\sum_i e^{\psi_i}\big)+ \varphi_L
	\end{equation} 
	and we summarise its properties in the following statement.
	\medskip
	\begin{lemme}\label{metric} { The following hold.
			\begin{enumerate}
				\smallskip
				
				\item[\rm (1)] The metric $h_{m, L}$ has logarithmic singularities and moreover the singular part of the weights in \eqref{simp47} is independent of the variable $t\in \Jac(X)$ (cf. notations in section 1).
				\smallskip
				
				\item[\rm (2)] Let $\omega$ be a fixed metric on $X$. For any global holomorphic section $\sigma$ of $m K_X + L_\tau$ there exists a positive constant $C> 0$ so that we have 
				\begin{equation}\nonumber
					\sup_X\frac{|\sigma|^2_{h_{m,L}}}{dV_\omega}\leq C.
				\end{equation} 
				\smallskip
				
				\item[\rm (3)] The curvature of the bundle $\displaystyle \big(m K_X + L, h_{m, L}\big)$   
				is semi-positive when restricted to the fiber ${X\times \{\tau\}}$ for any $\tau\in \Jac(X)$.
			\end{enumerate}
		}
	\end{lemme}
	
	\noindent All these properties (1)-(3) are simple consequence of the definitions and we will not provide any explanation.
	\medskip

	\noindent We need the following lemma.

\begin{lemme}\label{L2}
In the setting of \Cref{jumppluri}, consider a germ of a holomorphic disk 
$\gamma: (\mathbb{D}, 0) \to Z_{k,m}$ such that 
\begin{equation}\label{simp34}
	t \longmapsto h^0\big(X, mK_X + L_{\gamma(t)}\big)
\end{equation}
is constant for all $t \in \mathbb{C}$ with $|t|$ sufficiently small. 
We denote by $\alpha_0$ and $\dot{\alpha}_0$ the elements $\gamma(0)$ and $\gamma'(0)$, respectively. 
Let $u \in H^0(X, mK_X + L_{\alpha_0})$. Then the equation
\begin{equation}\label{simp36}
	\dbar w = \dot{\alpha}_0 \wedge u
\end{equation}
admits a solution $w$ that is $L^2$ with respect to the restriction of the 
metric $\displaystyle h_{m,L}$ to the fiber $X \times \{\alpha_0\}$.
\end{lemme}

	\begin{proof}
    We consider the familly $X\times \mathbb D$ together with the line bundle $L \to X\times \mathbb D$ induced by $\gamma$. 
		By the semi-continuity theorem of Grauert there exists a holomorphic section $U_\gamma$ of $\displaystyle m K_{X\times \mathbb D}+ L$ extending $u$. In particular, $U_\gamma$ varies smoothly with respect to the base variable 
		$t$ and it follows that we have 
		\begin{equation}\label{simp35}
			\int_{X\times \mathbb D}|U_\gamma|^2e^{-\varphi_{m, L}}< \infty
		\end{equation}
		by the property $(2)$ of Lemma \ref{metric}

		%
		
		\noindent Let $v$ be the holomorphic function representing $U_\gamma$ on some open subset 
		$\Omega\times \DD$. 
		Cauchy inequalities show that we have
		\begin{equation}\label{simp43}
			|\partial_tv(x, 0)|^2e^{-\frac{1}{2}\varphi_{m, L}(z)}\leq C\int_{\DD}|v(x, t)|^2e^{-\varphi_{m, L}}d\lambda(t)
		\end{equation}
		for some constant $C$, and by integrating this equality with respect to $z$ it follows that
		\begin{equation}\label{simp44}
			\int_{\Omega}|\partial_tv(x, 0)|^2e^{-\varphi_{m, L}}d\lambda(x)< \infty
		\end{equation}
		since the RHS of \eqref{simp43} is bounded by \eqref{simp35}.

		\medskip
		
		
		
		\noindent In order to obtain the Lie derivative operator $\Lie_\Xi$ we are using the Chern connection 
		$D'$ associated to the metric $h_{m, L}$ on the line bundle $(m-1)K_{\cX_\gamma}+ L$, whose local expression is given by the following
		\begin{equation}\label{simp42}
			\partial - \sum_j \frac{e^{\psi_j}}{\sum_i e^{\psi_i}}\partial \psi_j- \partial\varphi_L. 
		\end{equation}
		Then we have the equality
		\begin{equation}\label{simp41}
			\frac{\Lie_\Xi(U_\gamma)}{dt}= \partial_tv(x, 0)- v(x,0)\sum_j \frac{e^{\psi_j}}{\sum_i e^{\psi_i}}\partial_t \psi_j -v(x, 0)\partial_t\varphi_L
		\end{equation}
		locally on $\Omega$. 
		
		\noindent By equality \eqref{simp47} we see that $\displaystyle \partial_t \psi_j= -\frac{m-1}{m}\partial_t \varphi_j$ so we finally infer that we have
		\begin{equation}\label{simp37}
			\int_X\left|\frac{\Lie_\Xi(U_\gamma)}{dt}\right|^2e^{-\varphi_{m, L}}< \infty,
		\end{equation}
		which shows that Lemma \ref{L2} holds true. 
	\end{proof}
	\medskip
	
\noindent After these preparations, we next present our second proof of \Cref{jumppluri}.
\smallskip

\begin{proof}[Proof of \Cref{jumppluri}]
Let $\alpha_0 \in Z_{k,m}$ be a generic point. 
Consider a germ of a holomorphic disk $\gamma: (\mathbb{D}, 0) \to Z_{k,m}$ such that $\gamma(0) = \alpha_0$.
We denote by $\dot{\alpha}_0$ the derivative $\gamma'(0)$. 
Let $u \in H^0(X, mK_X + L_{\alpha_0})$.

\noindent We aim to solve the following system of equations:
\begin{equation}\label{addseqeq}
	\begin{cases}
		\dbar u_1 = \dot{\alpha}_0 \wedge u, \\
		\dbar u_2 = \dot{\alpha}_0 \wedge u_1, \\
		\vdots \\
		\dbar u_{i+1} = \dot{\alpha}_0 \wedge u_i, \\
		\vdots
	\end{cases}
\end{equation}
such that each $u_i$ is $L^2$ with respect to $h_{m,L}$ for every $i \in \mathbb{N}$.

Since $\alpha_0$ is generic in $Z_{k,m}$, we can apply Lemma~\ref{L2} to $\gamma$. 
Hence, the first equation of \eqref{addseqeq} admits an $L^2$ solution with respect to $h_{m,L}$.
Let $u_1$ be the solution of
\begin{equation}\label{simp51}
	\dbar u_1 = \dot{\alpha}_0 \wedge u,
\end{equation}
such that the $L^2$-norm of $u_1$ is minimal. 
Then $u_1$ is orthogonal to the harmonic part $\ker \Delta''$. 
By \cite[Thm.~2.2]{CP}, $u_1$ is $D'_{h_{m,L}}$-exact, and there exists a $(n-1,0)$-form $v$ with values in $(m-1)K_X + L_{\alpha_0}$ such that
\begin{equation}\label{simp52}
	u_1 = D'_{h_{m,L}} v, 
	\qquad 
	\int_X |v|_{\omega_0}^2 e^{-\varphi_{m,L}}\, dV_{\omega_0} < \infty.
\end{equation}
\smallskip

\noindent 
Now we solve the second equation. 
We claim that
\[
	\dot{\alpha}_0 \wedge u_1 \in \operatorname{Im}(\dbar).
\]
Indeed, the form $\dot{\alpha}_0 \wedge u_1$ is $L^2$ and $\dbar$-closed. 
It suffices to show that
\begin{equation}\label{simp53} 
	\int_X \langle \dot{\alpha}_0 \wedge D'_{h_{m,L}} v, \beta \rangle_{\omega_0} 
	e^{-\varphi_{m,L}}\, dV_{\omega_0} = 0
\end{equation}
for any $\Delta''$-harmonic $(n,1)$-form $\beta$ that is $L^2$. 
This is clear since $\beta$ automatically belongs to the domain of $(D'_{h_{m,L}})^\star$, and by the Bochner formula we have
\begin{equation}\label{simp53b} 
	(D'_{h_{m,L}})^\star \beta = 0,
\end{equation}
by the curvature property~(3) of Lemma~\ref{metric}. 

\noindent The form $\dot{\alpha}_0 \wedge v$ is $L^2$ by \eqref{simp52}, so we can write
\begin{equation}\label{simp54} 
	\int_X \langle \dot{\alpha}_0 \wedge D'_{h_{m,L}} v, \beta \rangle_{\omega_0} 
	e^{-\varphi_{m,L}}\, dV_{\omega_0}
	= - \int_X \langle \dot{\alpha}_0 \wedge v, (D'_{h_{m,L}})^\star \beta \rangle_{\omega_0} 
	e^{-\varphi_{m,L}}\, dV_{\omega_0} = 0,
\end{equation}
where the last equality follows from \eqref{simp53b}. 
As a consequence, $\dot{\alpha}_0 \wedge u_1 \in \operatorname{Im}(\dbar)$, and we can find $u_2$ such that $u_2$ is $L^2$ with respect to $h_{m,L}$ and
\[
	\dbar u_2 = \dot{\alpha}_0 \wedge u_1.
\]
By repeating the same argument, we can solve the system \eqref{addseqeq} to any order.

Finally, using \cite{ML} or the convergence argument in the proof of \Cref{twistedversion1}, we know that the segment $\alpha_0 + z \cdot \dot {\alpha}_0$ lies in $Z_{k,m}$. 
This implies that $Z_{k,m}$ is a finite union of translates of subtori in $\Jac(X)$.
\end{proof}

 \begin{remark}
 Note that we have the twisted strong Hodge decomposition from \cite[Thm.~2.2]{CP}. 
Together with the above argument, we can easily obtain the following twisted version of \Cref{jumppluri}.

Let $X$ be a compact Kähler manifold, and let $F$ be a pseudo-effective line bundle on $X$. 
Let $h_F$ be a possibly singular Hermitian metric on $F$ such that $i\Theta_{h_F}(F) \geq 0$. 
Assume that $m \in \mathbb{N}$ satisfies $\mathcal{I}_X(h_F^{\frac{2}{m}}) = \mathcal{O}_X$. 
Then, for every $k\in \mathbb{N}$, the jumping locus
\begin{equation}
	Z_{k,m} := \{ \alpha \in \Jac(X) \mid h^0(X, mK_X + F + L_\alpha) \geq k \}
\end{equation}
is a union of subtori in $\Jac(X)$.

In particular, let $F = F_1 + \sum a_i E_i$, where $F_1$ is a semipositive $\mathbb{Q}$-line bundle and 
$\sum \frac{a_i}{k} [E_i]$ is a klt $\mathbb{Q}$-effective divisor. 
Then $Z_{k,m}$
is a finite union of translates of subtori in $\Jac(X)$.

\end{remark}%



\section{Further results involving infinitesimal extension of sections}

\smallskip

\noindent In this section, we consider the following data.
\begin{itemize}
	
	\item $p:\cX\to \DD$ is a non-singular, proper family of Calabi-Yau Kähler manifolds, in the sense that $\displaystyle K_{X}\equiv 0$ is trivial, where $X$ is the central fiber.
	
	\item $L\to \cX$ is a holomorphic line bundle on $\cX$, and $\displaystyle L_0:= L|_{X}$ be the restriction of $L$ to the central fiber $X$ of the family $p$.
	
	\item There exists some $m\in\mathbb N^\star$ and $s_0,\dots, s_N \in H^0 (X, m L|_X)$ a basis such that $\{s_i= 0\}_{i=0}^N$ has no common zeroes. 
\end{itemize}
\medskip

\noindent An important problem in this context is the following.

\begin{conjecture}\label{qCH}
	The function $\displaystyle t\to h^0(\cX_t, L|_{\cX_t})$ is constant.
\end{conjecture}
\medskip

Conjecture \ref{qCH} holds when $X$ is hyperkähler by a result of Matsushita, \cite{Mat16} (see also \cite{Voi92}). The case $\kappa (L |_X) = \dim X -1$ is due to Kollár \cite{Kol15}.

\medskip

\noindent In this context, our result is as follows.

\begin{thm}\label{ext0} We assume that there exists a non-identically zero section $s$ of $L_0$ that admits an extension to $\cX$. Then every holomorphic section 
	of $\displaystyle L_0$ above extends to arbitrary order.
\end{thm}
\smallskip

\begin{proof} To begin with, recall the following result established in \cite{CP}. Let $\cE_1,\dots, \cE_K$ be a set of line bundles 
	on $\cX$ such that 
	\begin{equation}\label{intr5}
		(1- r_0)c_1(L)= r_0c_1(K_\cX)+ \sum_{i=1}^K r_ic_1(\cE_i),
	\end{equation}
	where $0\leq r_0< 1$ are rational numbers and  $r_j\geq 0$ for $j=2,\dots, K$.
	In \eqref{intr5} we denote by $c_1(\cE)$ the first Chern class of $\cE$. 
	Let $s$ be a section of 
	\[\cF_k:= \left(K_\cX+ L\right)\otimes \O_\cX/t^{k+1}\O_\cX,\]
	and and for each $i= 1,\dots ,K$ let $\sigma_i$ be a section of the sheaf $\displaystyle \cE_i\otimes \O_\cX/t^{k+2}\O_\cX$ (notice that this is one order higher than $s$). We denote by $h_L$ the metric induced on the bundle $L_0$ by the set of holomorphic sections $(s, \sigma_i)$ restricted to the central fiber
	--notice that this makes sense thanks to \eqref{intr5}. 
	\medskip
	
	\noindent Then the following holds true.
	\begin{thm}\label{007}\cite{CP} We consider $\displaystyle (s_0, \sigma_i)_{i=1,\dots, K}$ a family of holomorphic sections of $\displaystyle (\cF_k, \cE_i)_{i=1,\dots, K}$ respectively as above. We assume that $s_0$ admits a $\mathcal C^\infty$ extension $s_k$ such that if we write
		$\displaystyle \dbar s_k= t^{k+1}\Lambda_k$
		then \begin{equation}\label{intr6}\int_X\left|\frac{\Lambda_k}{dt}\right|^2e^{-(1-\ep)\varphi_L}dV< \infty\end{equation} for any positive $\ep> 0$. Then there exists a section $\wh s$ of $\cF_{k+1}$ such that $s= \pi_k(\wh s)$.
	\end{thm}
	\medskip
	
	\noindent In order to conclude the proof, it is sufficient to show that every section $s_0$ of $L_0$ with the following properties
	\begin{itemize}
		\smallskip
		
		\item the zero set $\Xi_0:= (s_0= 0)$ is non-singular;
		\smallskip
		
		\item the restriction $\displaystyle s|_{\Xi_0}$ is not identically zero
	\end{itemize}
	extends to $\cX$. This is clear, because these requirements are satisfied by the generic section of $L_0$, given the set-up at the beginning 
	of the current section.
	\medskip
	
	\noindent The fact that the $L^2$-hypothesis \eqref{intr6} in Theorem \ref{007} is satisfied follows from the following well-known statement.
	
	\begin{lemme}\label{int} Let $f\in \O_{(\CC^n, 0)}$ be a holomorphic function defined near the origin, whose restriction to $z_1= 0$ is not identically zero. 
		Then there exists $\ep_0> 0$ such that the integral
		\[\int_{(\CC^n, 0)}\frac{d\lambda}{|f|^{2\ep_0}|z_1|^{2-2\ep_0}}< \infty\]
		is convergent locally near zero.
	\end{lemme}
	
	\begin{proof}
		Let $\Omega\subset \CC^n$ be an open subset such that $\overline \Omega$ is contained in the definition domain of $f$. We denote by $\Omega_0:= \Omega\cap (z_1=0)$ the intersection of $\Omega$ with the hyperplane $(z_1=0)$. Then there exists a positive constant 
		$\ep_0> 0$ such that the integral
		\[\int_{\Omega_0}\frac{d\lambda}{|f|^{2\ep_0}}< \infty\]
		is convergent. This means that the function identically $1$ is integrable with respect to the weight $\displaystyle \varphi:= \ep_0\log|f|^2$. By Ohsawa-Takegoshi there exists a holomorphic function $F\in \O(\Omega)$ such that 
		\[F|_{\Omega_0}= 1, \qquad \int_{\Omega}|F|^2\frac{e^{-\varphi}d\lambda}{|z_1|^{2}\log^2|z_1|^2}< \infty.\]
		From this, our lemma follows immediately. 
	\end{proof}
	
	\noindent We now apply Theorem \ref{007} with the following data
	\[K:= 1, \quad \cE_1:= L, \quad \sigma_1:= s, \quad 0<r_1\ll 1,\]
	and by Lemma \ref{int}, the integrability requirements in \ref{007} are automatically satisfied. Repeated applications of this procedure shows that 
	$s_0$ extends to an open neighborhood of the central fiber $\cX_0$. The proof of \ref{ext0} is complete.
\end{proof}
\medskip



\section{Quasi-compact version}\label{S-qc}
In the remaining sections, we would like to generalise the previous results to the quasi-compact case. To begin with,
let $X$ be a compact K\"ahler manifold and let $D=\sum_{i=1}^{m}D_i$ be a SNC divisor on $X$. Write $X\setminus D :=X-\sum D_i$. A logarithmic flat bundle $(L,\nabla)$ consists of a holomorphic line bundle $L$ on $X$ and a logarithmic holomorphic connection
\begin{equation}\label{logholoconnecion}
\nabla:L\to L\otimes \Omega_X^{1}(\log D)
\end{equation}
so that $\nabla^2=0$.  The space of logarithmic flat bundles is denoted by $M_{\rm DR}(X/D)$. 
We consider $M_{\rm B}(X\setminus D)$ the space of rank $1$ local system on $X\setminus D$. Then we have 
$$M_{\rm B}(X\setminus D) \simeq \Hom(\pi_1(X\setminus D), \CC^\star)\simeq (\CC^\star)^\ell$$ 
where $\ell = b_1 (X\setminus D)$. 
Then there  is a natural map
\begin{equation}\label{surjj}
\rho: M_{\rm DR}(X/D)\to M_{\rm B}(X\setminus D) . 
\end{equation}

Note that  \cite[II, 6.10]{Del} implies that this map is surjective, namely give a flat bundle on $X_0$, we can extend it to be a holomorphic line bundle on the total space $X$ and the flat connection might have simple poles along $D$.
Moreover the kernel of this map is isomorphic to $\mathbb{Z}^m$: given a log flat bundle $(L, \nabla)$, we can construct  new log flat bundles $L\otimes \prod \mathcal{I}_{D_i} ^{a_i}$ for any $a_i \in \mathbb Z$ and $\nabla$ induces a  flat connection $\nabla'$ on
$L\otimes \prod \mathcal{I}_{D_i} ^{a_i}$ with log pole along $D$. We have $\rho ((L,\nabla)) =\rho (L\otimes \prod \mathcal{I}_{D_i} ^{a_i} , \nabla')$.

\medskip

Given a local system $\tau \in M_{\rm B} (X\setminus D)$, and let $(L, \nabla) \in M_{\rm DR} (X/D)$ such that $\rho ((L, \nabla)) =\tau$.  
Let $H^p (X\setminus D, \tau)$ be the cohomology of the locally constant sheaf of flat sections of $\tau$.
Note that the connection $\nabla$ induces a complex
\begin{equation}\label{derham}
	0 \to \mathcal{O} (L) \to\mathcal{O} (L) \otimes \Omega_X^{1}(\log D) \to  \mathcal{O} (L) \otimes \Omega_X^{2}(\log D) \to \cdots \mathcal{O} (L) \otimes \Omega_X^{n}(\log D) \to 0 .
	\end{equation}
Let  $\mathbb H^p (X, L \otimes \Omega_X ^{\bullet } (\log D))$ be the hypercohomology of the complex. 
We have the following result due to Deligne

\begin{thm}\cite[II, 6.10]{Del}\label{Deli}
	If the residue of $(L, \nabla)$ on $D_i$ is not a positive integer for every $i$, then we have
	$\mathbb H^p (X, L \otimes \Omega_X ^{\bullet }(\log D) )\simeq H^p (X\setminus D , \tau)$. 
	\end{thm}
The definition of the residue is as follows:
\begin{remark}
		Here if $D_i$ is locally defined by $z_1 =0$, by the compactness of $D_i$, there exists some constant $a_i$ such that $\nabla e_L = \frac{a_i dz_1}{z_1} \wedge e_L + C^\infty$. We call $a_i$ be the residue of $(L, \nabla)$ on $D_i$. It is independent of the choice of the base $e_L$. 
		
		If there is a metric $h$ on $L$ such that $\nabla =D' _h$, then $a_i$ equals to the Lelong number of $h$ on $D_i$ by using the connection formula. In the case when $\nabla$ does not come from a Chern connection, we have a similar result due to Y. Deng \cite[Lemma 2.3]{Deng}.
	\end{remark}
\medskip

\noindent The (lengthy) proof of the following result unfolds in the coming sections. More precisely, it follows directly from Theorem \ref{thm:convergence}.

\begin{thm}\label{complete}[=\Cref{introcomp}]
	Let $\tau \in M_{\rm B} (X\setminus D)$ and we consider a germ of holomorphic disk $\gamma: (\CC, 0)\to M_{\rm B} (X\setminus D)$ such that $\gamma(0)= \tau$ and such that 
\begin{equation} 
	\dim  H^p (X\setminus D, \gamma  (t)) \geq k
\end{equation} 
for every $|t| \ll 1$. 
We denote by $\dot\alpha$ the infinitesimal deformation  $d\gamma(0)$.  Then
\begin{equation} 
		\dim H^p (X\setminus D, \tau +z \dot\alpha )\geq k
\end{equation}
  for $ |z|\ll 1$.  Here $H^p (X\setminus D, \tau )$ is cohomology with respect to the local system $\tau$.
\end{thm}

\medskip


\subsection{Differential operators and harmonic metrics in log setting}\label{subsec:harmonic}
We discuss here the existence of harmonic metrics in the log setting, i.e. corresponding to $(X, D)$. For a general presentation  
we refer to \cite{Moc07b}. Since we are aiming at a reasonable self-contained paper, the proof of the special case we are interested is presented here in detail.
\smallskip

Let $X$ be a compact Kähler manifold and let 
$$\nabla: L \to L\otimes \Omega_X ^1 (\log \sum D_i)$$ 
be our flat holomorphic connection. Set $D_K : =\dbar + \nabla$, where $\dbar$ defines the complex structure of $L$. Then $D_K ^2=0$ on $X_0 := X\setminus \sum D_i$.  

We now recall the definition of harmonic metrics. 
Let $h$ be a metric on $L$ smooth on $X_0$ and with possible analytic singularity along $\sum D_i$. Let $D'_{h}$ be the $(1,0)$-part of the Chern connection on $(X, L, h)$. 
Set
$$-2\theta_0 :=D' _h - \nabla.$$
Then  $\theta_0$ is a $(1,0)$-form with log pole along $\sum D_i$. Since $\nabla^2=0$ and $(D' _h )^2 =0$, we have 
\begin{equation}\label{partialclose}
	\partial \theta_0 =0 .
\end{equation}
We now follow the definitions of \cite[Section 1]{Sim92} (but in order to be coincide with the our contexts, the notations of the connections are not the same). Note that 
$$\{ \theta_0 e, e\} = \{e,  \ol \theta_0 e\} .$$
Then the two connections $ D' _h + \dbar $ and $\nabla +(\dbar -2\ol \theta_0)$ preserve the metric $h$. We can thus define $D_K ''$ as follows:

\begin{defn}\label{def:connection}
	We set 
	\begin{equation}\label{dbarop}
		D'' _K := \frac{\dbar+ (\dbar -2\ol\theta_0)}{2} + \frac{\nabla - D' _{h}}{2} = \dbar + {\theta_0 - \ol \theta_0}  
	\end{equation}
	and 
	\begin{equation}\label{e1}
 D'_K:= D'_{h} +\theta_0 +\ol \theta_0
\end{equation}

\end{defn}
From the definition,  we know that 
$$D'' _K +D'_K =D_K =\dbar + \nabla .$$
however,  $D'' _K +D'_K$ does not preserve the metric.

\begin{defn}[Harmonic metric]\label{def:harmonic}
We set $G_K := (D'' _K )^2$. 
We say that $h$ is a harmonic metric, if $G_K =0$ on $X_0$.
\end{defn}

In our rank $1$ case,  we have the following direct proposition

\begin{proposition}
	$h$ is harmonic if and only if $i\Theta_{h}= 0$ on $X_0$.
	\end{proposition}

\begin{proof}
	Thanks to \eqref{partialclose}, we have $\dbar \ol \theta=0$. Then we obtain 
	\begin{equation}\label{curvformula}
		G_K = \dbar \theta_0  . 
	\end{equation}
	Note that $i\Theta_{h} = [\dbar, D'_{h}] = -2\dbar \theta_0$. Then $G_K =0$ on $X_0$ if and only if $i\Theta_{h}=0$ on $X_0$.
	\end{proof}
\medskip
 
\noindent We have the following important remarks.

\begin{proposition}\cite[section 1]{Sim92}
	Let $\omega$ be a Kähler metric on $X_0$ and let $(D'_K)^\star  , (D''_K)^\star $ be the formal adjoint of $D'_K, D''_K$ with respect to $(h, \omega)$.
Then the equalities
 \begin{equation}\label{kahleridentity}
  (D'_K)^\star  = i [\Lambda_\omega, D'' _K] \qquad\text{and} \qquad (D''_K)^{\star}  = -i [\Lambda_\omega, D'_K] 
 \end{equation}
 hold true.
 
 \noindent If $h$ is harmonic, then 
  \begin{equation}\label{comparelaplace}
   \Delta'' _K =\Delta'_K =\frac{1}{2} \Delta_K \qquad\text{on } X_0,  
  \end{equation}
where $\Delta_K :=[D_K, D^\star  _K]= [\partial +\nabla, (\partial +\nabla)^\star ]$. Moreover, $D''_K, D'_K,  (D''_K)^\star, (D'_K)^\star$ commute with the Lapalces $\Delta''_K, \Delta'_K$.
\end{proposition}

\begin{proof}
 Since $\theta$ is a $(1,0)$-form, we have 
 $$ \theta^\star  = -i [\Lambda_\omega, \ol \theta] \qquad \text{and}\qquad
 (\ol\theta)^\star = i [\Lambda_\omega,  \theta].$$
 The curical point is that the second equation changes the sign.
 Together with the basic Kähler identiy $(D'_h)^\star  = i [\Lambda_\omega, \dbar]$, we obtain
 $$(D'_K )^\star = i [\Lambda_\omega,  D''_K] .$$
 The second equality in \eqref{kahleridentity} is proved in the same way.
 \medskip

Suppose next that $h$ is harmonic. Since $D''_K +D'_K =D_K =\dbar + \nabla$, we have
 $$(D''_K +D'_K )^2 = (\dbar + \nabla)^2 =0 .$$
 Since $h$ is harmonic, we have $(D''_K)^2 =(D'_K)^2=0$.
 Then we obtain 
 $$[D''_K , D'_K] = D''_K D'_K + D'_K D''_K =0 .$$
 Together with \eqref{kahleridentity} and the Jacobi equality, we know that 
 $$\Delta'' _K = [D''_K, (D''_K) ^\star] =[D'_K, (D'_K) ^\star]  =\Delta' _K .$$
 
 Finally, \eqref{kahleridentity} and the Jacobi equality imply that $[(D'_K)^\star, D''_K]=0$.  Then
 $$ \Delta_K = [D'_K +D''_K , (D'_K + D''_K)^\star] = \Delta'' _K + \Delta'_K =2 \Delta''_K .$$ 
 
 By the Jacobi equality, we also have
 $$[D''_K , \Delta''_K]= 0 .$$
\end{proof}

\medskip

Before proving the existence of harmonic metrics, we first explain the relationship between $c_1 (L)$ and $\sum a_i [D_i]$ where $a_i$ is the residue of $\nabla$ along $D_i$.

\begin{proposition}\label{residuerelation}
	Let $a_i$ be the residue of $\nabla$ along $D_i$. Then we have 
	\begin{equation}\label{class}
		\sum \rel (a_i) [D_i] = c_1 (L) \in H^{1,1} (X, \mathbb R)  
	\end{equation} 
	and 
	\begin{equation}\label{class2}
		\sum \im (a_i) [D_i]=0 \in H^{1,1} (X, \mathbb R) .
	\end{equation} 
	\end{proposition}

\begin{proof}
Fix a smooth metric $h_0$ on $L$. Then $D' _{h_0} = \nabla +\theta_0$ for some $(1,0)$-form $\theta_0$ with log pole along $\sum D_i$. Since $h_0$ is smooth on $(X, L)$, we know that
near $D_i = \{ z_i =0\}$, $\theta_0 = - a_i \frac{dz_i}{z_i} + C^\infty$, where $a_i$ is the residue of $\nabla$ along $D_i$. 

Now as $\nabla$ is holomorphic, we have $[\dbar, \nabla]=0$ on $X_0$. Then
\begin{equation}\label{curv}
	i\Theta_{h_0} (L) = [\dbar, D'_{h_0}]=i \dbar \theta_0 \qquad\text{ on } X_0 .
\end{equation}
On the other hand, since $\theta_0 = - a_i \frac{dz_i}{z_i} + C^\infty$, the divisorial part of the current $i\dbar \theta_0$ is $- \sum a_i [D_i]$.
Note that $\dbar \theta_0$ is exact in the sense of currents on $X$. Then  \eqref{curv} implies that  
$$c_1 (L) = [ i \Theta_{h_0} (L) ] = - i [\dbar \theta_0]_{\text{Div part} } =\sum a_i [D_i] .$$
Note that $a_i$ is a complex number, the RHS is written as $\sum \rel (a_i) [D_i] + i (\sum  \im (a_i) [D_i])$. As $c_1 (L) \in H^{1,1} (X,\mathbb R)$,  we obtain \eqref{class} and \eqref{class2}.
\end{proof}
\medskip

\noindent We are now ready to establish the existence of harmonic metrics for line bundles.

\begin{thm}\label{existence}
Let $(L, \nabla) \in M_{\rm DR} (X/D)$. Then there exists a unique harmonic metric (modulo constant) $h$ on $L$ such that the Lelong number of $h$ along $D_i$ is $\rel (a_i)$ for every $i$, where $a_i$ is the residue of $\nabla$ along $D_i$.

Moreover, if we have $c_1 (L) = \sum b_i D_i$ for some $b_i \in \mathbb R$, then there exists a unique harmonic metric (modulo constant) $h'$ on $L$ such that the Lelong number of $h'$ along $D_i$ is $b_i$ for every $i$.
	\end{thm}

\begin{proof}
	We fix a smooth metric $h_0$ on $L$ and set $\theta_0 := D' _{h_0}- \nabla$ be a smooth $(1,0)$-form of log pole along $\sum D_i$.
Let $h := h_0 e^{-\phi}$ for some real function $\phi$ of analytic singularity along $\sum D_i$. Thanks to \eqref{curvformula}, we have 
$$G_K= \dbar (\theta_0 - \partial\phi) .$$Then  $h$ is harmonic if and only if we have
$$i\ddbar \phi = - i\dbar \theta_0 \qquad\text{ on } X_0 .$$
To prove the theorem, it is thus sufficient to solve the above $\ddbar$-equation.

In Proposition \ref{residuerelation}, we showed that the divisorial part of $i \dbar \theta_0$ is $-\sum a_i [D_i]$. Together with \eqref{curv}, we know that the current 
$$T: = i \dbar \theta_0  + i ( \sum  \im (a_i) [D_i])$$ is a real $(1,1)$-current on $X$. Recall that \eqref{class2} says that the class of $( \sum  \im (a_i) [D_i])$ is $0$. Then $[T] =0 \in H^{1,1} (X)$. In other word,  $T$ is $\dbar$-exact in the sense of current on $X$. Moreover, thanks to \eqref{partialclose}, $T$ is also $d$-closed.  Then $T$ is a real $(1,1)$-current, $d$-closed and $\dbar$-exact on $X$. By the $\ddbar$-lemma,  we can always find a real function $\phi$ such that 
\begin{equation}\label{ddbareq}
	i\ddbar \phi = -T =-i\dbar \theta_0  - i ( \sum  \im (a_i) [D_i]) \qquad\text{ on } X.
\end{equation}
As a consequence, for the metric $h:= h_0 e^{-\phi}$, the curvature 
$$2G_K = \dbar (\theta_0- \partial \phi) \equiv 0 \qquad\text{ on }X_0 .$$ 
Then $h$ is a harmonic metric. Moreover, as the divisorial part of $-T$ is $\sum \rel (a_i) [D_i]$, \eqref{ddbareq} implies that the Lelong number of $h$ along $D_i$ equals to $\rel (a_i)$.

\medskip

For the second part,   by using Proposition \ref{residuerelation}, we know that 
$$\sum \rel (a_i) [D_i] =\sum b_i [D_i] =c_1 (L) \in H^{1,1}  (X, \mathbb R) .$$
Then there exists a real function $\psi$ such that $i\ddbar \psi = \sum b_i [D_i]  -\sum \rel (a_i) [D_i] $. Then $h':= h \cdot e^{-\psi}$ is also harmonic and the Lelong number of $h'$ along $D_i$ is $b_i$.
\end{proof}
\smallskip

\noindent Let $h$ be a harmonic metric on $L |_{X\setminus D}$. By using \eqref{curvformula}, we know that
$$\theta := D'_{h} -\nabla \in H^{0} (X, \Omega_{X} (\log D))$$ 
is a holomorphic $1$-form with log pole.  We have the following easy observation.

\begin{proposition}
	Let $\theta \in  H^{0} (X, \Omega_{X} (\log D))$.
	There exists a complex-valued function $\phi$ of $\log |s_{D}|$-type singularity along  $\sum D_i $ such that $\partial \phi \in  H^{0} (X, \Omega_{X} (\log D))$ and
	$$\theta -\partial \phi \in H^0 (X ,\Omega^1 _X) .$$
\end{proposition}	

\begin{proof}
	Note that $\dbar \theta $ is a $d$-closed by \cite{Del71, Nog95} and $\dbar$-exact current. By the $\ddbar$-lemma, we have $\dbar \theta =\dbar \partial \phi $ on $X$ for some $\phi$. Then $\theta -\partial \phi$ is a $\dbar$-closed $(1,0)$-current on $X$. Then $\theta -\partial \phi \in H^0 (X ,\Omega^1 _X)$.
\end{proof}

\begin{lemme}\label{lem:auxiliary}
	Let $X$ be a compact K\"ahler manifold and  let $D$ be a simple normal crossing divisor. Let $A$ be a connected component of $M_{\rm DR}(X/D)$. 
	We decompose 
	$$D = \Delta_1 +\Delta_2 +\Delta_3,$$
	where $\Delta_1 +\Delta_3$ corresponds to the components such that $\Res_{\Delta_1 +\Delta_3} (L, \nabla)$ is invariant for every $(L, \nabla) \in A$.
	Here $\Delta_1$ (resp. $\Delta_3$) corresponds to the components such that $\Res_{\Delta_1} (L, \nabla)\in\mathbb Z$ ($\Res_{\Delta_3} (L, \nabla)\notin\mathbb Z$)
	
	We denote $\Delta_2 = \sum_{i=1}^k D_i$. Then there exists $(a_1, \cdots, a_k)\in \Q^k$ such that $\sum_{i=1}^{k}a_i c_1(D_i)=0$ and $a_i \neq 0$ for every $1\leq i \leq k$.
\end{lemme}

\begin{proof}
	
	Let $(L, \nabla), (L', \nabla') $ be two points in $A$. Since $c_1 (L)= c_1 (L')$ and $\Res$ is invariant on $\Delta_1 +\Delta_3$. We have
	$$0 = \sum_{i=1}^k (\rel (\Res (L, \nabla)_{D_i})  - \rel (\Res (L', \nabla')_{D_i}))  [D_i] .$$

	We claim that for every $1\leq i \leq k$, there exists a $\{a_{i,j}\}_{j=1,\ldots, k} \subset \mathbb \Q^k$ such that 
	$$\sum_{j=1}^{k}  a_{i,j } c_1(D_j) =0 $$
	and $a_{i,i} \neq 0$.   In fact, since $\Res_{D_i}$ is not invariant, we can find two $(L,\nabla), (L',\nabla') \in A$ such that $\Res_{D_i} ((L,\nabla)) \neq \Res_{D_i} ((L',\nabla'))$. Then we can find some real numbers $b_{i,j}$ such that
	$$\sum_{j=1}^{k}  b_{i,j} c_1(D_j) =0  $$ 	
	and $b_{i,i}\neq 0$.  Note that $c_1(D_1),\ldots,c_1(D_k)\in H^2(X,\Q)=H^2(X,\Z)\otimes_{\Q} \Q$. Then for every $i$, we can find $\{a_{i,j}\}_{j=1,\ldots, k} \subset \mathbb \Q^k$ such that 
	$$\sum_{j=1}^{k}  a_{i,j } c_1(D_j) =0 $$
	and $a_{i,i} \neq 0$. 
	
	By the above result, we can find rational numbers $\ep_k\gg \cdots\gg \ep_1>0$ such that for the sum 
	$$
	\ep_k  (\sum_{j=1}^{k}a_{k,j}c_1(D_j))+\cdots+ \ep_1 (\sum_{j=1}^{k}a_{1,j}c_1(D_j))=:\sum_{j=1}^{k}a_j'c_1(D_j),
	$$
	we have $a_j'\neq 0$ for any $j\in \{1,\ldots,k\}$.   
	The lemma  is thus proved. 
\end{proof}

\begin{proposition}\label{specialhar}
	Let $X$ be a compact K\"ahler manifold and let $D=\sum_{i=1}^{k}D_i$ be a simple normal crossing divisor. Let  $\Res: M_{\rm DR}(X/D)\to \mathbb C^k$ be the residue map, namely for every $(L, \nabla) \in A$, 
	$$\Res(L, \nabla) := (\Res_{D_1} (L,\nabla), \cdots, \Res_{D_k} (L,\nabla)) .$$ 
	We shall decompose 
	$$D = \Delta_1 +\Delta_2 +\Delta_3$$ 
	as in Lemma \ref{lem:auxiliary}, namely $\Delta_1 +\Delta_3$ corresponds to the components such that $\Res_{\Delta_1 +\Delta_3} (L, \nabla)$ is invariant for every $(L, \nabla) \in A$.
	Here $\Delta_1$ (resp. $\Delta_3$) corresponds to the components such that $\Res_{\Delta_1} (L, \nabla)\in\mathbb Z$ ($\Res_{\Delta_3} (L, \nabla)\notin\mathbb Z$)

Let $A$ be any connected component of $M_{\rm B}(X_0)$. 	Then for any representation $\varrho\in A$, there exists some  $(L,\nabla)\in M_{\rm DR}(X/D)$ whose monodromy representation is $\varrho$ and there is a (non-unique) harmonic metric $h$ for $(L,\nabla)$ such that   
\begin{itemize}
	\item For every $\beta\in H^0 (X, \Omega_X ^1 (\log D))$, $\Res\beta =0$ along $\Delta_1 +\Delta_3$.
	
	\item  $\Res\theta =0$ on $\Delta_1+\Delta_3$ and $\Res\theta \neq 0$ on every component of $\Delta_2$.
	
	\item The Lelong number $\nu (h_L) \in ]-1, 0] \cap \mathbb Q$ on every component of $D$.  Moroever,  $\nu (h_L)=0$ on $\Delta_1$ and  $\nu (h_L) \in ]-1,0[ \cap \mathbb Q$ for every component of $\Delta_2 +\Delta_3$.
	\end{itemize}
\end{proposition}

\begin{proof}
	
	For the first point, by \cite{Del71, Nog95}, we know that $\dbar \beta$ is $d$-closed and $\dbar$-exact on $X$. Then $\dbar \beta$ is $\ddbar$-exact.
	On the other hand,  we have $\dbar \beta= \sum \Res_{D_i} (\beta) [D_i]$. Then we have
	$$ \sum \Res_{D_i} (\beta)  c_1 ([D_i])=0  .$$
We now prove that $\Res(\beta)=0$ over $\Delta_1 +\Delta_3$.
Let $a_i \in \mathbb R$ such that $\sum a_i [D_i] =0$.  Let $L_0$ be a trivial line bundle on $X$. Then $e^{-\prod a_i \log |s_{D_i}|^2}$ defines a harmonic metric on $L_0$. It induces a $(1,0)$-connection $D'$. Then $(L \otimes L_0, \nabla \otimes D')$ is a flat log holomorphic connection whose residue is 
$$\Res(L \otimes L_0, \nabla \otimes D' )= (\Res_{D_1} (L,\nabla) + a_1, \cdots, \Res_{D_k} (L,\nabla) +a_k) .$$
In particular,  the pair
$$ (\Res_{D_1} (L,\nabla) + \rel (\Res_{D_i} (\beta)), \cdots, \Res_{D_k} (L,\nabla) +\rel (\Res_{D_k} (\beta)))$$ and 
$$ (\Res_{D_1} (L,\nabla) + \im (\Res_{D_i} (\beta)), \cdots, \Res_{D_k} (L,\nabla) +\im (\Res_{D_k} (\beta)))$$
belong to the image of $\Res: A \to \mathbb C^k$.
Then $\Res(\beta)=0$ over $\Delta_1 +\Delta_3$.
	
	\medskip
	
	For the second part,	 
we first fix a $(L, \nabla) \in M_{\rm DR} (X/D)$ such that its monodromy representation is $\rho$. 
By applying \Cref{lem:auxiliary}, there exists a harmonic metric $h_L$ such that the Lelong number $\nu (h_L)$ is rational on every component of $D$.
Set $\theta := \nabla - D'_{h_L}$.  By applying the first point, we know that $\Res\theta =0$ on $\Delta_1 +\Delta_3$. We assume that $\Delta_2 = \sum_{i \in I} [D_i]$. 
By \Cref{lem:auxiliary}, there exists $b_i\in \Q$ such that $\sum_{i\in I} b_i c_1 [D_i]=0$ and $b_i \in \mathbb Q^\star$.
Then $h_L \cdot e^{- \ep \sum_{i\in I} b_i\log |s_i|^2}$ defines also a harmonic metric on $L$.  For $\ep \in \mathbb Q$ generic, we know that $\Res(\nabla - D' _{h_L \cdot e^{- \ep \sum_{i\in I} b_i\log |s_i|^2}} ) \neq 0$ on $\Delta_2$. 
Replacing $h_L$ by $h_L \cdot e^{- \ep \sum_{i\in I} b_i\log |s_i|^2}$, we know that $h_L$ satisfies the second point.

\medskip

For the third point, by construction, $\nu (h_L)$ is rational on every component of $D$. 
	Then $i\Theta_h(L)=\sum_{i=1}^{k}a_i[D_i]$ where $a_i \in \mathbb Q$.  Replacing $L$ by $L\otimes  \O_X(-\sum_{i=1}^{k}\lfloor a_i+1\rfloor D_i))$, we may assume that $a_i\in ]-1,0]\cap \Q$ for each $i$. Now as the monodromy is zero along $\Delta_1$, we know that $a_i =0$ on $\Delta_1$. 
	
	For $\nu (h_L)$ along $\Delta_2$,   since $b_i \neq 0$ for every $i\in I$, for a generic $\ep\in \mathbb Q$, we know that the Lelong number of $h_L$ along $\Delta_2$ is not an integer. Then we know that $a_i \in ]-1,0[ \cap \mathbb Q$. 
\end{proof}
\smallskip

\begin{remark}
In general, we cannot modify further the component $\Delta_1$. For example, let $X$ be a compact Kähler manifold and let $D_1 \subset X$ be a smooth irreducible divisor. 
We consider $M_{\rm DR} (X/D_1)$.  Let $A$ be a connected component of $M_{\rm DR} (X/D_1)$ such that for every $(L, \nabla) \in A$, $c_1 (L) = [D_1] \in H^{1,1} (X, \mathbb R)$. Then the only harmonic metric on $L$ is the singular metric $h_L$ such that $i\Theta_{h_L} (L) =[D_1]$.  We cannot find a harmonic metric such that the Lelong number along $D_1$ is not integer.

Note that in this case,  for every $(L, \nabla) \in A$, the residue of $(L, \nabla)$ is an integer and the monodromy of  $(L, \nabla)$ along $D_1$ is trivial. Let $\tau$ be the local system of $(L, \nabla)$. If we would like to study the cohomology of $H^\bullet (X_0, \tau )$, since $\tau$ is trivial along $D_1$, if $X$ is simply connected, we have $H^\bullet (X_0, \tau )\simeq H^\bullet (X_0, \mathbb C)$. In general, it is not isomorphic to $H^\bullet (X, \mathbb C)$. In other words, the $\log D_1$ in the hypercohomology $\mathbb H^\bullet (X, \Omega^\bullet (\log D_1) )$ is necessary. 
\end{remark}
\smallskip

\begin{remark}
As we can see from the previous results, the three components $\displaystyle (\Delta_i)_{i}$ of the divisor $D$ have different geometric interpretations. Loosely speaking, the singularities of the metrics we will consider in the next subsections reflect this. 
\end{remark}

\medskip

\section{Technicalities I: elliptic estimates for $\Delta_K$ }\label{s-tech-I}

The aim of this section is to establish some elliptic estimates for the Laplace operator $\Delta_K$. 
In this section, we are always in the following setting.

Let $X$ be a compact K\"ahler manifold and let $D=\sum_{i=1}^{m}D_i$ a SNC divisor on $X$. Write $X_0 :=X-\sum D_i$. A logarithmic flat bundle $(L,\nabla)$ consists of a holomorphic line bundle $L$ on $X$ and a logarithmic holomorphic connection
$$
\nabla:L\to L\otimes \Omega_X^{1}(\log D)
$$
so that $\nabla^2=0$. 
Recall that thanks to Proposition \ref{specialhar}, after replacing $(L, \nabla)$ by another $(L', \nabla')$ with the same monodromy representation, 
we can find a harmonic metric $h_L$ on $L$ such that we have the decomposition 
\[D= \Delta_1 + \Delta_2 +\Delta_3\]
with the following properties.
Set $\theta:= D'_{h_L} -\nabla \in H^0 (X, \Omega^1_X (\log D))$. 
\begin{property}\label{property} We assume that the following hold.
	\begin{enumerate}
		
		\item[(1)] At a generic point of $\Delta_1$ we have  $\Res \theta =0$ and $h_L$ is smooth. 
		
		\item[(2)] We have $\Res \theta \neq 0$ for every component of $\Delta_2$, and $\Res \theta = 0$ for every component of $\Delta_3$.
		
		\item[(3)] The Lelong number of $h_L$ belongs to $]-1,0[ \cap \mathbb Q$ for every component of $\Delta_2+ \Delta_3$.
		
	\end{enumerate} 
\end{property}

\subsection{Laplace operators on smooth/log forms}\label{ellipticest}

Here we first define two types of Laplace operators $\Delta_K$, acting on $L$-valued smooth forms, and on 
$L$-valued smooth forms with log poles along the divisor $\Delta_1$, respectively. 

For the latter, our strategy is to take advantage of the vector bundle structure of the logarithmic tangent bundle $T_X\langle \Delta_1\rangle$. More precisely, a $(p, q)$-form with log poles is seen as $(0, q)$-form with values in the vector bundle $\Omega^p_X(\log \Delta_1)$. 
If we disregard $L$, as soon as we fix a Hermitian metric on $T_X\langle \Delta_1\rangle$ we can define the adjoint of the exterior derivative 
\[\partial: \Omega^p_X(\log \Delta_1)\to \Omega^{p+1}_X(\log \Delta_1)\]  and a corresponding Laplace operator.
\smallskip

Next we show that Hodge decomposition in the presence of the divisor $D$ still holds:
this will be our main result in Section \ref{HodgeDek}. Our arguments  mainly rely on the new a-priori inequalities we establish in the next subsection, cf. \ref{apriori}. This is the reason why we need a detailed analysis of the properties of the metrics we introduce next next. 
\smallskip

\noindent 

\subsubsection{Metric structure on the (orbifold) tangent bundle of $X_0$}\label{metricTX} Let $\omega_D$ be a metric with Poincaré singularities along the support of the divisor $\Delta_2$, and with conic singularities of angles 
$\displaystyle 2\pi(1-1/k)$ on each component of $\Delta_3$, where $k\geq 3$ is a positive integer, sufficiently divisible.  
In particular, $\omega_D$ is smooth at a generic point of $\Delta_1$.
\smallskip

More specifically, we set $\displaystyle \Delta_2 := \sum_{i=1}^m D_i$ and $ \Delta_3 := \sum_{i=m+1}^N D_i.$
Then we define 
\begin{equation}\label{m1} 
	\omega_{D}:= \omega_X+ \sum_{i=1}^m\sqrt{-1}\ddbar |s_i|^{2/k} - \sum_{i=m+1}^N\sqrt{-1}\ddbar\log\log \frac{1}{|s_i|^{2}}
\end{equation} 
where $\omega_X$ is a fixed Kähler metric on $X$, sufficiently positive so that $\omega_D$ above is greater than
$1/2\omega_X$. In the equality \eqref{m1} we denote by $s_i$ the defining equation for $D_i$, measured with respect to a non-singular metric
on $\O(D_i)$. The motivation for working with this precise type of metric will become clear later, in Theorem \ref{hodgedec}.
\medskip

\noindent Our next goal is to estimate the coefficients of the metric $\omega_D$ and their inverses. By a direct calculation, one easily sees that the following equalities 
\begin{equation}\label{m2}\ddbar |s_i|^{2/k}= \frac{1}{k^2}\frac{\langle D's_i, D's_i\rangle}{|s_i|^{2-2/k}}- \frac{|s_i|^{2/k}}{k}\theta_i\end{equation}
and 
\begin{equation}\label{m3}-\ddbar\log\log \frac{1}{|s_i|^{2}}= \frac{\langle D's_i, D's_i\rangle}{|s_i|^{2}\log ^2{|s_i|^{2}}}+ \frac{\theta_i}{\log {|s_i|^{2}}} \end{equation}
hold. Here we denote by $\theta_i$ the curvature form of a fixed, non-singular metric on $\mathcal O(D_i)$.
\medskip

\noindent We have the following statement.

\begin{lemme}\label{coeff}
Let $\Omega\subset X$ be an open subset. Assume that we can choose the coordinate system $z= (z_1,\dots, z_n)$ on $\Omega$ such that 
\[\Delta_2\cap \Omega= (z_1\dots z_p= 0), \qquad \Delta_3\cap \Omega= (z_{p+1}\dots z_{p+q}= 0).\]
Let $(\omega_{i\ol j})_{i,j}$ be the coefficients of the metric $\omega_D$ with respect to $z$. Then the following are true.
\begin{enumerate}
\smallskip

\item[\rm (1)] There exist smooth functions $A_k, A_{\ol m}, C_l$ and $\wt A_k, \wt A_{\ol m}, \wt C_l$ so that we have
\begin{equation}\label{m3}
\begin{split}
\omega_{i\ol j}(z)= & g_{i\ol j}+ \frac{\delta_{ij}\delta_{i,p}}{|z_i|^2\log^2|z_i|^2}+ \frac{\delta_{i,p}A_i}{z_i\log^2|z_i|^2}+
\frac{\delta_{j,q}A_{\ol j}}{\ol z_j\log^2|z_j|^2}\\
+ & \frac{\delta_{ij}\delta_{i,q}e^{\varphi_i}}{|z_i|^{2-2/k}}+ \frac{\delta_{i,q}|z_i|^{2/k}\wt A_i}{z_i}+ 
\frac{\delta_{j,q}|z_j|^{2/k}\wt A_{\ol j}}{\ol z_j}\\
+ & \sum_{l=1}^p\frac{C_l+ C_{l, 1}\log|z_l|^2}{\log^2|z_l|^2}+ \sum_{l= p+1}^{p+q}\wt C_l |z_l|^{2/l}\\
\end{split}
\end{equation}
where the matrix $(g_{\beta\ol \alpha})$ is positive definite and has smooth coefficients. We denote by $\delta_{ij}$ the Kronecker symbol, and $\delta_{i,p}$ equals 1 if $i\in \{1,\dots, p\}$ and zero if not. Also, $\delta_{i,q}$ equals 1 if $i\in \{p+1,\dots, p+q\}$ and zero if not.
\smallskip

\item[\rm (2)] Let $(\omega^{\ol j i})_{i, j}$ be the inverse of $(\omega_{i\ol j})$. Then for any $i, j\in \{1, \dots, p\}$ such that $i\neq j$ we have 
\[\omega^{\ol i i}(z)= (1+  \psi_i(z))|z_i|^2\log^2|z_i|^2, \qquad \omega^{\ol j i}(z)= \psi_{ij}(z)z_i\ol z_j.\]
The functions $\psi_i, \psi_{ij}$ are smooth on $\Omega\setminus D$, we have $|\psi_i|< 1/2$ and 
\[\sup_{\Omega\setminus D}(|d\psi_i|_{\omega_D}+ |d\psi_{ij}|_{\omega_D})< \infty.\] 
\smallskip

\item[\rm (3)] For each $i\in \{1, \dots, p\}$ and $j\in \{p+1,\dots, p+q\}$ we have 
\[ \omega^{\ol j i}(z)= \psi_{ij}(z)z_i\ol z_j.\]
\smallskip

\item[\rm (4)] For each $i\in \{1, \dots, p\}$ and $j\geq p+q+1$ we have 
\[ \omega^{\ol j i}(z)= \psi_{ij}(z)z_i.\]
\smallskip

\item[\rm (5)] For each $i, j\in \{p+1, \dots, p+q\}$ such that $i\neq j$ we have 
\[\omega^{\ol i i}(z)= (1+  \psi_i(z))|z_i|^{2-2/k}, \qquad \omega^{\ol j i}(z)= \psi_{ij}(z)z_i\ol z_j.\]

\item[\rm (6)] For each $i\in \{p+1, \dots, p+q\}$ and $j\geq p+q+1$ we have 
\[\omega^{\ol j i}(z)= \psi_{ij}(z)z_i.\]
\smallskip

\item[\rm (7)] Finally, in case both indexes $i\neq j$ are bigger than $p+1$, then we have 
\[\omega^{i\ol i}(z)= 1+  \psi_i(z), \qquad \omega^{i\ol j}(z)= \psi_{ij}(z)\]
where the functions $\psi$ have the same regularity properties as in {\rm (2)} and moreover are small in $\mathcal C^1$-norm.
\end{enumerate}
\end{lemme}

\begin{proof} The first point (1) is a simple, direct calculation following from \eqref{m2} and \eqref{m3}, see also
\cite{HenriII}. It follows that for $i= 1,\dots, p$ we have
\begin{equation}\label{h1} 
\omega_{i\ol i}(z)= \frac{1}{|z_i|^2\log^2|z_i|^2}\big(1+ \psi_i(z)\big)
\end{equation}
and for $i= p+1,\dots, p+q$ the corresponding diagonal coefficient can be written as follows
\begin{equation}\label{m4}
\frac{1}{|z_i|^{2-2/k}}\big(1+ \psi_i(z)\big).
\end{equation}

In case $i\geq p+q+ 1$, we have
\begin{equation}\label{h2} 
	\omega_{i\ol i}(z)= g_{i\ol i}(z)\big(1+ \psi_i(z)\big)
\end{equation}
where $g_{i\ol i}$ are smooth, strictly positive.
In the equations above \eqref{h1} and \eqref{h2} we denote by $(\psi_i)$ a set of smooth functions in the complement of $D$, such that 
\[\sup_{\Omega\setminus D}|\psi_i|_{\omega_D}< 1/2, \qquad \sup_{\Omega\setminus D}|d\psi_i|< \infty.\]

\noindent We have similar expression for the coefficients outside the diagonal, as follows.

If $i\neq j$ are both smaller than $p$, then we have
\begin{equation}\label{h3} 
	\omega_{i\ol j}(z)= \frac{\psi_{ij}(z)}{z_i\ol z_j\log^2|z_i|^2\log^2|z_j|^2}
\end{equation} 
with the same type of function $\psi$ as before. This is a consequence of the fact that the differential of
\[z_i\log^2|z_i|^2, \qquad \ol z_i\log^2|z_i|^2\]
for $i= 1, \dots, p$ is bounded on $\Omega\setminus D$ when measured with respect to the Poincaré metric.

If $i\leq p$ and $p+1\leq j\leq p+q$, then we have  
\begin{equation}\label{m6} 
	\omega_{i\ol j}(z)= \frac{|z_j|^{2/k}}{z_i\ol z_j\log^2|z_i|^2}\psi_{ij}(z).
\end{equation} 
Here (as always) the function $\psi_{ij}(z)$ is bounded in $\mathcal C^1$ norm: this is due to the fact that 
\[\sup_{\Omega\setminus D}|d(|z_j|^{-2/k} z_j)|_{\omega_D}< \infty\]
for $p+1\leq j\leq p+q$. It is at this point that the hypothesis $k\geq 3$ is required.

If $i\leq p$ and $j\geq p+q+1$, then we have  
\begin{equation}\label{m7} 
\omega_{i\ol j}(z)= \frac{1}{z_i\log^2|z_i|^2}\psi_{ij}(z)
\end{equation} 

If $i\neq j$ and $p\leq i, j\leq p+q$, then we have  
\begin{equation}\label{m8} 
	\omega_{i\ol j}(z)= \frac{|z_i|^{2/k} |z_j|^{2/k}}{z_i\ol z_j}\psi_{ij}(z)
\end{equation} 

If $p\leq i \leq p+q$ and $j\geq p+q+1$, then we have  
\begin{equation}\label{m9} 
	\omega_{i\ol j}(z)= \frac{|z_i|^{2/k}}{z_i}\psi_{ij}(z).
\end{equation} 

If $i\neq j$ are both greater than $p+ q$, then we have
\begin{equation}\label{h4} 
	\omega_{i\ol j}(z)= \psi_{ij}(z)
\end{equation} 
i.e. a function which is bounded in $\mathcal C^1$-norm.


\noindent From the relations \eqref{h1}--\eqref{h4} it follows that 
\begin{equation}\label{h6} 
	\det\big(\omega_{i\ol j}(z)\big)= \big(1+ \psi(z)\big)\prod_{i=1}^p\frac{1}{|z_i|^2\log^2|z_i|^2}\prod_{j=p+1}^{p+ q}\frac{1}{|z_i|^{2-2/k}}.
\end{equation} 
Indeed, we have
\begin{equation}\label{h7} 
\det\big(\omega_{i\ol j}(z)\big)= \sum_{\sigma} \ep(\sigma)\prod\omega_{i\ol {\sigma_i}}(z).
\end{equation} 
Consider first the set of permutations $\sigma$ such that $\sigma(i)=i$ for all $i=1,\dots, p+q$.
The corresponding sum will be of type
\begin{equation}\label{h8} 
\big(1+ \psi(z)\big)\prod_{i=1}^p\frac{1}{|z_i|^2\log^2|z_i|^2}\prod_{j=p+1}^{p+ q}\frac{1}{|z_i|^{2-2/k}}.
\end{equation} 
as it follows from the formulas \eqref{h1}, \eqref{m4} and \eqref{h2}. 

Let now $\sigma$ be a permutation such that at least one element of the set $\{1,\dots, p\}$ is not fixed.
Then we have 
\begin{equation}\label{h9} 
\prod\omega_{i\ol {\sigma_i}}(z)= \frac{\psi(z)}{\prod_{i=1}^p |z_i|^2\log^{m_i}|z_i|^2\prod_{j=p+1}^{p+q} |z_j|^{2-n_j/k}}
\end{equation}
where each integer $m_i, n_j$ in \eqref{h9} are greater than two, and at least one of them is equal to 4. Therefore, the sum 
\eqref{h8} is the dominant factor in the expression \eqref{h7}, and our claim is proved.
\smallskip

\noindent We will not do the complete calculations in order to establish the points (2)-(7) of the lemma, but simply evaluate 
$\omega^{\ol 1 1}$ and $\omega^{\ol 2 1}$. One can derive the expression for general $\omega^{\ol j i}$ in a similar, tedious manner.

Consider $\omega^{\ol 1 1}$ first, i.e. 
\begin{equation}\label{m10} 
\omega^{\ol 1 1}(z)= - \frac{1}{\det\big(\omega_{i\ol j}(z)\big)}\sum\ep(\tau)\prod\omega_{i\ol \tau_i}(z)
\end{equation}
where $\tau: \{2, 3,\dots, n\}\to \{2, 3,\dots, n\}$ is bijective. The terms appearing in the sum \eqref{m10} are 
\begin{equation}\label{m11} 
\frac{\psi(z)}{\prod_{i\neq 1} |z_i|^2\log^{m_i}|z_i|^2\prod_{j\neq 1} |z_j|^{2- n_j/k}}
\end{equation}
where each of the integers $m_i, n_j$ are greater or equal to two. By the expression of the determinant \eqref{h7}, we obtain 
\[\omega^{\ol 1 1}(z)= \big(1+ \psi(z)\big)|z_1|^2\log^2|z_1|^2.\]

Assume that $p\geq 2$. We have
\begin{equation}\label{h10} 
\omega^{\ol 2 1}(z)= - \frac{1}{\det\big(\omega_{i\ol j}(z)\big)}\sum\ep(\tau)\prod\omega_{i\ol \tau_i}(z)
\end{equation}
where $\tau: \{2, 3,\dots, n\}\to \{1, 3,\dots, n\}$ is bijective. Hence each of the terms of the sum in \eqref{h10}
can be written as
\begin{equation}\label{h11} 
 \frac{1}{z_2\ol z_1\log^2|z_2|^2\log^2|z_1|^2}\frac{\psi(z)}{\prod_{i\neq 1,2} |z_i|^2\log^{m_i}|z_i|^2\prod_{j\geq p+1} |z_j|^{2- n_j/k}}
\end{equation}
where each of the integers $m_i, n_j$ are at least two. The equality 
\begin{equation}\label{h12} 
\omega^{\ol 2 1}(z)= z_1\ol z_2\psi_{12}(z)
\end{equation}
follows, again by using \eqref{h7}. 

If for example $p= 0, q\geq 2$, then the expression \eqref{h11} above is to be replaced with 
\begin{equation}\label{h13} 
 \frac{|z_1|^{2/k}|z_2|^{2/k}}{z_2\ol z_1}\frac{\psi(z)}{\prod_{j\geq 3} |z_j|^{2- n_j/k}}
\end{equation} 
with $n_j\geq 2$. Multiplication with the inverse of the determinant shows that the point (5) of our lemma holds.
\end{proof}
\medskip

\noindent An important invariant associated with a Hermitian metric on a vector bundle is its curvature. In our setting, and despite the singularities of $\omega_D$, one has the next well-known fact.  

\begin{lemme}\label{mixbound}\cite{HenriII} There exists a positive constant $C> 0$ so that the inequality 
	\begin{equation}
		-C\omega_D\otimes \Id_{T_X}\leq \sqrt{-1}\Theta(T_X, \omega_D)\leq C\omega_D\otimes \Id_{T_X}
	\end{equation}
	holds on $X\setminus D$. 	
\end{lemme} 

\begin{proof}
There are many ways in which this statement can be established. For example, one could simply refer to \cite{HenriII} (and set the parameter $\ep$ in \emph{loc. cit.} to zero). An alternative argument is the following. We choose coordinates $z= (z_1,\dots, z_n)$ on an open subset $\Omega\subset X$ as in Lemma \ref{coeff}, and consider the 
finite map 
\begin{equation}\label{ram1}
\pi: \mathbb D\to \Omega, \qquad \pi(w):= (w_1^k,\dots, w_{p+q}^k, w_{p+q+1},\dots,  w_n).
\end{equation}
Then $\pi^\star \omega_D$ is a metric with Poincaré singularities along $\pi^{-1}(\Delta_2)$. If so, the curvature is indeed 
two-sided bounded as one can see by using the quasi-coordinates, see \cite{Kob84}.
\end{proof}
\medskip

\noindent It is easy to see that the Sobolev inequality does not hold for metrics with singularities of Poincar\'e-type -- which is the case for $\omega_D$. However, we have the following substitute.
\medskip

\begin{lemme}\label{sobolev}\cite{Biq97}
There exist positive constants $C> 0, \delta_n> 0$ as well as $a$ such that the following inequality holds true
\begin{equation}\label{end16} 
	\left(\int_X \rho^{2+ 2\delta_n}dV_{\omega_D}\right)^{\frac{1}{1+ \delta_n}}\leq C\int_X(\rho^2+ |d\rho|^2_{\omega_D})d\mu_1
\end{equation}
Here $d\mu_1$ is the measure given by $\prod \log^2|s_i|^2dV_{\omega_D}$, where $\prod s_i= 0$ is the equation of $\Delta_2$.
\begin{proof} It follows from the Euclidean version of this inequality, via the use of quasi-coordinates on the open subset 
	$\mathbb D^{\star p}\times \mathbb D^{n-p}.$
	
	\noindent For each $0< \delta< 1$ we define the map
	\begin{equation}\label{g22}
		\phi_{\delta}: \frac{3}{4}{\mathbb D}\to {\mathbb D}^{\star},\quad \xi\to \exp\left(-\frac{1+ \delta}{1-\delta}
		\frac{1+\xi}{1-\xi}\right).
	\end{equation}
	Then given $\eta= (\delta_1,\dots, \delta_p)$ we have an induced application
	\begin{equation}\label{end25}
		\Phi_{\eta}: \frac{3}{4}{\mathbb D}^n\to {\mathbb D}^{\star p}\times \mathbb D^{n-p},\quad \xi\to \big(\phi_{\delta_1}(\xi_i), \xi_{j}\big),
	\end{equation}
	where $i=1,\dots, p$ and $j=p+ 1,\dots, n$. Here the first $p$ coordinates are supposed to be the local equations of $\Delta_2$.
	
	Now for each $l\geq 1$ we define the parameter $\displaystyle \delta_l:= \frac{k 2^l-1}{k 2^l+ 1}$. The degree of the corresponding quasi-coordinate function $\displaystyle \phi_l:= \phi_{\delta_l}$ is of order $2^l$. In what follows, we will use the same letter $l$ to denote a multi-index \[l= (l_1,\dots, l_p)\]
	and the induced map will be $\Phi_l$.
	\smallskip
	
	\noindent In addition, there exists a constant $C> 0$ such that we have the following inequalities
	\begin{equation}\label{g24}
		C^{-1}\sum_{l}\frac{1}{2^{|l|}}\int F_ld\lambda\leq \int FdV_{\rm Poinc}\leq C\sum_{l}\frac{1}{2^{|l|}}\int F_ld\lambda
	\end{equation}  
	as well as
	\begin{equation}\label{g23}
		C^{-1}\sum_{l}\frac{1}{2^{|l|}}\Vert\nabla^j(f_l)\Vert_{L^2}\leq \Vert\nabla_{\rm Poinc}^j(f)\Vert_{L^2}\leq C\sum_{l}\frac{1}{2^{|l|}}\Vert\nabla^j(f_l)\Vert_{L^2}
	\end{equation}  
	where $F$ is a positive integrable function, $f_l:= f\circ \phi_l$ (by abuse of notation...). Here $f$ is defined on
	$\mathbb D^{\star p}\times \mathbb D^{n-p}$, endowed with the standard metric (Poincaré on the first $p$ components and Euclidean on the remaining ones), and $f_l$ is defined by
	$\frac{3}{4}\mathbb D$. The constant ``$C$'' is uniform as long as $j$ belongs to a finite family of indexes (which will be our case).
	
	\noindent Let now $f$ be a function defined on $\mathbb D^{\star p}\times \mathbb D^{n-p}$, which vanishes near the boundary $\partial \mathbb D^n$.
	If we denote by $\displaystyle \delta_n= \frac{n}{n-1}$, then we have 
	\begin{equation}\label{g25}
		\int_{\mathbb D^{\star p}\times \mathbb D^{n-p}}|f|^{2+ 2\delta_n}dV_{\rm Poinc}\leq C\sum_{l}
		\frac{1}{2^{|l|}}\int_{\frac{3}{4}\mathbb D^n} |f_l|^{2+ 2\delta_n}d\lambda
	\end{equation}  
	which implies the inequality
	\begin{equation}\label{g26}
		\left(\int_{\mathbb D^{\star p}\times \mathbb D^{n-p}}|f|^{2+ 2\delta_1}dV_{\rm Poinc}\right)^{1/1+\delta_n}\leq C\sum_{l}
		\frac{1}{2^\frac{|l|}{1+ \delta_1}}\left(\int_{\frac{3}{4}\mathbb D^n} |f_l|^{2+ 2\delta_1}d\lambda\right)^{1/1+\delta_n}.
	\end{equation}
	For each of the terms of the right hand side of \eqref{g26} we apply the usual
	Sobolev inequality, and we get
	\begin{equation}\label{g27}
		\left(\int_{\mathbb D^{\star p}\times \mathbb D^{n-p}}|f|^{2+ 2\delta_n}dV_{\rm Poinc}\right)^{1/2+2\delta_n}\leq
		C\sum_{l}
		\frac{1}{2^\frac{|l|}{2+ 2\delta_n}}\int_{\frac{3}{4}\mathbb D^n}(|f_l|+ |\nabla f_l|)d\lambda.
	\end{equation}
	We observe now that the following 
	\begin{equation}\label{g28}
		\frac{1}{2^\frac{l}{2+ 2\delta_1}}= \frac{1}{2^l} (2^l)^{\frac{1+ 2\delta_n}{2+ 2\delta_n}}\leq \frac{1}{2^l} \log\frac{1}{|\phi_l(\xi)|}
	\end{equation}
	holds for each $l$ and any $\xi\in \frac{3}{4}{\mathbb D}$, so summing up we get
	\begin{equation}\label{g29}
		\left(\int_{\mathbb D^{\star p}\times \mathbb D^{n-p}}|f|^{2+ 2\delta_n}dV_{\rm Poinc}\right)^{1/2+2\delta_n}\leq
		C\int_{\mathbb D^{\star p}\times \mathbb D^{n-p}}(|f(z)|+ |\nabla f(z)|)d\mu_1,
	\end{equation}
	where we have used the first part of the inequality \eqref{g23}. 
\end{proof}

\end{lemme}


 \subsubsection{Metric structure on the logarithmic tangent bundle of $X_0$}\label{metriclogTX}
 In what follows, an $L$-valued $k$-form with log poles along $\Delta_1$ on $X_0$ will simply be a global section of the bundle
\begin{equation}\label{log1}
F_k:= \oplus_{p+q= k}\Omega_X^p(\log\Delta_1)\otimes L\otimes\Lambda^{q}\ol T^\star_X.	
\end{equation}
That is to say, it represents a (formal) sum of $(0,q)$-forms with values in the holomorphic vector bundle $\Omega_X^p(\log\Delta_1)\otimes L$, such that the total degree $p+q$ is equal to $k$. 
\smallskip

\noindent Of course, in general differential forms with log poles are not $L^2$-integrable with respect to $\omega_D$ introduced above, so we first have to agree about the scalar product on $F_k$ we want to define. This will involve a modification of $\omega_D$ by adding so-called "cylindrical singularities" along $\Delta_1$. 
\smallskip 

\noindent On the log tangent bundle $T_X\langle\Delta_1\rangle$ 
we introduce the following Hermitian metric
\begin{equation}\label{log2}
g_D:= \omega_D+ \sqrt{-1}\sum_{i\in I} \frac{D's_i\wedge \ol {D's_i}}{|s_i|^2}
\end{equation}
where the $D'$ in the RHS of \eqref{log2} is induced by an arbitrary, smooth metric on the bundle $\mathcal O(Y_i)$, with $(s_i= 0)= Y_i$ and 
$\Delta_1= \sum_{i\in I} Y_i$. 

\begin{remark} As such, this is only defined in the complement of $\Delta_2+ \Delta_3$, with Poincar\'e singularities and conic singularities along $\Delta_2$ and $\Delta_3$, respectively. \end{remark}
\smallskip

Let $u, v$ be two sections of the bundle $F_k$. Their scalar product is defined as follows
 \begin{equation}\label{log3}
\langle u, v \rangle_{L^2}:= \int_{X_0} \langle u, v \rangle_{g_D, \omega_D}e^{-\varphi_L}dV_{\omega_D}
\end{equation}
where the notation is supposed to suggest that we are using the metric induced by $g_D$ for the "holomorphic vector bundle" part $\Omega_X^p(\log\Delta_1)$
and the metric $\omega_D$ for the conjugate-holomorphic factor of $u$ and $v$. 
\smallskip

\noindent In order to state the analogue of Lemma \ref{coeff} in this new setting, we first need a frame for $T_X\langle\Delta_1\rangle$, as follows.
Let $U\subset X$ be an open coordinate set, $w= (w_1, \dots, w_n)$ are local holomorphic coordinates defined on $U$, such that 
\begin{equation}\label{log222}U\cap \Delta_1= (w_1\dots w_r= 0).\end{equation} 
We also define 
\begin{equation}\label{log2007}
	e_i:= w_i\frac{\partial}{\partial w_i}
\end{equation}
for $ i=1,\dots r$ and 
\begin{equation}\label{log2008}
	e_j:= \frac{\partial}{\partial w_j}
\end{equation}
for $j\geq r+1$. 

Let $(h_{\beta\ol\alpha})$ be the coefficients of the Hermitian metric 
$g_D$ with respect to this frame. As usual we denote by $(h^{\ol\alpha \beta})$ the coefficients of the inverse matrix. We need a detailed understanding of these functions. The good news is that this was already done: $g_D$ is a -singular- Hermitian metric on the vector bundle $T_X\langle \Delta_1\rangle$, and it is the perfect analogue of $\omega_D$, in which $T_X$ was replaced with its logarithmic version. In what follows, we will simply write explicitly the estimates for the coefficients of $g_D$.
\smallskip

\noindent We can assume that the local coordinates $w$ on $U$ have been chosen such that the equalities
\begin{equation}\label{log223}
U\cap \Delta_2= (w_{r+1}\dots w_{r+ p}= 0)\quad U\cap \Delta_3= (w_{r+p+1}\dots w_{r+ p+q}= 0)
\end{equation}
hold, where we recall that we already assume $\displaystyle U\cap \Delta_1= (w_1\dots w_r= 0)$. Notice that these notations are consistent with the ones in the previous subsection, modulo the shift by $r$. With respect to this frame/coordinates, the metric $g_D$ can be written as follows (cf. \eqref{m3} for the analogous expression for $\omega_D$).
\begin{equation}\label{log30}
\begin{split}
h_{i\ol j}(w)= & \frac{\delta_{ij}\delta_{r+1, r+p}^i}{|w_i|^2\log^2|w_i|^2}+ \frac{\delta_{r+1, r+p}^iA_i}{w_i\log^2|w_i|^2}+
\frac{\delta_{r+1, r+p}^jA_{\ol j}}{\ol w_j\log^2|w_j|^2}\\
+ & \frac{\delta_{ij}\delta_{r+p+1, r+p+q}^ie^{\phi_i}}{|w_i|^{2-2/k}}+ \frac{\delta_{r+p+1, r+p+q}^i|w_i|^{2/k}\wt A_i}{w_i}+ 
\frac{\delta_{r+p+1, r+p+q}^j|w_j|^{2/k}\wt A_{\ol j}}{\ol w_j}\\
+ & \sum_{l=r+1}^{r+p}\frac{C_{l0}+ C_{l1}\log|w_l|^2}{\log^2|w_l|^2}+ \sum_{l= r+p+1}^{r+p+q}\wt C_l |w_l|^{2/l}
+  \rho_{i\ol j}(w)\\
\end{split}
\end{equation}
where the matrix $(\rho_{\beta\ol \alpha})$ is positive definite and has smooth coefficients. We denote by $\delta_{ij}$ the Kronecker symbol, and $\delta_{\alpha, \beta}^i$ equals 1 if $i\in \{\alpha,\dots, \beta\}$ and zero if not. 
The functions $A_k, A_{\ol m}, C_l$ and $\wt A_k, \wt A_{\ol m}, \wt C_l$ are smooth.
\medskip

\noindent We see that the expression of the coefficients $(h_{i\ol j})$ in \eqref{log30} is completely identical to \eqref{m3} (of course, the frame with respect to which these coefficients are computed is not the same).
The following statements hold, in which the functions $(\psi)$ are smooth in the complement of $D$, such that 
\begin{equation}\label{log31}\sup_{U\setminus D}|\psi|_{\omega_D}< 1/2, \qquad \sup_{U\setminus D}|d\psi|< \infty.\end{equation}
We will not give any detail about the proof, since it closely follows the one in Lemma \ref{coeff}.
\begin{itemize}

\item[\rm (1)] For $i= r+1,\dots, r+p$ we have
\[
h_{i\ol i}(w)= \frac{1}{|w_i|^2\log^2|w_i|^2}\big(1+ \psi_i(z)\big)
\]
and for $i= r+p+1,\dots, r+p+q$ the corresponding diagonal coefficient can be written as follows
\[
\frac{1}{|w_i|^{2-2/k}}\big(1+ \psi_i(w)\big).
\]

\item[\rm (2)] In case $i\leq r$ or $i\geq p+q+ r+1$, we have
\[ 
h_{i\ol i}(w)= \tau_{i\ol i}(w)\big(1+ \psi_i(w)\big)\]
where $\tau_{i\ol i}$ are smooth, strictly positive.

\item[\rm (3)] If $i\neq j$ are both in $r+1,\dots, r+p$, then we have
\[
h_{i\ol j}(w)= \frac{\psi_{ij}(w)}{w_i\ol w_j\log^2|w_i|^2\log^2|w_j|^2}.
\]

\item[\rm (4)] If $r\leq i\leq p+r$ and $p+r+1\leq j\leq p+q+r$, then we have  
\[ 
h_{i\ol j}(w)= \frac{|w_j|^{2/k}}{w_i\ol w_j\log^2|w_i|^2}\psi_{ij}(w).
\]

\item[\rm (5)] If $r\leq i\leq p+r$ and $j\geq p+q+r+1$ or $j\leq r$ then we have  
\[
h_{i\ol j}(w)= \frac{1}{w_i\log^2|w_i|^2}\psi_{ij}(w)
\] 

\item[\rm (6)] If $i\neq j$ and $p+r\leq i, j\leq p+q+r$, then we have  
\[ 
h_{i\ol j}(w)= \frac{|w_i|^{2/k} |w_j|^{2/k}}{w_i\ol w_j}\psi_{ij}(w)
\] 

\item[\rm (7)] If $p\leq i \leq p+q$ and $j\geq p+q+1$, then we have  
\[
h_{i\ol j}(w)= \frac{|w_i|^{2/k}}{w_i}\psi_{ij}(w).
\] 

\item[\rm (8)] If $i\neq j$ are both greater than $p+q+r+1$ or both smaller than $r-1$, then we have
\[
h_{i\ol j}(w)= \psi_{ij}(w)
\]
i.e. a function which is bounded in $\mathcal C^1$-norm.

\item[\rm (9)] Finally, if $r+1\leq i\leq p+r$ and $j\geq p+r+1$ or $j\leq r-1$ then we have 
\[ 
h_{i\ol j}(w)= \frac{\psi_{ij}(w)}{w_i\log^2|w_i|^2},\qquad  h_{j\ol i}(w)= \frac{\psi_{ji}(w)}{\ol w_i\log^2|w_i|^2}
\]

\item[$(\star)$] It follows that 
\[
\det\big(h_{i\ol j}(w)\big)= \big(1+ \psi(w)\big)\prod_{i=r+1}^{r+p}\frac{1}{|w_i|^2\log^2|z_i|^2}\prod_{j=p+r+1}^{p+q+r}\frac{1}{|w_i|^{2-2/k}}.
\]

\item[\rm (a)] Let $(h^{\ol j i})_{i, j}$ be the inverse of $(h_{i\ol j})$. Then for any $i, j\in \{r+1, \dots, r+p\}$ such that $i\neq j$ we have 
\[h^{\ol i i}(w)= (1+  \psi_i(w))|w_i|^2\log^2|w_i|^2, \qquad h^{\ol j i}(w)= \psi_{ij}(w))w_i\ol w_j.\]
The functions $\psi_i, \psi_{ij}$ are smooth on $U\setminus D$, and verify the condition \eqref{log31}.
\smallskip

\item[\rm (b)] For each $i\in \{r+1, \dots, r+p\}$ and $j\in \{r+p+1,\dots, r+p+q\}$ we have 
\[ h^{\ol j i}(w)= \psi_{ij}(w)w_i\ol w_j.\]
\smallskip

\item[\rm (c)] For each $i\in \{r+1, \dots, r+p\}$ and $j\geq p+q+r+1$ we have 
\[ h^{\ol j i}(w)= \psi_{ij}(w)w_i.\]
\smallskip

\item[\rm (d)] For each $i, j\in \{r+p+1, \dots, r+p+q\}$ such that $i\neq j$ we have 
\[h^{\ol i i}(w)= (1+  \psi_i(w))|w_i|^{2-2/k}, \qquad h^{\ol j i}(w)= \psi_{ij}(w)w_i\ol w_j.\]

\item[\rm (e)] For each $i\in \{r+p+1, \dots, p+q+r\}$ and $j\geq p+q+r+1$ we have 
\[h^{\ol j i}(w)= \psi_{ij}(w)w_i.\]
\smallskip

\item[\rm (f)] Finally, in case both indexes $i\neq j$ are smaller than $r$ or bigger than $r+p+1$, then we have 
\[h^{i\ol i}(w)= 1+  \psi_i(w), \qquad h^{i\ol j}(w)= \psi_{ij}(w)\]
where the functions $\psi$ have the same regularity properties as before.
\end{itemize}
\medskip

\noindent As a consequence, we have the following statement.

\begin{lemme}\label{logmixbound} There exists a positive constant $C> 0$ so that the inequality 
	\begin{equation}\label{form301}
		-C\omega_D\otimes \Id_{T_X\langle\Delta_1\rangle}\leq \sqrt{-1}\Theta(T_X\langle\Delta_1\rangle, g_D)\leq 
		C\omega_D\otimes \Id_{T_X\langle\Delta_1\rangle}
	\end{equation}
	holds on $X\setminus D$. 	
\end{lemme} 

\begin{proof}
The argument is identical to the one already given in Lemma \ref{mixbound}, so we will not repeat it here. We however add a (hopefully helpful) remark: with respect to the frames \eqref{log2007}, \eqref{log2008}, the coefficients of the metric $g_D$ are practically identical to those of the 
metric $\omega_D$. This explains \eqref{form301}.
\end{proof}


\subsection{A-priori inequalities for $\Delta_K$}\label{apriori}
Let $X$ be a compact Kähler manifold and let $D$ be a snc divisor in $X$. 
We denote by 
$$\nabla: L \to L\otimes \Omega_X ^1 (\log D)$$ 
a flat holomorphic connection, and let $h_L$ be a harmonic metric on $L$. Then $h_L$ is smooth on $X\setminus D$ and $i\Theta_{h_L} (L)=0$ on $X\setminus D$. Also, notice that $\displaystyle -2\theta_0 = D'_{h_L} -\nabla\in H^0 (X, \Omega^1 _X (\log D))$ has nothing to do with any metric on $X$.
\medskip

\noindent The motivation for what will follow is to be able to state all our results in a "uniform" manner, i.e. regardless to the fact that our forms have or haven't log poles along $\Delta_1$. First, consider a covering of $X$ with a finite set of coordinate open subsets $(\Omega_i, w_i)$, see \eqref{log222}, \eqref{log223}. If $(\Omega, w)$ is one of them, we define the local ramified maps
\begin{equation}\label{clog}
\pi: \DD\to \Omega, \qquad \pi(\xi)= (\xi_1,\dots, \xi_r, \xi_{r+1}^k,\dots,\xi_{r+p+q}^k, \xi_{r+p+q+1},\dots, \xi_n).
\end{equation}
\smallskip

\begin{convent}\label{conv}{\rm From this point on, and unless explicitly specified otherwise, we will consistently use the following conventions.
		Recall that $X_0= X\setminus |\Delta_2|$ is the complement of the support of $\Delta_2$.
		
\begin{enumerate}
\smallskip

\item[${\rm C}_0$] We denote by $\mathcal E_k$ one of the following vector bundles
\begin{equation}\label{firstcases}
	\oplus_{r+s= k}\Lambda^{r, s}T^\star_X\otimes L, 
	\end{equation}
	\begin{equation}\label{secondcases}
	\oplus_{r+s= k}\Lambda^{r, s}T^\star_X\langle \Delta_1\rangle\otimes L
\end{equation}

\item[${\rm C}_1$] \emph{An $L$-valued $k$-form on $X_0$} will be a smooth section $C^\infty (X_0, \cE_k)$ for $\cE_k$ equals to \eqref{firstcases}
or \eqref{secondcases}
\smallskip

\item[${\rm C}_2$] \emph{The holomorphic form $\theta_0$ operates naturally on $\cE_k$.} 
\smallskip

\item[${\rm C}_3$] \emph{The adjoint of the operator $\theta_0$ will be denoted by $\theta_0 ^\star$:}  
\begin{itemize}

\item In the case \eqref{firstcases},  its adjoint $\theta_0 ^\star$ is computed by using the metric
$\omega_D$ on $\cE_k$.

\item In the case \eqref{secondcases}, the adjoint $\theta_0 ^\star$ is computed by using the cylindrical metric $g_D$.
\end{itemize}
\smallskip

\item[${\rm C}_4$] \emph{We denote by $D_K= D'_K+ D''_K$ one of the following differential operators:} 
\begin{itemize}
\item In the case \eqref{firstcases}, we introduce
\[D'_K:= D'_{h_L}+ \theta_0+ \ol\theta_0, \qquad D''_K:= \dbar + \theta_0- \ol\theta_0\]
where locally $D'_{h_L}= \partial -\partial\varphi_L$ is the $(1,0)$-part of the Chern connection.

\item In the case  \eqref{secondcases}, our operator will be
\[D'_K:= D'_{h_L}+ \theta_0+ \ol\theta_0, \qquad D''_K:= \dbar + \theta_0- \ol\theta_0\]
where $D'_{h_L}$ is the operator on $\cE_k$ induced by the -twisted- exterior derivative $ \partial -\partial\varphi_L$ cf. Remark \ref{operdif} below. 
In particular, \emph{it is
not} the $(1,0)$-part of the Chern connection induced by the metrics $g_D$ and $h_L$. \end{itemize}
\smallskip

\item[${\rm C}_5$] \emph{An $L$-valued $k$-form $u$ with compact support on $X_0$ is called smooth with respect to the conic structure $\big(X, (1-1/k)\Delta_3\big)$ if:}
\begin{itemize}

\item In the case \eqref{firstcases},  we request that 
$u$ is smooth on $X\setminus D$ and moreover the quotient
\[u_i:= \frac{1}{w^{k\nu}}\pi_i^\star u\]
extends smoothly across the pre-image of $\Delta_1+ \Delta_3$. Here $\nu_\alpha$ are the Lelong numbers of the curvature current of $(L, h_L)$ along $D_\alpha$, and $\displaystyle w^{k\nu}:= \prod_{\alpha= 1}^{p+q} w_{\alpha}^{k\nu_\alpha}$.  Moreover, $\pi$ is  the ramified cover
\eqref{clog}.

\item In the case  \eqref{secondcases}, we request that 
$u$ is smooth on $X\setminus D$ and moreover the quotient
\[u_i:= \frac{1}{w^{k\nu}}\pi_i^\star u\]
defines a form with log-poles on the inverse image of $\Delta_1$. 
\end{itemize}
\end{enumerate}
}
\end{convent}

\medskip

\begin{remark}\label{operdif}
As one can see, in $C_4$, the differential operators $D'_K$ we are using in the smooth/logarithmic case are basically the same, modulo a small identification we now make precise. The $D'_K$ corresponding to forms with log-poles is defined by using the fact that sections of the dual of the vector bundle $T_X\langle \Delta_1\rangle$ identify with 
smooth forms with log poles. The exterior $D'_{h_L} :=\partial-\partial \varphi_L$-derivative, composed with the map 
\[\Omega_X^1\to \Omega_X^1(\log\Delta_1)\]
is our operator.

Moreover, note that this is equally consistent with the $(1,0)$-derivative of forms with log-poles in the sense of distributions (unlike the $(0,1)$-derivative, they don't have any residue on the components of $\Delta_1$). 
\end{remark}
\medskip

\noindent Now let $\Delta_K$ be the commutator $\displaystyle \Delta_K:= [D_K, D_K^\star].$
The key estimate in this article reads as follows.

\begin{thm}\label{keyinequ}
	In the setting of this subsection, there exists a positive constant $C> 0$ such that 
for any $L$-valued form $u$ with compact support in $X\setminus \Delta_2$ and which moreover is smooth in the orbifold sense, 	
	the inequality
	$$\|u\|^2_{L^2}  + \int_X\langle \Delta_K u, u\rangle dV_{\omega_D} \geq \frac{1}{C} \int_X \big(|u|^2 |\theta_0|^2+ \langle [D_{h_L} , D^\star _{h_L}] u, u\rangle \big)dV_{\omega_D} $$
holds. 
\end{thm}
\medskip

\noindent Mainly, due to the fact that the metric $\omega_D$ is Kähler and $g_D$ is not, we have adopted the following strategy for the proof we present next. We first complete the arguments in the smooth case \eqref{firstcases}. Then we explain how to adapt our arguments in the logarithmic case \eqref{secondcases}: the differences are insignificant, but they exist.
To begin with, we collect next a few facts about some commutator operators which will be needed in what follows.

\subsubsection{Commutators in the smooth case \eqref{firstcases}}
We expand the expression of $\Delta_K$ by using the equality $D_K =D_{h_L} +2\theta_0$, i.e.
\begin{equation}\label{form40}
[D_K, D_K ^\star]= [D_{h_L} , D_{h_L} ^\star] + 4[\theta_0, \theta_0^\star] + 2[D_{h_L}, \theta_0^\star] +2[\theta_0, D_{h_L} ^\star]\end{equation}

\noindent We evaluate next the operators $\theta_0^\star$ and $[D_{h_L}, \theta_0^\star]$ in the case \eqref{firstcases}.
		
\begin{lemme}\label{calc,I} Let $(\Omega, z)$ be a coordinate system of $X$, centred at an arbitrary point $x_0\in X$. We write 
\[\theta_0|_\Omega= \sum \theta_idz_i\]	
Then following formulas hold true
\begin{equation}\label{form1} 
\theta_0^\star= \sum_{i,k}\ol\theta_i \omega^{\ol i k}\frac{\partial}{\partial z_k}\rfloor	
\end{equation}		 
as well as
\begin{equation}\label{form5}
	\begin{split}
[D_{h_L}, \theta_0^\star]u= & \sum \ol\theta_i \omega^{\ol i k}(\partial_k u_{I\ol J}- u_{I\ol J}\partial_k\varphi) dz^I\wedge dz^{\ol J}\otimes e_L\\
+ & \sum u_{I\ol J}d(\ol\theta_i \omega^{\ol i k})\wedge \frac{\partial}{\partial z_k}\rfloor dz^I\wedge dz^{\ol J}\otimes e_L\\
\end{split}
\end{equation}
where $\displaystyle \frac{\partial}{\partial z_k}\rfloor$ means the contraction with respect to the respective vector field.  
\end{lemme}		

\begin{proof}
Indeed, the formula \eqref{form1} is a consequence of the fact that we have
\[\langle u, dz_i\wedge v\rangle= \langle \omega^{\ol i k}\frac{\partial}{\partial z_k}\rfloor u, \wedge v\rangle.\]

In order to establish the second formula, we perform next a calculation in coordinates. Let
\[u= \sum u_{I\ol J}dz^I\wedge dz^{\ol J}\otimes e_L\]
be a $k= |I|+ |J|$-form with values in $L$. Then we have 
\[D_{h_L} u= \sum (d u_{I\ol J}- u_{I\ol J}\partial\varphi_L)\wedge dz^I\wedge dz^{\ol J}\otimes e_L\]
and it follows that 
\[\theta_0^\star\circ D_{h_L} u= \sum \ol\theta_i \omega^{\ol i k}\frac{\partial}{\partial z_k}\rfloor (d u_{I\ol J}- u_{I\ol J}\partial\varphi_L)\wedge dz^I\wedge dz^{\ol J}\otimes e_L\]
and by the Leibniz rule we get
\begin{equation}\label{form3}
\begin{split}
\theta_0^\star\circ D_{h_L} u= & -\sum \ol\theta_i \omega^{\ol i k}(d u_{I\ol J}- u_{I\ol J}\partial\varphi_L)\wedge \frac{\partial}{\partial z_k}\rfloor dz^I\wedge dz^{\ol J}\otimes e_L \\
+ &\sum \ol\theta_i \omega^{\ol i k}(\partial_k u_{I\ol J}- u_{I\ol J}\partial_k\varphi_L) dz^I\wedge dz^{\ol J}\otimes e_L\\ 
\end{split}
\end{equation}
where we remark that the second line preserves the type of the form. Also, we use the notation $\partial_k$ for the holomorphic derivative in the direction 
$\displaystyle \frac{\partial}{\partial z_k}$.
Next we have
\begin{equation}\label{form4}
	\begin{split}
		D_{h_L} \circ \theta_0^\star u= & \sum \ol\theta_i \omega^{\ol i k}(d u_{I\ol J}- u_{I\ol J}\partial\varphi_L)\wedge \frac{\partial}{\partial z_k}\rfloor dz^I\wedge dz^{\ol J}\otimes e_L \\
		+ &\sum u_{I\ol J}d(\ol\theta_i \omega^{\ol i k})\wedge dz^I\wedge dz^{\ol J}\otimes e_L\\ 
	\end{split}
\end{equation}
and by adding the two expressions \eqref{form3} and \eqref{form4} we get
\begin{equation}
	\begin{split}\label{uniform}
[D_{h_L}, \theta_0^\star]u= & \sum \ol\theta_i \omega^{\ol i k}(\partial_k u_{I\ol J}- u_{I\ol J}\partial_k\varphi_L) dz^I\wedge dz^{\ol J}\otimes e_L\\
+ & \sum u_{I\ol J}d(\ol\theta_i \omega^{\ol i k})\wedge \frac{\partial}{\partial z_k}\rfloor dz^I\wedge dz^{\ol J}\otimes e_L\\
\end{split}
\end{equation}
which is what we wanted to prove.
\end{proof}
\smallskip

\noindent The formula \eqref{form5} can be used in order to express the "pure" part 
$[D_h^\prime, \theta^\star]$ of the commutator by using the covariant derivative.
We recall here that in general setting, given a $(p, q)$ form $u$ on a Kähler manifold $(X, \omega)$ 
with values in a line bundle $(L, h_L)$, its (1,0)-covariant derivative $\nabla^{1,0}$ is induced by the operators
\[\nabla^{1,0} e= -\partial\varphi_L\otimes e, \qquad \nabla_{k}^{1,0} dz_i= -\Gamma^{i}_{jk}dz_j, \qquad \nabla_{k}^{1,0} dz_{\ol i}= 0\] 
together with the Leibniz rule. We denote by $\Gamma^{i}_{jk}$ the Christoffel symbols of $(X, \omega_D)$. 

\begin{lemme}\label{calc,II}
Let $\nabla^{1,0}$ be the $(1, 0)$-covariant derivative associated to $(X, \omega_D)$ and $(L, h_L)$. Then we have
\[[D_{h_L}^\prime, \theta_0^\star]u= \theta_0^\star \rfloor \nabla^{1,0} u\]
at each point of $X\setminus D$, where $\nabla^{1,0} u$ is a (1,0)-form with values in the $\mathcal C^\infty$ bundle $E:= \Lambda^pT^\star_X\otimes \Lambda^q\ol T^\star_X\otimes L$ (so that the contraction with $\theta_0^\star$ gives a section of $E$).
\end{lemme}

\begin{proof}
It is a nice exercise to prove this identity working with an arbitrary coordinate system. However, choosing \emph{geodesic} coordinates at some point $x_0$ is much more efficient. Indeed, at $x_0$ formula \eqref{form5} reads
\[[D_{h_L} ^\prime, \theta_0^\star]u_{x_0}= \sum \ol\theta_i (\partial_i u_{I\ol J} - u_{I\ol J} \partial_i \varphi_L )dz^I\wedge dz^{\ol J}\otimes e_L\]
where we use that $\partial\ol \theta_i=0$ since the coefficients of $\theta_0$ are holomorphic in the complement of $D$. 

On the other hand, we have
\[\nabla^{1,0}u_{x_0}= \sum (\partial_i u_{I\ol J} - u_{I\ol J} \partial_i \varphi_L)dz_i\otimes dz^I\wedge dz^{\ol J}\otimes e_L\]
(because the other terms involve the Christoffel symbols, which vanish at $x_0$) so the lemma follows immediately.
\end{proof}
\smallskip

\noindent In conclusion, we have
\begin{equation}\label{form6}
[D_{h_L}, \theta_0^\star]u= \theta_0^\star \rfloor \nabla^{1,0} u+ [\dbar, \theta_0^\star]u
\end{equation}
and the local expression of the second term of the RHS of \eqref{form6} is equal to
\begin{equation}\label{form7}
[\dbar, \theta_0^\star]u= \sum u_{I\ol J}\dbar(\ol\theta_i \omega^{\ol i k})\wedge \frac{\partial}{\partial z_k}\rfloor dz^I\wedge dz^{\ol J}\otimes e_L.
\end{equation}
\medskip

\noindent Another simple calculation (which we skip) shows that the commutator $[\theta_0, \theta_0 ^\star]$ equals 
\begin{equation}\label{form15}
[\theta_0, \theta_0^\star]= |\theta_0|^2_{\omega_D}	
\end{equation}
that is to say, it is given by the norm of $\theta_0$ with respect to the Poincar\'e
metric $\omega_D$. 
\medskip

\subsubsection{Proof of Theorem \ref{keyinequ}, smooth case}\label{proofsmth} Now we would like to prove Theorem \ref{keyinequ} in the case \eqref{firstcases}. 
We will argue as follows:
\begin{enumerate}

\item[\rm (i)] We first show that there exists a constant $C> 0$ such that global scalar product $\langle [\dbar, \theta_0^\star]u, u\rangle$ is bounded by  
\[\frac{C}{\ep}\Vert u\Vert^2+  C\ep \int_X|\theta_0|_{\omega_D}^2|u|^2dV_{\omega_D},\]
for any positive $\ep> 0$. 

\item[\rm (ii)] For the term $\langle \theta_0^\star \rfloor \nabla^{1,0} u, u\rangle= \langle \nabla^{1,0} u, \theta_0\otimes u\rangle$, by Cauchy, we have
\[2|\langle \nabla^{1,0} u, \theta_0\otimes u\rangle|\leq \frac{3}{4}\Vert \nabla^{1,0} u\Vert^2+ \frac{4}{3}\int_X|\theta_0|^2|u|^2dV\]

\item[\rm (iii)] We finally play the joker: modulo terms of order zero $\Vert u\Vert^2$, we have cf.\cite{Siu82}
\[\Vert \nabla^{1,0} u\Vert^2\simeq \Vert \dbar u\Vert^2+ \Vert \dbar^\star u\Vert^2.\]

\item[\rm (iv)] We have \[\Vert \dbar u\Vert^2+ \Vert \dbar^\star u\Vert^2= \Vert D'_{h_L} u\Vert^2+ \Vert D_{h_L}^{\prime \star}  u\Vert^2= \frac{1}{2}\langle [D_{h_L}, D_{h_L} ^\star]u, u\rangle\] since the curvature of $(L, h_L)$ is non-singular across $\Delta_1$, and $u$ has conic singularities on $\Delta_3$ (see \cite{CP}).
\end{enumerate}

\begin{proof}
	The points $(ii)$ and $(iv)$ are clear. Concerning the point $(iii)$, it is a consequence of the Bochner-type result
	\begin{equation}\label{form31}
		\Vert \dbar u\Vert^2+ \Vert \dbar^\star u\Vert^2+ C\Vert u\Vert^2\geq  \Vert \nabla^{1,0} u\Vert^2 
	\end{equation} 
	for any form $u$ with compact support in $X\setminus \Delta_2$ and smooth in orbifold sense. We will explain it in more details in the subsection \ref{bochner1}. 
\smallskip

For the point $(i)$, we will use the equality \eqref{form7}, together with Lemma \ref{coeff}.
Given the equality \eqref{form7}, we have to find an upper bound for the expression
\begin{equation}\label{form27}
|\dbar(\ol\theta_i \omega^{\ol i j})|_{\omega_D}\left|\frac{\partial}{\partial z_j}\right|_{\omega_D}
\end{equation}
which we next do using a case-by-case analysis. 

\noindent $\bullet$ \emph{We assume that $i\neq j$.} If $\max (i, j)\leq p$, then \eqref{form27} becomes
\[\frac{1}{|z_j|\log1/|z_j|}|\dbar(z_j\psi_{ij})|_{\omega_D},\]
by the third point of Lemma \ref{coeff}, and the fact that the coefficients $\theta_i$ can be written as $f_i(z)/z_i$, with $f_i$ a local holomorphic function. Since the differential of $\psi_{ij}$ with respect to $\omega_D$ is bounded, we have 
\[\sup_{\Omega\setminus D}|\dbar(\ol\theta_i \omega^{\ol i j})|_{\omega_D}\left|\frac{\partial}{\partial z_j}\right|_{\omega_D}< \infty.\]

If $i\leq p$ and $p+q\geq j\geq p+ 1$, then \eqref{form27} is equal to $\displaystyle |\dbar \psi_{ik}|_{\omega_D}|z_j|^{1/k}$ (the $\ol z_i$ is cancelled by the pole of $\ol \theta_i$), so it is clearly bounded.

If $i\leq p$ and $j\geq p+ q+1$, then up to a bounded function the expression \eqref{form27} is equal to $\displaystyle |\dbar \psi_{ij}|_{\omega_D}$, so it is clearly bounded.

If $p+q\geq i, j\geq p+1$, then \eqref{form27} is equal to $\displaystyle \O(|z_j|^{1/k})$, and again, this is bounded.

If $p+q\geq i \geq p+1$, and $j\geq p+ q+1$ then \eqref{form27} is equal to $\displaystyle \O(1)$.

If both $i$ and $k$ are strictly greater than $p+ q$, things are clear.
\medskip

\noindent $\bullet$ \emph{We assume that $i= j$.} In case their common value is strictly greater than $p+q$, there is nothing to prove.  
The remaining cases are $i= j \leq p$, and $p+1\leq i= j \leq p+q$ so we have to deal with the following expressions
\[\frac{1}{|z_i|\log1/|z_i|}\dbar(\ol\theta_i |z_i|^2\log^2|z_i|^2), \qquad \frac{1}{|z_i|^{1-1/k}}\dbar (|z_i|^{2-2/k}\psi_i).\]
The residue $a_i$ of $\theta$ along $z_i= 0$ is non-zero for $i= 1,\dots , p$, the expression is bounded by $|a_i|\log1/|z_i|$ (up to a constant). In the second case a quick calculation shows that it is bounded as soon as $k\geq 2$.
\smallskip

\noindent Expressed in intrinsic terms, we have proved the existence of a constant $C> 0$ such that
\begin{equation}\label{form28}
\int_X\langle [\dbar, \theta_0^\star]u, u\rangle dV\leq C \int_X |u|^2\sum_{i=1}^m |a_i|\log1/|s_i| dV.
\end{equation}
The RHS of \eqref{form28} is clearly bounded by the RHS of (i) above, as consequence of the Cauchy inequality. 

\medskip

\noindent 
 \end{proof}
 
 Now we can prove Theorem \ref{keyinequ} in the case \eqref{firstcases}.
 
 \begin{proof}
 By combining the points $(iii)$ and $(iv)$, we have 
\begin{equation}\label{form32}
\frac{1}{2}(\Vert D_{h_L} u\Vert^2+ \Vert D_{h_L} ^\star u\Vert^2)+ C\Vert u\Vert^2\geq  \Vert \nabla^{1,0} u\Vert^2.
\end{equation} 
Together with $(i)$, $(ii)$ and Lemma \ref{calc,II},  by choosing $0<\ep\ll 1$ in $(i)$, we know that 
\begin{equation}\label{form33}
\begin{split}
\left|\int_X\langle \mathcal [D_{h_L}, \theta_0 ^\star]u, u\rangle dV\right|\leq &
\frac{3}{4}\int_X|\theta_0|_{\omega_D}^2|u|^2dV+ \frac{3}{16}\int_X\langle [D_{h_L}, D_{h_L} ^\star]u, u\rangle dV\\
& + C_\ep \Vert u\Vert^2 
\end{split}
\end{equation} 
for some constant $C_\ep$ depending on $\ep$.  By using the formula \eqref{form40}, we infer that  
\begin{equation}\label{form34}
C\Vert u\Vert^2+ \int_X\langle \Delta_K u, u\rangle dV\geq \frac{1}{4}\int_X\langle [D_{h_L}, D_{h_L} ^\star]u, u\rangle dV+ \int_X|\theta_0|_{\omega_D}^2|u|^2dV,
\end{equation} 
which ends the proof.
\end{proof}	
		
\medskip

\noindent Theorem \ref{keyinequ} admits the following version.

\begin{cor}\label{keylog} In the context of \Cref{keyinequ}, for any $a_\alpha\geq 0$ there exist positive constants $C_i$ such that the inequality
\begin{align}C_0\int_X|u|^2d\mu_a+ \int_X\langle \Delta_K u, u\rangle d\mu_a\geq & C_1 \int_X |\theta_0|_{\omega_D}^2|u|^2d\mu_a \cr 
+ & C_2 \int_X\langle [D_{h_L} , D^\star_{h_L}] u, u\rangle d\mu_a  
\nonumber
\end{align}	
holds. We denote by $d\mu_a$ the measure 
\[d\mu_a:= \prod_{i\in J}\log^{2a_i}(|s_i|^2+ \delta^2)dV_{\omega_D},\]
where $J\subset \{1,\dots m\}$ is an arbitrary subset, $\delta\in [0, 1]$. Notice that the constants above are independent of $\delta$. 	
\end{cor}		

\noindent The corollary follows immediately from Theorem \ref{keyinequ}, by applying it to the form
\[u_J:= u\prod_{i\in J}\log^{a_i}(|s_i|^2+ \delta^2).\]
Indeed, for example terms like $\displaystyle \int_X |D_K u_J|^2dV_{\omega_D}$ are handled by writing the formula
\[D_K u_J= \big(\prod_{i\in J}\log^{a_i}(|s_i|^2+ \delta^2)\big) D_Ku+ 
d\big(\sum a_i\log\log(|s_i|^2+ \delta^2)\big)\wedge u_J\]
and observing that the gradient of the function $\log\log^2(|s_i|^2+ \delta^2)$ with respect to $\omega_{D}$
is bounded independently of $\delta$ on $X\setminus \Delta_2$.

\subsubsection{Proof of Theorem \ref{keyinequ}, logarithmic case \eqref{secondcases}} In Subsection \ref{proofsmth}, 
Theorem \ref{keyinequ} was established for sections of the bundle
\[\oplus_{p+q= k}\Lambda^{p, q}T^\star_X\otimes L\]
with conic singularities on $\Delta_3$ and compact support in $X_0:= X\setminus |\Delta_2|$, that is to say, in the complement of the support of 
$\Delta_2$. Here we will do the same in logarithmic context  \eqref{secondcases}.

The next notions were already mentioned in Convention \ref{conv}, but nevertheless we present a few details about the differential operators on $(F_k)_{k\geq 0},$ see \eqref{log1}, which are of interest for us. The first one is the twisted exterior derivative
\begin{equation}\label{log4}
D_{h}: F_k\to F_{k+1}, \qquad D_{h}= D'_{h}+ \dbar 
\end{equation}
where $\dbar$ is interpreted as operator on the vector bundle $\Omega_X^p(\log\Delta_1)$ (therefore the residues along the components of $\Delta_1$ are not taken into account). 

The form $\theta$ equally induces a natural map
\begin{equation}\label{log5}
\theta: F_k\to F_{k+1}, 
\end{equation}
and let 
\begin{equation}\label{log6}
D_K: F_k\to F_{k+1}, \qquad D_K:= D_{h}+ 2\theta_0
\end{equation}
be the sum of \eqref{log4} and \eqref{log5}. Notice that the definition of $D_K$ is absolutely the same as in the absence of log poles. 

Let $D_{h}^\star$ be the formal adjoint of $D_{h}$ with respect to the scalar product 
\eqref{log3}, together with the corresponding Laplace operator
\[\Delta_K:= [D_K, D_K^\star].\]
As in the previous subsection, we evaluate the commutator $\displaystyle [D'_h, \theta_0^\star]$. To begin with, we introduce some notations necessary for the computations in coordinates that will follow. Let $U\subset X$ be an open coordinate set, $w= (w_1, \dots, w_n)$ are local holomorphic coordinates defined on $U$, such that 
\[U\cap \Delta_1= (w_1\dots w_r= 0).\] 
We also define 
\begin{equation}\label{log7}
e_i:= w_i\frac{\partial}{\partial w_i}, \qquad e^i:= \frac{dw_i}{w_i}
\end{equation}
for $ i=1,\dots r$ and 
\begin{equation}\label{log8}
e_j:= \frac{\partial}{\partial w_j}, \qquad e^j:= {dw_j}
\end{equation}
for $j\geq r+1$. For each ordered subset $I= (i_1<\dots < i_p)$ we denote by $e^I$ the exterior product 
$\displaystyle e^{i_1}\wedge\dots\wedge e^{i_p}$, and if we write 
\[\theta|_U= \sum \theta_i e^i, \qquad g_D|_U= \sqrt{-1}\sum g_{\alpha\ol\beta}e^\alpha\wedge \ol e^\beta,\]
then we have 
\begin{equation}\label{log9}
\theta^\star= \sum \ol\theta_i g^{\ol i k}e_k\rfloor .
\end{equation}
Moreover, let $u$ be a section of $F_k$, which we locally write as 
\[u|_U= \sum u_{I\ol J}e^I\otimes d\ol w^J\otimes e_L.\]
We have 
\begin{equation}\label{log10}
D'_h u= \sum (e_k\cdot u_{I\ol J}- u_{I\ol J}e_k\cdot \varphi_L)e^k\wedge e^I\otimes d\ol w^J\otimes e_L,
\end{equation}
and a short calculation which we skip shows that we have
\begin{align}\nonumber \label{log11}
[D'_h, \theta^\star]u= & \sum \ol \theta _i g^{\ol i p}(e_p\cdot u_{I\ol J}- u_{I\ol J}e_p\cdot \varphi_L)e^I\otimes d\ol w^J\otimes e_L\\
+ & \sum \ol\theta_i (e_q\cdot g^{\ol i p}) u_{I\ol J}e^q\wedge (e_p\rfloor e^I)\otimes d\ol w^J\otimes e_L\\
\nonumber \end{align}
We see that the formula \eqref{log11} is very similar to \eqref{form5}, except that the metric $g$ is not Kähler. Nevertheless, we introduce the operator
\begin{equation}\label{nabla1}
\nabla^{1,0}_{\rm log}: F_k\to \Omega^1_X(\log \Delta_1)\otimes F_k 
\end{equation}
by composing the full $(1, 0)$--covariant derivative $\displaystyle \nabla^{1,0}: F_k\to \Omega_X^1\otimes F_k$ with the 
natural map $\Omega_X^1\to \Omega^1_X(\log \Delta_1)$, and
evaluate 
$\displaystyle \theta^\star \rfloor \nabla^{1,0}_{\rm log}u$. The coefficients of the connection for $(\Omega_X(\log \Delta_1), g)$ are as follows
\begin{equation}\label{log12}
\Gamma^i_{jk}= \sum g_{k\ol\alpha} \partial_jg^{\ol\alpha i}
\end{equation}
and we have 
\begin{equation}\label{log13}
\nabla e^i= \sum \Gamma^i_{jk}dw_j\otimes e^k.
\end{equation}
Thus we can write
\begin{align}\nonumber \label{log14}
\nabla^{1,0}_{\rm log} u= &\sum (e_k\cdot u_{I\ol J}- u_{I\ol J}e_k\cdot \varphi_L)e^k\otimes e^I\otimes d\ol w^J\otimes e_L\\
+ & \sum u_{I\ol J}\Gamma^i_{jk} dw_j\otimes e^{I(i, k)}\otimes d\ol w^J\otimes e_L\nonumber  \\
= &\sum (e_k\cdot u_{I\ol J}- u_{I\ol J}e_k\cdot \varphi_L)e^k\otimes e^I\otimes d\ol w^J\otimes e_L\\
+ & \sum u_{I\ol J} g_{q\ol\alpha} (e_k\cdot g^{\ol\alpha p})e^k\otimes e^{I(p, q)}\otimes d\ol w^J\otimes e_L\nonumber  \\
\nonumber  \end{align}
where $I(p, q)$ means that the index $p$ of $I$ was replaced by $q$. Therefore we obtain
\begin{align}\nonumber \label{log15}
\theta_0^\star \rfloor  \nabla^{1,0}_{\rm log} u= &\sum \ol \theta_i g^{\ol ik}(e_k\cdot u_{I\ol J}- u_{I\ol J}e_k\cdot \varphi_L)e^I\otimes d\ol w^J\otimes e_L\\
+ & \sum u_{I\ol J} \theta_i g^{\ol ik} g_{q\ol\alpha} (e_k\cdot g^{\ol\alpha p})e^{I(p, q)}\otimes d\ol w^J\otimes e_L\\
\nonumber  \end{align}
at each point of $U$. 
\smallskip

\noindent In our current situation, we have
\begin{equation}\label{log16}
[D'_h, \theta_0^\star]u- \theta_0^\star \rfloor  \nabla^{1,0}_{\rm log} u= 
\sum \big(e_q\cdot g^{\ol i p}-  g_{q\ol\alpha} g^{\ol ik}e_k\cdot g^{\ol\alpha p}\big) \ol\theta_i u_{I\ol J}e^{I(p, q)}\otimes d\ol w^J\otimes e_L
\end{equation}
as consequence of \eqref{log15} and \eqref{log11}.
One can see that if the metric $g_D$ is Kähler, then the RHS of \eqref{log16} is identically zero. Unfortunately, this may not be the case here, so we will have to 
evaluate the contribution of the linear term \eqref{log16}: this is what we do next. In the first place, 
notice the following string of equalities
\begin{align}\nonumber \label{log17}
e_q\cdot  g^{\ol i p}-  g_{q\ol\alpha} g^{\ol ik}e_k\cdot  g^{\ol\alpha p}= & e_q\cdot g^{\ol i p}+  g^{\ol ik} g^{\ol\alpha p} e_k\cdot g_{q\ol\alpha}\nonumber  \\
= & g^{\ol ik} g^{\ol\alpha p} (e_k\cdot g_{q\ol\alpha}- e_q\cdot g_{k\ol\alpha})+ e_q\cdot g^{\ol i p}+  g^{\ol ik} g^{\ol\alpha p} e_q\cdot g_{k\ol\alpha}\nonumber \\
= & g^{\ol ik} g^{\ol\alpha p} (e_k\cdot g_{q\ol\alpha}- e_q\cdot g_{k\ol\alpha})+ e_q\cdot  g^{\ol i p}-  g^{\ol\alpha p} g_{k\ol\alpha} e_q\cdot g^{\ol ik} \\
= & g^{\ol ik} g^{\ol\alpha p} (e_k\cdot g_{q\ol\alpha}- e_q\cdot g_{k\ol\alpha})\nonumber \\
\nonumber  \end{align}
which will simplify a lot the calculations (given the expression of the metric $g_D$) as we next see.
\smallskip

\noindent The metric $\omega_D$ is Kähler, and we can assume that its restriction to $U$ is 
given by the Hessian of a function $\psi$. We then have 
\begin{equation}\label{log18}
\omega_D|_U= \sqrt{-1}\sum_{\alpha, \beta}e_\beta\cdot(e_{\ol\alpha}\cdot\psi) e^\beta\wedge \ol e^\alpha
\end{equation}
and moreover 
\begin{equation}\label{log19}
\sqrt{-1}\frac{D's_i\wedge \ol{D's_i}}{|s_i|^2}= \big(e^i- \sum_k(e_k\cdot\varphi_i) e^k\big)\wedge \big(\ol e^i- \sum_k(e_{\ol k}\cdot\varphi_i) \ol e^k\big)
\end{equation}
for all $i= 1,\dots, r$. Then the coefficients of the metric $g_D$ can be computed as follows
\begin{equation}\label{log20}
g_{\beta\ol\alpha}= e_\beta\cdot(e_{\ol\alpha}\cdot\psi)+ \delta_{\alpha\beta}\delta_{\alpha, r}
- \delta_{\beta, r} e_{\ol \alpha}\cdot \varphi_\beta - \delta_{\alpha, r} e_\beta\cdot \varphi_\alpha +\sum_i (e_\beta\cdot \varphi_i)(e_{\ol \alpha}\cdot \varphi_i)
\end{equation}
where $\delta_{\alpha\beta}$ is the Kronecker symbol, and $\delta_{\alpha, r}$ is equal to one if $\alpha\leq r$, and zero if not. 

The difference in the last line of \eqref{log17} is then easy to evaluate
\begin{align}\nonumber \label{log21}
e_k\cdot g_{q\ol\alpha}- e_q\cdot g_{k\ol\alpha}= & \delta_{k, r}e_q\cdot (e_{\ol\alpha}\cdot\varphi_k) -\delta_{q, r}e_k\cdot (e_{\ol\alpha}\cdot\varphi_q)
\\
 + & 
\sum_i(e_q\cdot\varphi_i) e_k\cdot (e_{\ol\alpha}\cdot \varphi_i)- (e_k\cdot\varphi_i) e_q\cdot (e_{\ol\alpha}\cdot \varphi_i)
\nonumber  \end{align}
where we see that the function $\psi$ is missing, because
\[e_k\cdot\big(e_q\cdot(e_{\ol\alpha}\cdot\psi)\big)- e_q\cdot\big(e_k\cdot(e_{\ol\alpha}\cdot\psi)\big)=0\]
for any $q, k$ and $\alpha$. 
\medskip

\noindent We will next detail the calculation for the difference $\displaystyle [D'_h, \theta^\star]u- \theta^\star \rfloor  \nabla^{1,0}_{\rm log} u$, since this is basically the only difference between the context here and in Section ??. We assume that the frames of $\mathcal O(Y_j)$ have been chosen so that the first order derivatives of the weights $\varphi_j$ at $w= 0$ vanish. Then we have
\[[D'_h, \theta_0^\star]u- \theta_0^\star \rfloor  \nabla^{1,0}_{\rm log} u= 
\sum g^{\ol ik} g^{\ol\alpha p}\big(\delta_{k, r}e_q\cdot (e_{\ol\alpha}\cdot\varphi_k) -\delta_{q, r}e_k\cdot (e_{\ol\alpha}\cdot\varphi_q)\big) \ol\theta_i u_{I\ol J}e^{I(p, q)}\otimes d\ol w^J\otimes e_L
\]
at $w=0$, and a case-by-case analysis which we next detail will show that 
\begin{equation}\label{log22}
|[D'_h, \theta^\star]u- \theta^\star \rfloor  \nabla^{1,0}_{\rm log} u|_{g_D, \omega_D}\leq C|\theta||u|
\end{equation}
where the constant $C> 0$ is uniform.

Consider first the expression $\sum_i g^{\ol ik}\ol\theta_i$, where $k$ is arbitrary. The inequality
\[|g^{\ol ik}|\leq C |e^i|_g\]
holds true, since the inverse of $(g^{\ol\alpha\beta})$ is positive definite and the diagonal coefficients are bounded from above. We infer that the relation
\begin{equation}\label{log23}
|\sum_i g^{\ol ik}\ol\theta_i|\leq C|\theta|_{g_D}
\end{equation}
is satisfied. We next have to evaluate
\[|g^{\ol\alpha p}|\big|\big(\delta_{k, r}e_q\cdot (e_{\ol\alpha}\cdot\varphi_k) -\delta_{q, r}e_k\cdot (e_{\ol\alpha}\cdot\varphi_q)\big)\big| |e^q||e_p|;\]
this quantity is clearly bounded from above, since the product
\[|g^{\ol\alpha p}|^2g_{p\ol p}\]
is uniformly bounded for any $\alpha, p$.
\medskip

\noindent Given that, the same arguments as in the smooth case of Theorem \ref{keyinequ} apply, finishing the proof. Indeed, the points (i) and (ii) in \ref{proofsmth} are completely identical. Then we write
\[\langle \theta_0^\star \rfloor \nabla^{1,0} u, u\rangle= \langle \theta_0^\star \rfloor \nabla^{1,0} u- [D'_h, \theta_0^\star]u, u\rangle+ \langle [D'_h, \theta_0^\star]u, u\rangle\]
and the first term of the RHS of this equality is estimated by using \eqref{log22}.

Finally, the point (iii) works as well in our context, since the curvature of $(T_X\langle \Delta_1\rangle, g_D)$ is bounded, and so does (iv): this is due to the definition of $D'_h$, see Convention \ref{conv}.
\medskip

\noindent We end this section by a few basic consequences of Bochner-type formulas, which were used in our previous arguments.

\subsubsection{Bochner formulas, I}
In this subsection we assume that $\Delta_1= 0$. The following statement is a direct consequence of \cite{Siu82}, together with the fact that the 
curvature of $(L, h_L)$ equals zero in the complement of $\Delta_2$, and the fact that the full curvature tensor of $(T_X, \omega_D)$ is bounded.

\begin{lemme}\label{bochner1} \cite{Siu82} Let $u$ be a form of total degree $k$, with compact support in $X\setminus D$ and values in $L$. We assume moreover 
that $u\in L^2$, as well as $\dbar u, \dbar^\star u\in L^2$. Then the following inequalities hold.
\begin{enumerate}

\item[\rm (a)] There exists a constant $C> 0$ only depending on $(X, \omega_D)$ and $(L, h_L)$ such that the inequality 
\[\Vert \nabla^{1,0}u\Vert^2\leq C(\Vert u\Vert^2+ \Vert \dbar u\Vert^2+ \Vert \dbar^\star u\Vert^2)\]
holds.

\item[\rm (b)] There exists a constant $C> 0$ only depending on $(X, \omega_D)$ and $(L, h_L)$ such that the inequalities  
\[-C\Vert u\Vert^2+ \Vert \dbar u\Vert^2+ \Vert \dbar^\star u\Vert^2\leq \Vert D'_h u\Vert^2+ \Vert {D'_h}^\star u\Vert^2\leq C\Vert u\Vert^2+ \Vert \dbar u\Vert^2+ \Vert \dbar^\star u\Vert^2\]
hold.
\end{enumerate}
\end{lemme}

\subsubsection{Bochner formulas, II} Here we discuss the logarithmic version of the previous inequalities. They are completely standard, 
but we will provide a complete treatment since unfortunately we were not able to find an appropriate reference. 

\begin{lemme}\label{bochner2} Let $u$ be an $L$-valued smooth form with logarithmic poles along $\Delta_1$, conic singularities on $\Delta_3$, and compact support in $X \setminus \Delta_2$. Then the following inequalities hold.
	
\begin{enumerate}

\item[\rm (a)] There exists a constant $C> 0$ only depending on $(X, \omega_D)$, $(T_X\langle \Delta_1\rangle, g_D)$ and $(L, h_L)$ such that the inequality 
\[\Vert \nabla^{1,0}_{\rm log}u\Vert^2\leq C(\Vert u\Vert^2+ \Vert D'_h u\Vert^2+ \Vert {D'_h}^\star u\Vert^2)\]
holds, where $\nabla^{1,0}_{\rm log}$ was defined in \eqref{nabla1} above.

\item[\rm (b)] Given $\ep_0> 0$, there exists a rescaling constant $\lambda> 0$ and a positive real number $C> 0$ only depending on $(X, \omega_D)$, $(T_X\langle \Delta_1\rangle, \lambda g_D)$ and $(L, h_L)$ such that the inequality  
\[\Vert D'_h u\Vert^2+ \Vert {D'_h}^{\star}u\Vert^2\leq C\Vert u\Vert^2+ \ep_0(\Vert \dbar u\Vert^2+ \Vert \dbar^\star u\Vert^2)\]
holds.
\end{enumerate}
\end{lemme}

\begin{remark}
In our current setting, the principal symbols of the commutators $[D'_h, {D'_h}^\star]$ and $[\dbar, \dbar^\star]$ are very different. Indeed, the former equals
\[g^{\ol\beta \alpha} e_\alpha \cdot (e_{\ol \beta} \cdot f)\]
and the later is
\[\omega^{\ol\beta \alpha}\frac{\partial^2 f}{\partial z^\alpha\partial z^{\ol \beta}}.\]
This accounts for the difference between Lemma \ref{bochner1} and Lemma \ref{bochner2}.
\end{remark}

\begin{proof}
We will discuss $(a)$ first. 
Let $x_0\in X\setminus D$ be a point in the complement of the support of the divisor $D$. We evaluate the intrinsic quantities
\begin{equation}\label{boch1}
[D'_h, {D'_h}^{\star}]u, \qquad \nabla^{1,0 \star}_{\rm log}\nabla^{1,0}_{\rm log}u
\end{equation}
by a direct calculation. Consider an arbitrary holomorphic frame $\displaystyle (\xi^{i})_{i=1,\dots, n}$ of the vector bundle $T_X\langle \Delta_1 \rangle$, as well as coordinates $(w^\alpha)_{\alpha= 1,\dots, n}$. Then we can write  
\begin{equation}\label{boch2}
u= \sum u_{I\ol J}\xi^I\otimes dw^{\ol J}\otimes e_L
\end{equation}
and (by the definition of the operator $D'_h$) we have
\begin{equation}\label{boch3}
D'_hu= \sum (\partial_\alpha u_{I\ol J}- u_{I\ol J}\partial_\alpha \varphi_L)dw^\alpha \wedge \xi^I\otimes dw^{\ol J}\otimes e_L,	
\end{equation}	
where the differential $dw^\alpha$ is interpreted as local section of the log cotangent bundle. Given a form $v$ with support near $x_0$, we have 
\[\int_{X}\langle v, D'_hu \rangle e^{-\varphi_L}dV= \int_{X}\langle {D'_h}^{\star}v, u \rangle e^{-\varphi_L}dV\]
by the definition of the formal adjoint. With respect to the coordinates fixed above, the LHS becomes
\[\int_{(X, x_0)}v_{K\ol L}\partial_{\ol \alpha}(e^{-\varphi_L}\ol u_{I\ol J})
\langle \xi^K, dw^\alpha \wedge \xi^I\rangle_g \langle dw^{J}, dw^{L}\rangle_{\omega} \det\omega d\lambda\]
and integration by parts shows that this expression is equal to 
\begin{equation}\label{boch4}
-\int_{(X, x_0)}e^{-\varphi_L}\ol u_{I\ol J}\partial_{\ol \alpha}\big(v_{K\ol L}\langle \xi^K, dw^\alpha \wedge \xi^I\rangle_g \langle dw^{J}, dw^{L}\rangle_{\omega} \det\omega\big) d\lambda.
\end{equation}
Now, in order to obtain the expression of ${D'_h}^{\star} v$ near the point $x_0$, we remark that we can write 
\begin{equation}\label{boch5}
dw^\alpha= \sum_i f^\alpha_i \xi^{i}
\end{equation}
for some locally defined holomorphic functions $(f^\alpha_i)_{i, \alpha= 1, \dots, n}$. 
Equation \eqref{boch3} becomes
\begin{equation}\label{boch11}
D'_hu= \sum (\partial_\alpha u_{I\ol J}- u_{I\ol J}\partial_\alpha \varphi_L)f^\alpha_i \xi^{i} \wedge \xi^I\otimes dw^{\ol J}\otimes e_L,	
\end{equation}
and moreover we have
\[\langle \xi^K, dw^\alpha \wedge \xi^I\rangle_g= \sum_i\ol{f^\alpha_i} \langle \xi^K,  \xi^{i} \wedge \xi^I\rangle_g= 
\sum_i\ol{f^\alpha_i} g^{\ol i p}\langle \xi_p\rfloor \xi^K,  \xi^I\rangle_g\]
where $g^{\ol i p}= \langle \xi^p,  \xi^i \rangle_g$ are the coefficients of the metric $g_D$ with respect to the basis $\displaystyle (\xi^{i})_{i=1,\dots, n}$, and $\displaystyle (\xi_p)_{p=1,\dots, n}$ is the dual basis.

\noindent By expanding the derivative under the integral in \eqref{boch4} we get a first term
\begin{align}\label{boch6}
- & \big(\langle dw^{J}, dw^{L}\rangle_{\omega} \det\omega \sum_i\ol{f^\alpha_i} g^{\ol i p}\langle \xi_p\rfloor \xi^K,  \xi^I\rangle_g \big)\partial_{\ol \alpha}v_{K\ol L}\\ 
- & \sum\big(v_{K\ol L}\langle dw^{J}, dw^{L}\rangle_{\omega} \det\omega g^{\ol i p}\langle \xi_p\rfloor \xi^K,  \xi^I\rangle_g \big)\ol{\partial_{\alpha}{f^\alpha_i}}\nonumber 
\end{align}
as well as the following term
\begin{equation}\label{boch7}
-\sum v_{K\ol L}\ol{f^\alpha_i} \partial_{\ol \alpha}\big(\langle dw^{J}, dw^{L}\rangle_{\omega} \det\omega g^{\ol i p}\langle \xi_p\rfloor \xi^K,  \xi^I\rangle_g\big).
\end{equation}
\noindent We infer that the adjoint of $D'_h$ can be expressed as follows 
\begin{align}\label{boch8}
 {D'_h}^{\star} v\simeq & -\sum \partial_{\ol \alpha}v_{K\ol M}\ol{f^\alpha_i} g^{\ol i p} \xi_p\rfloor \xi^K\otimes dw^{\ol M}\otimes e_L \\
 & -\sum v_{K\ol M}\ol{\partial_{\alpha}{f^\alpha_i}}g^{\ol i p} \xi_p\rfloor \xi^K\otimes dw^{\ol M}\otimes e_L\nonumber 
\end{align}
modulo terms of order zero depending on the anti-holomorphic derivatives of $g^{\ol i p}$ and $\omega^{\ol \mu \rho}$ (which vanish at $x_0$). 
\smallskip

\noindent Next, assume that the frame $\displaystyle (\xi^{i})_{i=1,\dots, n}$ is normal at $x_0$, and that the coordinates $\displaystyle (w^{\tau})_{\tau=1,\dots, n}$ are geodesic at $x_0$. Then as consequence of the fact that $g_D\gg \omega_D$ (where we interpret here $g_D$ as 
singular Hermitian metric) il follows that 
\begin{equation}\label{boch9}
\sup|f^\alpha_i|_{\mathcal C^1(X, x_0)}\leq C
\end{equation}
for some \emph{uniform constant} $C$. Indeed, this is immediately seen by using quasi-coordinates, with respect to which $\omega_D$ becomes 
a standard, smooth metric on $(\mathbb C^n, 0)$.
\smallskip

\noindent At $x_0$ we can write
\begin{align}\label{boch10}
D'_h{D'_h}^{\star}u\simeq  & - \sum \partial^2_{\beta \ol \alpha}u_{K\ol M}\ol{f^\alpha_p}f^\beta_q \xi^q\wedge (\xi_p\rfloor \xi^K)\otimes dw^{\ol M}\otimes e_L \\
& - \sum \partial_{\beta}u_{K\ol M}\ol{\partial_{\alpha}{f^\alpha_p}}f^\beta_q \xi^q\wedge (\xi_p\rfloor \xi^K)\otimes dw^{\ol M}\otimes e_L\nonumber 
\end{align}
modulo a term of order zero involving the curvature forms of $(T_X, \omega_D)$ and $(T_X\langle \Delta_1\rangle, g_D)$ at $x_0$.
\smallskip

\noindent Similarly, we have
\begin{align}\label{boch12}
{D'_h}^{\star}D'_hu\simeq  & - \sum \partial^2_{\beta \ol \alpha}u_{K\ol M}\ol{f^\alpha_p}f^\beta_q \xi_p\rfloor (\xi^q\wedge  \xi^K)\otimes dw^{\ol M}\otimes e_L \\
& - \sum \partial_{\beta}u_{K\ol M}\ol{\partial_{\alpha}{f^\alpha_p}}f^\beta_q \xi_p\rfloor (\xi^q\wedge  \xi^K)\otimes dw^{\ol M}\otimes e_L\nonumber 
\end{align}
which gives
\begin{align}\label{boch13}
{D'_h}^{\star}D'_hu+ D'_h{D'_h}^{\star}u\simeq  & - \sum \partial^2_{\beta \ol \alpha}u_{K\ol M}\ol{f^\alpha_p}f^\beta_p \xi^K\otimes dw^{\ol M}\otimes e_L \\
& - \sum \partial_{\beta}u_{K\ol M}\ol{\partial_{\alpha}{f^\alpha_p}}f^\beta_p \xi^K\otimes dw^{\ol M}\otimes e_L\nonumber 
\end{align}
as before, modulo linear terms only involving the curvature at $x_0$. Notice that the second line in this expression \eqref{boch12} marks an important difference with respect to the usual context (in which we only have a Laplace-like term plus curvature).
\smallskip

\noindent One can evaluate $\nabla^{1,0}_{\rm log}$ and its adjoint in a similar manner, which is what we next do. We have 
\begin{align}\label{boch20}
\nabla^{1,0}_{\rm log}u = & \sum (\partial_{\alpha}u_{K\ol M}- u_{K\ol M}\partial_{\alpha}\varphi_L)f^\alpha_p \xi^p\otimes \xi^K\otimes dw^{\ol M}\otimes e_L \\
+ & \sum u_{K\ol M}\Gamma^i_{\alpha j}f^\alpha_p \xi^p \otimes \xi^{K(i, j)}\otimes dw^{\ol M}\otimes e_L\nonumber
\end{align} 
where $\Gamma^i_{\alpha j}$ are the coefficients of the connection on the log cotangent bundle and $K(i, j)$ is the 
ordered multi-index deduced from $K$, in which $i$ was replaced with $j$.

\noindent Furthermore, if $\displaystyle v:= \sum v_{p I \ol J} \xi^p\otimes \xi^I\otimes dw^J\otimes e_L$, then we have 
\begin{equation}\label{boch21} 
\nabla^{1,0 \star}_{\rm log} v= -\sum \partial_{\ol\alpha}(v_{p I \ol J} \ol{f^\alpha_p})\xi^I\otimes dw^J\otimes e_L
\end{equation}
at $x_0$ and 
a quick calculation shows that we have  
\begin{align}\label{boch14}
\nabla^{1,0 \star}_{\rm log}\nabla^{1,0}_{\rm log}u \simeq & - \sum \partial^2_{\beta \ol \alpha}u_{K\ol M}\ol{f^\alpha_p}f^\beta_p \xi^K\otimes dw^{\ol M}\otimes e_L \\
& - \sum \partial_{\beta}u_{K\ol M}\ol{\partial_{\alpha}{f^\alpha_p}}f^\beta_p \xi^K\otimes dw^{\ol M}\otimes e_L\nonumber
\end{align}
modulo linear terms only involving the curvature at $x_0$. We remark that this is the same expression as in \eqref{boch13}, and we infer that the following holds
\begin{equation}\label{boch15}
\big| \langle {D'_h}^{\star}D'_hu+ D'_h{D'_h}^{\star}u, u \rangle - \langle \nabla^{1,0 \star}_{\rm log}\nabla^{1,0}_{\rm log}u, u \rangle\big|\leq C
\langle u, u\rangle
\end{equation}
at $x_0$, where the constant $C$ is independent of $x_0$ (no matter what is the distance between this point and the support of $D$). This is due to two facts: first one is \eqref{boch9}, and the second one is that if we denote by 
\begin{equation}\label{boch16}
\mathscr R^i_{j \alpha\ol\beta} \xi^j\otimes \xi_i\otimes dw^{\alpha}\wedge dw^{\ol \beta}
\end{equation}
and 
\begin{equation}\label{boch17}
\Theta^i_{j \alpha\ol\beta} \frac{\partial}{\partial w^i}\otimes dw^j \otimes dw^{\alpha}\wedge dw^{\ol \beta}
\end{equation}
the curvature tensors of $(T_X\langle \Delta_1\rangle, g_D)$ and $(T_X, \omega_D)$ at $x_0$, respectively, then we have
\[\max (|\mathscr R^i_{j \alpha\ol\beta}|, |\Theta^i_{j \alpha\ol\beta}|)\leq C\]
for some constant independent of $x_0$. Recall that we have already proved that the curvature forms are two-sided bounded, see Lemma \ref{mixbound} as well as Lemma \ref{logmixbound}. 

\noindent The proof of point $(a)$ is completed by integrating over $X$. 
\smallskip

\noindent As for the point $(b)$, it is a simple rescaling matter. Assume that the form $u$ has total degree $k$, then we write 
\[u= \sum_{p+q= k} u_{pq}\]
and thus the global $L^2$ norm of its derivative equals
\begin{equation}\label{boch18}\Vert D'_h u\Vert^2_{\lambda g_D, \omega_D}= \sum \frac{1}{\lambda^{p+1}}\Vert D'_h u_{pq}\Vert^2_{g_D, \omega_D}.\end{equation}
A quick calculation which we skip shows that we have 
\[{D'_h}^{\star \lambda} u= \frac{1}{\lambda}{D'_h}^{\star} u\]
where, as notation suggests, ${D'_h}^{\star \lambda}$ is the adjoint of $D'_h$ with respect to the rescaled metric. We then have 
\begin{equation}\label{boch19}\Vert {D'_h}^{\star \lambda} u\Vert^2_{\lambda g_D, \omega_D}= \sum \frac{1}{\lambda^{p+1}}\Vert {D'_h}^{\star} u\Vert^2_{g_D, \omega_D}\end{equation}
and the conclusion follows by choosing $\lambda$ in an appropriate manner.
\end{proof}

Now we can prove Theorem \ref{keyinequ} in the log case \eqref{secondcases}.

	\begin{proof}[Proof of Theorem \ref{keyinequ} in the log case \eqref{secondcases}]
	Fix a $\ep_0 \ll 1$. By Lemma \ref{bochner2}, we have 
		\begin{equation}
			\ep_0 (\Vert D_{h_L} u\Vert^2+ \Vert D_{h_L} ^\star u\Vert^2)+ C\Vert u\Vert^2\geq  \Vert \nabla^{1,0} _{log}u\Vert^2.
		\end{equation} 
		Together \eqref{log22}, we know that 
	$$|\int_X\langle \mathcal [D'_{h_L}, \theta_0 ^\star]u, u\rangle dV|\leq 
	C \int_X|\theta_0| \cdot |u|^2dV + |\int_X\langle  \nabla^{1,0} _{log}  u, \theta_0 \wedge u\rangle dV|$$
				$$\leq C \int_X|\theta_0| \cdot |u|^2dV+2 \Vert  \nabla^{1,0} _{log}u\Vert^2 +\frac{1}{8} \int_X|\theta_0|^2 |u|^2 dV$$
		$$\leq 	(\frac{1}{8} +\ep_0) \int_X|\theta_0|^2 \cdot |u|^2dV+2 \ep_0 \int_X\langle [D_{h_L}, D_{h_L} ^\star]u, u\rangle dV
		+ \frac{C}{\ep_0}\Vert u\Vert^2 $$
	Finally, by using the formula \eqref{form40}, we infer that  
		\begin{equation}
			C_{\ep_0}\Vert u\Vert^2+ \int_X\langle \Delta_K u, u\rangle dV\geq \frac{1}{4}\int_X\langle [D_{h_L}, D_{h_L} ^\star]u, u\rangle dV+ (\frac{1}{2}-\ep_0)\int_X|\theta_0|^2|u|^2dV,
		\end{equation} 
		which ends the proof.
	
	\end{proof}

\section{Technicalities II: Hodge theory for $\Delta_K$}\label{HodgeDek}

By using the estimate in Theorem \ref{keyinequ}, we establish in this section the Hodge decomposition for the Laplace operator $\Delta_K$. We first address this in the context of $L^2$-forms. Next, we define the sheaves of smooth forms $\cC^\infty _K (X, L)$ and $\cC^\infty _K (X, \Delta_1, L)$ in the setting of $(X, \omega_D, L, h_L)$, and we show that the Hodge decomposition holds for both $\cC^\infty _K (X, L)$ and $\cC^\infty _K (X, \Delta_1, L)$. An important remark, going back at least to the seminal article of Noguchi \cite{Nog95}, is that a form with logarithmic poles naturally defines a current. In the last subsection, we extend the Hodge decomposition to the space of currents, relying on a technique due to de Rham and Kodaira. This additional result will be crucial in the proof of our version of the $\ddbar$-lemma.

\subsection{Hodge decomposition for $L^2$-forms}

We start this section by some details about the definition of $\Delta_K$ as an unbounded operator.
Recall that $\mathcal E_k$ denotes one of the following vector bundles
\begin{equation}\label{twocases}
	\oplus_{r+s= k}\Lambda^{r, s}T^\star_X\otimes L, \qquad  
	\oplus_{r+s= k}\Lambda^{r, s}T^\star_X\langle \Delta_1\rangle\otimes L .
\end{equation}
 Consider the set
\begin{equation}\label{d1}
\dom(\Delta_K):= \{u\in L^2(X, \cE_k) : D_Ku, D_K^\star u \hbox{ and } \Delta_K u \in L^2(X, \cE_k)\}
\end{equation}
where $L^2$-norm is with respect to the metrics $\omega_D$ and $h_L$ in the smooth case, and 
$(g_D, \omega_D, h_L)$ in the logarithmic case.
  
Notice that for any $u\in \dom(\Delta_K)$, we have the equality
\begin{equation}\label{d2}
\int_X\langle \Delta_Ku, v\rangle dV_{\omega_D}=  \int_X\langle D_Ku, D_Kv\rangle dV_{\omega_D}+ \int_X\langle D_K^\star u, D_K^\star v\rangle dV_{\omega_D},
\end{equation}
provided that $v\in \dom(D_K)\cap \dom(D_K^\star)$. This is a consequence of the definition of $\Delta_K$, the so-called \emph{Friedrichs extension} of the commutator 
$[D_K, D^\star_K]$ acting on smooth forms with compact support. It is self-adjoint, and by elementary Hilbert space theory its domain -defined above in \eqref{d1}- consists in 
$u\in L^2(X, \cE_k)$ such that 
\begin{equation}\label{d8}u\in \dom(D_K)\cap \dom(D_K^\star), \qquad D_Ku\in \dom(D_K^\star),\qquad D_K^\star u\in \dom(D_K)\end{equation} 
see e.g. \cite{bookJP}.
\medskip

\noindent We next consider the following Sobolev-type space.
	
\begin{defn}\label{regularity}
Let $H^i $ be the subspace of $L^2 (X, \cE_k)$ such that the derivatives with respect to $D_K$ and $D^\star _K$ of order $\leq i$ are $L^2$  with respect to the metrics $\omega_D$ and $h_L$ in the smooth case, and $(g_D, \omega_D, h_L)$ in the logarithmic case.
\end{defn}

\smallskip

\noindent In this context, Theorem \ref{keyinequ} admits the following more general version.
\begin{thm}\label{KeyDom}
	There exists a positive constant $C> 0$ such that 
for any $u\in H^1$, the inequality
	$$\|u\|^2_{L^2}  + \int_X(|D_K u|^2+ |D_K^\star u|^2)dV_{\omega_D} \geq \frac{1}{C} \int_X \big(|u|^2 |\theta_0|^2+ |D_h u|^2+ |D_h^\star u|^2 \big)dV_{\omega_D} $$
holds. Moreover, as in Corollary \ref{keylog}, the same inequality is true if we replace the volume element of $\omega_D$ with $d\mu_a$. 
\end{thm}
\medskip

\begin{remark}
In particular, it follows that for $u\in \dom(\Delta_K)$ we automatically have 
\[D_h u, D^\star _hu \in L^2.\]  
\end{remark}

\begin{proof} 
As expected, we will proceed by approximation. In the first place, we can assume that the form $u\in \dom (\Delta_K)$ has compact support in $X_0= X\setminus \Delta_2$. Indeed, this is due to the fact that the metric $\omega_D$ is complete near the support of the divisor $\Delta_2$. 

Next, we cover $X$ with a finite set of coordinate open subsets $(\Omega_i, z_i)$, as in Lemma \ref{coeff}. If $(\Omega, z)$ is one of them, we assume that the support of $u$ is contained in $\Omega \setminus \Delta_2$ and consider the local ramified maps
\begin{equation}\label{d4}
\pi: \DD\to \Omega, \qquad \pi(w)= (w_1,\dots, w_r, w_{r+1}^k\dots, w_{r+p+q}^k, w_{r+p+q+1},\dots, w_n).
\end{equation}
Recall the following notation
\[u_1:= \frac{1}{w^{k\nu}}\pi ^\star u\]
cf. the point $C_5$ of Convention \ref{conv}.

Corresponding to each ramified map $\pi$ we introduce the operators 
\begin{equation}\label{d6}
d'_K:= \partial - \alpha', \qquad d''_K:= \dbar + \alpha''
\end{equation}
acting on forms defined on $\DD$, where $\displaystyle \alpha':= -\pi^\star(\theta_0+ \ol\theta_0)$ and 
$\displaystyle \alpha'':= \pi^\star(\theta_0- \ol\theta_0)$. We 
set \[d_K:= d_K'+ d_K''\]
and a simple verification shows that the equalities
\begin{equation}\label{d7}
d_K u_1= \frac{1}{w^{k\nu}}\pi^\star D_Ku, \qquad d_K^\star  u_1= \frac{1}{w^{k\nu}}\pi^\star (D_K^\star u)
\end{equation}
hold (see \cite{CP} for a more detailed presentation).
\smallskip

\noindent Moreover, we have the following density result.
\begin{lemme}\label{dens}
Let $u$ be a form with compact support in the set $\Omega$. If we assume that $u$, together with its derivatives 
$D_Ku$ and $D_K^\star u$ are in $L^2(X, \cE_k)$ then we have
\[u_1= \lim_\alpha u_\alpha\]  
where $u_\alpha$ are smooth and have compact support in $\DD\setminus \pi^{-1}(\Delta_2)$. The limit above is in the graph norm corresponding to the operators $d_K, d^\star_K$.
\end{lemme}
\begin{proof} We first remark that, by hypothesis combined with the equality 
	\eqref{d7} the relations
\begin{equation}\label{d3}
u_1\in L^2(\DD, \omega_1),\qquad d_K u_1\in L^2(\DD, \omega_1), \qquad 	d_K^\star u_1\in L^2(\DD, \omega_1)
\end{equation}
hold.  Notice that the pull-back 
\begin{equation}\label{d5}
	\omega_1:= \pi^\star\omega_D
\end{equation}
is a metric with Poincaré singularities along the $\pi$-inverse image of $\Delta_2$. In particular $\omega_1$ is complete near the $\pi$-inverse image of $\Delta_1\cap \Omega$, and the form $u_1$ has compact support in $\DD$. 

By completeness, there exists a family of cutoff functions $(\rho_\ep)_{\ep> 0}$ as in \cite[VIII, Lemma 2.4]{bookJP},  we have 
\begin{equation}
	\rho_\ep u_i\to u_1,\qquad d_K (\rho_\ep u_1)\to d_K u_1, 
	\qquad d_K ^\star(\rho_\ep u_1)\to d_K ^\star u_1
\end{equation}
in $L^2(\DD, \omega_1)$.  Note that $\rho_\ep u_i$ is of compact support in $\DD\setminus \pi^{-1}(\Delta_2)$. 

Finally, on any compact subset in the complement of $\pi^{-1}(\Delta_2)$ the operators $d_K$ and $d^\star_K$ are 
smooth, of first order. The convolution of $\rho_\ep u_i$ with a smoothing kernel together with Friedrichs lemma (see \cite{bookJP}) allow us to conclude. 
\end{proof}	
\medskip

\noindent The proof of Theorem \ref{KeyDom} is ended by applying Theorem \ref{keyinequ} (or better say, its local version) to the regularisation of $\rho_\ep u_i$ and taking the limit.
\end{proof}


\smallskip

\noindent The next statement is a G\aa rding-type inequality. 

\begin{lemme}\label{h2norm} There exists a positive constant $C> 0$ such that 
for any $\cE_k$-valued section $f$ with compact support in $X\setminus \Delta_2$, which is moreover smooth in the conic $(X, \Delta_3)$-sense, we have   
	$$\|f\|_{H^2} \leq  C (\|f\|_{L^2} + \|g\|_{L^2}) $$
where $g:=\Delta_K f$ is the Laplacian of $f$.
\end{lemme}
	
\begin{proof} In the first place, the hypothesis implies that $f\in \dom(\Delta_K)$. This is immediately seen by using the local ramified maps 
$\pi$ in \eqref{d4}.
It therefore follows that the equality \[\|D_K f\|^2 +\|D^\star _K f\|^2 = \langle g, f\rangle\] 
holds, and therefore we have 
	$$\|f\|_{H^1} \leq  C (\|f\|_{L^2} + \|g\|_{L^2}) .$$
	
	For the second order derivatives, since $D_K, D^\star_K$ commute with $\Delta_K$, and $D_K f$ is still smooth in the conic sense and with compact support in $X\setminus \Delta_2$ we have 
	$$D_K g = \Delta_K D_K f \qquad\text{and} \qquad D^\star _K g = \Delta_K D^\star_K f.$$
It follows that the equalities 
	$$\langle D_K g, D_K f\rangle = \langle \Delta_K D_K f , D_K f\rangle =  \langle D^\star _KD_K f ,D^\star_K D_K f\rangle $$
	and 
		$$\langle D^\star _K g, D^\star _K f\rangle = \langle \Delta_K D^\star_K f , D^\star _K f\rangle =  \langle D_KD^\star_K f ,D_K D^\star_K f\rangle
		$$
	hold.	
		Since $\langle D_K g, D_K f\rangle  + \langle D^\star _K g, D^\star _K f\rangle  = \langle g, \Delta_K f\rangle =\|g\|^2$, we have 
		$$ \langle D^\star _KD_K f ,D^\star_K D_K f\rangle  + \langle D_KD^\star_K f ,D_K D^\star_K f\rangle  = \|g\|^2.$$
In conclusion, we infer that the inequality 
			$$\|f\|_{H^2} \leq  C (\|f\|_{L^2} + \|g\|_{L^2})$$
is true, which ends the proof of our statement.			
		\end{proof}

\medskip

\noindent Our next target is the following estimate, a consequence of Theorem \ref{KeyDom}.

\begin{thm}\label{logpole} 
	Let $\omega_D$ be the metric considered in \eqref{m1}. We suppose that $\Delta_2 =\sum_{i=1}^m D_i$
	and $s_i$ is the canonical section of $D_i$.
Let $g, f$ be two $\cE_k$-valued sections on $X$ such that $f\in \dom(\Delta_K)$ and such that $\Delta_K f =g$. Then the following hold.
\begin{enumerate} 

\item[\rm (1)] The inequalities
$$\int_X (|f|^2 + |D^\star_K f|^2 +|D_K f|^2) (\sum_{i=1}^m \log |s_i|^2 )\ dV_{\omega_D} <+\infty. $$
and
		$$\int_X (|D_K D^\star_K f|^2 + |D^\star_K D_K f|^2 )dV_{\omega_D} <+\infty$$
hold.

\item[\rm (2)] Consider a non-empty subset $I \subset \{1,\cdots, m\}$ as well as positive reals $a_i\geq 0$. If 
\[\int_X |g|^2 \prod_{i\in I} \log^{2a_i} |s_i| dV_{\omega_D} <+\infty\] then we have 
$$\int_X (|f|^2 + |D^\star_K f|^2 +|D_K f|^2) (\sum_{i=1}^m \log |s_i|^2 )\cdot   \prod_{i\in I} \log^{2a_i} |s_i| dV_{\omega_D} <+\infty$$
and 
\begin{equation}\label{addnew}
\int_X (|D^\star_K D_K f|^2 +|D_K D_K ^\star  f|^2)   \prod_{i\in I}   \log^{2a_i} |s_i|  dV_{\omega_D} <+\infty
\end{equation}	

\end{enumerate}	
	\end{thm}
	
\begin{proof} We first remark that the second relation of the first point $(1)$ follows from the fact that $f$ belongs to the domain of $\Delta_K$, together with \eqref{d8}. 
\medskip
	
\noindent Concerning the first part of $(1)$, 
it is a direct consequence of the main estimate in Theorem \ref{KeyDom} and the fact that $\Res \theta_0$ is non-vanishing on every component of $\Delta_2$, cf. Property \ref{property}.
\medskip

\noindent The argument showing that $D_K f$ and $D_K^\star f$ verify a similar estimate is the same, given that $f\in \dom(\Delta_K)$, see \eqref{d8}. That is to say, we simply replace $f$ with $D_Kf$ in Theorem \ref{KeyDom}.
\medskip

\noindent
The second point $(2)$ follows along the same lines as the first one, by using the Corollary \ref{keylog} instead of Theorem \ref{keyinequ}. We have
\begin{equation}\label{d100}\int_X|f|^2e^{\psi_\delta}dV_{\omega_D}  + \int_X|D_Kf|^2e^{\psi_\delta}dV_{\omega_D}+  \int_X|D_K^\star f|^2e^{\psi_\delta}dV_{\omega_D}\geq \frac{1}{C}\int_X |f|^2 |\theta|^2e^{\psi_\delta}dV_{\omega_D}\end{equation} 
where $\displaystyle e^{\psi_\delta}:= \prod_{i\in I} \log^{2a_i}(\delta^2+ |s_i|^2)$. Notice that all terms are convergent since $\psi_\delta$ is smooth.
\smallskip

\noindent We next claim that the inequality 
\begin{equation}\label{d101} \int_X|D_Kf|^2e^{\psi_\delta}dV+  \int_X|D_K^\star f|^2e^{\psi_\delta}dV \leq 
C\int_X|f|^2e^{\psi_\delta}dV+ \int_X |g|^2e^{\psi_\delta}dV \end{equation}
holds, for some constant $C> 0$.

This is easy to verify: integration by parts -justified by the fact that $f\in \dom(\Delta_K)$- shows that we have 
\begin{equation}\label{d103} \int_X|D_Kf|^2e^{\psi_\delta}dV= 
\int_X\langle D_Kf, D_K(e^{\psi_\delta}f)\rangle dV- \int_X\langle D_Kf, d\psi_\delta\wedge f\rangle e^{\psi_\delta} dV
\end{equation}
The norm of $d\psi_\delta$ with respect to $\omega_D$ is pointwise uniformly bounded, so we write 
\begin{equation}\label{d104} \big|\int_X\langle D_Kf, d\psi_\delta\wedge f\rangle e^{\psi_\delta} dV\big|\leq \frac{1}{2}\int_X|D_Kf|^2e^{\psi_\delta}dV+ C\int_X|f|^2e^{\psi_\delta}dV.
\end{equation}
We have similar inequalities for $D_K^\star f$; together with the fact that 
\begin{equation}\label{d105}
\int_X\langle D_Kf, D_K(e^{\psi_\delta}f)\rangle dV+ \int_X\langle D_K^\star f, D_K^\star (e^{\psi_\delta}f)\rangle dV= 
\int_X\langle g, f\rangle e^{\psi_\delta}dV, 
\end{equation}
claim \eqref{d101} is established. 
\medskip

\noindent We therefore obtain the inequality
\begin{equation}\label{d106}C\big(\int_X|f|^2e^{\psi_\delta}dV+ \int_X|g|^2e^{\psi_\delta}dV\big) \geq  \int_X |f|^2 |\theta|^2_{\omega_D}e^{\psi_\delta}dV,\end{equation} 
as consequence of 
 \eqref{d101} and \eqref{d100}. Near the divisor $D$, it is clear that the RHS of \eqref{d106} dominates the norm of $f$ on the LHS, so we infer the inequality 
\begin{equation}\label{d107}C\big(\int_X|f|^2 dV+ \int_X|g|^2e^{\psi_\delta}dV\big) \geq  \int_X |f|^2 |\theta|^2_{\omega_D}e^{\psi_\delta}dV,\end{equation} 
for some other constant $C> 0$. We take the limit $\delta\to 0$ and the first part of (2) follows. 
\smallskip

For the remaining part, notice that we have 
\begin{equation}\label{d108}
\int_X\langle \Delta_Kf, e^{\psi_\delta}f\rangle dV= \int_X\langle g, e^{\psi_\delta}f\rangle dV
\end{equation}
by hypothesis. We infer that the inequality 
\begin{equation}\label{d109}
\int_X |D_Kf|^2 e^{\psi_\delta}dV+  \int_X |D_K^\star f|^2 e^{\psi_\delta}dV\leq 
C\big(\int_X|g|^2e^{\psi_\delta}dV+ \int_X|f|^2e^{\psi_\delta}dV\big)
\end{equation}
holds (here no further details are given, since the argument the same as above). In particular, there exists a positive constant $C> 0$ such that  we have  
\begin{equation}\label{d110}
\int_X |D_Kf|^2 e^{\psi_\delta}dV+  \int_X |D_K^\star f|^2 e^{\psi_\delta}dV\leq 
C
\end{equation}
for every $\delta\geq 0$.   

In order to use Corollary \ref{keylog} with e.g. $u= D_Kf$, we need to obtain a uniform bound for the integrals
\begin{equation}\label{d111}
I_\delta:= \int_X |D_KD_K^\star f|^2 e^{\psi_\delta}dV+  \int_X |D_K^\star D_Kf|^2 e^{\psi_\delta}dV.
\end{equation}

This is done as follows: first we remark that by hypothesis 
\begin{equation}\label{d111}
\int_X |D_KD_K^\star f + D_K^\star D_Kf|^2 e^{\psi_\delta}dV= \int_X|g|^2e^{\psi_\delta}dV \leq  \int_X|g|^2 d\mu< \infty,
\end{equation}
and thus we have 
\begin{equation}\label{d112}
I_\delta\leq C- 2\rel\int_X\langle D_KD_K^\star f,  D_K^\star D_Kf \rangle e^{\psi_\delta} dV.
\end{equation}
On the other hand, integration by parts gives
\begin{equation}\label{d113}
\begin{split}
\int_X\langle D_KD_K^\star f,  D_K^\star D_Kf \rangle e^{\psi_\delta} dV= & \int_X\langle D_K(e^{\psi_\delta} D_K^\star f),  D_K^\star D_Kf \rangle  dV\\ - & \int_X\langle d\psi_\delta\wedge D_K^\star f,  D_K^\star D_Kf \rangle e^{\psi_\delta}dV \\
\end{split}
\end{equation}
and since $\displaystyle \sup_{X\setminus D}|d\psi_\delta|_{\omega_D}< \infty$, we infer that 
\begin{equation}\label{d114}
I_\delta\leq C
\end{equation}
by using \eqref{d110}.

\noindent Finally, we invoke Corollary \ref{keylog} and infer that 
\begin{equation}\label{d115}
\int_X(|D^\star_K f|^2 +|D_K f|^2) (\sum_{i=1}^m \log |s_i|^2 ) d\mu_a <+\infty.
\end{equation}
Accidentally, the proof just finished settles \eqref{addnew} as well.
\end{proof}

\medskip

\noindent Our next statement deals with the Hodge decomposition for the operator $\Delta_K$.

\begin{thm}\label{hodgedec}
Let $L^2 (X, \cE_k)$ be the space of sections, which are $L^2$ integrable with respect to $(\omega_D, h_L)$ in the first case of \eqref{twocases} and $(\omega_D, g_D, h_L)$ in the second case of \eqref{twocases}.	
Then the following decomposition 
	$$L^2 (X, \cE_k) = \ker \Delta_K \oplus \im \Delta_K ,$$
holds, where $\ker \Delta_K$ is the space of forms $u\in \dom(\Delta_K)$ such that $\Delta_K u=0$
and $\im \Delta_K$ is the space of $\Delta_K (v)$ where $v\in \dom(\Delta_K)$. The sum is 
is orthogonal with respect to the natural $L^2$-product.
\end{thm}

\begin{proof}
Since $\Delta_K$ is a densely defined, self-adjoint operator on a Hilbert space, the decomposition 
	$$L^2 (X, \cE_k) = \ker \Delta_K \oplus \overline{\im \Delta_K}$$
holds, by general results in functional analysis (since in general, for any closed, densely defined operator $T$ the orthogonal of $\ker T$ is equal to the $\overline{\im T^\star}$). The completion is with respect to the $L^2$-norm.

Let $g\in L^2(X, \cE_k)$ such that $g$ is orthogonal to $\ker \Delta_K$. In order to conclude, we have to show that $g$ belongs to image of $\Delta_K$. 	
By the general results invoked above, there exists a sequence of forms $f_m \in \dom(\Delta_K)$ such that 
\[\Delta_K f_m \to g\] in $L^2$ and moreover, $f_m$ is orthogonal to  $\ker \Delta_K$.
\smallskip
	
\noindent We claim that \emph{there exists a positive constant $C> 0$ such that
\begin{equation}\label{d10}
\|f_m\|^2_{H_1}:= \|f_m\|^2+ \|D_Kf_m\|^2+ \|D^\star_Kf_m\|^2\leq C
\end{equation}
for all} $m$ (up to the choice of a subsequence).

Indeed, in case \eqref{d10} holds, we infer that $(f_m)$ converges to an element $f_\infty \in \dom(\Delta_K)$ such that $\Delta_K f_\infty =g$ and this would end the proof of our result. 
\smallskip	
	
\noindent Suppose (by contradiction) that we cannot find a constant as in \eqref{d10}.
By passing to a subsequence and renormalisation, it follows that there exists $f_m \in \dom(\Delta_K)$ such that 
\[\Delta_K f_m \to 0, \qquad \|f_m\|_{H^1} =1, \qquad f_m \perp \ker \Delta_K\] 
hold.

We can then assume that $f_m$ converges to some $f_0\in \dom(\Delta_K)$. It follows that $\Delta_K f_0=0$, and on the other hand
 $f_0$ is orthogonal to $\ker\Delta_K$. This can only happen if $f_0 \equiv 0$.
\smallskip
	
\noindent In order to derive a contradiction, notice that the equality
	$$\|D_K f_m\|^2 + \|(D_K)^\star  f_m\|^2= \langle \Delta_K f_m, f_m\rangle$$
holds, since $f_m\in \dom(\Delta_K)$. Given that $\Delta_K f_m \to 0$, we infer that
	$$\|D_K f_m\|^2 + \|(D_K)^\star  f_m\|^2 \to 0$$
is true. Therefore, since $\|f_m\|_{H_1}= 1$ we must have
\begin{equation}\label{d11}
\lim_m\|f_m\|^2= 1.
\end{equation}
On the other hand, notice that $f_m$ together with its Laplacian $\Delta_K$ converges weakly to zero, and in the complement of 
$D$, the metrics $\omega_D, h_L$ and the form $\theta$ are smooth. Therefore, we have
\begin{equation}\label{conve0}
\lim_{m\to +\infty}\|f_m \|_{L^2 (V)} =0
\end{equation}
for any $V\Subset X\setminus D$. The main point for defining the metric $\omega_D$ as we did in \eqref{m1} is here: the relation 
\eqref{conve0} equally holds for 
\[V\Subset X\setminus \Delta_2.\]
(this is immediately verified by using the family of local uniformisations $\pi_i$).

\noindent By Theorem \ref{keyinequ} the inequality
	\begin{equation}\label{ellipest}
	\langle |\theta|_{\omega_D}^2f_m , f_m \rangle + \|D_h f_m\|^2 + \|(D_h)^\star f_m\|^2 \leq C \|f_m\|_{L^2} + C \langle\Delta_K f_m, f_m \rangle.	 
	\end{equation}
holds. Since by Property \ref{property}, we have 
\begin{equation}\label{d12}
|\theta|_{\omega_D}^2\geq \ep_0\sum \log^2|s_i|^2
\end{equation}
for some positive constant $\ep_0> 0$, it follows that for any small enough open subset $W$ containing the support of the divisor $\Delta_2$ we have
\begin{equation}\label{d14}
\lim_{m\to +\infty}\|f_m \|_{L^2 (W)} =0 .
\end{equation}
This, together with \eqref{conve0} and \eqref{d11} gives a contradiction and our result is proved. 
\end{proof}
\medskip

\noindent We derive the following useful consequence. 
	\begin{proposition}\label{effectest}
	Let $\omega_D$ be the metric considered in \eqref{m1}.
	Let $g, f \in L^2 (X, \cE_k)$ such that $f\in \dom(\Delta_K)$ and such that $\Delta_K f =g$. If $f$ is $L^2$-orthogonal to $\ker\Delta_K$, then we have 
	$$\|f\|_{H^2} \leq C \|g\|_{L^2} $$
	for some uniform constant $C$. 
\end{proposition}

\begin{proof}
	We assume by contradiction there is no uniform constant. Then we can find a sequence $f_m \in H^2$ with $\|f_m\|_{H^2} =1$ such that $\Delta_K f_m =g_m$, $\|g_m\|_{L^2} \to 0$ and $f_m$ is $L^2$-orthogonal to $\ker\Delta_K$.
	By using Lemma \ref{h2norm}, we know that $\|f_m\|_{L^2} \geq C$ for some uniform constant $C>0$. 
	Then the argument in Theorem \ref{hodgedec} shows that $f_m \to 0$ in $L^2$. We obtain thus a contradiction.
\end{proof}


\subsection{Hodge decomposition in smooth setting}\label{Hodsmo}

As usual, we denote by $\mathcal E_k$ one of the following vector bundles
\begin{equation}\label{end1}
\oplus_{r+s= k}\Lambda^{r, s}T^\star_X\otimes L, \qquad  
\oplus_{r+s= k}\Lambda^{r, s}T^\star_X\langle \Delta_1\rangle\otimes L
\end{equation}
on $X_0$. The log pair $(X, D)$ is endowed with the metric $\omega_D$ defined in  \eqref{m1}, and $L$ with the metric $h_L$ defined in \Cref{specialhar}, whose curvature current is equal to $\sum a_i [D_i]$. We recall that  
the coefficients $a_i$ are rational numbers in $]-1, 0]$. The logarithmic tangent bundle $T_X\langle \Delta_1\rangle$ is endowed with the metric $g_D$.

\noindent These objects induce a connection 
\begin{equation}\label{end2}
	\nabla: \cE_k\to T^\star_X\otimes \cE_k\oplus \ol T^\star_X\otimes \cE_k;
\end{equation}
for example, in the absence of $L$ and in the first case in \eqref{end1} the connection $\nabla$ is induced by Levi-Civita (= Chern, since the metric $\omega_D$ is Kähler).  

Note that \eqref{end2} is not the logarithmic holomorphic connection defined in \eqref{logholoconnecion}. In this section, whenever we use $\nabla$, it always refers to \eqref{end2}.
 
\medskip

\noindent The following space is plays the role of smooth forms in our current setting. For similar considerations (in a simplified context) we refer to the elegant article \cite{Biq97}.
\smallskip 

\begin{defn}\label{defend1}
Let $X$ be a compact K\"ahler manifold and let $D$ be a simple normal crossing divisor on $X$.  Fix any $(L,\nabla)\in M_{\rm DR}(X/D)$ and let $h$ be a harmonic metric satisfying the properties in \Cref{property}. 

We denote by  $C^\infty _{K} (X, \cE_k)$ be the sheaf of $L$-valued forms $u$ of  total degree $k$ such that for any integer $r\geq 0$, the iterated 
	covariant derivative $\nabla^r(u)$ is $L^2$ with respect to any measure 
$d\mu_a$.	
More precisely, this means that 
\[
  \int_X |\nabla^r(u)|^2\, d\mu_a < +\infty
\]
for all \(r \in \mathbb N\) and multi-indices \(a = (a_i)_{i \in I} \in \mathbb N^I\). 
Here we use the notation 
\begin{align}\label{eq:mu}
   d\mu_a := \prod_{i\in I} \log^{2a_i}\!\bigl(|s_i|^2\bigr)\, dV_{\omega_D},
\end{align}
where we decompose \(\Delta_2 = \sum_{i \in I} D_i\) into prime components, and each \(s_i\) denotes the canonical section defining \(H^0(X, \O_X(D_i))\). 
The norm \(|s_i|\) is taken with respect to a fixed smooth metric on \(\O_X(D_i)\).

\medskip

Moreover,  we use the following notations:
\begin{itemize}
	\item If $\displaystyle \cE_k= \oplus_{r+s= k}\Lambda^{r, s}T^\star_X\otimes L$, we set $C^\infty _{k,K}(X,L) := C^\infty _{K} (X, \cE_k)$
	
	\item If $\displaystyle \cE_k = \oplus_{r+s= k}\Lambda^{r, s}T^\star_X\langle \Delta_1\rangle\otimes L$, we set $C^\infty _{k,K}(X, \Delta_1, L) := C^\infty _{K} (X, \cE_k)$.
	\end{itemize}
	We sometimes write simply  $C_{K}^\infty(X,  L)$ (resp. $C_{K}^\infty(X,\Delta_1, L)$) when the degree is not essential. 
\end{defn}

\begin{remark}  In order to motivate the definition above, we have the next observations.
	
\begin{itemize}
	
\item In case $\Delta_2+ \Delta_3= 0$ and $\displaystyle \cE_k= \oplus_{r+s= k}\Lambda^{r, s}T^\star_X\otimes L$, then $\cC^\infty _{\bullet,K}(X,L) $ coincides with the sheaf of smooth forms.

\item In case $\Delta_2+ \Delta_3= 0$ and $\displaystyle \cE_k= \oplus_{r+s= k}\Lambda^{r, s}T^\star_X\langle \Delta_1\rangle\otimes L$, then $\cC^\infty _{\bullet,K}(X,\Delta_1,L) $ coincides with the sheaf of smooth forms with logarithmic poles.

\item Of course, $\cC^\infty _{K}(X, \cE_i) $ contains the smooth forms with compact support in $X\setminus (\Delta_2 +\Delta_3)$. 
\end{itemize}
 \end{remark}
 
 \begin{remark}\label{useful}
 	The sheaf $\cC^\infty_{K}(X,\cE_k)$ appears to be well-adapted to our context. We will next show that the space $C^\infty_{K}(X,\cE_k)$ enjoys several very useful properties.
 	
 	\begin{itemize}
 		\item In Proposition~\ref{smoothhodge}, we show that if $f \in \dom(\Delta_K)$ and $g \in L^2$ with
 		\[
 		\Delta_K f = g,
 		\]
 		and if moreover $g \in C^\infty_K(X, \cE_k)$, then $f \in C^\infty_K(X, \cE_k)$.  
 		In particular, any harmonic form $f \in \ker \Delta_K$ belongs to $C^\infty_K(X, \cE_k)$.
 		
 		\item In \Cref{lem:residue}, we show that $\cC^\infty_{K}(X, \Delta_1, L)$ behaves well with respect to the residue map.
 		
 		\item Assume $u$ is an $L^2$-section of $\cE_k$. In general, the wedge product $\theta_0 \wedge u$ is not $L^2$. However, this is the case if $u \in \cC^\infty_{K}(X, \cE_i)$.  
 		Namely, in \Cref{propend1} we prove that if $u \in \cC^\infty_{K}(X, \cE_i)$, then $\theta_0 \wedge u \in \cC^\infty_{K}(X, \cE_i)$.
 		
 		\item In \Cref{quasi}, we show that $\cC^\infty_{\bullet, K}(X, \Delta_1, L)$ provides a resolution of the complex
 		\begin{equation}
 			\nabla: \cdots \to \mathcal{O}_X(L) \otimes \Omega_X^{\bullet}(\log D) 
 			\to \mathcal{O}_X(L) \otimes \Omega_X^{\bullet}(\log D) \to \cdots 
 		\end{equation}
 	\end{itemize}
 \end{remark}
 
\medskip

\noindent In order to establish part the properties of $\cC^\infty _{K}(X, \cE_i)$ stated in Remark \ref{useful} above, we will 
use coordinates adapted to $(X, D)$, i.e. fix a finite covering 
\[X= \bigcup U_\alpha\]
of $X$ with coordinates sets $(U_\alpha, w_\alpha)$, where the coordinates 
\[w_\alpha= (w_\alpha^1,\dots,  w_\alpha^n)\]
are such that 
\begin{equation}
	D\cap U_\alpha= \big(\prod_1^{p+q+r}w_\alpha^i = 0\big)
\end{equation}
where the first $r$ coordinates are the local equations of $\Delta_1$, the next $p$ are the local equations of $\Delta_2$ and finally $\Delta_3$ is given locally by 
$\prod_{r+p+1}^{r+ p+q}w_\alpha^i = 0$. 
\smallskip

\noindent By considering the usual ramified covering \[z\to (z_1,\dots, z_r, z_{r+1}^k,\dots, z_{r+p+q}^k, z_{r+p+q+1}, \dots, z_n),\] for any local calculations we can assume that the following holds. 
\begin{itemize} 

\item The divisor $\Delta_3= 0$;

\item The metric $h_L$ is smooth;

\item The form $\theta_0$ is holomorphic, with log poles (along the pre-image of $\Delta_2$).
\end{itemize}
\emph{For the rest of this subsection we take this as granted.}
\medskip

\noindent Fix an index $\alpha$ and call $(\Omega, z)$ the corresponding coordinate system. 
We next introduce new frames for the tangent bundle and its log version which are well-suited for the computations to follow.
\smallskip

\noindent $\bullet$ In the case of $T_X|_{\Omega\setminus D}$ we define
\begin{equation}\label{end6}
	e_i:= z^i\log|z^i|\frac{\partial}{\partial z^i},\qquad e_i^\star:= \frac{dz^i}{z^i\log|z^i|},
\end{equation}
for each index $i= r+1,\dots, p+r$ 
whereas for $i\geq p+r+1$ or $i\leq r$ we introduce
\begin{equation}\label{end7}
	e_i:= \frac{\partial}{\partial z^i},\qquad e_i^\star:= {dz^i}.
\end{equation}
\smallskip

\noindent $\bullet$ In the case of $T_X\langle\Delta_1\rangle|_{\Omega\setminus D}$ we introduce
\begin{equation}\label{logend8}
	l_i:= z^i\frac{\partial}{\partial z^i},\qquad l_i^\star:= \frac{dz^i}{z^i},
\end{equation}
for each index $i= 1,\dots, r$ 
and then

\begin{equation}\label{logend6}
	l_i:= z^i\log|z^i|\frac{\partial}{\partial z^i},\qquad l_i^\star:= \frac{dz^i}{z^i\log|z^i|},
\end{equation}
for $i= r+1,\dots, r+p$ 
whereas for $i\geq r+p+1$ we introduce
\begin{equation}\label{logend7}
	l_i:= \frac{\partial}{\partial z^i},\qquad l_i^\star:= {dz^i}.
\end{equation}

The main motivation for this definition is that the vectors $(e_i)_{i=1,...,n}$ induce an orthonormal basis with respect to the model metric with Poincar\'e singularities along $\Delta_2$, and therefore they are 'almost orthogonal' with respect to $\omega_D$ in the former case, and with respect to the metric $g_D$ in the latter. Let
\begin{equation}\label{met1}
a_{\beta\ol\alpha}:= \big\langle e_\beta, e_\alpha\big\rangle_{\omega_D}, \qquad b_{\beta\ol\alpha}:= \big\langle l_\beta, l_\alpha\big\rangle_{g_D}
\end{equation}
be the coefficients of the metric $\omega_D$ and $g_D$ with respect to the frames above. As consequence of the evaluations in 
sections \ref{metricTX} and \ref{metriclogTX} we have the following statement.
\begin{lemme}\label{lemend1}
Corresponding to each indexes $\alpha,\beta$ exists smooth a function $F$ such that we have
\[a_{\beta\ol\alpha}(z)= F\big(\delta_{\beta, p}z^\beta\log|z^\beta|, \delta_{\alpha, p}\ol z^\alpha\log|z^\alpha|, \frac{\delta_{i, p}}{\log|z^i|}, z\big).\] 
A similar statement holds true for the functions $b_{\beta\ol\alpha}$ and their inverses $(a^{\ol\alpha \beta})$ and $(b^{\ol\alpha \beta})$. 
Here $\delta_{i, p}$ is equal to one if $i\in \{r+1,\dots, r+p\}$ and zero if this is not the case.
\end{lemme}

\noindent Indeed, in the sections mentioned above the evaluations are done with respect to different frames, but it is immediate to "switch" to the 
$(e_\alpha)$ or $(l_\alpha)$. Of course the functions generically denoted by $F$ in Lemma \ref{lemend1} depend on the indexes $\alpha,\beta$. 
\smallskip

We use the same notation $\nabla$ for the connection on $(T_X, \omega_D)$ and $(T_X\langle\Delta_1\rangle, g_D)$, and
write
\begin{equation}\label{end8}
	\nabla e_\alpha^\star= \sum_\beta A^\alpha_{\beta} e_\beta^\star, \qquad \nabla l_\alpha^\star= \sum_\beta B_\beta^{\alpha} l_\beta^\star.
\end{equation}
where 
\begin{equation}\label{met2}
A^\alpha_{\beta}:= \sum a_{\beta\ol k}\partial a^{\ol k \alpha},\qquad B^\alpha_{\beta}:= \sum b_{\beta\ol k}\partial b^{\ol k \alpha}.
\end{equation}
The local $1$-forms $A$ and $B$ can be written as follows
\begin{equation}\label{met3}
A^\alpha_{\beta}= \sum A^\alpha_{\beta \gamma} e^\star_\gamma, \qquad B^\alpha_{\beta}= \sum B^\alpha_{\beta \gamma} e^\star_\gamma
\end{equation}
i.e. with respect to the frame we fixed above for $T^\star_X|_\Omega$, and as consequence of Lemma \ref{lemend1} we see that each of the coefficients in \eqref{met3} can be written as follows
\begin{equation}\label{met4}
A^\alpha_{\beta \gamma}= F\big(z^{r+j}\log|z^{r+j}|, \ol z^{r+j}\log|z^{r+j}|, \frac{1}{\log|z^{r+j}|}, z\big)
\end{equation}
where it is understood that $j= 1,\dots, p$.


\begin{defn}\label{pogrow} Consider the following class of functions 
\begin{equation}\label{met5}
z\to F\big(\log|z^{r+j}|, z^{r+j}\log^{k_j}|z^{r+j}|, \ol z^{r+j}\log^{m_j}|z^{r+j}|, \frac{1}{\log|z^{r+j}|}, z\big)
\end{equation}
where $k_j, m_j$ are positive integers. Moreover, $F$ is \emph{smooth and polynomial with respect to the first $p$ variables $\log|z^{r+1}|, \dots, \log|z^{r+p}|$}. In what follows we will refer to such $F$ as a \emph{function with logarithmic growth}.
\end{defn}
\noindent It has the following simple, but very important property.

\begin{lemme}\label{lemend2}
	Let $\Psi$ be a function with logarithmic growth, and consider its differential
	\[d\Psi= \sum_i \Psi_i e_i^\star+ \sum_i \Psi_{\ol i} \ol e_i^\star\]
	expressed in the base $(e_i^\star)_{i=1,\dots, n}$.
	Then each of the coefficients $\Psi_i$ have logarithmic growth.
\end{lemme}
\begin{proof}
	Indeed, for $\alpha = r+1,\dots, r+p$ we have 
	\[\partial \log|z^\alpha|^2= \log|z^\alpha| e_\alpha^\star, \qquad  \partial \big(z^\alpha \log|z^\alpha|^2\big)= (1+ \log|z^\alpha|^2)z^\alpha \log|z^\alpha| e_\alpha^\star\]
from which our assertion follows.
\end{proof}
\medskip

\noindent A first easy consequence of these considerations is the following "denseness" result.

\begin{lemme}\label{lemend5} The space of smooth forms with compact support in $X\setminus \Delta_2$ is dense in $\cC^\infty _{K}(X, \cE_i) $, in the following sense.
	If $u\in\cC^\infty _{K}(X, \cE_i) $ is an arbitrary form, then there exists a family $(u_\ep)\subset \cC^\infty _c (X\setminus \Delta_2, \cE_i)$ such that
	\[\lim_{\ep\to 0}\int_X|\nabla^r u- \nabla^r u_\ep|d\mu_a= 0\]
	for any $r\geq 0$ and for any multi-index $a$.
	
	Here $\cC^\infty _c(X\setminus \Delta_2, \cE_i)$ is the space of the smooth sections (in the conic sense) and of compact support in $X\setminus \Delta_2$.
\end{lemme}

\noindent Before explaining the proof of this statement, we remark that it has the following consequence. 

\begin{remark}
	Let $T:\cC^\infty _{K}(X,\cE_i)\to \mathbb C$ be a current of finite order, so that $T$ is linear and we moreover have 
\[|T(u)|\leq C\sum_{r=1}^d\Vert \nabla^r u\Vert_{2, a}\]  
for all $u\in \cC^\infty _{K}(X,\cE_i)$. If the support of $T$ is contained in $\Delta_2$, then $T$ is identically zero. Indeed, we have $T(u_\ep)= 0$ for all $\ep> 0$, hence $T(u)= 0$ as consequence of Lemma \ref{lemend5}.
\end{remark}

\begin{proof}
	We consider the truncation functions $\displaystyle \rho_\ep:= \Xi_\ep \big(\log\log\frac{1}{|s_D|^2}\big)$, together with 
	\[u_\ep:= \rho_\ep u.\]
	The norm of any covariant derivative of $\rho_\ep$ with respect to the Poincar\'e metric $\omega_D$ is bounded by a polynomial expression in $\displaystyle \log\frac{1}{|s_i|^2}$,
	where the $s_i$ are the sections defining the components of $D$. Given this, the claim in our lemma is obvious.	
\end{proof}

\noindent Let $u$ be a section of the bundle $\cE_i$, such that $\Supp(u)\subset \Omega$, so that we can write
\begin{equation}\label{end10}
	u= \sum u_{I\ol J}e_I^\star \wedge \ol e_J^\star \otimes e_L, \qquad u= \sum u_{I\ol J}l_I^\star \wedge \ol e_J^\star \otimes e_L
\end{equation}
with respect to the frames introduced in \eqref{end6}, \eqref{end7} or their logarithmic versions. Then we have the following simple criteria for $u$ to belong to the space 
$\cC^\infty _{K}(X,\cE_i)$.
\begin{lemme}\label{lemend3}
	For each couple of multi-indexes $M= (m_i), N= (n_j)$ of finite length and whose coefficients belong to $\{1,\dots, p\}$ we introduce the following functions
	\[z^M\log|z^M|:= \prod z^{m_i}\log|z^{m_i}|, \qquad \ol z^N\log|z^N|:= \prod \ol z^{n_i}\log|z^{n_i}|.\]
	Let $u$ be a form as in \eqref{end10}. Then $u$ defines an element of the space $\cC^\infty _{K}(X,\cE_i)$ if and only if 
	\[z^M\ol z^N\log|z^M| \log|z^N| \partial_{M\ol N} \partial_{P\ol Q}u_{I\ol J}\in L^2(\Omega, e^{-\varphi_L}d\mu_a)\]
	for any $a, M, N, I, J$ and $P, Q\subset \{p+1,\dots, n\}$. In the expression above we denote by $\partial_{M\ol N} u_{I\ol J}$ the partial derivative of $u_{I\ol J}$ of total order 
	$|M|+ |N|$, where we take holomorphic derivatives with respect to the variables corresponding to the indexes of $M$ and anti-holomorphic derivative with respect to the variables corresponding to the indexes of $N$.
\end{lemme}
\begin{proof}
	For sections $u$ on the LHS of \eqref{end10}, we write
	\[\nabla^{1, 0} u= \sum_{I, J, i} u_{I\ol J i} e_i^\star\otimes \big(e_I^\star \wedge \ol e_J^\star\big)+  \sum_{I, J, \alpha, k}u_{I\ol J}A^{i_\alpha}_{ik}e_i^\star\otimes \big(e_{I(\alpha, k)}^\star \wedge \ol e_J^\star\big), \]
	where if $I=(i_1,\dots, i_r)$, then $I(\alpha, k)$ is the multi-index obtained by replacing $i_\alpha$ by $k$.
	
	\noindent Assume next that $u_{I\ol J}\in L^2(\Omega, e^{-\varphi_L}d\mu_a)$. The previous formula -and the similar one for $\nabla^{0,1}$-show that $\nabla u\in L^2(\Omega, e^{-\varphi_L}d\mu_a)$ for all $a$ if and only if $u_{I\ol J i}$ and $u_{I\ol J \ol i}$ are in $L^2(\Omega, e^{-\varphi_L}d\mu_a)$
	for all $i, a, I, J$. On the other hand, we have $u_{I\ol J i}= z^i\log|z^i|\partial_i u_{I\ol J}$ for $i=1,\dots, p$ and $u_{I\ol J i}= \partial_i u_{I\ol J}$ for $i\geq p+ 1$.
	
	\noindent The proof is completed by iterating this argument (i.e. using induction), where we use the fact that the coefficients $A^i_{jk}$ have logarithmic growth
	(cf. Lemma \ref{lemend1}) combined with Lemma \ref{lemend2}. The same arguments apply in case of sections with log-poles, cf. the RHS of \eqref{end10}.
\end{proof}

 Yet another confirmation that the sheaf $\cC^\infty _{K}(X,\cE_k)$ is \emph{the right one} in our context is provided by the next statement, showing that this space behaves well with respect to the residue map 
(corresponding to the components 
of $\Delta_1$ and to any of their complete intersections).

\begin{lemme}\label{lem:residue}
	 Let  $u\in C_{\bullet,K}^\infty(X,\Delta_1, L)$. Then for any irreducible component $Y$ of $\Delta_1$, the residue 
$u_{Y}\in C_{\bullet-1,K}^\infty(Y, (\Delta_1 -Y)\cap Y, L)$ of $u$ along $Y$.
Equivalently, for $u\in C_{\bullet,K}^\infty(X, L)$, the restriction (trace) $u|_{Y}$ is well defined and $u |_Y \in C_{\bullet,K}^\infty(Y, L)$.
\end{lemme}

\begin{proof} Given that we define $\cC_{\bullet,K}^\infty(X,\Delta_1, L)$ and $\cC_{\bullet,K}^\infty(X, L)$ via $L^2$ conditions, a-priori it is not clear that this space should behave well by restriction to a lower-dimensional sub-manifold. 
	
\noindent Nevertheless, for any $u\in C_{\bullet,K}^\infty(X, L)$, we have the following inequalities
\begin{align}
\int_Y\big|\nabla^r u_Y\big|^2d\mu_a\leq & \int_Y\big|\nabla^r u\big|^2_{|Y}d\mu_a \label{rest1}\\
= & \int_X\big|\nabla^r u\big|^2\log^{2a}|s_{\Delta_2}|^2\omega_D^{n-1}\wedge \big(\Theta(Y)+ \sqrt{-1}\ddbar \log|s_Y|^2\big) \label{rest2}\\
\nonumber
\end{align}
where $u$ is assumed to have compact support in the complement of $\Delta_2$. The first inequality \eqref{rest1} is clear, and in \eqref{rest2} we are using the Poincar\'e-Lelong theorem. 

Since $\Theta(Y)\leq C\omega_D$ for some constant $C> 0$, the first term in \eqref{rest2} is bounded by the following quantity
\[C\int_X\big|\nabla^r u\big|^2d\mu_a\]	
where we stress on the fact that $C$ is independent of the form $u$.

The singular term in \eqref{rest2} is handled via integration by parts: it equals
\begin{equation}\label{rest5}\int_X\log|s_Y|^2\sqrt{-1}\ddbar\big(\big|\nabla^r u\big|^2\log^{2a}|s_{\Delta_2}|^2\big)\wedge\omega_D^{n-1}\end{equation}	
since the metric $\omega_D$ is K\"ahler. 

As $\omega_D$ has Poincar\'e singularities along $\Delta_2$, we have 
\begin{equation}\label{rest3}
\big|d\log^{2a}|s_{\Delta_2}|^2\big|_{\omega_D}\leq C(a)\log^{2a}|s_{\Delta_2}|^2
\end{equation}
and moreover, we have the point-wise inequality
\begin{equation}\label{rest4}
\big|\Delta''_{\omega_D}|\nabla^r u|^2\big|\leq C\sum_{k=r}^{r+2}\big|\nabla^k u|^2.
\end{equation} 
Because the support of $u$ is contained in the complement of $\Delta_2$,
the singular part of $\sqrt{-1}\ddbar \log|s_{\Delta_2}|^2$ can be ignored.
\smallskip

\noindent Therefore, the integral \eqref{rest5} is bounded by 
\begin{equation}\label{rest6}
C\sum_{k=r}^{r+2}\int_X\log|s_Y|^2\big|\nabla^k u|^2d\mu_{a},\end{equation}
as we see from \eqref{rest3}, \eqref{rest4}. 

By the snc condition on the support of $D$, for any positive $N\geq 0$ there exists a constant $C(N)$ such that we have 
\begin{equation}\label{rest7}
	\int_X\log^N\frac{1}{|s_Y|^2}dV_{\omega_D}\leq C(N).
\end{equation}
Combined with the H\"older inequality, this shows that the integrals in \eqref{rest6} are bounded by something like the $\left( \displaystyle \frac{1}{1+\delta_n}\right)^{\rm th}$ root of 
\begin{equation}\label{rest8}
\int_X\big(\big|\nabla^k u\big|^2\log^{2a}|s_{\Delta_2}|\big)^{2+ 2\delta_n}dV_{\omega_D}	
\end{equation}	 
where we adjust $N$ on \eqref{rest7} so that $\delta_n$ in 
\eqref{rest8} equals the one in the Sobolev inequality of Lemma \ref{sobolev}. 
By Lemma \ref{sobolev} and a few more calculations (which we skip since they are completely similar to \eqref{rest3}, \eqref{rest4}), we are done.	
\end{proof}

\medskip

\noindent After all these preparations the following statement is easy to establish.
\begin{proposition}\label{propend1}
	Let  $u\in \cC^\infty _{K}(X,\cE_i)$ and let $f$ be a function on $X$, which has smooth on $X\setminus \Delta_2$ and with logarithmic growth along $\Delta_2$. Then $f \cdot u \in \cC^\infty _{K}(X,\cE_i)$. Moreover, $\theta\wedge u$ and $\theta^\star(u)\in\cC^\infty _{K}(X,\cE_{\bullet})$. 
\end{proposition}

\begin{proof} By using a partition of unity, we can assume that $\Supp(u)\subset \Omega$, where $\Omega$ is one of the coordinate subsets $U_\alpha$. 
	The coefficients of $f\cdot u$ are simply $f\cdot u_{I\ol J}$, and by hypothesis, derivatives of any order
	\[e_{i_1}\cdot\dots e_{i_m}\cdot f\]
	of $f$ still have logarithmic growth. It then follows that $f \cdot u \in \cC^\infty _K$.
	
	The same considerations apply for $\theta\wedge u$: locally we have
	\[\theta|_\Omega= \sum_{i\leq p}\log|z^i| \tau_i e_i^\star+ \sum_{i\geq p+1}\tau_i e_i^\star\]
	where the functions $(\tau_i)$ are smooth -actually they are holomorphic, but this is irrelevant here. Since the expressions
	\[z\to \log|z^i| \tau_i(z), \qquad z\to \tau_j(z)\]
	clearly define functions with logarithmic growth for $i= 1,\dots p$ and $j\geq p+1$,
	the conclusion follows from Lemma \ref{lemend3}. 
	
	Finally, we recall that the formula 
	\begin{equation}\label{end11} 
		\theta^\star(u)= \sum_{i,k}\ol\theta_i \omega^{\ol i k}\frac{\partial}{\partial z_k}\rfloor u	
	\end{equation}	
	holds, cf. \eqref{form1}. A closer look at the coefficients above shows that
	\begin{equation}\label{end12} 
		\ol\theta_i \omega^{\ol i k}\frac{\partial}{\partial z_k}= \psi_{ik}e_k	
	\end{equation}
	where $\psi_{ik}$ is a function with logarithmic growth. From this point on, we invoke Lemma \ref{lemend3} again to deduce that $\theta^\star(u)\in\cC^\infty _K (X, \cE_{i-1})$.
	
\end{proof}

\medskip

\noindent The following regularity result will play a key role in the next sections.
\begin{thm}\label{smoothhodge}
	Let $g\in L^2$ and $f\in \dom(\Delta_K)$ such that $\Delta_K f=g$. If $g\in  \cC^\infty _{K}(X,\cE_i)$, then 
 $$f\in  \cC^\infty _{K}(X,\cE_i) ,\quad D_K f \in  \cC^\infty _{K}(X,\cE_{i+1}),\quad D_K ^\star f\in  \cC^\infty _{K}(X,\cE_{i-1}).$$
	In particular, if we consider the Hodge decomposition
	\[u = \mathcal H(u)+ \Delta_K(v)\]
	of a form $u\in  \cC^\infty _{K}(X,\cE_i)$, then its harmonic component $\mathcal H(u)$ as well as $v$ are elements of $ \cC^\infty _{K}(X,\cE_i)$. 
\end{thm}	

\begin{proof} A first simple remark is that in case we are able to show 
	that $f\in  \cC^\infty _{K}(X,\cE_i)$, then follows from Proposition \ref{propend1}
	that this is equally the case for its differentials $D_K f$ and $D_K ^\star f$. In what follows we will prove that $f\in  \cC^\infty _{K}(X,\cE_i)$ 
	
	The proof is based on the classical boot-strapping argument.
	In the first place, Theorem \ref{logpole} implies that the following
	\[f, D_K f, D_K^\star f\in L^2\big(X, \cE_{\bullet}, d\mu_a\big)\] holds.
	
The Bochner formulas of Lemma \ref{bochner1} show that we have 
	\[\nabla f\in L^2\big(X,\cE_{\bullet} , d\mu_a\big).\]
	The last part of Theorem \ref{logpole} equally shows that
	\[D_KD_K^\star f, D_K^\star D_K f \in L^2\big(X, \cE_{\bullet},  d\mu_a\big),\]
	from which we infer that 
	\begin{equation}\label{end13} 
		\nabla (D_K^\star f), \nabla (D_K f)\in L^2\big(X, \cE_{\bullet}, d\mu_a\big)\end{equation}
	as well.
	
	\noindent Assume now that we are in the local setting, and let $\xi$ be one of $(e_i)$. Then we have the following statement 
	\begin{claim}\label{claimend1}
		The commutation relations
		\begin{equation}\label{end14} 
			[D_K, \nabla_{\xi}]u= \mathcal P(u, \nabla u), \qquad [D_K^\star, \nabla_{\xi}]u= \mathcal P(u, \nabla u)
		\end{equation}
		as well as
		\begin{equation}\label{end15} 
			[\Delta_K,\nabla_{\xi}]u= \mathcal P(u, \nabla u, \nabla^2 u)	
		\end{equation}
		hold true, where $\mathcal P$ is a differential operator whose coefficients are logarithmic growth functions.
	\end{claim} 
	
	\noindent We can restart the same procedure by replacing $f$ with $\nabla_{\xi} f$, and since we can repeat this indefinitely, we are done, \emph{provided that} Claim \ref{claimend1} holds, which is what we next prove.
	
	\noindent First we consider differential operators of the following type 
	\begin{equation}\label{end20}  
		\mathcal P(f)= \sum_{\Lambda}a_{\Lambda}e_\Lambda \cdot f	
	\end{equation}
	acting on functions $f$, where $\Lambda$ is an ordered collection of indexes among $1,\dots, n$ together with $ \ol 1,\dots, \ol n$, and 
	\[e_\Lambda:= e_{\lambda_1}\cdot \big(\dots (e_{\lambda_d}\cdot f)\big) \]
	with the convention that $e_{\ol \alpha}= \ol e_{\alpha}$. Here we interpret the vectors $e_\alpha, \ol e_\alpha$ as derivations. Since they don't commute, the order is relevant (but not too important). Finally, the coefficients $(a_{\Lambda})$ in the expression \eqref{end20} above are assumed to have logarithmic growth. 
	
	A very simple remark, consequence of Lemma \ref{lemend2} is that in case $\mathcal P_1, \mathcal P_2$ are differential operators as in \eqref{end20}, whose respective orders are $d_1$ and $d_2$, then so is the commutator 
	\[[\mathcal P_1, \mathcal P_2]\]
	and moreover its order is at most $d_1+ d_2- 1$. 
	
	Therefore it is sufficient to show that the operators appearing in the claim are a vector-valued version of \eqref{end20}. The verification is done by writing in coordinates each of the operators involved. If $u= \sum u_{I\ol J}e_I^\star\wedge \ol e_J^\star$, then we have 
	\[\partial u= \sum_{I, J, \alpha} e_\alpha (u_{I\ol J})e_\alpha^\star\wedge e_I^\star\wedge \ol e_J^\star- \sum_{I, J, j\in J} \delta_{j, p}u_{I\ol J} e_j^\star\wedge e_I^\star\wedge \ol e_J^\star\]
where the second term is due to the fact that $\partial \ol e_k^\star= - e_k^\star \wedge \ol e_k^\star$ for $k= r+1,\dots, r+p$, and we denote 
by $\delta_{j, p}= 1$ if $j\in \{r+1,\dots, r+p\}$ and zero otherwise. A similar expression can be obtained for the antiholomorphic derivative
	$\dbar u$. Moreover, any form $\tau$ of the following type
	\[\tau = \sum a_i \frac{dz^i}{z^i}+ \sum b_i \frac{d\ol z^i}{\ol z^i}\]
	can be written as 
	\[\tau = \sum a_i\log|z^i| e_i^\star+ \sum b_i \log|z^i| \ol e_i^\star\]
	which accounts for the difference between $D_K$ and $d$. 
	
	The operator $\nabla_\xi$ is clearly of type \eqref{end20}, as consequence of \eqref{end8}. For the operator
	$D_h^\star$, one can argue either directly or by using the commutation relation $D_h^\star= \sqrt{-1}[\Lambda_\omega, \dbar]$. Claim 
	\ref{claimend1} is therefore established.
	\end{proof}

\medskip

\noindent
Theorem \ref{smoothhodge} has several important consequences.
\begin{proposition}\label{finitdim}
	The kernel $\ker(\Delta_K)$ is a finite dimensional subspace contained in $\cC^\infty _K (X, \cE_i)$. 
\end{proposition}

\begin{proof}
	By  \Cref{smoothhodge}, $\ker(\Delta_K)$ is a subspace in  $\cC^\infty _K (X, \cE_i)$. Now we would like to show that it is of finite dimension.
	Let $f\in \ker \Delta_K$. Thanks to \Cref{keyinequ}, we know that 
	\begin{equation}\label{l2control} \|D_h f\|_{L^2} +  \|D^\star _h f\|_{L^2} +\int_X |\theta|^2 |f|^2 \omega_D ^n  \leq C \|f\|_{L^2} .
	\end{equation}
	Set $\|f\|^2 _{H^1} := \|D_h f\|^2 _{L^2} +  \|D^\star _h f\|^2 _{L^2} +\|f\|^2_{L^2}$.
	
	Then we can find a sequence $f_i\in \ker\Delta_K$  such that $ \|f_i\|^2 _{L^2}=1$ for every $i$. To prove the proposition, it is equivalent to show that by passing to some subsequence, $f_i$ converges  in $L^2$.
	
	Note that by passing to some subsequence,  $f_i$ converges weakly to some $f_0$ in $L^2$. Then $f_i -f_0 $ converges weakly to $0$ in $L^2$.
	Now $\|f_i -f_0\|_{L^2} \leq 2$ for every $i$.
	Then on any open set $K\Subset X\setminus \Delta_2$, thanks to \eqref{l2control} and the compact embedding theorem with respect to smooth metric, we know that $f_i -f_0$ converges to $0$ in the $L^2$-norm on $K$. In a small $\ep$-neighborhood $U_\ep$ of $\Delta_1$, we know that $|\theta|^2 \geq \log^2 \ep$. By applying \eqref{l2control} to $f_i -f_0$, we obtain
	$$\int_{U_\ep} |f_i -f_0|^2 \omega_D ^n \leq \frac{C}{\log^2 \ep}  \|f_i -f_0\|_{L^2} =\frac{C}{\log^2 \ep} . $$
	Therefore, we obtain
	$$\limsup_{i\to \infty} \int_X   |f_i -f_0|^2 \omega_D ^n \leq \frac{C}{\log^2 \ep}  .$$
	By letting $\ep \to 0$, we know that $f_i -f_0 \to 0$ in $L^2$ on $X$. The proposition is thus proved. 
\end{proof}

Following the notations in \cite{Sim92}, we define $D^c _K := D''_K -D'_K$ and we have the relation 
\begin{equation}\label{identity}
	D^{c \star}_K =-\sqrt{-1} [\Lambda_\omega, D_K] \qquad D_K ^\star = \sqrt{-1} [\Lambda, D^c_K] .
\end{equation}
The reason why we introduce it is that, as $D^c _K D_K = D''_K D'_K$, if we would like to solve $D_K$-equation in the above setting,  so we need a $D^c _K D_K$-lemma which plays the same role as the $\ddbar$-lemma.

\noindent We introduce $\Delta_K^c := [D^c_K, D^{c\star}_K]$, i.e., the Laplace operator corresponding to the derivative $D^c_K$,  
and take $\cE_k = \bigoplus_{r+s=k} \Lambda^{r,s} T^\star_X \otimes L$, then we have  
\begin{equation}\label{end18} 
	\Delta_K = \Delta_K^c
\end{equation}
since $[D'_K, D_K^{''\star}] = 0$.  
This gives rise to the following proposition.

\begin{proposition}[$\ddbar$-lemma for $C^\infty_K (X, L)$]\label{ddbar1}
	Let $u \in C^\infty_{\bullet, K} (X, L)$ such that $u$ is $D_K$-closed and $D^c _K$-exact. Then 
	$$u = D_K D^c _K v$$
	for some $v\in C^\infty_{\bullet -2, K} (X, L)$.
	\end{proposition}
\begin{proof}
Note that we have $\Delta_K = \Delta^c _K$ for $C^\infty _K (X, \cE_k)$ when $\cE_k =\oplus_{r+s= k}\Lambda^{r, s}T^\star_X\otimes L$.
Then \Cref{smoothhodge} together with the proof of standard $\ddbar$-lemma implies this proposition.
\end{proof}
	
\begin{cor}\label{decopols}
Assume that $u\in \cC^\infty _{\bullet,K}(X,\Delta_1,L)$. Then $u$ is $L^2$ with respect to 
$(g_D, \omega_D, h_L)$, and consider the corresponding Hodge decomposition
\[u= \mathcal H(u)+ \Delta_K(v).\]
The components $v, \mathcal H(u) \in \cC^\infty _{\bullet,K}(X,\Delta_1,L)$. 
\end{cor}

\begin{defn}
	Let $u$ be an $L^2$-section with values in $\cE_i$. There exists a unique form $\mathcal G(u)\in \dom(\Delta_K)$ such that 
	\[u= \mathcal H(u)+ \Delta_K\mathcal G(u)\]
	and such that $\mathcal G(u)$ is orthogonal to the space of $\Delta_K$-harmonic forms. 
\end{defn}	

\begin{remark}
	Let $\cE_i = \oplus_{r+s= i}\Lambda^{r, s}T^\star_X\langle \Delta_1\rangle\otimes L$.
	As consequence of Proposition \ref{propend1} we see that the image of the restriction of the Green operator $\mathcal G$ to the space $\cC^\infty _{\bullet,K}(X,\Delta_1,L)$ is included in $\cC^\infty _{\bullet,K}(X,\Delta_1,L)$, i.e. the map
	\[\mathcal G: \cC^\infty _{\bullet,K}(X,\Delta_1,L)\to \cC^\infty _{\bullet,K}(X,\Delta_1,L)\]
	is well-defined.
\end{remark}	


\subsection{Hodge decomposition for currents and $\ddbar$-lemma} 
In this subsection, we focus on $L$-valued smooth forms $u$ with logarithmic poles along the divisor $\Delta_1$ and fixed degree $k$. We have already seen that $u$ induces a section of the vector bundle 
\[
\bigoplus_{r+s=k} \Lambda^{r,s} T_X^\star\langle \Delta_1 \rangle \otimes L,
\]
and moreover that it is $L^2$ with respect to the metric induced on this bundle by $g_D$. However, the curvature of the bundle
\[
(T_X\langle \Delta_1 \rangle, g_D)
\]
is, in general, nonzero.

However, it would be useful to have an analogue of the $\ddbar$-lemma in this context. But since the adjoint $D_K^\star$ of $D_K$ in the logarithmic case is computed with respect to $g_D$, the commutator $[D_K^c, D_K^\star]$ might not vanish (precisely because of the curvature form). Recall that the proof of the $\ddbar$-lemma relies on the vanishing of the curvature.

In order to establish a $\ddbar$-lemma for forms with log poles along $\Delta_1$, we will consider $u$ as an element in the dual space of $C^\infty _K (X, L^\star)$. This follows the ideas in \cite{Nog95}. It is one of the main motivations for discussing next the following notion.
\begin{defn}\label{currentspace}
An $L$-valued current of degree $k$ on $X$ is a linear form 
\[T: \cC^\infty _{2n-k, K} (X, L^\star)\to \CC\]
which is continuous with respect to the topology induced by the family of semi-norms in Definition \ref{defend1}. Here we denote by 
$\cC^\infty _{2n-k, K} (X, L^\star)$ the space of $L^\star$-valued forms of degree $2n-k$, whose covariant derivative of an arbitrary order 
is $L^2$-integrable with respect to $d\mu_a$ for any $a$. The space of currents of degree $k$ will be denoted by $\mathscr D_K(X, L)$.
\end{defn} 

The continuity requirement we impose here means that there exist $N, a$ such that for any form $\phi\in \cC^\infty _{2n-k, K} (X, L^\star)$ we have
\[|T(\phi)| \leq C\big(\sum_{r=1}^N \Vert \nabla^r \phi\Vert_{L^2, a}\big).\]
\medskip

\noindent Let $u \in C^\infty _K (X, \Delta_1, L)$. We show next that $u$ defines an element in $\mathcal D_K(X, L)$.

\begin{proposition}\label{propend3}
	Let $u \in C^\infty _K (X, \Delta_1, L)$. 
	There exists a positive constant $C> 0$ and a set of natural numbers $a_i$ such that 
	\[\Big|\int_X u\wedge \phi\Big|\leq C(\Vert \phi\Vert_{L^2, a}+ \Vert \nabla \phi\Vert_{L^2, a})\]
	for any $\phi\in \cC^\infty _{K} (X, L^\star)$.  In particular, $u$ defines an element in $\mathcal D_K(X, L)$.
	
	Here denote by $\displaystyle \Vert\cdot \Vert_{L^2, a}$ the $L^2$ norm on forms with respect to $(\omega_D, h_L^\star)$ and the volume form $d\mu_a$. 
\end{proposition}

\begin{proof}
	There are probably many ways in which this result can be established, but here we see it as direct consequence of Sobolev inequality 
	for $(X, \omega_D)$, cf. Lemma \ref{sobolev}
 combined with the Hölder inequality. Indeed, the form $u$ belongs to $L^p$ for any $p< 2$, from which we easily conclude via \eqref{end16}.
\end{proof}

\medskip

\noindent This proposition can be generalised for forms with log-poles defined on an intersection of components of $\Delta_1$.
\begin{proposition}\label{propend4}
	We suppose that $\Delta_1 =\sum_{i=1}^m D_i$ and $\Xi := \cap_{i\in I} D_i$ for some $I\subset \{1, \cdots, m\}$. Let $\eta \in C^\infty _K (\Xi, (\Delta_1 \setminus \sum_{i\in I} D_i) |_\Xi , L)$. We moreover assume that $\eta\in L^{2-\ep}(\Xi, L, \omega_D|_\Xi)$ for any positive $\ep> 0$.
	Then $\eta$ induces a current on $X$.
\end{proposition}

\begin{proof}
	In order to illustrate the argument, assume that $\eta$ has log-poles and is defined locally near $\Xi= (w_1= w_2= 0)$, and let $u$ be a test form. The linear form induced by $\eta$ is the following
	\[ T(u):= \int_\Xi\eta\wedge u|_\Xi.\] 
	
	\noindent By the Sobolev inequality applied on $\Xi$, we infer that the following
	\begin{equation}\label{end39} 
		|T(u)|^2\leq C\int_\Xi (|u|^2+ |\nabla u|^2)\log^2|s|^2{\omega_D^{n-2}}	
	\end{equation}
	holds, where $C> 0$ is a positive real constant and $\log^2|s|^2$ is an ad-hoc notation for the product $\prod \log^2|s_i|_D|^2$.
	
	Since the metric $\omega_D$ has Poincar\'e singularities, there exists a constant $C> 0$ such that  
	\begin{equation}\label{end40} 
		(|u|^2+ |\nabla u|^2)\log^2|s|^2\leq C(|u\log|s||^2+ |\nabla (u\log|s|)|^2)	
	\end{equation}
	is verified. Therefore, we obtain
	\begin{equation}\label{end37} 
		|T(u)|^2\leq C\int_\Xi(|u_s|^2+ |\nabla (u_s)|^2){\omega_D^{n-2}}	
	\end{equation}
	where $u_s:= u\log|s|$, as consequence of \eqref{end39}. The RHS of this inequality coincides with
	\begin{equation}\label{end41} 
		\int_X(|u_s|^2+ |\nabla (u_s)|^2)dd^c\frac{1}{|w_1|^2+ |w_2|^2}\wedge dd^c (|w_1|^2+ |w_2|^2)\wedge \omega_D^{n-2}
	\end{equation}
	and integration by parts shows that \eqref{end41} is bounded by the integral
	\begin{equation}\label{end38} 
		\int_X \frac{1}{|w_1|^2+ |w_2|^2}\big(\sum_{k= 1}^4 |\nabla ^k(u_s)|^2\big) dV_{\omega_D}	
	\end{equation}
	up to a constant.
	
	The function $\frac{1}{|w_1|^2+ |w_2|^2}$ belongs to the space $L^{2-\ep}$ for every $\ep > 0$, so Sobolev's inequality shows that \eqref{end38} is bounded by
	\begin{equation}\label{end42} 
		\int_X\big(\sum_{k= 0}^6 |\nabla ^k(u_s)|^2\big)d\mu_1.	
	\end{equation}
	again, up to a constant. Recalling the definition of $u_s$, this is smaller than
	\begin{equation}\label{end43} 
		\sum_{k= 0}^6 \int_X|\nabla ^k(u)|^2d\mu_a,	
	\end{equation}
	for some $a$, and the proposition is proved.
\end{proof}

\begin{remark}
	If instead of $L^{2-\ep}$ we have the hypothesis $\displaystyle \frac{1}{\log^a(1/|s|)}\eta\in L^2$, the proof just finished still works and we derive the same conclusion as in Proposition \ref{propend4}.
\end{remark}

\subsubsection{Two applications}

\noindent We derive two important consequences: a De Rham-Kodaira type decomposition for forms with log-poles along $\Delta_1$, and more importantly, the
$D_KD^c_K$ lemma, crucial in the proof of the next result. These two results will be discussed next.
\smallskip

Recall that by Proposition \ref{finitdim}, $\ker(\Delta_K)$ is of finite dimension. Let $(\xi_i)_{i\in I}$ be a basis of $\ker(\Delta_K)$, which moreover is orthonormal with respect to the Hermitian structure on $\cE_k$. 

\noindent Given a current $T\in \mathscr D_K(X, L)$ we introduce the following objects.

\begin{defn} The harmonic projection of $T$ is equal to
	\[\mathcal H(T)= \sum c^i\xi_i,
	\]	
	where the coefficients $c^i\in \CC$ are defined by the equality $\displaystyle c^i:= T(\xi_i)$.
\end{defn}
\medskip

\noindent Recall that given an element $u\in \cC^\infty _{K}(X, L^\star)$,  thanks to \Cref{smoothhodge}, we have the decomposition 
\[u= \mathcal H(u)+ \Delta_K\mathcal G(u)\]
and $\mathcal G(u)\in \cC^\infty _{K}(X,L^\star)$.
Now given current $T\in \mathscr D_K(X, L)$, we define $\mathcal G(T)$ to be the current defined by duality as follows
\begin{equation}\label{greendef}
	\mathcal G(T)(u):= T\big(\mathcal G(u)\big).
	\end{equation}
Notice that this is well-defined, since $\mathcal G(u)\in \cC^\infty _{\bullet,K}(X,L^\star)$.
\smallskip

\noindent We discuss the following very useful statement.
\begin{thm}[De Rham-Kodaira decomposition for currents]\label{currentdecomp}
Given a current $T\in \mathscr D_K(X, L)$, we have the following Hodge decomposition
\begin{equation}\label{end17} 
	T= \mathcal H(T)+ \Delta_K\mathcal G(T)
\end{equation}
and $\mathcal G(T) (\xi)=0$ for every $\xi \in \ker(\Delta_K)$. 
\end{thm}

\begin{proof}
	By Definition \eqref{greendef}, we know that $\mathcal G(T) (\xi)=0$ for every $\xi \in \ker(\Delta_K)$.
	 
	We consider the current $T_1:= T- \mathcal H(T)$.  By construction, $T_1 (\xi)= 0$ for every $\xi \in \ker(\Delta_K)$. So we have
	\[T_1(u)= T_1\big(\Delta_K\mathcal G(u)\big)  \qquad\text{ for every } u\in \cC^\infty _{K}(X, L^\star).\]
	By  $[\mathcal  G, \Delta_K]=0$ and \eqref{greendef}, we have 
	$$T_1\big(\Delta_K\mathcal G(u)\big)= T_1\big(\mathcal G \Delta_K (u)\big) = \mathcal{G} (T_1) (\Delta_K u) = \Delta_K (\mathcal G (T_1)) (u) .$$
	By combining the above two equations, we obtain 
	$$T_1 (u)=  \Delta_K (\mathcal G (T_1)) (u) \qquad\text{ for every } u\in \cC^\infty _{K}(X, L^\star) .$$
	Then $T_1 = \Delta_K (\mathcal G (T_1)) $. This implies \eqref{end17}.
	\end{proof}

Now we would like to prove a $D_KD^c_K$-type lemma for $u \in C^\infty_K(X, \Delta_1, L)$. Recall that in \Cref{ddbar1}, we proved a $D_KD^c_K$-type lemma for $C^\infty_K(X, L)$. The key point in the proof is that we have $\Delta_K = \Delta^c_K$ when $\cE_k = \bigoplus_{r+s=k} \Lambda^{r,s} T^\star_X \otimes L$. If instead 
$$\cE_k = \bigoplus_{r+s=k} \Lambda^{r,s} T^\star_X \langle \Delta_1 \rangle \otimes L ,$$
then we no longer have $\Delta_K = \Delta^c_K$ since the metric $g_D$ is not flat in general.

In order to establish a $D_KD^c_K$-type lemma for the sections in $\cC^\infty_K(X, \Delta_1, L)$, we follow the idea in \cite{Nog95}, making use of the above De Rham--Kodaira decomposition \Cref{currentdecomp}. We then obtain the following result.

\begin{lemme}[A $D_KD^c_K$-type lemma]\label{lemend4} Let $u\in \mathscr D_K(X, L)$ such that $D_K u= 0$ as currents on $X$.
	We assume moreover that $u= D_K^c v$ on $X$ for some current $v$. Then $u$ is in the image of $D_KD_K^c$ on the total space $X$. 
\end{lemme}
\begin{proof} 
	By hypothesis $u= D_K^c v$ for some current $v$. Now	
	consider the decomposition 
	\[v= \mathcal H(v)+ D_KD^{\star}_K\mathcal G(v)+ D^{\star}_KD_K\mathcal G(v)\]	
	of $v$. We deduce the equality
	\[u= D_K^c\mathcal H(v)+ D_K^cD_K w + D_K^cD^{\star}_K\tau\]
	where 
	\[w:= D^{\star}_K\mathcal G(v), \qquad \tau:= D_K\mathcal G(v).\]
	By the equality \eqref{end18} we deduce that $D_K^c\mathcal H(v)= 0$. Moreover,
	we have 
	\[[D_K^c, D^{\star}_K]= 0\]
	and therefore the relation
	\begin{equation}\label{end19} 
		u= D_K^cD_K w - D^{\star}_KD_K^c\tau
	\end{equation}
	holds. Finally, the last term in \eqref{end19} is equal to zero, since 
	$D_K u= 0$, so we obtain $u= D_K^cD_K w$, which is the conclusion we wanted to reach.
\end{proof}

\medskip

\noindent In this respect, we have the following observation.

\begin{remark} Let $\Theta$ be an $L$-valued current, such that we have
	\[D_KD_K^c \Theta= \Lambda\]
	in the sense of currents, where $\Lambda$ is an $L$-valued differential form,
	such that the following holds
	\[\Lambda \hbox{ and } D_K^{\star}D_K^{c \star}\Lambda \in L^p(X, \omega_D)\]for any $p< 2$ (for example, a smooth form with log-poles). 
	
	Given the equation satisfied by $\Theta$, we can assume that 
	$\Theta\in \IM(D_K^{\star}D_K^{c \star})$. Indeed, we first write the $D_K$-Hodge 
	decomposition for $\Theta$ and obtain that 
	\[D_KD_K^c \Theta= D_KD_K^c D_K^\star \Theta_1\]
	and then the $D_K^c$-Hodge 
	decomposition for $\Theta_1$. 
	
	If that is the case, we have 
	\[\Delta_K\circ\Delta^c_K\Theta= D_K^{\star}D_K^{c \star}\Lambda\]
	and by our assumption, the RHS of this equality belongs to $L^p$. Then so does $\Theta$. Indeed, we have 
	\[\Delta^c_K\Theta= \mathcal G(D_K^{\star}D_K^{c \star}\Lambda)\]
	and then we argue as follows. In order to simplify the writing, we denote by 
	$\xi:= D_K^{\star}D_K^{c \star}\Lambda$ and write 
	\[\int_X\langle \mathcal G(\xi), u \rangle dV_{\omega_D}= \int_X\langle \xi, \mathcal G(u) \rangle dV_{\omega_D}\]
	for any form $u\in  \cC^\infty _{\bullet,K}(X,L)$, by the definition of the Green operator. 
	
	H\"older's inequality -combined with Sobolev's inequality- implies that the square of the absolute value of the LHS of the previous equality is bounded by  
	\[\int_X(|\mathcal G(u)|^2+ |\nabla \mathcal G(u)|^2)dV_{\omega_D}.\]
	Now Theorem \ref{logpole} implies that we have 
	\[\int_X(|\mathcal G(u)|^2+ |\nabla \mathcal G(u)|^2)dV_{\omega_D}\leq C\int_X |u|^2d\mu_a,\]
	taking into account the fact that $\Delta_K\mathcal G(u)= u- \mathcal H(u)$, together with \[ \int_X |\mathcal H(u)|^2d\mu_a\leq C(a)\int_X |\mathcal H(u)|^2dV_{\omega_D}\] again by Theorem \ref{logpole}. 
	\medskip
	
	\noindent Thus, all in all, we have
	\begin{equation}\label{end40} 
		\big|\int_X\langle \mathcal G(\xi), u \rangle dV_{\omega_D}\big|^2\leq C\int_X |u|^2d\mu_a
	\end{equation}
	Since this is true for all $u\in  \cC^\infty _{\bullet,K}(X,L)$, it implies that $\displaystyle \frac{1}{\log^a(1/|s|)}\mathcal G(\xi)\in L^2$. 
	The same argument shows that \[ \frac{1}{\log^a(1/|s|)}\Theta\in L^2.\] 
\end{remark}

\medskip


\section{Resolution of Hypercohomology and Harmonic Representatives} \label{hypercoho}

\noindent We consider a pair $(L, \nabla) \in M_{\rm DR}(X/D)$. The aim of this section is to construct a natural resolution of the complex 
\begin{equation}\label{derhamcom}
	\nabla:\ \cdots \to \Omega^\bullet_X(\log D) \otimes L \to \Omega^{\bullet+1}_X(\log D) \otimes L \to \cdots 
\end{equation}	
induced by the flat connection $\nabla$. 

Let $h_L$ be a harmonic (possibly singular) metric on $L$ satisfying \Cref{property}. Recall that we have a decomposition
\[
D = \Delta_1 + \Delta_2 + \Delta_3
\]
as in \Cref{property}, such that
\[
\theta := D'_{h_L} - \nabla \in H^{0}\big(X, \Omega_X(\log D)\big)
\]
has nonzero residue along $\Delta_2$ and vanishing residue along $\Delta_1 + \Delta_3$. 
Let $\cC^\infty_{\bullet,K}(X,\Delta_1,L)$ be the space defined in \Cref{defend1}.
\medskip

\noindent One of the main results of this section is the following.
\begin{thm}\label{quasi}
	Let $(X, D, L, \nabla)$ be as above. 
	Then the following two complexes
\[
\begin{tikzcd}[row sep=large, column sep=small, ampersand replacement=\&]
	\Omega^\bullet_X(\log D)\otimes L 
	\arrow[r, "\nabla"] 
	\arrow[d, "i"'] 
	\& \Omega^{\bullet+1}_X(\log D)\otimes L 
	\arrow[r, "\nabla"] 
	\arrow[d, "i"'] 
	\& \cdots 
	\arrow[r] 
	\& \Omega^n_X(\log D)\otimes L 
	\arrow[d, "i"'] 
	\arrow[r] 
	\& 0 
	\arrow[d, "i"'] \\
	\cC^\infty_{\bullet,K}(X, \Delta_1, L) 
	\arrow[r, "D_K"] 
	\& \cC^\infty_{\bullet+1,K}(X, \Delta_1, L) 
	\arrow[r, "D_K"] 
	\& \cdots 
	\arrow[r] 
	\& \cC^\infty_{n,K}(X, \Delta_1, L) 
	\arrow[r, "D_K"] 
	\& \cC^\infty_{n+1,K}(X, \Delta_1, L) 
	\arrow[r] 
	\& \cdots
\end{tikzcd}
\]
	are quasi-isomorphic via the natural inclusion $i$. 
\end{thm}		

\begin{remark}\label{realpol} 
	Before giving the proof, we make a few comments about this statement.
	\begin{enumerate}
		
		\item[(1)]  The inclusion $i$ is well defined: let 
		\[
		\frac{dz_I}{z_I} \wedge \frac{dz_J}{z_J} \otimes e_L
		\]
		be a logarithmic form, where $z_I = 0$ defines some components of $\Delta_1$, and $z_J = 0$ defines some components of $\Delta_2 + \Delta_3$. By the definition of $\cC^\infty_{\bullet,K}(X,L)$, we know that 
		\[
		\frac{dz_J}{z_J} \otimes e_L \in \cC^\infty_{\bullet,K}(X,L).
		\]
		Thus, the wedge product
		\[
		\frac{dz_I}{z_I} \wedge \frac{dz_J}{z_J} \otimes e_L \in \cC^\infty_{\bullet,K}(X,\Delta_1,L).
		\]
		In summary, a form with poles along $\Delta_2 + \Delta_3$ is in fact “smooth’’ in our setting. The only logarithmic poles we need to consider are those along $\Delta_1$.
		
		\item[(2)] As a consequence of Theorem~\ref{quasi}, we obtain an explicit description of the hypercohomology groups, namely
		\[
		\mathbb H^{\bullet}\big(X, \Omega^{\bullet}_X(\log D) \otimes L\big) \simeq 
		\frac{\ker D_K}{\im D_K},
		\]
		where 
		\[
		D_K: \cC^\infty_{\bullet,K}(X,\Delta_1,L) \longrightarrow \cC^\infty_{\bullet+1,K}(X,\Delta_1,L)
		\]
		is considered on $X \setminus \Delta_1$ (ignoring the divisorial part). Together with \Cref{decopols}, each element of 
		\[
		\mathbb H^{\bullet}\big(X, \Omega^{\bullet}_X(\log D) \otimes L\big)
		\]
		admits a unique harmonic representative.
	\end{enumerate}
\end{remark}

\medskip

Before proving Theorem~\ref{quasi}, we need some preparations. Throughout this section, we denote
\[
\cE_{p,q} := \Lambda^{p,q}T^\star_X\langle \Delta_1\rangle \otimes L ,\qquad 
\cE_k := \bigoplus_{p+q=k}\Lambda^{p,q}T^\star_X\langle \Delta_1\rangle \otimes L .
\]
The log pair $(X, D)$ is endowed with the metric $\omega_D$ defined in~\eqref{m1}.  
Let $g_D$ be the metric on $T^\star_X\langle \Delta_1\rangle$ defined in~\eqref{log2}.  
Then $\cE_{p,q}$ is equipped with a smooth Hermitian metric: $g_D$ on $\Lambda^p T^\star_X\langle \Delta_1\rangle$ and $\omega_D$ on $(\Lambda^q T_X^{0,1})^\star$.  
Let $L^2(X, \cE_{p,q})$ denote the space of $L^2$-sections with respect to the above data, and let $C^\infty_K(X, \cE_k)$ be the space defined in \Cref{defend1}.  

For technical reasons, we also consider the following subsheaf, which, loosely speaking, interpolates between $L^2(X, \cE_k)$ and $C^\infty_K(X, \cE_k)$.

\begin{defn}\label{ssheaf}
	Let $\Delta_2 = \sum_{i=1}^m D_i$ be the sum of the components on which the residue of $\theta$ is nonzero.  
	Assume that each hypersurface $D_i$ is locally defined by $z_i=0$. 
	
	Let $U \subset X$ be an open set.  
	We denote by $\mathcal{L}^2(U, \cE_k)$ the subsheaf of $L^2(U, \cE_k)$ whose sections can be written locally as
	\[
	f = \sum_{|I|+|J|=s} f_{I,J}\, dz_I \wedge d\overline{z}_J \otimes e_L,
	\]
	where the coefficients $f_{I,J}$ satisfy the condition
	\[
	\int_V \big| f_{I,J}\, dz_I \wedge d\overline{z}_J \big|^2_{h_L, g_D, \omega_D}
	\prod_{i \in \{1,\dots,m\}\setminus I} \log^2|z_i|\, dV_D < +\infty
	\]
	for every relatively compact open set $V \Subset U$, where $dV_D := \omega_D^n$.
\end{defn}
\medskip

There are several elementary properties of the space $\mathcal{L}^2(X, \cE_k)$.

\begin{proposition}\label{comefromdefn}
	If $f \in \mathcal{L}^2$, then $\theta \wedge f \in \mathcal{L}^2$.
\end{proposition}

\begin{proof}
	This is a direct consequence of the definition of the sheaf $\mathcal{L}^2$.
\end{proof}

\smallskip

\noindent The next result is a characterization of the sheaf of holomorphic forms in terms of $\mathcal{L}^2$.

\begin{proposition}\label{lelongnumb}
	Let $U \subset X$ be an open set and let $f \in H^0(U, L \otimes \Omega^p_X(\log D))$.  
	Then $f \in \mathcal{L}^2$.  
	
	Conversely, let $f$ be an $L$-valued $(p,0)$-form on $X \setminus D$ such that $\dbar f = 0$ on $U \setminus D$.  
	If $f \in \mathcal{L}^2$, then $f \in H^0(U, L \otimes \Omega^p_X(\log D))$.
\end{proposition}

\begin{proof}
	Write $D = \sum D_i$, where each $D_i$ is locally defined by $z_i=0$.  	
	Consider the logarithmic form
	\[
	  \frac{dz_J}{z_J} \wedge \frac{dz_I}{z_I},
	\]
	where $J$ corresponds to the components in $\Delta_1$ and $I$ corresponds to the components in $\Delta_2 + \Delta_3$.  
	
	Since the Lelong number of $h_L$ is strictly negative along $\Delta_2 + \Delta_3$, and $g_D$ is a smooth metric on $T_X\langle \Delta_1 \rangle$, we have
	\[
	\Bigg|\frac{dz_I}{z_I} \wedge \frac{dz_J}{z_J}\Bigg|^2_{g_D, \omega_D, h_L} \, \omega_D^n 
	\leq \frac{dV_\omega}{\prod |z_i|^{2-\varepsilon}},
	\]
	where $dV_\omega$ is a smooth volume form and $\varepsilon > 0$ is a small positive constant.  
	Thus $f \in \mathcal{L}^2$.  
	\medskip
	
	Conversely, let $f \in \mathcal{L}^2(U)$ such that $\dbar f = 0$ on $U \setminus D$.  
	Then $f$ extends to a meromorphic form on $U$ with poles along $D$.  
	Now, since the Lelong number of $h_L$ is greater than $-1$, $f$ can only have logarithmic poles of type $\Omega^p_X(\log D)$.  
	In fact, for a meromorphic form of type $\frac{dz_I}{z_i}$ with $i \notin I$, a direct calculation shows that $\frac{dz_I}{z_i}$ is not $L^2$ near $D_i$.  
\end{proof}

\noindent 

The following $L^2$-estimate is crucial in the proof of Theorem~\ref{quasi}.  
We first fix some notation.  
Assume that $\Delta_2 = \sum_{i=1}^m D_i$, where each $D_i$ is locally defined by $z_i = 0$.  
Let $J \subset \{1,\dots, m\}$ be a subset, and denote by $d\mu_J$ the measure
\begin{equation}\label{fuj1}
	d\mu_J := e^{-\varphi_L} \Bigg( \prod_{i \in J} \log^2|z_i|^2 \Bigg) dV_{\omega_D},
\end{equation}
where $\varphi_L$ is the local weight of $h_L$, and $\omega_D$ is the model Poincaré/conic metric.  
Using a result of A.~Fujiki~\cite{Fuj}, we obtain the following statement.

\begin{lemme}\label{fujiki}  
	Let $v \in L^2(U, \cE_{p,q}, g_D, \omega_D, d\mu_J)$ with $\dbar v = 0$, where $J \subset \{1,\dots, m\}$.  
	Then there exists a $(p,q-1)$-form $u$ such that
	\begin{equation}\label{L2fujiki}
		\dbar u = v, \qquad 
		\int_U \Big( |u|^2_{g_D, \omega_D} + |D'_{h_L}u|^2_{g_D, \omega_D} \Big) d\mu_J < \infty.
	\end{equation}
	Here $D'_{h_L}$ is the $(1,0)$-part of the Chern connection on a $L$-valued form with respect to $h_L$.
\end{lemme}

\begin{proof} 
	First assume that $\omega_D$ has only Poincaré singularities (so that $\Delta_3 = 0$).  
	As a consequence of \cite[Prop.~2.1]{Fuj}, there exists a solution $u_0$ of the equation $\dbar u_0 = v$ such that 
	$u_0 \in L^2(U, \cE_{p-1,q}, g_D, \omega_D, d\mu_J)$. 
	\smallskip
	
	\noindent Let $u \in (\ker \dbar)^\perp$ be the projection of $u_0$ onto the orthogonal complement of $\ker \dbar$, 
	where the scalar product is defined by the measure $d\mu_J$ together with $(g_D,\omega_D)$ on $\cE_{p-1,q}$.  
	We now show that this particular solution of the $\dbar$-equation satisfies \eqref{L2fujiki}.
	
	\medskip 
	
	Note first that $\dbar^\star u = 0$, where the adjoint is taken with respect to the weight $\varphi_L + \Xi$ on $L$, with 
	\[
	\Xi := -\log\!\Bigg(\prod_{i\in J} \log^2|z_i|^2\Bigg).
	\]
	Let $\Psi_\varepsilon$ be the standard cutoff function on $U \setminus \Delta_2$, which is equal to zero near the 
	$\varepsilon$-neighborhood of $\Delta_2$.  
	Since the norm of $\nabla \Psi_\varepsilon$ with respect to $\omega_D$ is uniformly bounded, we have
	\begin{equation}\label{dbarpart}
		\int_U |\dbar(\Psi_\varepsilon u)|^2 d\mu_J 
		+ \int_U |\dbar^\star(\Psi_\varepsilon u)|^2 d\mu_J 
		\leq C \Bigg( \int_U |\dbar u|^2 d\mu_J + \int_U |u|^2 d\mu_J \Bigg),
	\end{equation}
	for some constant $C$ independent of $\varepsilon$.
	
	\medskip
	We next control the $L^2$-norm of $D'_{h_L} u$.  
	Compared with the situation in \Cref{bochner2}, there is an additional curvature term 
	\[
	[dd^c \Xi, \Lambda_{\omega_D}],
	\]
	which is pointwise bounded, as shown by the following calculation of the Hessian of $\Xi$.  
	Indeed, we have
	\begin{equation}\label{fuj2}
		\dbar \Xi = -2 \sum_{k\in J} \frac{d\overline{z}_k}{\overline{z}_k \log |z_k|^2},	
	\end{equation}
	which implies
	\begin{equation}\label{derest}
		\sup_{U\setminus |D|} |\partial \Xi|_{\omega_D} < \infty.
	\end{equation}
	Moreover, $dd^c \Xi$ contains terms of the form
	\begin{equation}\label{fuj3}
		\sum_{k\in J} \frac{\sqrt{-1}\,dz_k \wedge d\overline{z}_k}{|z_k|^2 \log^2 |z_k|^2}, 
	\end{equation}
	each of which is uniformly bounded in the $L^\infty$ norm relative to $\omega_D$.  
	In particular, the eigenvalues of $dd^c \Xi$ with respect to $\omega_D$ are uniformly bounded, and hence the curvature operator above is bounded.
	
	\medskip
	Applying point (b) of \Cref{bochner2}, together with the above uniform estimate, we obtain
	\begin{equation}\label{barpart}
		\int_U \big|D'_{h_L}(\Psi_\varepsilon u) - \partial \Xi \wedge (\Psi_\varepsilon u)\big|^2 d\mu_J   
		\leq 
		C \Bigg( \int_U |\dbar(\Psi_\varepsilon u)|^2 d\mu_J 
		+ \int_U |\dbar^\star(\Psi_\varepsilon u)|^2 d\mu_J \Bigg).
	\end{equation}
	Combining with \eqref{dbarpart}, we obtain
	\[
	\int_U \big|D'_{h_L}(\Psi_\varepsilon u) - \partial \Xi \wedge (\Psi_\varepsilon u)\big|^2 d\mu_J  \leq C,
	\]
	for some constant $C$ independent of $\varepsilon$.  
	Then, together with \eqref{derest}, letting $\varepsilon \to 0$ shows that 
	$D'_{h_L} u \in L^2(U, \cE_{p,q}, g_D, \omega_D, d\mu_J)$.  
	This proves \eqref{L2fujiki}.
	
	\medskip
	The general case reduces to this one: via the local ramified cover $\pi$ defined in \eqref{d4},
	the metric $\pi^\star \omega_D$ has only Poincaré singularities.  
	We construct a solution $\widetilde{u}_0$ of the equation $\dbar \widetilde{u}_0 = \pi^\star v$, 
	and by the usual averaging procedure we may assume that
	\begin{equation}\label{d15}
		\widetilde{u}_0 = \pi^\star u_0
	\end{equation}
	for a form $u_0$ defined on $U$, such that $\dbar u_0 = v$, together with the required integrability properties.
\end{proof}

\medskip

\noindent We next derive the following consequence.

\begin{proposition}\label{debareq}
	Let $U$ be a Stein open set and let $f \in \mathcal{L}^{2}(U, \cE_k)$.
	If $\dbar f = 0$ on $U \setminus D$, then there exists a $g \in \mathcal{L}^{2}(U, \cE_{k-1})$ such that 
	$f = \dbar g$ on $U \setminus D$ and $D_K g \in \mathcal{L}^{2}(U, \cE_k)$.
\end{proposition}

\begin{proof}
	Write $f = \sum f_{I,J} \, dz_I \wedge d\overline{z}_J$. By definition,
	\[
	\int_{U} 
	\big| f_{I,J} \, dz_I \wedge d\overline{z}_J \big|^2_{h_L, g_D, \omega_D} 
	\prod_{i \in \{1,\dots, m\} \setminus I} \log^2 |z_i| \, dV_D 
	< +\infty .
	\]
	
	For each fixed index $I$, since $\dbar f = 0$ on $U \setminus D$, we know that 
	\[
	\sum_J f_{I,J} \, dz_I \wedge d\overline{z}_J
	\]
	is $\dbar$-closed on $U \setminus D$.  
	Applying Lemma~\ref{fujiki} to this form, with respect to the weight 
	$\prod_{i \in \{1,\dots,m\} \setminus I} \log^2 |z_i|$, we obtain 
	$g_I \in \mathcal{L}^2(U, \cE_{k-1})$ such that 
	\[
	\sum_J f_{I,J} \, dz_I \wedge d\overline{z}_J = \dbar g_I \qquad \text{on } U,
	\]
	and moreover $D'_h g_I \in \mathcal{L}^2(U, \cE_k)$.  
	Set $g := \sum_I g_I$. Then $f = \dbar g$ and $g \in \mathcal{L}^2(U, \cE_{k-1})$.
	
	Finally, we check that $D_K g \in \mathcal{L}^2(U, \cE_k)$.  
	Since $\dbar g = f$, we already have $\dbar g \in \mathcal{L}^2(U, \cE_k)$.  
	By Proposition~\ref{comefromdefn}, $\theta \wedge g \in \mathcal{L}^2(U, \cE_k)$.  
	Together with $D'_h g \in \mathcal{L}^2(U, \cE_k)$, this shows that indeed 
	$D_K g \in \mathcal{L}^2(U, \cE_k)$.
\end{proof}

\noindent The next regularity lemma is a consequence of Theorem \ref{smoothhodge}.
\begin{lemme}\label{locreg}
	Let $U \subset X$ be some open subset and let $f\in \cC^\infty _{\bullet,K}(U, \Delta_1, L)$. Assume that there exists $g_0\in L^2 (U, \cE_{\bullet-1})$ such that the equality $f= D_K g_0$ holds on $X\setminus \Delta_1$. Then there exists a  $g\in \cC^\infty _{\bullet-1,K}(U, \Delta_1, L)$ such that
	$$f= D_K g \qquad\text{on } X\setminus \Delta_1 .$$ 
\end{lemme}

\begin{proof}
	Let $g$ be the projection of $g_0$ on $\ker(D_K|_U)^{\perp}$, the orthogonal of the kernel of the operator $D_K$ restricted to $U$. Then we automatically have 
	\[D_K^{\star}g=0\]
	and as consequence, a quick calculation shows that for any cutoff function $\psi$ 
	with support in $U$ we have 
	\[\psi g\in \dom \Delta_K, \qquad \Delta_K(\psi g)- \Delta(\psi) g\in \cC^\infty _{\bullet,K}(X,\Delta_1,L)(U).\]
	The conclusion follows by the usual boot-strapping procedure, using Theorem \ref{smoothhodge}.  
\end{proof}
By combining \Cref{debareq} and \Cref{locreg}, we obtain the following consequence.

\begin{cor}\label{localsolution}
	Let $U \subset X$ be a small Stein open set and let $f \in C^\infty_{s,K}(U, \Delta_1, L)$ such that $D_K f = 0$ and $f \in \mathcal{L}^2(U)$.  
	
	\begin{enumerate}
		\item If $s > n$, there exists $g \in C^\infty_{s-1,K}(U, \Delta_1, L)$ such that $g \in \mathcal{L}^2(U)$ and $D_K g = f$ on $U$.
		\item If $s \leq n$, there exists $g \in C^\infty_{s-1,K}(U, \Delta_1, L)$ such that $g \in \mathcal{L}^2(U)$ and
		\[
		f - D_K g \in H^0(U, \Omega^s_X(\log D) \otimes L), \qquad \nabla(f - D_K g) = 0.
		\]
	\end{enumerate}
\end{cor}

\begin{proof}
	Write $f = \sum_{p+q = s} f^{p,q}$, where $f^{p,q}$ is an $L$-valued $(p,q)$-form.  
	Then
	\[
	f^{0,s} \in \mathcal{L}^2(U), \qquad \dbar f^{0,s} = 0 \text{ on } U \setminus D.
	\]  
	Applying Proposition~\ref{debareq}, there exists a $(0,s-1)$-form $g^{0,s-1} \in \mathcal{L}^2(U)$ such that
	\[
	\dbar g^{0,s-1} = f^{0,s} \quad \text{on } U \setminus D, \qquad D_K g^{0,s-1} \in \mathcal{L}^2(U).
	\]
	Then $f - D_K g^{0,s-1}$ is $D_K$-closed and belongs to $\mathcal{L}^2(U)$.  
	
	\medskip
	We repeat the above argument for $f - D_K g^{0,s-1}$. By construction, its $(0,s)$-part is zero.  
	Then $D_K$-closeness implies that the $(1,s-1)$-part is $\dbar$-closed.  
	Applying Proposition~\ref{debareq} again, we find a $(1,s-2)$-form $g^{1,s-2} \in \mathcal{L}^2(U)$ such that
	\[
	f - D_K g^{0,s-1} - D_K g^{1,s-2} \in \mathcal{L}^2(U)
	\]
	has no $(0,s)$ or $(1,s-1)$ component.  
	Iterating this process $\min\{s,n\}$ times, we reach one of the following two situations:
	
	\begin{enumerate}
		\item If $s > n$, we obtain $g \in \mathcal{L}^2(U)$ such that $f = D_K g$ on $U$.  
		By \Cref{locreg}, we can choose $g \in C^\infty_{s-1,K}(U, \Delta_1, L)$, proving the corollary.
		\item If $s \leq n$, we obtain $g \in \mathcal{L}^2(U)$ such that $f - D_K g$ is a $D_K$-closed $(s,0)$-form.  
		In particular, $f - D_K g \in H^0(U, \Omega^s_X(\log D) \otimes L)$ and $\nabla(f - D_K g) = 0$. 
		Then $D_K g \in C^\infty_{s,K}(U, \Delta_1, L)$.
		By applying \Cref{locreg} to $D_K g$, there exists $g_0 \in C^\infty_{s-1,K}(U, \Delta_1, L)$ such that $D_K g_0 = D_K g$.  
		Then $f - D_K g_0 \in H^0(U, \Omega^s_X(\log D) \otimes L)$ and $\nabla(f - D_K g_0) = 0$, which proves the corollary.
	\end{enumerate}
\end{proof}

\begin{remark}
The previous two statements, i.e. Proposition \ref{debareq} and Lemma \ref{locreg} together with their consequence we have just established show clearly the relevance of the space of forms $\mathcal{L}^2$. Indeed, it would be more difficult (but not impossible) to establish for example Proposition \ref{debareq} directly in the context of 
the spaces $\cC^\infty_{\bullet,K}$.\end{remark}
\medskip

\noindent Now we are ready to prove the main theorem of this section.

\begin{thm}\label{quasiiso}
	Consider the natural inclusion morphism $i$ between the following complexes:	
\[
\begin{tikzcd}[row sep=large, column sep=small, ampersand replacement=\&]
	\Omega^\bullet_X(\log D)\otimes L 
	\arrow[r, "\nabla"] 
	\arrow[d, "i"'] 
	\& \Omega^{\bullet+1}_X(\log D)\otimes L 
	\arrow[r, "\nabla"] 
	\arrow[d, "i"'] 
	\& \cdots 
	\arrow[r] 
	\& \Omega^n_X(\log D)\otimes L 
	\arrow[d, "i"'] 
	\arrow[r] 
	\& 0 
	\arrow[d, "i"'] \\
	\cC^\infty_{\bullet,K}(X, \Delta_1, L) 
	\arrow[r, "D_K"] 
	\& \cC^\infty_{\bullet+1,K}(X, \Delta_1, L) 
	\arrow[r, "D_K"] 
	\& \cdots 
	\arrow[r] 
	\& \cC^\infty_{n,K}(X, \Delta_1, L) 
	\arrow[r, "D_K"] 
	\& \cC^\infty_{n+1,K}(X, \Delta_1, L) 
	\arrow[r] 
	\& \cdots
\end{tikzcd}
\]

	Then $i$ induces an isomorphism of sheaves:
	\begin{equation}\label{inducedmor1}
		[i]: \frac{\ker \nabla}{\im \nabla} \longrightarrow \frac{\ker D_K}{\im D_K}.
	\end{equation}
	In other words, the two complexes above are quasi-isomorphic.
\end{thm}

\begin{proof} 
	As a consequence of Proposition~\ref{lelongnumb}, a holomorphic log form $f$ belongs to $\cC^\infty_{\bullet, K}(X, \Delta_1, L)$.  
	Therefore, the inclusion $i$ is well-defined.
	\medskip
	
	We first prove the surjectivity of \eqref{inducedmor1}. Let $U \subset X$ be a small Stein open set and let $f \in C^\infty_{k, K}(U, \Delta_1, L)$ such that $D_K f = 0$.  
	We follow the argument used in the proof of the usual Dolbeault lemma (see, e.g., \cite[Cor.~1.3.9]{Huy}).  
	Let $\{U_i\}_{i=1}^\infty$ be an increasing sequence of Stein open subsets of $U$ with $\bigcup_{i=1}^\infty U_i = U$.
	
	\medskip
	Consider first the case $k \geq n+1$. For each $m \in \mathbb{N}$, since $f\in \mathcal{L}^2 (U_m)$, by \Cref{localsolution} we can find $g_m \in \cC^\infty_{k-1,K}(U_m, \Delta_1, L)$ such that
	\[
	f = D_K g_m \qquad \text{on } U_m.
	\] 
	We claim that we can choose a sequence $\widetilde{\beta}_m \in \cC^\infty_{k-1,K}(U_m, \Delta_1, L)$ such that
	\[
	f = D_K \widetilde{\beta}_m \qquad \text{on } U_m, \qquad
	\widetilde{\beta}_m = \widetilde{\beta}_{m+1} \quad \text{on } U_{m-1}.
	\]
	
	\medskip
	To prove the claim, assume by induction that $\widetilde{\beta}_1, \dots, \widetilde{\beta}_{m-1}$ have already been constructed.  
	Consider $\widetilde{\beta}_{m-1} - \beta_m$. By construction,
	\[
	D_K(\widetilde{\beta}_{m-1} - \beta_m) = 0 \quad \text{on } U_{m-1}.
	\]  
	Since $k \geq n+1$, applying \Cref{localsolution} to $\widetilde{\beta}_{m-1} - \beta_m$ on $U_{m-1}$, we can find $\gamma \in \cC^\infty_{k-1,K}(U, \Delta_1, L)$ such that
	\[
	\widetilde{\beta}_{m-1} - \beta_m = D_K \gamma \quad \text{on } U_{m-2}.
	\]  
	Set
	\[
	\widetilde{\beta}_m := \beta_m + D_K \gamma.
	\]  
	Then $\widetilde{\beta}_m = \widetilde{\beta}_{m-1}$ on $U_{m-2}$ and $D_K \widetilde{\beta}_m = f$ on $U_m$. The claim is proved.
	
	\medskip
	Thanks to the claim, the sequence $\widetilde{\beta}_m$ converges to some $\beta \in \cC^\infty_{k-1,K}(U, \Delta_1, L)$ such that
	\[
	D_K \beta = f \quad \text{on } U.
	\]  
	This proves surjectivity for $k > n$.
	
	\medskip
	If $k \leq n$, the only difference is that the $D_K$-closeness of $\widetilde{\beta}_{m-1} - \beta_m$ on $U_{m-1}$ implies
	\[
	\widetilde{\beta}_{m-1} - \beta_m = D_K \gamma + F_m \quad \text{on } U_{m-2},
	\]  
	for some $F_m \in H^0(U_{m-2}, \Omega^{s-1}_X(\log \sum D_i) \otimes L)$ with $\nabla F_m = 0$.  
	Since $\ker \nabla$ is coherent, there exists $F \in H^0(U, \Omega^{s-1}_X(\log \sum D_i) \otimes L)$ with $\nabla F = 0$ such that $\|F - F_m\| \leq 2^{-m}$ on $U_{m-2}$.  
	Then set
	\[
	\widetilde{\beta}_m := \beta_m + D_K \gamma + (F_m - F).
	\]  
	Then $D_K \widetilde{\beta}_m = 0$ on $U_m$, and $\widetilde{\beta}_m$ converges on compact subsets of $U$.
	
	\medskip
	For injectivity, let $F \in H^0(U, L \otimes \Omega^k_X(\log D))$ satisfy $F = D_K g$ for some $g \in C^\infty_{k-1,K}(U, \Delta_1, L)$.  
	As $F$ is a $(k,0)$-form, all $(p,q)$-components of $D_K g$ vanish for $q \geq 1$.  
	Using the previous argument, we can find $\gamma \in C^\infty_{k-2,K}(U, \Delta_1, L)$ such that $g - D_K \gamma$ is a $(k-1,0)$-form.  
	Then
	\[
	F = D_K(g - D_K \gamma),
	\]  
	so $g - D_K \gamma$ is holomorphic on $U \setminus D$.  
	By the second part of Proposition~\ref{lelongnumb}, we conclude $g - D_K \gamma \in H^0(U, L \otimes \Omega^{k-1}_X(\log D))$.  
	Hence $F$ is $\nabla$-exact.
\end{proof}

\noindent For further reference, we record the following consequence.

\begin{thm}\label{L2cohoisom} 
	The following spaces are isomorphic: 
	\[
	\ker \Delta_K \simeq 
	\frac{\ker (\dbar + \nabla): \cC^\infty_{\bullet,K}(X, \Delta_1, L) \to \cC^\infty_{\bullet+1,K}(X, \Delta_1, L)}
	{\im (\dbar + \nabla): \cC^\infty_{\bullet-1,K}(X, \Delta_1, L) \to \cC^\infty_{\bullet,K}(X, \Delta_1, L)}
	\simeq \mathbb{H}^k(X, L \otimes \Omega_X^\bullet(\log \sum D_i)).
	\]
\end{thm}

\begin{proof}
	This follows directly from \Cref{decopols} and \Cref{quasiiso}.
\end{proof}

Finally, we point out a variant of \Cref{quasiiso}. Nothing that the complex \eqref{derhamcom} does not change if we replace $\nabla$ by
$D_K = \nabla + \dbar$. Therefore, we have the complex
\begin{equation}\label{derhamcom11}
	\cdots \to \Omega^\bullet_X(\log D) \otimes L \stackrel{D_K}{\to} \Omega^{\bullet+1}_X(\log D) \otimes L \to \cdots.
\end{equation}	

Now, for each $p \geq 0$, we have the following complex of sheaves:
\begin{equation}\label{derhamcom2}
	\dbar : C^\infty_K(X, \Lambda^{p,q} \Omega_X(\log \Delta_1) \otimes L) \to C^\infty_K(X, \Lambda^{p,q+1} \Omega_X(\log \Delta_1) \otimes L),
\end{equation}	
which is locally exact. This follows from a version of Fujiki's lemma in the context of $C^\infty_K$, by considering solutions with minimal norm and applying the Bochner formula. The consequence is the same: the hypercohomology of \eqref{derhamcom11} is computed by the cohomology of the double complex \eqref{derhamcom2}.

\medskip

 \section{An extension theorem}\label{sec:solve}
\subsection{Main results}\label{subsec:main}
\noindent In this section we prove the extension Theorem \ref{solveequs} below, which plays an important role in the "jumping of the cohomology locus" statement. The usual notation/conventions here are the usual ones fixed in Section ???




		
%
%
\smallskip

We introduce a few more notations: consider the decomposition
	\[
	\Delta_1 := \sum_{i=1}^{\ell} Y_i
	\] 
of $\Delta_1$ as sum of prime divisors. For any subset 
	\(I=\{i_1,\ldots,i_p\}\subset \{1,\ldots,\ell\}\), we denote 
	\[
	Y_I := Y_{i_1}\cap \cdots\cap Y_{i_p},
	\] 
	which is a smooth closed submanifold of \(X\) of codimension \(p\). We set 
	\[
	\Delta_{I;1} := \sum_{i\notin I} (Y_i \cap Y_I), \quad  
	\Delta_{I;2} := \Delta_2 \cap Y_I, \quad  
	\Delta_{I;3} := \Delta_3 \cap Y_I.
	\]
	Then 
	\[
	D_I := \Delta_{I;1} + \Delta_{I;2} + \Delta_{I;3}
	\] 
	is a simple normal crossing divisor on \(Y_I\). We define 
	\[
	C^\infty_{i,K}(Y_I, \Delta_{I;1},L)
	\] 
	the space of smooth forms in \Cref{defend1} for the log pair \((Y_I,D_I)\).  
	
	Note that, by \Cref{specialhar}, for any 
	\(\beta \in H^0\big(X, \Omega_X^1(\log D)\big)\), we have 
	\(\Res_{Y_i}\beta = 0\) for each \(i\). Hence, 
	\[
	\iota_I^*\beta \in H^0\big(Y_I, \Omega_{Y_I}^1(\log D_I)\big),
	\] 
	where \(\iota_I:Y_I\hookrightarrow X\) is the inclusion map. Therefore, 
	\(\iota_I^*(L,\nabla)\) defines a well-defined object in 
	\(M_{\mathrm{\rm DR}}(Y_I/D_I)\). Moreover, the metric \(\iota_I^*h\) is a harmonic metric for \(\iota_I^*L\), verifying the conditions of \Cref{specialhar}.  
	
	In what follows, we write \((L,h)\) for \((\iota_I^*L,\iota_I^*h)\) on \(Y_I\) by abuse of notation.
\medskip

\noindent Our next result states as follows.

\begin{thm}\label{solveequs} 
	In the above setting,  let	\[\beta\in H^0 \big(X, \Omega_X^1 (\log D)\big)\] be a log form and let $\xi $ be a smooth $(0,1)$-form on $X$ such that $d\xi=0$. 
	Let $u_0 \in  C^\infty _{k,K}(X,\Delta_1,L)$ such that the equality
	\[D_K u_0 =0 \qquad\text{ on } X\setminus D .\] 
	We suppose that the following hold.
	\begin{itemize}
		
		\item  The form $u_0$ extends to order one, i.e. there exists a log form $v_1\in C^\infty _{k,K}(X,\Delta_1,L)$ such that 
		\[D_K v_1 = -(\xi +\beta ) \wedge u_0 \qquad\text{ on } X\setminus D\]
		
		\item For every non-empty $I\subset \{1,\ldots,\ell\}$,  any $D_K$-closed form in  $C^\infty_{k-|I|,K}(Y_I, L)$  extends to order one, i.e, for any  $s\in C^\infty_{k-|I|,K}(Y_I,L)$   if $D_K s =0$, then $-\iota_I^*(\xi +\beta ) \wedge s$ is $D_K$-exact on $Y_I$.  Here $|I|$ denotes the cardinality of $I$.
	\end{itemize}
	Then for any $k\geq 0$ we can solve the equations
	\begin{equation}\label{simp4611}
		\begin{cases}
			D_K(u_1) =- (\xi +\beta ) \wedge u_0\\
			D_K (u_2)=- (\xi +\beta ) \wedge u_1\\
			\ldots\\
			D_K (u_{k+1})=- (\xi +\beta ) \wedge  u_k \\
		\end{cases} 
	\end{equation} 
	on $X\setminus D$ such that $u_i \in  C^\infty _{\bullet,K}(X,\Delta_1,L)$ and such that  
the series $ \sum_{i\geq 0} t^i u_i$ is convergent for $|t| \ll 1$.
\end{thm}
\medskip

\begin{remark}\label{realsmooth}
	By the first point of Proposition \ref{specialhar}, we know that $\beta$ may have a logarithmic pole along $\Delta_2$ and is smooth along $\Delta_1 + \Delta_3$. Together with the definition of $\cC^\infty_{\bullet,K}(X,L)$ (resp.\ $\cC^\infty_{\bullet,K}(X,\Delta_1,L)$), it follows that if 
	\[
	u \in \cC^\infty_{\bullet,K}(X,L) \quad \text{(resp.\ } u \in \cC^\infty_{\bullet,K}(X,\Delta_1,L)\text{)},
	\] 
	then the product
	\[
	(\xi + \beta) \wedge u
	\]
	belongs to $\cC^\infty_{\bullet,K}(X,L)$ (resp.\ $\cC^\infty_{\bullet,K}(X,\Delta_1,L)$).
\end{remark}

\begin{remark}\label{equivthm}
	For each $t\in \CC$ such that $|t|\ll 1$ we consider the twisted operator 
	\begin{equation}\label{par1}
		D_{t, K}:= D_K+ t(\xi +\beta )\wedge \cdot 
	\end{equation}
	and notice that $D_{t, K}^2= 0$ for each $t$. We will be interested in the dimension of the kernel of $D_{t, K}$ acting on the space of $k$-forms $u$ with log poles along $\Delta_1$ as the parameter $t$ varies. 
	\medskip
    
	\noindent In this setting, Theorem \ref{solveequs} can be reformulated as follows.
	
	\begin{thm} Assume that an $L$-valued $D_{0, K}$-closed logarithmic form $u_0$ on $X$ extends to order one in the direction $\xi +\beta$, and that the same is true for every $L$-valued $D_{0,K}$-closed smooth form on every stratum of $\Delta_1$. Then 
    $u_0$ extends (as element of the kernel of $D_{t, K}$). 
	\end{thm}	
\end{remark}
\medskip

\noindent The proof of Theorem \ref{solveequs} will take us quite some time. The main steps are as follows.  We show first that the system \eqref{simp4611} can be solved in the space of currents, and then we show that we can "convert" (so to speak) currents to forms with log-poles. After that, the convergence issue is addressed by using the elliptic estimates in Section \ref{HodgeDek}.  
\medskip

\textbf{Notations and conventions}
To begin with, we fix some notations. Let $X$ be a compact K\"ahler manifold endowed with the metric $\omega_D$ as above. We introduce
\[\alpha:= -\xi -\beta\]
in order to simplify the notations, and let 
\[X_0:= X\setminus |\Delta_2|\]
be the complement of the divisor on which the singularities of $\omega_D$ are of Poincaré-type. In the absence of the divisor $\Delta_1$, the operator $D_K$ acts as follows
\begin{equation}
	D_K: \cC^\infty _{\bullet,K}(X,L)\to \cC^\infty _{\bullet+1,K}(X,L)
\end{equation} 
where $r$ represents the total degree of the forms under considerations. The extension of this operator to forms with log poles 
will be denoted by the same symbol
\begin{equation}
	D_K: \cC^\infty _{\bullet,K}(X,\Delta_1,L)\to \cC^\infty _{\bullet+1,K}(X,\Delta_1,L)
\end{equation} 
by choosing to ignore the residues along the components of $\Delta_1$. 

Next, each element $u$ of $\cC^\infty _{\bullet,K}(X,\Delta_1,L)(X_0)$ defines a twisted, smooth form with log poles along $\Delta_1$. By \Cref{lem:residue},  the residue of $u$ along the component $Y_i$ of $\Delta_1$ is well defined and belongs to $C^\infty _K (Y_i, \Delta_1 , L)$. 
We denote it by $\CR_i(u)$.
We define
\begin{equation}\label{end3}
	\cD_K(u):= D_K(u)+ \sum_i \CR_i(u)
\end{equation} 
The RHS of \eqref{end3} is a current on $X_0$. 
Notice that we have 
\begin{equation}\label{end4}
	D_K^2= 0, \qquad \cD_K^2= 0 \qquad\text{on } X_0. 
\end{equation} 
\bigskip

To illustrate the idea, we first present the proof of \Cref{solveequs} under the assumption that $\Delta_1 = 0$, which is essentially the same as in the compact case.
\begin{proof}[Proof of Theorem \ref{solveequs} if $\Delta_1= 0$ ]
\hspace{5pt}
	\medskip
	
	\noindent $\bullet$ Let $u_0$ be a $D_K$-closed form with values in $L$, of total degree equal to $r$. By changing the representative (notice that this does not change anything to the problem we have to solve) we can assume that $u_0$ is harmonic
	--and therefore, we equally have $D_K^cu_0= 0$.
	We now use the hypothesis, and infer the existence of a $r$-form $v_1 \in  C^\infty _{\bullet,K}(X,L)$ so that the equality 
	\begin{equation}\label{par3}
		D_Kv_1- \alpha \wedge u_0= 0  . 
	\end{equation}
	\smallskip
	
	\noindent $\bullet$
	We show next that we can find a solution $v_1 \in  C^\infty _{\bullet,K}(X,L)$ of \eqref{par3} which is actually $D_K^c$-exact.  
	Let $v_1$ be a solution of \eqref{par3}. We now consider $D_K^c v_1$. It is $D_K ^c$-exact. Moreover, $D_K^c v_1$ is $D_K$--closed, since
	\begin{equation}\label{par4}
		D_K(D_K^cv_1)= -D_K^c(D_K v_1)= D_K^c(\alpha\wedge u_0)= 0,
	\end{equation}
	where the last equality comes from the fact that $\alpha$ and $u_0$ belong to the kernel of $D_K^c$. By using the $D^c _K D_K$-lemma  \Cref{ddbar1}, we see that there exists a form $w_1 \in C^\infty _{\bullet,K}(X,L)$ that
	\begin{equation}\label{par5}
		v_1- D_K w_1 \in \ker D_K^c
	\end{equation}
	holds. 
	Therefore, we can assume that the solution $v_1$ in \eqref{par3} is $D_K^c$-closed. Finally, by  \Cref{smoothhodge}, we have the Hodge decomposition
	\begin{equation}\label{par6}
		v_1= \mathcal H(v_1)+ D_K^cw_2
	\end{equation}
	of $v_1$, where $\mathcal H(v_1)$ is the harmonic part of $v_1$ and $w_2 \in C^\infty _{\bullet,K}(X,L)$.
	 Set $u_1 := D^c _K w_2$. Then $u_1$ is $D_K^c$-exact and 
	$$D_Ku_1- \alpha \wedge u_0= 0 \qquad\text{ on } X . $$
	\medskip
	
	\noindent$\bullet$ Now we solve the remaining equations. Consider next the exterior product
	\begin{equation}\label{par7}
		\alpha\wedge u_1.
	\end{equation}
	By Remark \ref{realsmooth},  $\alpha\wedge u_1 \in  C^\infty _{\bullet+1,K}(X,L)$.
	Since $u_1$ is $D_K^c$-exact,  $\alpha\wedge u_1$ is also $D_K^c$-exact. $\alpha\wedge u_1$ is equally $D_K$-closed, since the derivative of $u_1$ is a multiple of $\alpha$. By the  $D^c _K D_K$-lemma \Cref{ddbar1}, we can therefore find a $v_2 \in  C^\infty _{\bullet,K}(X,L)$ such that
	\begin{equation}\label{par8}
		D_KD_K^c v_2= \alpha\wedge u_1 .
	\end{equation}
	In particular we found a  $u_2 \in C^\infty _K (X) $ such that
	\begin{equation}
		D_Ku_2= \alpha\wedge u_1, \qquad u_2= D_K^c v_2.
	\end{equation}
	
	\noindent By this procedure we obtain successively
	\begin{equation}
		D_Ku_m= \alpha\wedge u_{m-1}, \qquad u_m= D_K^c v_m
	\end{equation}
	for all $m\geq 1$. This means precisely that $u_0$ can be extended to an arbitrary order.  Finally, thanks to Proposition \ref{mainpropind}, we have the convergence of $\sum_i t^i u_i$ for $|t|\ll 1$.
\end{proof}

\begin{remark}\label{induct}
	In particular we see that if $u_0$ is $D_K$ \emph{and} $D_K^c$--closed and moreover it deforms to the first order, then 
	the equations \eqref{simp4611} can be solved, with the additional constraint $u_m\in \im D_K^c$ for any $m\geq 1$. 
\end{remark}

In the remaining part of this section, we will give a proof of Theorem \ref{solveequs} in the general case. The structure is as follows.
In subsection \ref{currentsol}, we will solve the equations \eqref{simp4611} in the sense of current. In subsection \ref{regular} , we will establish a regularity proposition, which tells us that we can transform the current solution to a log smooth solution. In the subsection \ref{finshproof}, we will complete the proof of Theorem \ref{solveequs}.
\subsection{Solutions in the sense of currents}\label{currentsol}
Note that  in the proof of Theorem \ref{solveequs} when $\Delta_1 =0$, the conclusion of the first bullet is a consequence of the fact that we have $\Delta''_K =\Delta' _K$ on $X$.
In the case when $\Delta_1 \neq 0$, the following lemma plays a similar role.
\medskip

\begin{lemme}\label{initiald}
	In the setting of Theorem \ref{solveequs},  we suppose that $\Delta_1 =\sum_{i=1}^m Y_i$. Then there exists a current $u_{\Delta_1}$ with support in $\Delta_1$ of order $0$ and a current $v_0$ on $X$ such that $\wh u_0$ defined by
	\[\wh u_0= u_0+ \cD_K v_0+ u_{\Delta_1}\]
	satsifes the following properties:
	
	\begin{itemize}
		\item We have 
		\begin{equation}\label{par10}
			\mathcal D_K \wh u_0= \sum_i \CR_{i; 0}\qquad \Supp(\CR_{i; 0})\subset Y_i
		\end{equation}
		together with the successive relations
		\begin{equation}\label{par12}
			\mathcal D_K \CR_{i_1, i_2, \cdots, i_p; 0}= \sum_{j}\CR_{j, i_1, i_2, \cdots, i_p;  0}, \qquad \Supp(\CR_{ i_1, i_2, \cdots, i_p; 0})\subset Y_{i_1, i_2, \cdots, i_p }
		\end{equation}
		
		\item For any $p \geq 1$ and any multi-index $i_1, i_2, \cdots, i_p$, we have 
		\[\cD_K^c \wh u_0= 0, \quad \cD_K^c \CR_{i_1\dots i_p, 0}= 0.\]
		Moreover, the currents $\CR_{i_1, i_2, \cdots, i_p; 0}$ satisfy the skew-symmetry relations.
		
		\item For every $p$, as a current on $X$, $\CR_{i_1, i_2, \cdots, i_p; 0}$ is supported in $\Delta_1$ of order $0$.
	\end{itemize}	
\end{lemme}

\begin{proof}
	Note first that we have the relation
	\begin{equation}\label{par133}
		\cD_K u_0= \sum_i \wt\CR_{i; 0},
	\end{equation}
	where $\wt\CR_{i; 0}$ are supported in $Y_i$, and in general for each $p\geq 1$, we have
	\begin{equation}\label{par134new}
		\cD_K \wt \CR_{i_1 \dots i_p; 0}= \sum_j \wt \CR_{j i_1 \dots i_p; 0},
	\end{equation}
	where the currents $ \wt \CR_{i_1 \dots i_p; 0}$ are supported in $ \cap_{j=1}^pY_{i_j}$.
	Notice that the skew-symmetry equality $\wt \CR_{ \dots i j \dots; 0}=- \wt \CR_{ \dots j i \dots; 0}$ is automatic.
	
	\medskip
	
  Here we discuss the construction of $\widetilde{R}_{i;0}$.   
	By \Cref{lem:residue}, we note that for any 
	\[
	u \in C^\infty _{k,K}(X,\Delta_1,L)
	\]  
	we have 
	\[
	\Res_{Y_i}(u) \in C^\infty _{k-1,K}(Y_i,\Delta_{i;1},L).
	\]   	We then define 
	\[
	\widetilde{R}_{i;0} := (\iota_i)_*\bigl(\Res_{Y_i}(u)\bigr),
	\]  
	which is a current on $X$ of degree $k+1$ with coefficients in $L$.  
	
	Moreover, by \Cref{lem:residue} again,  we have 
	\[
	\Res_{Y_i \cap Y_j}\Res_{Y_i}(u) \in 
	C^\infty _{k-2,K}(Y_{i,j},\Delta_{i,j;1},L).
	\]   
	Since $ D_K$ has no poles along $\Delta_1$, its restriction to each $Y_i$ is well defined.  
	In this case, using $D_K^2=0$, we obtain
	\[
	\cD_K \widetilde{R}_{i;0} 
	= \sum_{j;j\neq i} (\iota_i)_*(\tilde{\iota}_{ji})_*\bigl(\Res_{Y_i \cap Y_j}\Res_{Y_i}(u)\bigr)=  ({\iota}_{ji})_*\bigl(\Res_{Y_i \cap Y_j}\Res_{Y_i}(u)\bigr)
	\]  
	where each $\tilde{\iota}_{ji}: Y_i \cap Y_j \hookrightarrow Y_i$ and $\iota_{ji}:Y_i\cap Y_j\hookrightarrow X$  are inclusions. 
	
	We then set
	\[
	\widetilde{R}_{ji;0} := (\iota_{ji})_*\bigl(\Res_{Y_i \cap Y_j}\Res_{Y_i}(u)\bigr),
	\]  
	which is also a current of degree $k+2$ on $X$ with coefficients in $L$.   By construction, 	$\widetilde{R}_{ji;0} $ is supported in $Y_i\cap Y_j$, and is skew-symmetric; i.e. we have
	$$
	\widetilde{R}_{ji;0}=-\widetilde{R}_{ij;0}. 
	$$
	By iterating this procedure, for any $I=\{i_1,\ldots,i_p\}\subset \{1,\ldots,m\}$,  we obtain $\widetilde{R}_{i_1 \dots i_p;0}$, which is 
	\begin{itemize}
		\item  a current of degree $k+p$ with coefficients in $L$;
		\item  supported in $Y_{i_1,\ldots,i_p}:=Y_{i_1}\cap\cdots\cap Y_{i_p}$ for distinct $i_1,\ldots,i_p$. 
		\item skew-symmetric.
	\end{itemize}

	By using \eqref{par134new}, we have 
	\[\cD_K\wt \CR_{1\dots m; 0}= 0,\]
	and so we can write
	\begin{equation}\label{par136new}
		\wt \CR_{1\dots m; 0}= \cD_K \Xi_{1 \dots m}+ \CR_{1 \dots m; 0}, \qquad \cD_K^c \CR_{1 \dots m; 0}= 0.
	\end{equation}
	For any permutation $(i_1,\dots, i_m)$ of the set $\{1, \dots, m\}$ we define  
	$\CR_{i_1, \dots, i_m;0}$ and $\Xi_{i_1, \dots, i_m}$ by skew-symmetry.
	
	\noindent Then we have 
	$$\cD_K  \CR_{i_1,\dots, i_m; 0} =0 \qquad\text{and } \qquad\CR_{i_1,\dots, i_m; 0}  \in \ker \cD_K^c$$
	with the additional condition that $\CR_{i_1,\dots, i_m; 0} $ and $\Xi_{i_1,\dots, i_m}$ are skew-symmetric.
	\medskip
	
	Now we move the next "floor": considering the currents supported on the complete intersection of $(m-1)$-components $Y_{i_1, \dots, i_{m-1}}$. By combining \eqref{par136new} and \eqref{par134new} we obtain
	\begin{equation}\label{par37new}
		\cD_K \big(\wt \CR_{i_1 \dots i_{m-1}; 0}- \sum_j \Xi_{j i_1, \dots i_{m-1}}\big) = \sum_j \CR_{j i_1 \dots i_{m-1}; 0}, 
	\end{equation}
	for which we notice that the RHS  is $\cD_K^c$-closed. Then we can write 
	\begin{equation}\label{par381}
		\wt \CR_{i_1\dots i_{m-1}; 0} = \sum_j \Xi_{j i_1 \dots i_{m-1} } + \cD_K \Xi_{i_1 \dots i_{m-1}}+ \CR_{i_1 \dots  i_{m-1}; 0} .
	\end{equation}
	where $\CR_{i_1 \dots  i_{m-1}; 0} \in \ker \cD_K^c $. 
	Note that $\wt \CR_{i_1 \dots i_{m-1}; 0}$ and $\Xi_{j i_1 \dots i_{m-1} }$ are skew-symmetric. Then we can assume that this is also the case for $\Xi_{i_1, \dots,  i_{m-1}}$ and $\CR_{i_1, \dots,  i_{m-1}; 0}$. Indeed, this can be seen by first defining these objects for
	\[i_1< \dots < i_{m-1}\]
	and then imposing the skew-symmetry condition (precisely as we did in the proof of the first point).
	\smallskip
	
	\noindent We obtain
	\begin{equation}\label{end30}
		\cD_K  \CR_{i_1 \dots i_{m-1}; 0}= \sum_j  \CR_{j i_1 \dots i_{m-1}; 0}, \qquad \CR_{i_1, \dots, i_{m-1}; 0} \in \ker \cD_K^c
		\end{equation}
	as it follows from \eqref{par37new} and \eqref{par381}.
	\medskip

	\noindent Consider next $\displaystyle \sum_j \CR_{j i_1 \dots i_{m-2};0}$ as current on $Y_{i_1, \dots, i_{m-2}}$. By construction, we know that 
	\[\sum_j \CR_{j, i_1, \dots, i_{m-2};0}\in \ker \cD_K^c\] on $Y_{i_1 \dots i_{m-2}}$. 
	
	Moroever, by \eqref{par381}, we have
	\begin{align}\nonumber
		\sum_j \CR_{j i_1 \dots i_{m-2}; 0} = & \sum_j \wt \CR_{j i_1 \dots i_{m-2}; 0} - \sum_{k,j}\Xi_{k j i_1 \dots i_{m-2} } - 
		\cD_K \big (\sum_j \Xi_{j i_1 \dots i_{m-2} }\big) \cr
		= & \sum_j \wt \CR_{j i_1 \dots i_{m-2}; 0} - \cD_K \big(\sum_j \Xi_{j i_1 \dots i_{m-2} }\big).\cr
	\end{align}
	The equality
	\begin{equation}\label{end31}
		\cD_K \big(\wt \CR_{i_1 \dots i_{m-2}; 0}- \sum_j \Xi_{j i_1 \dots i_{m-2} }\big) = \sum_j \CR_{ji_1 \dots i_{m-2}; 0} 
	\end{equation}
	follows.
	As the RHS of \eqref{end31} is $\cD_K^c$-closed, we can write 
	\begin{equation}\label{par3811}
		\wt \CR_{i_1 \dots i_{m-2}; 0} = \sum_j \Xi_{ji_1 \dots i_{m-2} } + \cD_K \Xi_{i_1 \dots  i_{m-2}}+ \CR_{i_1 \dots  i_{m-2}; 0} 
	\end{equation}
	such that $\CR_{i_1 \dots  i_{m-2}; 0}  \in \cD_K^c$. As usual, we can impose that 
	$\Xi_{i_1 \dots  i_{m-2}}$ and $\CR_{i_1, \dots,  i_{m-2}; 0}$ are 
	skew-symmetric. \medskip
	
	By iterating this procedure (i.e. using induction), we can thus construct $\Xi_{i_1 \dots  i_p}$ such that 
	\begin{equation}\label{induct11}
		\CR_{i_1 \dots i_p; 0} := \wt \CR_{i_1 \dots i_p; 0} - \sum_j \Xi_{j i_1 \dots i_p } + \cD_K \Xi_{i_1 \dots  i_p}
	\end{equation}
	satisfies 
	\begin{equation}\label{induct121}
		\cD_K  \CR_{i_1 \dots i_p; 0}= \sum_j  \CR_{j i_1 \dots i_p; 0}, \qquad \CR_{i_1 \dots i_p; 0} \in \ker \cD_K^c
	\end{equation}
	such that moreover $\Xi_{i_1 \dots  i_p}$ and $\CR_{i_1 \dots i_p; 0}$ are skew-symmetric.
	\medskip
	
	\noindent Finally, we consider the current
	\begin{equation}\label{par471}
		u_0- \sum_{i=1} \Xi_{i}
	\end{equation}
	and evaluate its differential, which equals
	\begin{equation}\label{par481}
		\cD_K\big(u_0- \sum_i \Xi_{i}\big)=\sum_i \wt\CR_{i;0} -  \sum_i  \cD_K\Xi_{i}= \sum_i \CR_{i; 0},
	\end{equation} 
	where the last equality is a consequence of \eqref{induct11} together with the skew-symmetry of $\Xi_{i,j}$.
	The RHS of \eqref{par481} is $\cD_K^c$-closed, so we can write 
	\begin{equation}\label{par491}
		u_0- \sum_i \Xi_{i}= \cD_K(v_0)+ \wh u_0, \qquad \cD_K^c \wh u_0= 0.	
	\end{equation}
	\medskip
	
	\noindent In conclusion, from the relations above we get 
	\begin{equation}\label{par501}
		\cD_K \wh u_0 = \sum_{i=1} \CR_{i; 0},
	\end{equation} 
	and 
	\begin{equation}\label{par511}
		\cD_K \CR_{i_1\dots i_k; 0} = \sum_j \CR_{j i_1\dots i_k; 0}
	\end{equation}
	where $\wh u_0$ and $ \CR_{i_1 \dots i_k; 0} $ are $\cD_K^c$-closed and $\CR_{i_1 \dots i_k; 0}$ are skew-symmetric.  This ends the proof of the second point.
	
	\medskip
	
	For the last point, it comes from the fact that every current $\CR_{i_1 \dots i_k; 0}$ is constructed as a current on $Y_{i_1 \dots i_k}$. Therefore, if we consider $\CR_{i_1 \dots i_k; 0}$ as a current on $X$, it is supported in $\Delta_1$ of order $0$.
\end{proof}

Now we would like to solve \eqref{simp4611} in the sense of currents. More precisely, we have 

\begin{proposition}\label{solucurrent}
	We are in the setting of Theorem \ref{solveequs}. Let $\wh u_0$ and $\CR_{i_1\dots i_p; 0}$ be the currents constructed in Lemma \ref{initiald}. Then for every $p$ and $k\geq 1$, there exists currents $u_k$ on $X$ and currents $\CR_{i_1\dots i_p; k}$ supported on $\cap_{s=1}^p Y_{i_s} $ such that 
	
	\begin{itemize}
		\item For every $p$ and $k\geq 1$, we have 
		$$\cD_K \CR_{i_1, \dots i_{p};  k}= \alpha\wedge \CR_{i_1, \dots i_{p};  k-1}+ \sum_j \CR_{j, i_1, \dots i_{p};  k},$$
		with $\CR_{i_1, \dots i_{p};  k} \in \im \cD_K^c$ on $\cap_{s=1}^p Y_{i_s}$ and the currents $\CR_{i_1, \dots i_{p};  k}$ are skew-symmetric.
		
		\item We have 
		\begin{equation}
			\mathcal D_K u_1= \alpha\wedge \wh u_0+ \sum_j \CR_{j, 1}
		\end{equation}
		and	for every $k\geq 2$, we have
		\begin{equation}
			\mathcal D_K u_k= \alpha\wedge u_{k-1}+ \sum_j \CR_{j, k}.
		\end{equation}
		Moreover, for every $k\geq 1$, we have $u_k \in \im \cD_K^c$.
		
		\item  For every $k, j$, $\CR_{j; k}$ is supported in $\Delta_1$ of order $0$.
	\end{itemize}
\end{proposition}

\begin{proof}
	
	\bigskip
	
	Recall that by Lemma \ref{initiald}, we have  
	\begin{equation}\label{par110}
		\mathcal D_K \wh u_0= \sum_i \CR_{i; 0}\qquad \Supp(\CR_{i; 0})\subset Y_i
	\end{equation}
	together with the successive relations: for every $p\geq 1$
	\begin{equation}\label{par112}
		\mathcal D_K \CR_{i_1 \cdots i_p; 0}= \sum_{j}\CR_{j, i_1, \cdots, i_p ;0}, \qquad \Supp(\CR_{i_1 \cdots i_p; 0})\subset \cap_{s=1}^p Y_{i_s}
	\end{equation}
	and the relation $ \CR_{\cdots ij \cdots; 0}=  -\CR_{\cdots ji \cdots; 0}$. 
	
	\noindent Note that our second assumption concerning the first order extension translates as follows: for every $p\geq 1$, and every \emph{harmonic form} $u_{i_1 \cdots i_p}$ on $\cap_{s=1}^p Y_{i_s}$ the following holds 
	\begin{equation}\label{addjumpingas}
		\alpha \wedge u_{i_1, \cdots, i_p} \in \im \cD_K.
	\end{equation}
	\medskip

	We start with the "bottom floor" $Y_{1\dots m}:= \cap_{i=1}^m  Y_i$. 
	Since \eqref{par112} implies that $\CR_{1\dots m; 0} \in \ker \cD_K$,
	by applying \eqref{addjumpingas}, we infer that
	$\alpha \wedge \CR_{1\dots m; 0}$ is $\cD_K$-exact on $Y_{1\dots m}$.
	The case $\Delta_1= \emptyset$ being already settled, we obtain solutions of the following equations
	\begin{equation}\label{par116}
		\cD_K\CR_{1\dots m; k}= \alpha \wedge \CR_{1\dots m; k-1}, \qquad \CR_{ 1\dots m;  k}\in \im \cD_K^c
	\end{equation}
	for any $k\geq 1$.
	
	Since the initial family of residues $(\CR_{i_1\dots i_m; 0})$ satisfies the skew-symmetry property, we can assume that the same is true for $(\CR_{i_1\dots i_m; k})$, for any $k\geq 1$. This is immediate: we first construct 
	$\CR_{1\dots m; k}$, for any $k\geq 1$, then impose the relation
	\[\CR_{\dots ij \dots; k}=  -\CR_{\dots ji \dots; k}\] 
	and remark that the corresponding equalities \eqref{par116} are satisfied for 
	any permutation $(i_1\dots i_m)$ of $\{1,\dots, m\}$. 
	\medskip
	
	Now we move to the next floor and consider next the equations 
	
	\begin{equation}\label{par117}
		\cD_K \CR_{ i_1 \dots i_{m-1}; 1}= \alpha\wedge \CR_{i_1\dots i_{m-1}; 0}+ \sum_j \CR_{j i_1 \dots i_{m-1}; 1}.
	\end{equation}
	We intend to show that the RHS is $\cD_K\cD_K^c$-exact. 
	
	First of all, we claim that the RHS of \eqref{par117} belongs to $\im \cD_K +\im \cD_K^c$.  
	In fact we have $\CR_{j i_1 \dots i_{m-1}; 1} \in \im \cD_K^c$ by  \eqref{par116}. 
	To treat the first term in RHS of \eqref{par117}, by Lemma \ref{initiald}, we have $\CR_{i_1 \dots i_{m-1}; 0}  \in \ker \cD_K^c$. Therefore we can write 
	$$\CR_{i_1 \dots i_{m-1}; 0}= a+b ,$$ 
	where $a \in \im \cD_K^c$ and $b$ is harmonic. As $\alpha \wedge b$ is $\cD_K$-exact by \eqref{addjumpingas}, we know that 
	$$\alpha\wedge \CR_{i_1, \dots i_{m-1}; 0} \in \im \cD_K +\im \cD_K ^c .$$
	As a consequence, the RHS of \eqref{par117} belongs to $\im \cD_K +\im \cD_K^c$.  
	\smallskip
	
	Secondly, we claim that the RHS of \eqref{par117} is $\cD_K$ and $\cD_K^c$-closed. This is verified by a direct calculation, cf.
	\begin{align}\nonumber \cD_K (\alpha\wedge \CR_{i_1 \dots i_{m-1}; 0}+ \sum_j \CR_{j i_1 \dots i_{m-1}; 1})
		& = - \alpha \wedge \sum_j  \CR_{j i_1 \dots i_{m-1}; 0} + \alpha \wedge \sum_j  \CR_{j i_1 \dots i_{m-1}; 0}
		 = 0\end{align}
	as well as 
	\begin{align}\nonumber
		\cD_K^c \big(\alpha\wedge \CR_{i_1 \dots i_{m-1}; 0}+ \sum_j \CR_{j i_1 \dots i_{m-1}; 1}\big)& = -\alpha\wedge \cD_K^c\CR_{i_1 \dots i_{m-1}; 0}+ \sum_j \cD_K^c\CR_{j i_1 \dots i_{m-1}; 1}= 0.\end{align}
	
As a consequence of the above two claims, by using the $\ddbar$-lemma, the RHS of \eqref{par117} is $\cD_K \cD_K^c$-exact on $Y_{ i_1, \dots i_{m-1}}$, for each multi-index $(i_1,\dots, i_{m-1})$ and it follows that we can find a solution $\CR_{ i_1 \dots i_{m-1}; 1}$ of \eqref{par117} such that
	\[\CR_{ i_1, \dots i_{m-1}; 1}\in \im \cD_K^c .\]
	\smallskip
	
	\noindent Our next claim is that we can determine $\CR_{i_1, \dots, i_{m-1};1}$
	in such a way that they are skew-symmetric. Indeed, for each ordered multi-index
	\[j_1<\dots< j_{m-1}\]
	we take an arbitrary solution $\CR_{j_1, \dots, j_{m-1};0}$ of \eqref{par117}, which belongs to the image of $\cD_K^c$. Then, for each permutation $\sigma$ of 
	$\{1, \dots, m-1\}$ we define
	\[\CR_{j_{\sigma(1)}\dots j_{\sigma(m-1)}; 1}:= \epsilon(\sigma)\CR_{j_1\dots j_{m-1}; 1}\]
	where $\epsilon(\sigma)$ is the signature of $\sigma$. We then have
	\begin{align}\nonumber
		\cD_K \CR_{j_{\sigma(1)}\dots j_{\sigma(m-1)}; 1}= & \epsilon(\sigma)\cD_K \CR_{j_1\dots j_{m-1}; 1}\cr
		= & \epsilon(\sigma)\alpha\wedge \CR_{j_1\dots j_{m-1}; 0} + \epsilon(\sigma)\sum_j \CR_{j j_1 \dots j_{m-1}; 1}\cr
		= & \alpha\wedge \CR_{j_{\sigma(1)}\dots j_{\sigma(m-1)}; 0}+ \sum_j \CR_{jj_{\sigma(1)}\dots j_{\sigma(m-1)}; 1}.\cr
	\end{align}
	In conclusion, we can solve the equation \eqref{par117} in such a way that the solution belongs to the image of $\cD_K^c$ and it is skew-symmetric with respect to the indexes $(i_1,\dots, i_{m-1})$.
	\medskip
	
	\noindent Now we solve the successive equations 
	\begin{equation}\label{moreeq}
		\cD_K \CR_{i_1, \dots i_{m-1};  k}= \alpha\wedge \CR_{i_1, \dots i_{m-1};  k-1}+ \sum_j \CR_{j, i_1, \dots i_{m-1};  k},
	\end{equation}
	for $k\geq 2$. In this case (i.e. $k\geq 2$) the RHS of \eqref{moreeq} is $\cD_K^c$-exact by induction. Moreover, it is $\cD_K$-closed by the following calculation, where we argue by induction on $k$ in \eqref{moreeq}. 
	
	\begin{align}\nonumber \cD_K ( \alpha\wedge \CR_{i_1, \dots i_{m-1};  k-1}+ \sum_j \CR_{j, i_1, \dots i_{m-1};  k}) = & 
		-\alpha \wedge  \sum_{j} \CR_{j, i_1, \dots i_{m-1};  k-1} \cr  & + \alpha\wedge  \sum_{j} \CR_{j, i_1, \dots i_{m-1};  k-1} + \sum_{s,j} \CR_{s,j, i_1, \dots i_{m-1};  k}\cr
		= & 0.
	\end{align} 
	By the $\ddbar$-lemma we infer the existence of $\CR_{i_1, \dots i_{m-1};  k}$ such that
	$$\cD_K \CR_{i_1, \dots i_{m-1};  k}= \alpha\wedge \CR_{i_1, \dots i_{m-1};  k-1}+ \sum_j \CR_{j, i_1, \dots i_{m-1};  k},$$
	with $\CR_{i_1, \dots i_{m-1};  k} \in \im \cD_K^c$. By the same argument as before, we can assume that the currents
	$\CR_{ i_1, \dots i_{m-1}; k}$ are skew-symmetric for all $k\geq 2$. 
	\medskip
	
Induction and the same arguments as above show that we can find the solutions $\CR_{i_1\dots i_p, k}$ for any $1\leq p\leq m$ and any $k$. The first point of the proposition is proved.

	\bigskip

For the second point of the proposition, we  determine first $u_1$ so that 
	\begin{equation}\label{par311}
		\mathcal D_K u_1= \alpha\wedge u_0+ \sum_j \CR_{j, 1}
	\end{equation}
	holds. As before, the RHS is $\cD_K^c \cD_K$-exact. Then we can find a solution $u_1$ of \eqref{par311} such that $u_1 \in \im \cD_K^c$. To finish we remark that we can iterate this and solve the equations
	\begin{equation}
		\mathcal D_K u_k= \alpha\wedge u_{k-1}+ \sum_j \CR_{j; k}
	\end{equation}
	for all $k\geq 2$. Since the support of the currents $\CR_{j; k}$ is contained in $\Delta_1$, the second point is completely established. 
	\bigskip
	
	For the last point, it comes from the fact that every current $\CR_{i_1 \dots i_j; k}$ is constructed as a current on $Y_{i_1 \dots i_j}$. Therefore, if we consider $\CR_{i_1 \dots i_j; k}$ as a current on $X$, it is supported in $\Delta_1$ of order $0$.
	
\end{proof}


\subsection{The regularity statement}\label{regular} 

By Proposition \ref{solucurrent},  Theorem \ref{solveequs}
is almost proved, modulo the fact that instead of being forms with log poles, the solutions $(u_k)_{k\geq 1}$ are currents. We explain next how this last issue is addressed.
In order to illustrate the method we will use, we first analyze the following particular case. 
\begin{lemme}\label{initalcase}
	Let $X$ be a Kähler manifold, and let $\Delta$ be a snc divisor on $X$. We denote by $V$ a holomorphic vector bundle on $X$, and
	 let $u_X$ be a $V$-valued smooth $(p,q)$-form with log poles along $\Delta$. We suppose that there exists a $V$-valued current $T$ on $X$ such that the following equality
	\begin{equation}\label{eqcond}
		\dbar T= u_X+ T_\Delta	
	\end{equation}
holds, where $T_\Delta$ is a $V$-current with support in $\Delta$ of order $0$, namely  
$$\int_X T_\Delta\wedge \phi =0$$ 
for any test $V^\star$- valued form $\phi$ such that $\phi|_\Delta =0$. Then there exists a smooth $V$-valued $(p, q-1)$-form with log poles $v$ such that the equality
	\begin{equation}\label{reg2}
		\dbar v= u_X
	\end{equation}
holds on $X\setminus \Delta$.
\end{lemme}

\begin{proof}
	One of the key points in the next arguments is taken from the paper by K. Liu, S. Rao and X. Wan \cite[Prop 2.4]{LRW}. 
	In order to avoid any risk of confusion, we denote by \[C^\infty_{0, q}(X, \Omega_X^p(\log \Delta) \otimes V)\] the space of $V$-valued smooth $(p,q)$-form with log poles along $\Delta$. 
	
	Let
	\[\cI: C^\infty_{0, q}(X, \Omega_X^p(\log \Delta) \otimes V  )\to C^\infty_{0, q}(X, \Omega^p_X\langle\Delta\rangle \otimes V)\]
be the map which identifies a logarithmic $(p,q)$-form with a 
$(0,q)$-form with values in the holomorphic vector bundle $\Omega^p_X\langle\Delta\rangle \otimes V$.	
Let $ \Lambda^pT_X\langle\Delta\rangle $ be the dual of $\Omega^p_X\langle\Delta\rangle $.
We equally introduce the natural map
\[\iota: C^\infty_{n, n-q}(X, \Lambda^pT_X\langle\Delta\rangle \otimes V^\star  )\to  C^\infty_{n-p, n-q}(X, V^\star)\]
obtained by contracting with the vector fields.
	
	\noindent  We suppose that $\Delta= \sum D_i$. We recall the statements Lemma 2.1 and Lemma 2.2 in \cite{LRW}: we have 
	\begin{equation}\label{reg3}
		\iota(\gamma)|_{D_i}= 0, \qquad \int_X \alpha\wedge \iota(\beta)= \int_X \cI(\alpha)\wedge \beta
	\end{equation}
	where $\gamma$ is a $\Lambda^pT_X\langle\Delta\rangle \otimes V^\star$-valued form of type $(n, n-q)$, and the second equality in \eqref{reg3} has the following meaning:
	
	\noindent $\bullet$ On the LHS, $\alpha$ is a $V$-valued smooth $(p,q)$--form with log poles which is wedged with the $V^\star$-valued $(n-p, n-q)$ form $\iota(\beta)$.
	
	\noindent $\bullet$ On the RHS, $\cI(\alpha)$ is a $(0,q)$-form with valued in $\Omega^p_X\langle\Delta\rangle \otimes V$, and $\beta$ is a $(n,n-q)$
form with values in the dual bundle (so the wedge makes perfectly sense).	
\medskip
	
	\noindent We remark that by \eqref{eqcond}, we have $\dbar \cI(u_X)= 0$. 
	 Let us fix an arbitrary, Hermitian metric $h$ on $\Omega^p_X\langle\Delta\rangle \otimes V$, as well as a Kähler metric on $X$. The Hodge decomposition (for the Laplace operator associated to $\dbar$) induced by this data gives 
	\begin{equation}\label{reg4}
		\cI(u_X)= \mathcal H+ \dbar v
	\end{equation}
	where $\mathcal H$ is $\Delta''$-harmonic and $v$ corresponds to a form with log poles $w$. 
	
	\medskip
	
	We introduce the following antilinear-operator
	$$\sharp : \cC^\infty_{(0,q), K} (X, \Omega^p_X\langle\Delta\rangle \otimes V )  \to \cC^\infty_{(n,n-q), K} (X, \Lambda^pT_X\langle\Delta\rangle \otimes V^\star )$$
	such that the equality
	\begin{equation}
		\int_X \langle s_1, s_2 \rangle_h \omega_X ^n = \int_X s_1 \wedge \sharp s_2
		\end{equation}
	holds for every $s_1, s_2 \in  C^\infty_{(0,q), K} (X, \Omega^p_X\langle\Delta\rangle \otimes V )$.

	We first show that 
	\begin{equation}\label{harmpart}
		 \dbar ( \sharp \mathcal H) =0 .
		\end{equation}
		In fact, for any $s' \in  C^\infty_{(0,q), K} (X, \Omega^p_X\langle\Delta\rangle \otimes V )$, we have
	$$ \int_X s' \wedge \dbar \sharp ( \mathcal H) = (-1)^\bullet \int_X \dbar s' \wedge \sharp \mathcal H   = (-1)^\bullet \int_X  \langle \dbar s' , \mathcal H \rangle_{h} \omega^n_X$$
	$$ = (-1)^\bullet  \int_X  \langle s' , \dbar^\star \mathcal H\rangle_{h} \omega^n_X =0 .$$ 
	Therefore we obtain \eqref{harmpart}.
	\medskip
	
	To prove the Lemma, we need show that the harmonic part $\mathcal H$ in \eqref{reg4} must vanish. 
	We have
	\begin{equation}\label{reg5}
		\int_X|\mathcal H|^2dV= \int_X\langle \cI(u_X), \mathcal H \rangle dV= \int_X \cI(u_X) \wedge \sharp \mathcal H
	\end{equation}
	which equals 
	\begin{equation}\label{reg6}
		\int_X u_X\wedge \iota( \sharp \mathcal H)
	\end{equation} 
	by \eqref{reg3}. The main point is that we have 
	\begin{equation}\label{reg7}
		\int_X u_X\wedge \iota(\sharp \mathcal H)= \int_X (u_X+ T_\Delta)\wedge \iota(\sharp \mathcal H)
	\end{equation}
	by the first part of \eqref{reg3}, and the properties of the current $T_\Delta$. By hypothesis, the RHS of \eqref{reg7} is equal to 
	\begin{equation}\label{reg8}
		\int_X \dbar T\wedge \iota(\sharp \mathcal H)= (-1)^\bullet \int_X T\wedge \dbar \big( \iota \circ \sharp \mathcal H \big) = 
		 (-1)^\bullet \int_X T\wedge \iota \big( \dbar ( \sharp \mathcal H) \big) =0
	\end{equation}
where the last equality is a consequence of \eqref{harmpart}.
\end{proof}
\medskip

\subsubsection{A few more notations} We are next aiming at the version of \Cref{initalcase} in our setting and as a first step we introduce/recall some notations.
We consider the holomorphic vector bundles
$$E_p := \Omega^p _X\langle\Delta_1\rangle \otimes L, \qquad E^\star _p := \wedge^p T_X\langle\Delta_1\rangle \otimes L^\star.$$
The metric $g_D$ on $T_X\langle\Delta_1\rangle$ together with $h_L$ induce a Hermitian structure on $E_p$ and its dual. 
Furthermore, the space $X_0$ is endowed with the metric $\omega_D$ and let
$$\cC^\infty_{(0,q), K} (X, E_p),\qquad \cC^\infty_{(n, q), K} (X, E^\star _p)$$
be the spaces obtained as in Definition \ref{defend1}.

\medskip
Let $C^\infty _{\bullet, K}(X, \Omega^\bullet _X (\log \Delta_1)\otimes L)$ be the space of $L$-valued log smooth form and let 
$\cD_K (X, L)$ be the space of $L$-currents, which is dual to $\cC^\infty _K (X, L^\star)$ cf. \Cref{currentspace}.
The identification map
\begin{equation}\label{identi}
\cI:   C^\infty _{\bullet , K}(X, \Omega^\bullet _X (\log \Delta_1)\otimes L)  \simeq \oplus_{p+q =\bullet }\cC^\infty_{(0,q), K} (X, E_p)  
	\end{equation}
is defined by applying the $"\cI"$ in the proof of Lemma \ref{initalcase} component-wise.
If the $\bullet$ on the RHS of \eqref{identi} is equal $i$, then this coincides with $\mathcal E_i$ in \ref{secondcases} of the Convention \ref{conv}.
\smallskip

\noindent We introduce a differential on the space $\oplus_{p+q =i }\cC^\infty_{(0,q), K} (X, E_p)$ via the operator $\cI$	
\begin{equation}\label{diffE}
D_{K,E}:  \oplus_{p+q =i }\cC^\infty_{(0,q), K} (X, E_p) \to \oplus_{p+q =i +1}\cC^\infty_{(0,q), K} (X, E_p) 
\end{equation}	
by imposing that the equality
\begin{equation}\label{diffE1}
D_{K,E} \circ \cI= \cI\circ D_K
\end{equation}
holds.
	
\noindent In a similar way, the operator 
\begin{equation}\label{diffE2}
D_{K, E^\star}: \oplus_{p+q =i}\cC^\infty_{(n,q), K} (X, E^\star _p) \to \oplus_{p+q =i-1}\cC^\infty_{(n,q), K} (X, E^\star _p)\end{equation}
is defined 
so that we have
\begin{equation}\label{diffE3}
\int_X D_{K,E}  (s)\wedge t = (-1)^{q+1} \int_X s\wedge D_{K, E^\star} (t) \end{equation} 
for any $s\in  \cC^\infty_{(0,q), K} (X, E_p), t \in  \cC^\infty_{(n,n-q), K} (X, E^\star _p)$.
\smallskip

\noindent It will be useful to introduce the following anti-linear-operator
$$\sharp : \cC^\infty_{(0,q), K} (X, E_p)  \to \cC^\infty_{(n,n-q), K} (X, E^\star_p)$$
so that the equality
\begin{equation}\label{diffE4}\int_X \langle s_1, s_2 \rangle_{g_D, h_L} \omega_D ^n = \int_X s_1 \wedge \sharp s_2\end{equation}
is verified for any $s_1, s_2 \in  \cC^\infty_{(0,q), K} (X, E_p)$. Note that although the two metrics $g_D$ and $\omega_D$ are 
not smooth, the following holds.

\begin{lemme} We consider a smooth form $u \in \cC^\infty_{(0,q), K} (X, E_p)$. Then we automatically have $\sharp u\in \cC^\infty_{(n,n-q),K}(X, E_p^\star)$.
\end{lemme}
\begin{proof} The matter is very easy: by using the coordinates/frames as in \eqref{end10}, we write locally
\begin{equation}\label{diffE5}
u= \sum_{|I|=p, |J|= q} u_{I\ol J}l_I^\star \wedge \ol e_J^\star \otimes e_L.
\end{equation}
Then we have the formula 
\begin{equation}\label{diffE6}
\sharp u= \sum_{|I|=p, |J|= n-q} \ol u_{K\ol L}a^{\ol K I} b^{\ol {\mathcal C}_J L}\det (a_{\beta\ol\alpha})l_I \wedge e\wedge \ol e_J^\star \otimes e_L
\end{equation}
where we use the notations in Section \ref{Hodsmo}. This quickly follows from the definition \eqref{diffE4}. The criteria of Lemma \ref{lemend3} together with the fact that the coefficients $a_{\beta\ol\alpha}, b_{\beta\ol\alpha}$ and each of their inverses are functions with polynomial growth 
finish the proof.
\end{proof}
\smallskip

\begin{remark}
The three operators we have just defined are not unrelated. One can immediately verify that the equality 
\[\sharp \circ D_{K, E}^\star = D_{K, E^\star}\circ \sharp \]
holds (possibly up to a sign...), where $D_{K, E}^\star$ is computed with respect to the metrics $\omega_D, g_D$ on $X_0$ and $\mathcal E_\bullet$. We will not discuss this equality here, because it follows from the proof of Proposition \ref{listprop} below. 
\end{remark}

\medskip

\subsubsection{Proof of the regularity result}
Consider next $u \in C^\infty _{\bullet , K}(X, \Omega^\bullet _X (\log \Delta_1)\otimes L)$, a smooth form with log poles.
Then as consequence of \Cref{propend3}, the form $u$ induces a current which will be denoted by  
\begin{equation}\label{diffE7}
T_u\in \cD_K (X, L).
\end{equation}
Recall that we have a natural contraction map $\iota$ defined on the space $C^\infty_{(n, n-q), K}(X,  E_p ^\star)$ as follows 
\begin{equation}\label{diffE8} 
\iota: \cC^\infty_{(n, n-q), K}(X,  E_p ^\star)= \cC^\infty _{(n,n-q), K} (X, \Lambda^p T_X \langle\Delta_1 \rangle \otimes L^\star)\to  \cC^\infty_{(n-p, n-q), K}(X, L^\star).
\end{equation}

\noindent 
\begin{remark}
Let $u \in C^\infty _{\bullet , K}(X, \Omega^\bullet _X (\log \Delta_1)\otimes L)$ be a smooth form with log poles. Then with respect to the notations above we can write
\[\cD_K \circ T_u = T_{D_K u} + \Res_{\Delta_1} u.\]
\end{remark}

\medskip

\noindent Let $\cD_{K}$ and $\cD_{K, E}$ be the operators on the spaces of currents $\cD_{\bullet, K} (X, L)$ and $\cD_{\bullet , K} (X, E_{\bullet})$
induced by $D_K$ and $D_{K, E}$.
We list next several properties which will be used later.

\begin{proposition}\label{listprop}
	 Let $s\in \cC^\infty_{(0,q), K} (X, E_p) $ and $t\in \cC^\infty_{(n,n-q), K} (X, E^\star_p) $.  Then we have
	 
	 \begin{itemize}
\item[\rm (1)] $ D_K  \circ \iota (t) = \iota  \circ D_{K, E^\star} (t) $.

\item[\rm (2)] If $D^\star_{K,E} s =0$,  then $D_{K, E^\star} \circ \sharp s =0$.
\end{itemize}
\end{proposition}

\begin{proof} We establish the first point by duality. Note that $\iota (t)  \in C^\infty _{\bullet ,K} (X, L^\star)$. 
	Let $w \in C^\infty_{\bullet ,K} (X, L)$. We have
	$$\int_X w \wedge D_K  \circ \iota (t) =  (-1)^\bullet \int_X D_K w \wedge  \iota (t) = (-1)^\bullet \int_X  D_{K, E} \cI (w) \wedge t =  \int_X   \cI (w) \wedge D_{K, E^\star} (t) .$$
	Then we obtain $ D_K  \circ \iota (t) = \iota  \circ D_{K, E^\star} (t) $.
\smallskip

For the second point, note that for every $s' \in  \oplus_{p+q =\bullet}\cC^\infty_{(0,q), K} (X, E_p)$, we have
	$$ \int_X s' \wedge D_{K, E^\star} \circ \sharp s = (-1)^\bullet \int_X D_{K, E}s' \wedge \sharp s  = (-1)^\bullet \int_X  \langle D_{K, E}s' , s\rangle_{g_D, h_L} \omega^n_D$$
	$$ = (-1)^\bullet  \int_X  \langle s' , D^\star_{K,E} s\rangle_{g_D, h_L} \omega^n_D =0 .$$ 
	Then $D_{K, E^\star} \circ \sharp s =0$. 
\end{proof}	
\medskip

\noindent Let $\oplus_{p+q =\bullet }\cD_{ (0,q), K}(X, E_p)$ be the dual space of $\oplus_{p+q =\bullet }\cC^\infty _{(n,n-q), K} (X, E^\star _p)$. We show next that there exists a natural morphism
\begin{equation}\label{diffE9} 
\wt \cI: \cD_{\bullet, K} (X, L) \to \oplus_{p+q =\bullet }\cD_{ (0,q), K}(X, E_p)
\end{equation}
induced by the identification map $\cI$ (which justifies the use of almost the same notation).

\begin{proposition}\label{currentcI}
	Let $\tau \in \cD_{\bullet, K} (X, L)$ be a $L$-valued current. 
	Then it induces an element $\wt \cI (\tau) \in \oplus_{p+q =\bullet }\cD_{ (0,q), K}(X, E_p)$ by the formula
	$$\int_X \wt \cI (\tau) \wedge \phi := \int_X \tau \wedge \iota (\phi),$$
	for every $\phi \in C^\infty _{(n,q) ,K} (X, E^\star_p )$.
	Moreover, we have
	$$\int_X \cD_{K, E} (\wt\cI (\tau)) \wedge \phi = (-1)^{\bullet }\int_X \tau \wedge D_{K} \circ \iota (\phi)$$
so that we can write 
	$$ \cD_{K, E} \circ \wt \cI (\tau)= \wt \cI \circ \cD_K (\tau).$$
	\end{proposition}
	
	\begin{proof}
		It is clear that $\wt \cI(\tau)$ is linear. The continuity follows from inequalities of type
		\begin{equation}\label{conj3}
			\sum_{r=0}^N\Vert \nabla^r \iota(\phi)\Vert^2_{a}\leq C(N, a)\sum_{r=0}^N\Vert \nabla^r \phi\Vert^2_{a}	
		\end{equation}
	In fact, when $r=0$, by the construction of $g_D$, there exists some constant $C>0$ such that
	$$ |e|_{g_D} \geq C |e|_{\omega_D} \qquad\text{for every } e \in T_X \langle \Delta_1 \rangle .$$
As a consequence, for any $u\in \Lambda^p T_X \langle \Delta_1 \rangle$ and $v\in K_X$, we have
	$$|u (v)|_{\omega_D} \leq C |v|_{\omega_D} \cdot |u|_{g_D} . $$
We obtain thus \eqref{conj3} for $r=0$. If $r\geq 1$, this is immediately verified by using the coordinates as in Section \ref{Hodsmo}, so we don't provide any further details.
		
The last part of our statement is a direct consequence of Proposition \ref{listprop}.
		\end{proof}
		
\begin{remark}
Note that the right-hand side of \eqref{conj3} cannot be controlled by the left-hand side of \eqref{conj3}. As a matter of fact, we will show in \Cref{vanish} that $\wt \cI$ is not injective, so the inverse of $\wt \cI$ does not exist. 
\end{remark}
\medskip

\noindent Before proving the regularity result Proposition \ref{regularprop}, we need a few auxiliary statements. The first of them states as follows.

\begin{proposition}\label{cumm}
We consider the map
$$\rho: \oplus_{p+q=\bullet }\cC^\infty _{(0,q) , K} (X, E_p) \to \oplus_{p+q=\bullet }\cD^\infty _{(0,q), K} (X, E_p),$$
as in \eqref{diffE7}, 
so that $\rho(s)$ is the current associated to a section $s\in \oplus_{p+q=\bullet }\cC^\infty _{(0,q) , K} (X, E_p)$.
Let $s\in \cC^\infty _{(0,q) , K} (X, E_p)$. If $\rho (s)$ is $\cD_{K, E}$-exact, then $s$ is $D_{K, E}$-exact.

In particular, let $u \in C^\infty _{\bullet, K}(X, \Omega^\bullet _X (\log \Delta_1)\otimes L)$. If $\wt \cI \circ T_u$ is $\cD_{K, E}$-exact, then there exists a $v \in C^\infty _{\bullet -1, K}(X, \Omega^\bullet _X (\log \Delta_1)\otimes L)$ such that 
$$u  = D_K v \qquad\text{ on }  X\setminus \Delta_1 .$$
\end{proposition}
			
			\begin{proof}
Note first that since $\rho (s)$ is $\cD_{K, E}$-exact, the equality 
\[D_{K, E} s= 0\]
follows. The Laplace operator $\Delta_{K, E}= [D_{K, E}, D_{K, E}^\star]$ is elliptic (where the adjoint is computed with respect to the usual metrics) and we have already seen that the Hodge decomposition holds.
We can thus write
				$$s= \mathcal{H} + D_{K, E} v$$
				where $\mathcal{H}$ is $\Delta_{K, E}$-harmonic and  $v\in \oplus_{p+q=\bullet }\cC^\infty _{(0,q) , K} (X, E_p)$. We show next that the harmonic part $\mathcal{H}$ is equal to zero.
				
\noindent To this end, remark that the equalities
\begin{align}				
\int_X \langle \mathcal{H}, \mathcal{H} \rangle_{g_D, h_L} \omega^n_D = & \int_X \langle s, \mathcal{H} \rangle_{g_D, h_L} \omega^n_D = \int_X s \wedge \sharp \mathcal{H} \nonumber \\
= & \int_X \rho (s) \wedge \sharp \mathcal{H} 
\nonumber
\end{align}
hold. Since $\Delta_{K, E}  \mathcal{H}=0$, by $(2)$ of \Cref{listprop} the form $\sharp \mathcal{H} $ is $D_{K, E^\star}$-closed. Together with the fact that $\rho (s)$ is $\cD_{K, E}$-exact, we obtain
				$$\int_X \rho (s) \wedge \sharp \mathcal{H}  =0 .$$
				Therefore $\mathcal{H}=0$ and $s$ is $D_{K, E}$-exact.
				
				\medskip
				
				For the second part, note that we have the equality $\rho \circ \cI (u) = \wt \cI \circ T_u$. Then $\rho \circ \cI (u)$ is $\cD_{K, E}$-exact. 
				By the first part, we know that $ \cI (u)$ is $D_{K, E}$-exact, namely
				$$ \cI (u) = D_{K, E} f$$
				for some $f \in \oplus_{p+q=\bullet }\cC^\infty _{(0,q) , K} (X, E_p)$.  Set $v:= \cI^{-1} (f)$. Then we have
				$$u  = D_K v \qquad\text{ on }  X\setminus \Delta_1,$$
which concludes the proof.				
				\end{proof}
\smallskip
			
\noindent The following simple remark will equally be useful. 
\begin{proposition}\label{vanish}
				Let $\tau \in \cD_{\bullet, K} (X, L)$ be a $L$-valued current of order zero. Assume moreover that the support of $\tau$ is contained in $\Delta_1$. Then we have 
				$$\cI (\tau)=0 \in \oplus_{p+q =\bullet }\cD_{ (0,q), K}(X, E_p) .$$
				\end{proposition}
				
				\begin{proof}
					Let $\phi \in C^\infty _{(n,q) ,K} (X, E^\star_p )$. By definition we have
					$$\int_X \cI (\tau) \wedge \phi = \int_X \tau \wedge \iota (\phi)  =0,$$
					where the last equality is a consequence of the fact that $\iota (\phi)$ vanishes on $\Delta_1$.
					 We deduce that $\cI (\tau)=0$.
					\end{proof}
\medskip

\noindent Finally, we establish the following version of Lemma \ref{initalcase} in our setting.

\begin{proposition}\label{regularprop}
	Let $u_X \in C^\infty _{K}(X, \Omega^\bullet _X (\log \Delta_1)\otimes L)$ be an $L$-valued, smooth form with log poles along $\Delta_1$. We suppose that there exists a current $T$ on $X$ with the property that 
	\begin{equation}\label{reg1}
		\cD_K T= T_{u_X}+ T_{\Delta_1}
	\end{equation}
	where $T_{\Delta_1}$ is a current with support in $\Delta_1$ of order $0$. 
	Then there exists a form $v\in C^\infty _{K}(X, \Omega^\bullet _X (\log \Delta_1)\otimes L)$ such that the equality
	\begin{equation}\label{reg2}
		 u_X =D_K v 	\qquad\text{ on } X\setminus \Delta_1.
	\end{equation}
\end{proposition}

\begin{proof}
We have $\displaystyle T_{u_X} \in \cD_{\bullet, K} (X, L)$ -cf. \eqref{diffE7}- and via the map $\wt \cI$ in \eqref{diffE9} we have
$$ \wt \cI \circ T_{u_X} \in \oplus_{p+q =\bullet }\cD_{ (0,q), K}(X, E_p).$$ 
On the other hand, by applying $\wt \cI$ in \eqref{reg1}, we see that the equality 
$$\wt \cI \circ \cD_K T= \wt \cI \circ T_{u_X} + \wt \cI (T_{\Delta_1})$$
holds.
Thanks to the second part of \Cref{currentcI} and \Cref{vanish}, we obtain
$$ \cD_{K, E} \circ \wt \cI (T) = \wt \cI \circ T_{u_X}.$$
Therefore $\wt\cI \circ T_{u_X}$ is $\cD_{K, E}$-exact in the sense of currents.  
 Thanks to the second part of \Cref{cumm}, we obtain that
	$$u_X = D_K (v )  \qquad\text{ on } X\setminus \Delta_1$$
	for some $v\in C^\infty _{K}(X, \Omega^\bullet _X (\log \Delta_1)\otimes L)$.
The proof is finished.
\end{proof}



\subsection{End of the proof of Theorem \ref{solveequs}}\label{finshproof} 
In order to make the presentation more fluid, we will next use the following terminology. Given a smooth form $u\in C^\infty _K(X, \Delta_1, L)$, we say that it is $D_K$-closed if $D_K u =0$ on $X\setminus \Delta_1$. 
Similarly, we say 
that $u$ is $D_K$-exact (resp. $D^c _K$-exact, $D_K D^c _K$-exact) if $u= D_K v$ (resp. $u= D^c _K v$, $u= D_K D^c _K v$) on  $X\setminus \Delta_1$ for some $v\in C^\infty _K(X, \Delta_1, L)$.
\smallskip

\noindent We first \emph{convert} the currents obtained at Section
\ref{currentsol} into log smooth forms, disregarding the convergence issues.  
Let $ \wh u_0$ and $\{u_k\}_{k=1}^\infty \subset \cD_K (X, L)$ be the currents constructed in \Cref{solucurrent} so that we have 
$$\cD_K u_1= \wh u_0 \wedge \alpha + T_0$$
and for $k \geq 1$
\begin{equation}\label{conj0}
\cD_K u_k = u_{k-1} \wedge \alpha + T_k
\end{equation}
(up to a sign), where the support of $\{T_k\} \subset  \cD_K (X, L)$ are currents supported in $\Delta_1$ of order $0$. 
Recall that by \Cref{currentcI}, we have a natural morphism
	\[
\wt \cI : \cD_{\bullet, K}(X, L) \to \cD_{\bullet, K}(X, E_\bullet).
\]
\smallskip

\noindent We claim that there exist a set of $E_\bullet$-valued currents
	$\{v_k\}_{k=0}^\infty \subset \cD_{\bullet, K}(X, E_\bullet)$ such that the form
	\begin{equation}\label{indution1}
	s_0:= \wt\cI (\wh u_0) + \cD_{K,E} v_0  \in \oplus_{p+q =\bullet }C^\infty _{(0,q), K} (X, E_p) = \cI\big(C^\infty _{\bullet ,K} (X, \Delta_1, L)\big)
	\end{equation}
	as well as   
	\begin{equation}\label{indution2}
	s_k:= \wt\cI (u_k) + \cD_{K, E} v_k  +  v_{k-1}\wedge \alpha \in \oplus_{p+q =\bullet }C^\infty _{(0,q), K} (X, E_p)  = \cI\big(
	C^\infty _{\bullet ,K}(X, \Delta_1, L)\big),
	\end{equation}
	where $k\geq 1$.
\smallskip
	
\noindent Before providing the arguments for this claim, note that these assumptions imply that $\{s_k\}$ solve the equations \eqref{simp4611}. Indeed, we have the string of equalities
\begin{align}\nonumber
D_K s_k =  & \cD_K (u_k + \cD_k v_k + v_{k-1}\wedge \alpha )|_{X\setminus \Delta_1}\nonumber \\
= & (\cD_K (u_k) + \cD_K v_{k-1}\wedge \alpha)  |_{X\setminus \Delta_1}\nonumber \\ = & (u_{k-1} \wedge \alpha +\cD_K v_{k-1}\wedge \alpha ) |_{X\setminus \Delta_1} \nonumber \\
= & ((u_{k-1} + \cD_K v_{k-1} +v_{k-2} \wedge \alpha) \wedge \alpha)  |_{X\setminus \Delta_1} \nonumber \\= &
s_{k-1} \wedge \alpha \nonumber
\end{align}
as required. Next we go ahead and prove the claim by induction.

\begin{proof} For $k=0$, we just take the $v_0$ to be the $\wt \cI (v_0)$ (after changing sign) for the $v_0$ in \Cref{initiald}. Note that by \Cref{vanish}, $\cI (u_{\Delta_1})=0$. In other words, $s_0$ equals to the original log smooth form $u_0$.
	
We assume that $v_0, \cdots, v_k$ have been constructed and satisfy \eqref{indution2}. To obtain $v_{k+1}$, we use the equality \eqref{conj0}. Together with \Cref{currentcI} and \Cref{vanish}, this implies that the equality 
	\begin{equation}\label{addeqq}
		\cD_{K,E} (\wt\cI (u_{k+1}) +  v_k \wedge \alpha ) = (\wt \cI (u_k) +  \cD_{K,E} v_k  + v_{k-1} \wedge \alpha) \wedge  \alpha =s_k \wedge \alpha
	\end{equation}
holds. Then $s_k \wedge \alpha$ is $\cD_{K,E}$-exact in the sense of currents. The induction hypothesis says that $s_k \in \oplus_{p+q= \bullet }C^\infty _{(0,q), K} (X, E_p)$ and the second part of \Cref{cumm}, implies the existence of some form 
$f  \in  \oplus_{p+q= \bullet -1 }C^\infty _{(0,q), K} (X, E_p)$  such that  
\begin{equation}
	\cD_{K,E} f =s_k \wedge \alpha. 
\end{equation}
This equality and \eqref{addeqq} imply that $\wt\cI (u_{k+1}) +  v_k \wedge \alpha -f $ is $\cD_{K, E}$-closed on $X$, hence $\cD_{K, E}$-exact up to a smooth form, as it follows from Hodge decomposition of this current. 
In particular, there exists a current $v_{k+1} \in  \cD_{\bullet, K}(X, E_\bullet)$ and a smooth form $f_1$ such that
$$\wt\cI (u_{k+1}) +  v_k \wedge \alpha -f_1 = -  \cD_{K,E} v_{k+1} $$
and thus the equality
$$ \cI (u_{k+1}) +  \cD_{K,E} v_{k+1} +  v_k \wedge \alpha = f_1  \in  \oplus_{p+q= \bullet }C^\infty _{(0,q), K} (X, E_p)$$
follows. Our claim is thus proved.
\end{proof}
\medskip

\noindent In this setting, we propose the following problem.
\begin{conjecture}
	Let $u\in C^\infty _K(X, \Delta_1, L)$ such that $u$ is $D_K$-closed and $D^c _K$-exact. 
	Then $u$ is $D_KD_K^c$-exact if and only if $\Res_{Y_I} u$ is $D_K$-exact on $Y_I$ for every index $I$, where $Y_I :=\cap_{i\in I} Y_i $ and $\Delta_1 =\sum Y_i$.
\end{conjecture}

\noindent Notice that one direction is almost immediate: the form $u$ is in particular $D_K$-exact, thus 
$$u = D_K v$$
for some $v\in C^\infty _K(X, \Delta_1, L)$. We write locally 
$v= \sum_I \frac{dz_I}{z_I} \wedge v_I$ for some $v_I \in C^\infty _K(X ,L)$.
Then 
$$u=  \sum_I \ep(I)\frac{dz_I}{z_I}\wedge D_K v_I,$$
from which we infer that $\Res_{Y_I} u$ is $D_K$-exact for any multi-index
$I$.

\noindent We show next that we can take $k= \infty$, so to speak and end the proof of the main theorem. 
First, Theorem \ref{solveequs} has the following consequence.
\begin{cor}\label{maincor1}
	In the setting of Theorem \ref{solveequs},  let $u \in  C^\infty _K (X, \Omega^\bullet _X (\log \Delta_1) \otimes L)$. If $- (\xi +\beta ) \wedge u$ is $D_K$-exact,  then there exists a form $v\in C^\infty _K (X, \log (\Delta_1))$ such that $- (\xi +\beta ) \wedge u = D_K v$ on $X\setminus D$ and $(\xi +\beta )  \wedge v$ is $D_K$-exact. 
\end{cor}

\noindent A quick remark: given a form $u \in C^\infty _K (X, \Omega^\bullet _X (\log \Delta_1) \otimes L)$, we say that $u$ is $D_K$-exact, if the equality $\displaystyle u=D_K v$ holds on $X\setminus D$, where $v\in C^\infty _K \big(X, \log (\Delta_1)\big)$.
The following statement is as consequence of the results in Section \ref{HodgeDek}.

\begin{proposition}\label{estiell}
There exists a uniform constant $C$ such that given any $D_K$-exact form \[u\in C^\infty _K (X, \Omega^\bullet _X (\log \Delta_1) \otimes L),\]  there exists $v\in C^\infty _K (X, \Omega^\bullet _X (\log \Delta_1) \otimes L)$ such that the equality
$$u =D_K v $$
holds, and moreover $\|v\|_{H^1} \leq C\|u\|_{H^1}$.
\end{proposition}

\begin{proof} This is a direct consequence of   \Cref{effectest}, as we will see next. We select the solution $v$ 
of the equation above such that $v\perp \ker(D_K)$. Then we have
\[D_K^\star u= \Delta_K v,\]
and Proposition \ref{effectest} with $f:= v$ and $g:= D_K^\star u$ gives the result.
\end{proof}
\medskip	

\noindent By combining the two proposition above, we have the following statement.
\begin{proposition}\label{mainpropind}
	In the setting of Theorem \ref{solveequs},	let $f \in   C^\infty _{k,K}(X,\Delta_1,L)$. 
	If $(\xi +\beta ) \wedge f $ is $D_K$-exact, then there exists a form $g\in  C^\infty _{k,K}(X,\Delta_1,L)$ such that 
the following hold.	
\begin{itemize}
		\item $D_K g  =- (\xi +\beta ) \wedge f$ on $X\setminus D$.
		
		\item  $(\xi +\beta )  \wedge g$ is $D_K$-exact.
		
		\item $\|g\|_{H^1} \leq C\|f\|_{H^1}$ for some uniform constant $C$.  
	\end{itemize} 
	Here $C$ depends on $X$ and $(\xi+\beta)$, but is independent of $f$.
	\end{proposition}
	
	\begin{proof}
		Let $\mathcal{H}$ be the space of harmonic forms of total degree $k$: notice that it is finite dimensional. 
		Let $\mathcal{H}_1 \subset \mathcal{H}$ be the space such that $(\xi +\beta) \wedge u$ is $D_K$-exact for every $u\in \mathcal{H}_1$.
		We have an orthogonal decomposition  $\mathcal{H} =\mathcal{H}_1 \oplus \mathcal{H}_2$ with respect to the $L^2$-norm. 
		
		Thanks to Corollary \ref{maincor1}, there exists $g_1$ such that 
		$$D_K g_1  =- (\xi +\beta ) \wedge f$$ 
		on $X\setminus D$ and $(\xi +\beta )  \wedge g_1$ is $D_K$-exact.  Now by Proposition \ref{estiell}, we can find a $g_2$ such that  
		$$D_K g_2  =- (\xi +\beta ) \wedge f$$ 
		on $X\setminus D$ and 
		\begin{equation}\label{ellest}
			\|g_2\|_{H^1} \leq C\|f\|_{H^1} .
			\end{equation}

		Then $g_1 -g_2$ is $D_K$-closed and we have a decompoisiton
		 $$g_1 =g_2 + a_1 +a_2 + D_K u $$where $a_1 \in \mathcal{H}_1$, $a_2 \in \mathcal{H}_2$ and $u\in  C^\infty _{k,K}(X,\Delta_1,L)$.
		 We claim that 
		 $$g:= g_2 +a_2$$ satisfies the three points in our statement.
		 
		By the definition of $\mathcal{H}_1$, the solution $g$ satisfies the first two points.  For the third point, we would like to $\|a_2\|_{H^1} \leq C \|g_2\|_{H^1}$ for some uniform constant $C$. It is a consequence of the following argument.
		
		We define a linear morphism 
		$$\varphi:  \mathcal{H}_2 \to  \ker  C^\infty _{\bullet,K}(X,\Delta_1,L)$$
		by taking $\varphi (u) :=(\xi +\beta )  \wedge u$ for $u\in \mathcal{H}_2$. By the definition of $\mathcal{H}_2$,  $\varphi$ is injective. By the construction of $\mathcal{H}_1$, we know that $Im (\varphi)  \cap Im D_K = 0$.  Then we can equip the finite linear space $Im (\varphi)$ with the following norm $\|\cdot\|_2$:  for every $a \in Im (\varphi)$, $\|a\|_2$ is defined to be the $L^2$-norm of the harmonic representation of $a$. $\|\cdot\|_2$ is norm on $Im (\varphi)$ since we have 
		$$Im (\varphi) \subset \ker D_K \qquad\text{and } \qquad Im (\varphi)\cap Im D_K = 0 .$$
		
		Note that $\varphi$ induces an isomorphism between two finite dimensional spaces $ \mathcal{H}_2 $ and $Im (\varphi)$. 
		Then any two norms on these two spaces are equvalent. Therefore,  there is a constant $C$ such that for every $u \in  \mathcal{H}_2$, we have
		$$\| u\|_{H_1} \leq C \|\varphi (u)\|_2  .$$
	By applying it to $a_2$, we know that 
		\begin{equation}\label{comparnorm}
			 \|a_2\|_{H_1} \leq C \| \varphi (a_2) \|_2  .
			 \end{equation}
		On the other hand, as $\varphi (a_1)$ is $D_K$-exact, we know that  $\varphi (a_2) - \varphi (g)$ is $D_K$-exact. Then the harmonic reprensetaive of $\varphi (a_2)$ and $\varphi (g)$ coincide. The $L^2$-norm of the harmonic representative of $\varphi (a_2) $ equals to $\|\varphi (a_2)\|_2$ by the definition of $\|\cdot\|_2$. Moreover, the $L^2$-norm of the harmonic representation of $\varphi (g)$ is less than the $L^2$-norm of $\varphi (g)$. As these two harmonic representations coincides, we have thus 
		$$ \|\varphi (a_2)\|_2  \leq \| (\xi +\beta ) \wedge g_2\|_{L^2}$$
	Together with \eqref{comparnorm}, we have 
		$$ \|a_2\|_{H_1}  \leq C \|(\xi +\beta ) \wedge g_2\|_{L_2} \leq C' \|g_2\|_{H_1} . $$
		Together with \eqref{ellest}, we obtain 
		$$\|g\|_{H_1} = \|g_2 +h_2\|_{H^1} \leq C\|f\|_{H^1} .$$ The proposition is proved.
		\end{proof}
\medskip
	
\noindent We infer the following consequence.
	
	\begin{thm}\label{thm:convergence}
		In the setting of Theorem \ref{solveequs}, we can find a sequence 
		$$\{v_k\}_{k=0}^{+\infty} \subset  C^\infty _{\bullet,K}(X,\Delta_1,L)$$ 
		such that $v_0 =u_0$ and for every $k$, we have  
		$$D_K v_{k+1} =- (\xi +\beta ) \wedge v_k \qquad\text{on }X\setminus D$$ and 
		$$\|v_{k+1}\|_{H^1} \leq C\|v_k\|_{H^1}$$ for some uniform constant $C$. 
		
		In particular, $\sum_{k \geq 0} t^k v_k$ is convergent when $|t| \ll 1$. 
		\end{thm}
		
		\begin{proof}
It is a direct consequence of Proposition \ref{mainpropind}. We can construct $v_k$ by using an induction on $k$ and by taking $f= v_k$ in Proposition \ref{mainpropind}.
			\end{proof}

\section{Jumping loci: main result and its proof}
The main goal of this section is to provide a new proof of \Cref{complete2}, based on the techniques developed in the preceding sections.  
We note that \Cref{complete2} was previously established by Budur–Wang \cite{BW15,BW20} using entirely different methods.   

\subsection{Some properties of De Rham moduli spaces and the Riemann-Hilbert map}\label{modulispace}
In this subsection, we recall the definition of De Rham moduli spaces $M_{\rm DR}(X/D)$ and the Riemann-Hilbert map and $D=D_1+\cdots+D_k$.
Let $\Res$ be the residue map of a logarithmic connection.   Let $\ell_1,\ldots,\ell_k$ be small loops around $D_i$. Let ${\rm Ev}:M_{\rm B}(X\setminus D)\to \mathbb{C}^k$ be the evaluation of any representation given by $\varrho\mapsto (\varrho(\ell_1),\ldots,\varrho(\ell_k))$.  For $M_{\rm DR}(X/D)$, it can be defined explicitly as follows. 
\begin{equation}\label{diag}
	\begin{tikzcd}
		&&0\arrow[d]&0\arrow[d]\\
		&&\mathbb{Z}^k\arrow[d]\arrow[r,equal]&\mathbb{Z}^k\arrow[d]&\\
		0\arrow[r]& M_{\rm DR}(X)\arrow[d,"{\rm RH}"]\arrow[r,hook]& M_{\rm DR}(X/D)\arrow[d,"{\rm RH}_D"]\arrow[r,"{\rm Res}"]&\CC^k\arrow[d,"\exp"]\\
		0\arrow[r]&M_{\rm B}(X)\arrow[r,hook]&M_{\rm B}(X\setminus D)\arrow[r,"{\rm Ev}"]\arrow[d]&(\CC^\star)^k\arrow[d]\\
		&&0&0
	\end{tikzcd}
\end{equation}
Here  $M_{\rm DR}(X)$, $M_{\rm DR}(X/D)$ are complex Lie groups if $X$ is a compact K\"ahler manifold. Note that $M_{\rm DR}(X)$ and $M_{\rm DR}(X/D)$ are complex algebraic groups   if $X$ is projective. In any case,  $M_{\rm B}(X\setminus D)$  and $M_{\rm B}(X)$ are  complex algebraic tori.  The Riemannian-Hilbert maps ${\rm RH}:M_{\rm DR}(X)\to M_{\rm B}(X)$ and ${\rm RH}_D:M_{\rm DR}(X/D)\to M_{\rm B}(X\setminus D)$ are defined by the monodromy representation, which are holomorphic maps. Let us denote by $M_{{\rm DR}}^0(X/D) $ and $M_{\rm B}^0(X\setminus D)$ the identity components of the corresponding complex Lie groups. We shall give an explicit description of $M^0_{\rm DR}(X/D)$. 

First, we note the following fact.
\begin{lemme}\label{lem:surj}
	For any $(L,\nabla)\in M_{\rm DR}(X/D)$, there exists some $a_1,\ldots,a_k\in \Q$ such that  $$c_1(L) = a_1c_1(D_1)+\cdots+ a_kc_1(D_k).$$  Conversely, if $c_1(L) = b_1c_1(D_1)+\cdots+ b_kc_1(D_k)$ for some $b_1,\ldots,b_k\in \CC$, then there exists a logarithmic connection $\nabla$ of $L$ such that $(L,\nabla)\in M_{\rm DR}(X/D)$. 
\end{lemme}
\begin{proof} 
	If $(L,\nabla)\in M_{\rm DR}(X/D)$, then the curvature current of the connection $\nabla+\db_L$ is  $\frac{i}{2\pi}[\db_L,\nabla]=-{\rm Res}(\nabla,D_i)[D_i]$.  Therefore, we have
	$$
	c_1(L)=-{\rm Res}(\nabla,D_i)[D_i].
	$$ 
	Note that the $c_1(L)$ and $c_1(D_i)$ are $\Q$-points in $H^2(X,\Q)$. The above equality implies that  $c_1(L)\in  \Q c_1(D_1)+\cdots+\Q c_1(D_k)$. 
	
	Conversely, if $c_1(L) = b_1c_1(D_1)+\cdots+ b_k c_1 (D_k)$ for some $b_1,\ldots,b_k\in \CC$, then we have
	$$
	c_1(L)=\rel (b_1)c_1(D_1)+\cdots+ \rel (b_k) c_1 (D_k) 
	$$
	as $c_1(L)\in H^{1,1}(X,\R)$. 
    This can be sen as a family of rational equations, which has a solution over $\mathbb R$. Then it has a solution over $\mathbb Q$, so we can write 
    $$
	c_1(L)= a_1c_1(D_1)+\cdots+ a_k c_1 (D_k) 
	$$
    for some rational numbers $a_i$.
    Moreover,
	there exists a singular hermitian metric $h$ for  $L$ such that the curvature current
	$$
	i\Theta_h(L)=\rel (b_1)[D_1]+\cdots+ \rel (b_k)[D_k].
	$$
	In particular, $i\Theta_h(L)=0$ on $X\setminus D$. Then $(L,D_h')\in M_{\rm DR}(X/D)$. 
	The lemma is proved.
\end{proof}

Let $(L, \nabla) \in M_{\rm DR} (X/ D)$ and let $(L' , \nabla') \in M_{\rm DR} (X/D)$ close to $(L, \nabla)$. Let $\dbar$ be the complex structure of $L$. Since $c_1 (L)= c_1 (L')$ on $X$, there exists a $\xi \in H^{0,1} (X)$ such that $\dbar + \xi$ defines the complex structure of $L'$.  We can let $\xi$ be a harmonic form. Then we have $d\xi=0$. Now $\nabla' = \nabla + \beta$ where $\beta$ is a $(1,0)$-form with log pole along $D$.
Moreover, the flat and holomorphic conditions  imply that 
\begin{equation}
	(\nabla + \beta)^2 =0  \qquad \text{and} \qquad [\nabla + \beta, \dbar +\xi] =0 \text{ on } X_0 .
\end{equation}
The first condition means that $\partial \beta= 0$ on $X_0$. For the second condition, since $\partial \xi=0$, we have $\dbar \beta =0$. In particular, we we know that $\beta \in H^0 (X, \Omega_X ^1 (\log D))$.  
To conclude, we have 

\begin{proposition}\label{identif}
	Fix a point $(L, \nabla) \in M_{\rm DR} (X/D)$ and let $(L' , \nabla') \in M_{\rm DR} (X/D)$ close to $(L, \nabla)$.  Then 
	$$(L', \nabla' ) \simeq (L, \dbar + \xi, \nabla + \beta )$$
	for some $\xi \in H^{0,1 } (X)$ smooth form with $d\xi=0$ and $\beta \in H^0 (X, \Omega_X ^1 (\log D))$.
\end{proposition}

\begin{proposition}
Let $(L, \nabla) \in M_{\rm DR} (X/D)$ and let $(L' , \nabla') \in M_{\rm DR} (X/D)$ close to $(L, \nabla)$.  Let $h_L$ be a harmonic metric on $L$ in Proposition \ref{specialhar}.  Then there exists a unique harmonic metric (up to a constant) $h_{L'}$ on $(L' , \nabla')$ such that 
\begin{equation}\label{samecurv}
	i\Theta_{h_L} (L) = i\Theta_{h_{L'}} (L') \qquad\text{ on } X .
	\end{equation} 

Morevoer, if $(L, \nabla, h_L)$ satisfies Proposition \ref{specialhar}, then  $(L' , \nabla' , h_{L'})$ satisfes the same properties in Proposition \ref{specialhar}. In particular,  the following two complexes 
\begin{equation}
	\nabla' :	\cdots \to \Omega^\bullet _X (\log D) \otimes L' \to \Omega^{\bullet+1} _X (\log D) \otimes L' \to \cdots 
\end{equation}
	\[D_{K, L'}: \cdots \to C^\infty _{\bullet, K} (X,  \Delta_1, L' )\to C^\infty _{\bullet +1, K} (X,  \Delta_1, L') \to \cdots \]
	are quasi-isomorphic.
\end{proposition}	

\begin{proof}
Since $h_L$ is harmonic, we have $i\Theta_{h_L} (L)  =\sum a_i [D_i]$ on $X$. Now as $(L', \nabla') \in M_{\rm DR} (X/D)$ is in the same component of $(L, \nabla)$, we have $c_1 (L') =\sum a_i [D_i] $. Then there exists a unique harmonic (after a constant) $h_{L'}$ on $(L' , \nabla')$ such that 
$$ i\Theta_{h_{L'}} (L') =\sum a_i [D_i]  = i\Theta_{h_L} (L) .$$
The first part is proved.  

For the second part, thanks to \eqref{samecurv}, we know that $ \nu (h_{L'}) = \nu (h_L)$ for every component of $D$.  Then the third point of Proposition \ref{specialhar} is satisified for $(L', \nabla')$.  Set $\theta' := \nabla' - D' _{h_{L'}}$.  By the first point of Proposition \ref{specialhar}, we know that $Res \theta' = 0$ on $\Delta_1 +\Delta_3$.  As $\nabla'$ is close to $\nabla$ and $ \nu (h_{L'}) = \nu (h_L)$,  then $\Res \theta'$ is close to $\Res \theta$.   Note that $\Res \theta \neq 0$ on $\Delta_2$. Then $\Res \theta \neq 0$ on $\Delta_2$ when $\nabla'$ is close to $\nabla$.  The second point of Proposition \ref{specialhar} is thus verified for $(L', \nabla')$.  
\end{proof}	

The objects in $M^0_{\rm DR}(X/D)$ correspond to a pair $(L,\nabla) \in M_{\rm DR}(X/D)$ with $c_1(L)=0$.  By Deligne's mixed Hodge theory, we have an isomorphism 
$$
H^1(X\setminus D,\CC)\simeq H^0(X,\Omega_X^1(\log D))\oplus H^1(X,\mathcal{O}_X)\simeq H^0(X,\Omega_X^1(\log D))\oplus \mathcal{H}^{0,1}(X). 
$$
For any $\xi\in H^1(X\setminus D,\CC)$, let $(\xi_1,\xi_2)\in H^0(X,\Omega_X^1(\log D))\oplus \mathcal{H}^{0,1}(X)$ be the corresponding element of $\xi$. This gives rise to a pair $(L,\nabla)$ as follows. The line bundle $L$ is defined by $\bar{\d}+\xi_2$.  Namely, locally there exists a function $f$ such that $\xi_2=\bar{\d}f$. Then  a local holomorphic section of $L$  is given by $e^{-f}$ as 
$$
(\bar{\d}+\xi_2)(e^{-f})=0.
$$
Then $\nabla:=\partial+\xi_1$  is a holomorphic connection for $L$. Indeed, one has
$$
(\partial+\xi_1)(e^{-f})=(-\d f+\xi_1)e^{-f},
$$ 
and by $\db(-\d f+\xi_1)=0$, it follows that $(-\d f+\xi_1)e^{-f}$ is a local holomorphic section of $\Omega_X^1(\log D)\otimes L$.

Note that $(L,\nabla)$ is isomorphic to $(\mathcal{O}_X,d)$ if   there is a nowhere vanishing smooth function $f \in C^0(X)$  that is  holomorphic  with respect to $L$ and flat with respect to $\nabla$ (i.e. $f$ is gauge transformation). Therefore, for any local smooth function $e$, one has  $$de=f^{-1}\circ (d+(\xi_1+\xi_2))\circ (f\cdot e)=de+f^{-1}df+(\xi_1+\xi_2).$$ This implies that $\xi_1\in H^0(X,\Omega_X^1)$ (i.e. it does not have log poles along $D$), and  $(d+\xi_1+\xi_2)(f)=0$.  If we impose the condition that $f\left(x_0\right)=1$, the solution is necessarily given by the integral

$$
f=\exp \left(-\int_{x_0}^x (\xi_1+\xi_2)\right)
$$

This is well-defined on $X$ if and only if all periods of $\xi_1+\xi_2$ belong to $ 2 \pi i  \mathbb{Z}$. Therefore, we have
\begin{align}\label{eq:DR}
	M_{\rm DR}^0(X/D)\simeq \frac{H^0(X,\Omega_X^1(\log D))\oplus \mathcal{H}^{0,1}(X)}{\{(\xi_1,\xi_2)\in H^0(X,\Omega_X^1)\oplus \mathcal{H}^{0,1}(X)\mid (\xi_1,\xi_2) \mbox{ has periods in }2\pi i\Z\}}.
\end{align} 
Let $\gamma_1,\ldots,\gamma_m\in \pi_1(X_0,x_0)$ such that they generate a basis for $H_1(X_0,\mathbb{Z})/{\rm torsion}$.   Then $M^0_{\rm B}(X_0)\simeq (\CC^\star)^m$ with the isomorphism  being $\varrho\mapsto (\varrho(\gamma_1),\ldots,\varrho(\gamma_m))$. Then the Riemann-Hilbert map is given by the holonomy representation.  It can be made explicitly as follows.  Take $(\xi_1,\xi_2)\in H^0(X,\Omega_X^1(\log D))\oplus \mathcal{H}^{0,1}(X)$.  We denote by $[(\xi_1,\xi_2)]$  its image in $M_{\rm DR}^0(X/D)$. Then we have
\begin{align}\nonumber
	M_{\rm DR}^0(X/D)&\to M_{\rm B}^0(X_0)\\\label{eq:RH}
	{\rm RH}_D:[(\xi_1,\xi_2)]&\mapsto(\exp(-\int_{\gamma_1}(\xi_1+\xi_2)),\ldots,\exp(-\int_{\gamma_m}(\xi_1+\xi_2))). 
\end{align}

\begin{lemme}\label{lem:iso}
	${\rm RH}_D|_{M_{\rm DR}^0(X/D)}$ is a surjective group morphism   between complex Lie groups, that is a covering map. 
\end{lemme}
\begin{proof}
	It is easy to see that it is a group morphism of complex Lie groups.  Since their Lie algebras are isomorphic, it is an covering map.  Choose any $(a_1,\ldots,a_m)\in (\CC^*)^m$. Since we have the isomorphisms $$(\CC^*)^m\simeq {\rm Hom}(H_1(X_0,\CC),\CC)\simeq H^1(X_0,\CC)\simeq H^0(X,\Omega_X^1(\log D))\oplus \mathcal{H}^{0,1}(X),$$
	there exists $(\xi_1,\xi_2)\in H^0(X,\Omega_X^1(\log D))\oplus \mathcal{H}^{0,1}(X)$ such that 
	$$
	(\exp(-\int_{\gamma_1}(\xi_1+\xi_2)),\ldots,\exp(-\int_{\gamma_m}(\xi_1+\xi_2)))=(a_1,\ldots,a_m).
	$$
	Therefore, ${\rm RH}_D|_{M_{\rm DR}^0(X/D)}$ is surjective. 
\end{proof}

\begin{remark}
	Note that in general, ${\rm RH}_D|_{M_{\rm DR}^0(X/D)}$ is not injective and it can be an infinite covering. For example, let $X=\P^1$ and $D=\{0,\infty\}$. Then $\pi_1(X_0)\simeq \Z$. Consider the pair $(\O_X,d+\log z)$. Then   ${\rm RH}_D([(\O_X,d+\log z)])=1$, that  is the trivial representation.  However, $(\O_X,d+\log z)$ is not isomorphic to $(\O_X,d)$. Note that $M_{\rm DR}(X/D)=H^0(X,\Omega_{X}^1(\log D))\simeq \CC$, and $M_{\rm B}(X_0)\simeq \CC^*$, with ${\rm RH}_D$ being the exponential map. Hence ${\rm RH}_D|_{M_{\rm DR}^0(X/D)}$ is an infinite covering map. 
\end{remark}

Let us state the following criterion for the  translate of subtori. 
\begin{lemme}\label{lem:criteria}
	Let $X\subset (\CC^*)^m=:T$ be an irreducible Zariski closed subset of the complex algebraic torus $T$. Assume that there exists a Zariski open subset $X^\circ\subset X$ satisfying the following properties
	\begin{itemize}
		\item $X^\circ$ is smooth,
		\item For any $p\in X^\circ$, and any $\xi\in T_{p}X$, there exists some $  \ep$ such that 
		\begin{align}\label{eq:tori}
			\{z\in \DD_\ep \mid\exp(z\xi_0)p\}\subset X.
		\end{align}  Here $\xi_0\in \CC^m$ is uniquely  defined so that \begin{align}\label{eq:xi0}
			\frac{d}{dz}|_{z=0}\exp(z\xi_0)p=\xi 
		\end{align}
	\end{itemize}
	Then $X$ is a translate of a subtorus in $T$.  
\end{lemme}
\begin{proof} 
	By \eqref{eq:tori}, there exists some $\ep>0$ such that 
	\begin{align}\label{eq:tori3}
		\{z\in \DD_\ep  \mid\exp(z\xi_{0})p\}\subset X 
	\end{align}  Since $X$ is Zariski closed, it follows that 
	\begin{align}\label{eq:tori2}
		\{z\in  \CC \mid\exp(z\xi_{0})p\}\subset X 
	\end{align}
	for any $\xi\in T_pX$ with $\xi_0$ defined by \eqref{eq:xi0}. 
	
	Let $\xi_1,\ldots,\xi_k\in T_pX$ be  a basis for $T_pX$. Since $p$ is a smooth point of $X$, then $\dim X=k$.  $\xi_{i,0}\in \CC^m$ is uniquely  defined so that $\frac{d}{dz}|_{z=0}\exp(z\xi_{i,0})p=\xi_i$. Then by \eqref{eq:tori2}, we have
	\begin{align}\label{eq:tori1}
		\{(z_1,\ldots,z_k)\in  \CC^k \mid\exp(\sum_{i=1}^{k}z\xi_{i})p\}\subset X. 
	\end{align}
	Consider a group morphism $f$ defined by 
	\begin{align}
		f:	 (\CC^k,+)&\to (\CC^*)^m\\\nonumber
		(z_1,\ldots,z_k)&\mapsto \exp(\sum_{i=1}^{k}z\xi_{i})p
	\end{align} 
	By construction, $f$ is a \emph{local} immersion at the origin.  By \eqref{eq:tori1},  the image  $f(\CC^k)$ is contained in an irreducible Zariski closed subset $X\cdot p^{-1}$ of dimension $k$. Therefore,  $f(\CC^k)$ is Zariski closed, and a subtorus of $(\CC^*)^m$.  and thus $X=f(\CC^k)\cdot p$ is a translate of subtori.  
\end{proof}


\subsection{A fine resolution for relative log De Rham complex}\label{subsec:universal}
Fix a  $(L,\nabla)\in M_{\rm DR}(X/D)$.   Let $(\xi, \omega):\DD\to \mathcal H^{0,1}(X)\times H^0(X,\Omega_{X}(\log D))$ be a holomorphic map such that $(\xi(0),\omega(0))=(0,0)$.  Let $p_1:X\times\DD \to X$ be the projection map. 
 We consider the local universal  holomorphic line bundle 
$
\mathcal L \to X\times \DD
$  defined by 
$$\mathcal L := (p_1^*L, p_1^*\db_L+   \xi(z)) .$$
Namely, the complex structure on $p_1^*L$ is given by $p_1^*\db_L+   \xi(z)$.  The aim of this subsection is to generalise the constructions in Section \ref{hypercoho} to $\mathcal L$.

\medskip

To begin with, we first construct a metric on $\mathcal L$.
Let us equip $L$ with a harmonic metric $h$ defined in \Cref{specialhar}. We equip  $\Lie$ the pullback metric $p_1^*h$. Then 
\begin{lemme} \label{lem:perturb}
	There exists some $\ep>0$ such that  for any $z\in \DD_\ep$, $p_1^*h |_{X\times \{z\}}$ is  a harmonic metric for  $(L,\db_L+ \xi(z), \nabla+ \omega(z))\in M_{\rm DR}(X/D)$    satisfying the properties in \Cref{specialhar}. 
\end{lemme}
\begin{proof}
	Write $L_z$ for the holomorphic line bundle $(L,\db_L+ \xi(z))$. 	Let $\theta_z$ be the corresponding logarithmic Higgs field of $(L_z, \nabla+ \omega(z))$ with respect to the metric $h$, and let $D_z'$ be the $(1,0)$-part of the Chern connection of $(L,\db_L+ \xi(z),h)$.   By the construction in \Cref{subsec:harmonic}, we have
	$$
	D_z'=D_0'- \overline{\xi(z)},
	$$
	and 
	$$
	-2\theta_z=D_0'-\overline{\xi(z)}-\nabla-\omega(z)=-2\theta_0-\overline{\xi(z)}-\omega(z).
	$$
	Then 
	$$
	i\Theta_{h}(L_z)=i[\db_L+\xi(z), D_0'-\overline{\xi(z)}]=i\Theta_h(L),
	$$ 
	and 
	$$
	{\rm Res}_{D_i}\theta_z={\rm Res}_{D_i}\theta_0+\frac{1}{2}{\rm Res}_{D_i}\omega(z)
	$$ 
	Hence there are some $\ep>0$ such that the three conditions in \Cref{specialhar} are fulfilled  for any $z\in \DD_\ep$.
\end{proof}

Now we would like to construct a relative logarithmic De-Rham complex on $\mathcal{L}$.
We shall abusively write $\omega(z)$ and $\xi(z)$ for $p_1^*\omega(z)$ and $p_1^*\xi(z)$.  Set $\cX:=X\times \DD $.   Note that $p_1^*(\nabla+\db_L)+ \omega(z)+\xi(z)$ is a  connection for the holomorphic line bundle $\Lie$.  Note that this connection is  not flat, since its curvature is $dz\wedge  (\frac{\d \omega} {\d z}+\frac{\d \xi} {\d z})$.  However, if we consider the relative logarithmic complex, it will be flat.

Set $\cD:=D\times \DD$.  We define  $\Omega^1_{\cX/\DD}(\log \cD)$ to be the quotient
$$
p_2^*\Omega_\DD^1\to \Omega^1_{\cX}(\log \cD) \to \Omega^1_{\cX/\DD}(\log \cD) 
$$ 
where $p_2:X\times \DD\to \DD$ is the projection on the second factor.  	Since $\Omega^1_{\cX}(\log \cD)=p_1^*\Omega_X^1(\log D)\oplus p_2^*\Omega_\DD^1$, it follows that  $$\Omega^1_{\cX/\DD}(\log \cD)=p_1^*\Omega^1_X(\log D).  $$ 
Then 
$$
\Omega^i_{\cX/\DD}(\log \cD):= \Lambda^i \Omega^1_{\cX/\DD}(\log \cD)=p_1^*\Omega^i_X(\log D).
$$
Note that
$$
\Omega^i_{\cX}(\log \cD)=\oplus_{k+\ell=i}p_1^*\Omega_X^k(\log D)\otimes_{\mathcal O_\cX} p_2^*\Omega_\DD^\ell.
$$ 	
Hence $\Omega^i_{\cX/\DD}(\log \cD)$ is a direct factor of  $	 \Omega^i_{\cX}(\log \cD)$, and there is thus a natural quotient
$$
q:	 \Omega^i_{\cX}(\log \cD)\to  \Omega^i_{\cX/\DD}(\log \cD).
$$For any open subset $U$ of $\cX$, and any holomorphic section 
$$e\in \Lie\otimes \Omega_{\cX/\DD}^i(U)= \Lie\otimes \Omega_{X}^i(U),$$ we define 
$$
\tilde{\nabla}(e):=q((p_1^*\nabla+  \omega(z))(u)). 
$$  
Then we have $\tilde{\nabla}^2=0$.   Let $\gamma:\DD\to M_{\rm DR}(X/D)$ be the holomorphic map defined by $\gamma(z):=(L,\db_L+\xi(z),\omega(z))$.  
We can thus define a \emph{relative logarithmic De-Rham complex}  by
\begin{align}\label{eq:relative}
	\mathcal L^\bullet_{\gamma}:=\mathcal L \stackrel{\tilde{\nabla}}{\to  }  \mathcal L\otimes \Omega_{\cX/\DD}^1(\log D)\to \cdots\stackrel{\tilde{\nabla}}{\to  }   \mathcal L\otimes \Omega_{\cX/\DD}^n(\log D)  
\end{align} 

\medskip

As in Section \ref{hypercoho},  we would like construct a fine resolution of the complex \eqref{eq:relative} . To begin with, we have the following easy lemma. 
\begin{lemme}\label{lem:invariant}
Set $L_z:=\Lie|_{X\times \{z\}}$.  	Consider the complex  of sheaves
	\begin{align}\label{eq:complex}
	\cC^\infty _{\bullet,K}(X,\Delta_1,L): \cC^\infty _{0,K}(X,\Delta_1,L)\stackrel{D_K}{\to}\cdots \stackrel{D_K}{\to} \cC^\infty _{2n,K}(X,\Delta_1,L)
	\end{align} 
	defined in \Cref{defend1}.    	For any $z\in \DD$, we have the identification
	$$
  \cC^\infty _{\bullet,K}(X,\Delta_1,L)= \cC^\infty _{\bullet,K}(X,\Delta_1,L_z)
	$$
\end{lemme}
\begin{proof}
This follows from the definition of \Cref{defend1}.    
\end{proof}
We now replace $\cX$ by $X\times \DD_\ep$.   Write $\tilde{\Delta}_1:=\Delta_1\times \DD_\ep$.  Consider the sheaf   $ \cC^\infty _{\bullet,K}(\cX,\tilde{\Delta}_1,\Lie)$ defined \eqref{defend1}.  By \Cref{lem:invariant}, we have
a splitting
\begin{align}\label{eq:split}
	\cC^\infty _{\bullet,K}(\cX,\tilde{\Delta}_1,\Lie)=p_2^* dz\wedge    	\cC^\infty _{\bullet,K}(\cX,\tilde{\Delta}_1,\Lie) 
\oplus p_2^* d\bar{z}\wedge    p_1^* \cC^\infty _{\bullet-1,K}(X,\Delta_1,L)\oplus  p_1^* \cC^\infty _{\bullet-1,K}(X,\Delta_1,L)
\end{align} 
where 
$$
p_1^* \cC^\infty _{\bullet,K}(X,\Delta_1,L)=p_1^{-1}\cC^\infty _{\bullet,K}(X,\Delta_1,L)\otimes_{p^{-1}C^\infty_X}C^\infty_{\cX} 
$$ 
Let 
$ \cC^\infty _{\bullet,K}(\cX/\DD_\ep,\tilde{\Delta}_1,\Lie)$ be the quotient of  $ \cC^\infty _{\bullet,K}(\cX,\tilde{\Delta}_1,\Lie)$ by  $p_2^* dz\wedge    	\cC^\infty _{\bullet,K}(\cX,\tilde{\Delta}_1,\Lie) $. Then it is isomorphic to 
\begin{align}\label{eq:iso}
p_2^* d\bar{z}\wedge    p_1^* \cC^\infty _{\bullet-1,K}(X,\Delta_1,L)\oplus  p_1^* \cC^\infty _{\bullet-1,K}(X,\Delta_1,L).
\end{align} 

For the connection $\tilde{D}_K:=p_1^*(\nabla+\db_L)+  \omega(z)+\xi(z) $,   we note that 
\begin{align}\label{eq:preserve2}
	\tilde{D}_K:p_2^* d\bar{z}\wedge    p_1^* \cC^\infty _{\bullet-1,K}(X,\Delta_1,L)\to p_2^* d\bar{z}\wedge    p_1^* \cC^\infty _{\bullet,K}(X,\Delta_1,L),\\\label{eq:preserve}
	\db_\Lie:  p_2^* d\bar{z}\wedge    p_1^* \cC^\infty _{\bullet-2,K}(X,\Delta_1,L)\oplus  p_1^* \cC^\infty _{\bullet-1,K}(X,\Delta_1,L)\to \\\nonumber
p_2^* d\bar{z}\wedge    p_1^* \cC^\infty _{\bullet-1,K}(X,\Delta_1,L)\oplus  p_1^* \cC^\infty _{\bullet,K}(X,\Delta_1,L)
\end{align}
and 
$$
\tilde{D}_K^2= dz\wedge  (\frac{\d \omega} {\d z}+\frac{\d \xi} {\d z}).
$$
It thus induces an integrable connection on $ \cC^\infty _{\bullet,K}(\cX/\DD_\ep,\tilde{\Delta}_1,\Lie)$, denoted abusively by $\tilde{D}_K$. Hence we have a complex 
\begin{align}\label{eq:complex2}
	 \cC_{K,\gamma}^\bullet := \cC^\infty _{\bullet,K}(\cX/\DD_\ep,\tilde{\Delta}_1,\Lie)\stackrel{\tilde{D}_K}{\to}\cdots  \stackrel{\tilde{D}_K}{\to} \cC^\infty _{\bullet,K}(\cX/\DD_\ep,\tilde{\Delta}_1,\Lie)
\end{align} 

As an analogue of Theorem \ref{quasi}, we have the following relative version.
\begin{proposition}\label{prop:qis}
	The two complexes  \eqref{eq:relative} and \eqref{eq:complex2} are quasi-isomorphic via the  natural morphism.
\end{proposition}
\begin{proof}
	Fix any $x\in X$, we take an open neighborhood $U$ of $x$, that is biholomorphic to some polydisk. Pick any $e\in \cC_{K,\gamma}^i(U) $ such that $\tilde{D}_K(e)=0$.  By \eqref{eq:split}, we can decompose $$e=\sum_{k=0,\ldots,\ell} e^{i-k,k}\quad  {\rm modulo }\quad  p_2^* dz\wedge   \cC^\infty _{i-1,K}(\cX,\tilde{\Delta}_1,\Lie)$$ such that each $e^{k,i-k}$ are in \eqref{eq:iso}. Then $\tilde{D}_K(e)=0$ together with \eqref{eq:preserve} and \eqref{eq:preserve2} implies that 
	$$
	\db_\Lie(e^{i-\ell,\ell})=0.
	$$ 
	By \Cref{debareq}, there exists $f$ such that $\db_\Lie f=-e^{i-\ell,\ell}$. So we have
	$$e+\tilde{D}_K(f)=\sum_{k=0,\ldots,\ell-1} e_1^{i-k,k}   \quad  {\rm modulo }\quad  p_2^* dz\wedge   \cC^\infty _{i-1,K}(\cX,\tilde{\Delta}_1,\Lie),$$
	with each $e_1^{i-k,k}$ lying in \eqref{eq:iso}.   
	Keep this algorithm in this manner, we finally find some smooth $f$ such that
	$$
	e+\tilde{D}_K(f)=e_0^{i,0}     \quad  {\rm modulo }\quad  p_2^* dz\wedge   \cC^\infty _{i-1,K}(\cX,\tilde{\Delta}_1,\Lie),$$
	such that $\tilde{D}_K(e_0^{i,0})=0$, and $e_0^{i,0}$ is a smooth $(i,0)$-form lying in \eqref{eq:iso}. This implies that $\db_{\Lie}(e_0^{i,0})=0$, and $\tilde{\nabla}(e_0^{i,0})=0$. Hence we prove the quasi-isomorphism between the two complexes. 
\end{proof}

\subsection{Cohomology jumping loci property}
Consider the subset
\[
\Sigma_k^i(X/D)_{\rm DR}
:= \left\{ (L,\nabla) \in M_{\rm DR}(X/D) \;\middle|\; 
\dim \mathbb{H}^i \!\left(X, \Omega_X^\bullet(\log D) \otimes L\right) \geq k \right\}.
\]
It is an analytic subset of $M_{\rm DR}(X/D)$. 
Note that we have the descending chain
\[
\Sigma_k^i(X/D)_{\rm DR} \supset \Sigma_{k+1}^i(X/D)_{\rm DR} \supset \cdots.
\]
Therefore, for a \emph{generic} point $(L,\nabla) \in \Sigma_k^i(X/D)_{\rm DR}$, there exists a neighborhood $U \subset \Sigma_k^i(X/D)_{\rm DR}$ such that 
\[
\dim \mathbb{H}^i \!\left(X, \Omega_X^\bullet(\log D) \otimes L'\right)
\]
remains constant for all $(L',\nabla') \in U$.  

Let $\Delta_1=\sum_{i=1}^{m}Y_i$ and we consider any $I=\{i_1,\ldots,i_p\}\subset \{1,\ldots,m\}$.   We shall adopt the notations in \Cref{subsec:main}.   Then for a \emph{generic} point $(L,\nabla) \in \Sigma_k^i(X/D)_{\rm DR}$, there exists a neighborhood $U \subset \Sigma_k^i(X/D)_{\rm DR}$ such that  \[
\dim \mathbb{H}^i \!\left(Y_I, \Omega_{Y_I}^\bullet(\log (D_I \setminus \Delta_1)) \otimes L'\right)
\]
remains constant for all $(L',\nabla') \in U$.

\begin{thm}\label{nearbygeneric}
	Let  $(L,\nabla) \in \Sigma_k^i(X/ D)_{\rm DR}$ be a generic point and we suppose that 
	$$\rel (\Res_{D_k} (L, \nabla)) \in ]-1,0]$$
	for every component $D_k$ of $D$.
Let $\mathbb D$ be the unit disc in $\mathbb C$ and	we consider a  holomorphic map 
$$\gamma: \mathbb D \to \Sigma_k^i(X/ D )_{\rm DR}$$ 
such that $\gamma(0)= (L,\nabla)$. 
	We denote by $\gamma'(0)\in \mathcal H^{0,1}(X)\times H^0(X,\Omega_{X}(\log D))$ the first order deformation of  $\gamma$.  Then 
	\begin{equation} \label{uppertorus}
		\dim \mathbb H^i (X , \Omega_X ^\bullet (\log D )\otimes L_{\gamma (0) +z \gamma'(0)})\geq k
	\end{equation}
	for every $ |z|\ll 1$.   Here $\Omega_X ^\bullet (\log D )\otimes L_{\gamma (0) +z \gamma'(0)}$ is the logarithmic De Rham complex induced by $\gamma (0) +z \gamma'(0) \in M_{\rm DR} (X/D) $. 
\end{thm}

\begin{proof} 
By the above arguments, for a generic point $(L,\nabla)\in \Sigma_k^{i}(X/D)_{\rm DR}$,  there exists a neighborhood $U \subset \Sigma_k^i(X/D)_{\rm DR}$ of $(L,\nabla)$  such that 
	\[
	\dim \mathbb{H}^i \!\left(X, \Omega_X^\bullet(\log D) \otimes L'\right)
	\]
	and 
	\begin{align} \label{eq:strata}
	\dim \mathbb{H}^{i-|I|} \!\left(Y_I, \Omega_{Y_I}^\bullet(\log (D_I\setminus \Delta_1)) \otimes L'\right)
	\end{align}
all remain  constant for all $(L',\nabla') \in U$ and any non-empty $I\subset \{1,\ldots,m\}$.

  As $\gamma:\DD\to \Sigma_k^i(X/D)_{\rm DR}$ is a holomorphic curve with $\gamma(0)=(L,\nabla)$, there exists some $\ep>0$ and holomorphic maps $(\alpha,\beta):\DD_\ep\to \mathcal H^{0,1}(X)\times H^0(X,\Omega_{X}(\log D))$ such that $(\alpha(0),\beta(0))=(0,0)$,   $\gamma(z)=(L,\db_L+\alpha(z),\nabla+\beta(z))$, and $\gamma(\DD_\ep)\subset U$.      Let $\Lie_{\gamma}^\bullet$ and $C^\bullet_{K,\gamma}$ be the two complexes defined in \Cref{eq:relative} and \eqref{eq:complex2}.   By applying \Cref{prop:qis}, after shrinking $\ep$, we can assume that $\Lie_{\gamma}^\bullet\to C^\bullet_{K,\gamma}$ is a quasi-isomorphism.  
	
	As $(L,\nabla)$ is a generic point in $\Sigma$, we know that the restiction 
	$$ \Gamma (\mathbb D_\ep, \mathbb R^i _* (\mathcal{L}^\bullet _\gamma )) \to \mathbb H^i(X, \mathcal{L}^\bullet_{K,\gamma (0)} ) $$
	is surjective by the cohomology and base change theorem.  Note that we have
	\begin{equation*}
		\begin{tikzcd}
		\Gamma(\mathbb D_\ep, \mathbb R^i _* (\mathcal{L}^\bullet _\gamma ))\arrow[d,"\simeq"] \arrow[r] & \mathbb H^i(X, \mathcal{L}^\bullet_{K,\gamma (0)} )\arrow[d,equal] \\
		\bH^i(\cX, \mathcal{L}^\bullet _\gamma )\arrow[r]\arrow[d,"\simeq"]&  \mathbb H^i(X, \mathcal{L}^\bullet_{K,\gamma (0)} )\arrow[d,"\simeq","\varphi"']	 \\
H^i(\cX, C^\bullet_{K,\gamma})\arrow[r,"\psi"] &  H^i( C^\infty _{\bullet,K}(X, {\Delta}_1,L))
		\end{tikzcd}
	\end{equation*}
	where the first vertical isomorphism comes from the definition. 
		Note that the relative complexes $\Lie_{\gamma}^\bullet$ and $C^\bullet_{K,\gamma}$ are quasi-isomorphic. Hence  the vertical arrows in the above diagram are all isomorphisms.  It follows that the bottom arrow is also surjective. Therefore,  for an element $e\in \mathbb H^i(X, \mathcal{L}^\bullet_{K,\gamma (0)} ) $, we can find a section  $u \in  C^\infty _{\bullet,K}(\cX/\DD_\ep,\tilde{\Delta}_1,\Lie)$ such that $\psi(u)=k(e)$.  

	Now we denote $u (z) :=  u |_{X\times \{z\}}$. Then  $u (0)$ represents $e$ and $D_K u =0$ in the sense of the relative complexe. In particular, we have
 $D_{K,z}u(z)=0$  for each $z\in \DD_\ep$. Here $D_{k,z}:=\nabla+\db_L+\alpha(z)+\omega(z)$.   Set $\xi:=\frac{\d \alpha}{\d z}(0)$ and $\omega:=\frac{\d \beta}{\d z}(0)$. If we take the first order expansion for $
	D_{K,z}u(z)=0
	$, it follows that
\begin{align}\label{eq:extension}
	 D_Ku(0)=0, \quad 	D_K \frac{\d u }{\d z}(0) =-(\xi+\omega)\wedge u(0).
\end{align} 
	Moreover, by \eqref{eq:strata} and the same arguments as above, we can show that, for any non-empty $I=\{i_1,\ldots,i_p\}\subset \{1,\ldots,m\}$,    \eqref{eq:extension} holds for any $Y_I$. In other words,   every $D_K$-closed form on $C_{i-|I|, K}^\infty(Y_I, L)$ can be extended   to order one. Then the second condition in \Cref{solveequs} is verified. Hence we can apply \Cref{solveequs}. 
	
	Let us consider another holomorphic map $\gamma_2:\DD_\ep\to M_{\rm DR}(X/D)$ defined by $\gamma_2(z)=(L,\db_L+z\xi,\nabla+z\omega)$.   We can shrink $\ep$ such that for the two complexes $\Lie_{\gamma_2}^\bullet$,  $C_{K,\gamma_2}^\bullet$, the inclusion $C_{K,\gamma_2}^\bullet\to \Lie_{\gamma_2}^\bullet$  is also a quasi-isomorphism by \Cref{prop:qis}.

	By \Cref{solveequs}, there exists   $u_j\in C^\infty_{i,K}(X,  \Omega^{\bullet}_X (\log \Delta_1 ) \otimes L)$ such that  they satisfy \eqref{simp4611}. Moreover, after possibly shrinking $\DD_\ep$, 
	for 	$$
	\sigma:=	u(0)+zu_1+\cdots+\cdots+z^ju_j+\cdots, 
	$$ 
	we have  $\sigma\in \cC_{K,\gamma_2}^i(\cX)$ by \Cref{thm:convergence}. 
	Let $\tilde{D}_{K}:C_{K,\gamma_2}^{\bullet}\to C_{K,\gamma_2}^{\bullet+1}$ be the  relative differential defined in  \Cref{subsec:universal}.  By \eqref{simp4611}, we have 
	$$
	\tilde{D}_K(\sigma)=0.
	$$ 
	By \Cref{prop:qis}, this implies that	 
	$$
	\mathbb R^i (p_1)_*(\mathcal L^\bullet_{\gamma_2})(\DD_\ep)\to  \bH^i(X, \Lie_{\gamma_2}^\bullet|_{X\times\{0\}})=\bH^i(X,L^\bullet)
	$$
	is surjective. Set $L_z:=(L,\db_L+z\xi)$ and $\nabla_z:=\nabla+z\omega$. Then $(L_z,\nabla_z)=\gamma_2(z)=\gamma(0)+z\gamma'(0)$. Let $L_z^\bullet$ be the logarithmic De Rham complex
	$$
	L_z\stackrel{\nabla_z}{\to}	L_z\otimes \Omega_{X}(\log D)\stackrel{\nabla_z}{\to}\cdots \stackrel{\nabla_z}{\to} L_z\otimes \Omega^{\dim X}_{X}(\log D).
	$$  Hence, by the \emph{cohomology and base change theorem}, after possibly shrinking $\ep$,  we have
	$\mathbb R^i (p_1)_*(\mathcal L^\bullet_{\gamma})$ is locally free over $\DD_\ep$, and $\bH^i(X, L_z^\bullet)$  is constant for $z\in \DD_\ep$.     
	   The theorem is thus proved.
\end{proof}

Now we would like to prove the main theorem in this section.

\begin{thm}\label{complete2}
	Let $X$ be a compact K\"ahler manifold and let $D=\sum_{i=1}^{k}D_i$ be a simple normal crossing divisor.    Then the jumping locus $\Sigma_k^i(X\setminus D)_{\rm B}$ is a finite union of translate of subtori. 
\end{thm}
\begin{proof}
	Let $\Sigma$ be an irreducible component of  $\Sigma_k^i(X\setminus D)_{\rm B}$. 	We pick a general  representation $\varrho\in \Sigma$, that is contained in the smooth locus of $\Sigma$.  Let $(L, \nabla) \in M_{\rm DR} (X/D)$ such that $\varrho = {\rm RH}_D( (L, \nabla))$ and $\rel (\Res (L,\nabla))\in ]-1,0]$ for every component of $D$. 
	
	Pick any $\xi\in T_\varrho\Sigma$. By \Cref{lem:iso},   there exists  $(\xi_1,\xi_2)\in H^0(X,\Omega_X^1(\log D))\oplus \mathcal{H}^{0,1}(X)$ such that, for the   family of representations $\varrho_t:\pi_1(X_0)\to \CC^*$ defined by  $$\varrho_t(\gamma)=\varrho(\gamma)\cdot \exp(-z\int_{\gamma}(\xi_1+\xi_2)), \quad \ \forall  \gamma\in \pi_1(X_0), $$ we have 
	$$
	\frac{d}{d t}|_{t=0}\varrho_t=\xi\in T_\varrho\Sigma.  
	$$ We also consider a family of log flat connections parametrized by $z\in \CC$ defined by $$(E_z,\nabla'_z):= (\db+z\xi_2,\d+z\xi_1).$$ Then $${\rm RH}_D(E_z,\nabla'_z)(\gamma_i)= \exp(-z\int_{\gamma_i}( \xi_1+ \xi_2))$$ by \eqref{eq:RH}.  We define a family of pairs $(L_z,\nabla_z):=(L,\nabla)\otimes  (E_z,\nabla_z')$ parametrized by $z\in \CC$. Then we have  
	$$
	{\rm RH}_D(L_z,\nabla_z)(\gamma)=\varrho (\gamma)\cdot \exp(-z\int_{\gamma}( \xi_1+ \xi_2))=\varrho_t(\gamma)\quad \forall\ \gamma\in \pi_1(X_0)
	$$
	with $t=\exp z$.  This implies that 
	$$
	{\rm RH}_D(L_z,\nabla_z)=\varrho_t. 
	$$
	Note that the residue map $M_{\rm DR}(X)\to \CC^k$ is a holomorphic map. Then for any representation $(L',\nabla')$ in a small neighborhood of $(L,\nabla)$ in $\Sigma_k^i(X/ D)_{\rm DR}$,  its   residues  are not positive integers along every component of $D$.  
	Therefore, by Theorem \ref{Deli}, we have
	$$
	\dim \mathbb H^{i} (X, \Omega_X ^\bullet (\log D)\otimes L') = \dim   H^{i} (X_0 ,  \mathcal L'),
	$$
	where $\mathcal L'$ is the local system associated with $(L',\nabla')$.  It follows from \Cref{lem:iso} that 
	$$
	(\xi_1,\xi_2)\in T_\varrho \Sigma_k^i(X/ D)_{\rm DR}.
	$$ 
	Therefore, we conclude from Theorem \ref{nearbygeneric}   that, there exists some $\ep>0$ such that $\varrho_t\in \Sigma$ for any $t\in \DD_\ep$. By \Cref{lem:criteria}, we conclude that   $\Sigma_k^i(X\setminus D)_{\rm B}$  is a translate of subtori of $M_{\rm B}(X\setminus D)$.   
\end{proof}
\begin{proof}[Proof of Theorem \ref{complete}] Let $W$ be an irreducible component of $\Sigma ^p_k(X\setminus D)_{\rm B}$ containing $\gamma(\mathbb C)$. By Theorem \ref{complete2} $W$ is a translate of a sub-torus of $M_{\rm B}(X\setminus B)$ whose tangent space contains $\dot\alpha=d\gamma (0)$. In particular $\tau +z\dot\alpha\subset W$ and the result follows. 
    
\end{proof}


	 \section{Generic vanishing Theorem}
	 Based on the above ideas, in this section, we can give    straightforward proofs for two  generic vanishing theorems by Green-Lazarsfeld \cite{GL87}. 
	 
	 The first one is the Kodaira type vanishing theorem. 
	 \begin{thm}[Green-Lazarsfeld]\label{thm:GL}
	 	Let $X$ be  a compact K\"ahler $n$-fold with maximal Albanese dimension.  
        Let \[V^i(\mathcal O _X)=\{L\in {\rm Pic}^0(X) \ s.t.\ h^i(X,L)\ne 0\},\] then ${\rm codim}_{{\rm Pic}^0(X)}\mathcal W\geq n-i$ for any irreducible component of $V^i(\mathcal O _X)$.
        In particular, for a general $L\in {\rm Pic}^0(X)$,  we have $H^{j}(X,\O_X(K_X+ L))=0$ for any $j>0$. 
	 \end{thm}
	 \begin{proof}
	 	Fix any $i>0$ then $V^i(\mathcal O _X)$ is an analytic subset of ${\rm Pic }^0(X)$. Let $L\in \mathcal W\subset V^i(\mathcal O _X)$ be a general point of the component $\mathcal W$,  then for any $[\xi] \in T_L\mathcal W\subset T_L {\rm Pic }^0(X)=H^1(\mathcal O _X)$ we pick a representative $\xi\in A^{0,1}(X)$ such that $\db \xi=0$.
        For any $t\in \mathbb C$, consider the holomorphic line bundles $L_{t\xi}$ with  complex structure defined by $\db_t:=\db_L+t\xi$.   Then there exists $\ep>0$ such that if $t\in \DD_\ep$, $\dim H^i(X,L_{t\xi})$ is constant.  
	 	
	 	We choose any $u\in A^{0,i}(X,L)$ such that $\db u=0$. By  \cite[Theorem 7.4]{Kod86}, there exists $u(t)\in A^{0,i}(X,L_{t\xi})$ for any $t\in \DD_\ep$, such that $u(0)=u$, $\db_t u(t)=0$, that deforms smoothly. Therefore, one can write 
	 	$$
	 	u(t)=u_0+t u_1+t^2 u_2+\cdots, 
	 	$$
	 	and we have
	 	$$
	 	\db u_1+\xi \wedge u_0=0. 
	 	$$
	 	Equivalently,  the first order deformation of $u_0$ is always unobstructed along any direction $\xi$ in $W:=T_L \mathcal W\subset H^{0,1}(X)$.   Therefore, the natural map
	 	$$
	 	H^{0,i}(X,L)\times W\to H^{0,i+1}(X,L)
	 	$$
	 	defined by the cup product is always zero. 
	 	Since $L$ has a  hermitian flat metric, if we take the conjugate $\bar W\subset H^{0}(\Omega^1_X)$, the map 
	 	$$
	 	H^{0}(X,\Omega_X^i\otimes L^\star)\times \bar W\to H^{0}(X,\Omega_X^{i+1}\otimes L^\star)
	 	$$
	 	is trivial.  Since $X$ has maximal Albanese dimension, there exists $\omega_1,\ldots,\omega_n\in H^0(X,\Omega_X^1)$ such that $\omega_1\wedge\cdots\wedge\omega_n\neq 0$. If $w=\dim W$, then we may assume that $\omega _1,\ldots,\omega _w$ is a basis of $\bar W$. If $v\in H^{0}(X,\Omega_X^i\otimes L^\star)$  is a non-zero section, it follows that, at a general point $x\in X$,  that $v\wedge\omega_i(x)=0 $ for $i=1,\ldots, w$ and so $i\geq w$, i.e. 
        \[{\rm codim}_{{\rm Pic}^0(X)}\mathcal W=n-w\geq n-i.\] The first assertion is thus proved.  The second assertion follows from hard Serre duality.  The theorem is thus proved.  
	 \end{proof}
	 
	 
	 In \cite{GL87}, they also proved a Nakano type generic vanishing theorem as follows. 
	 \begin{thm}[Green-Lazarsfeld]\label{thm:Nakano}
	 Let $X$ be  a compact K\"ahler $n$-fold.  Define 
	 $$
	 w(X):={\rm max}\{{\rm codim}_X{Z(w)} \mid 0\neq \omega\in H^0(X,\Omega_X^1) \},
	 $$ 
	 where $Z(\omega)$ is the zero locus of $\omega$.  Then for a general $L\in {\rm Pic}^0(X) $, we have
	 $
	 H^i(X,\Omega_X^j\otimes L)
	 $ 
	 for $i+j<w(X)$.  If there is a nowhere vanishing one-form on $X$, then the Theorem holds with $w(X)=\infty$. 
	 \end{thm}
	 We will  prove \Cref{thm:Nakano} in the rest of the section. We first prove the fllowing lemma. 	
	 \begin{lemme}\label{lem:vanish}
	 Let $\omega\in H^0(X,\Omega_X^1)$ be a holomorphic form such that ${\rm codim}_X{Z(w)}=m$, and let $L\in {\rm Pic}^0(X)$. Then the hypercohomology $\mathbb H^i(X,L^\bullet_\omega)$ of the Dolbeault complex 
	 \begin{align}\label{eq:Dol}
	 	L^\bullet_\omega:= 	  L\stackrel{\wedge\omega}{\to} L\otimes \Omega_X^1\to\cdots\to \cdots  \stackrel{\wedge\omega}{\to}L\otimes \Omega_X^n, 
	 \end{align} 
	  vanishes for $i<m$. 
	 \end{lemme}
	 \begin{proof}
	  By \cite[Proposition 3.4]{GL87},  we know that $L^\bullet_\omega$ is exact for $i<m$. By \cite[Section 2.4.1]{EZT14}, there exists a natural spectral sequence for $L^\bullet_\omega$ such that $E_2^{p,q}=H^p(X,\mathcal H^q)$, where $\mathcal H^q$ is the $p$-th cohomology sheaf of the complex of $L^\bullet_\omega$.  Therefore, $E_2^{p,q}=0$ if $q<m$. Note that $E_3^{p,q}$ is given by  $$\frac{\ker d_2:E^{p,q}\to E^{p+2,q-1}}{{\rm Im} d_2:E^{p-2,q+1}\to E^{p,q}},$$  thus is zero  for $q<m$. By virtue of the same manner,  we can show that $E_\infty^{p,q}=0 $ for $q<m$. Note that 
	 $$
	 F^p\mathbb H^{i}(X,L^\bullet_\omega)/	F^{p+1}\mathbb H^{i}(X,L^\bullet_\omega)=E_\infty^{p,i-p}.
	 $$  
	 for any $p\in \{0,\ldots,i\}$. It follows that if $i<m$, they all vanish. Hence $\mathbb H^i(X,L^\bullet_\omega)=0$ for $i<m$. The lemma is proved. 
	 \end{proof}
	 \begin{proof}[Proof of \Cref{thm:Nakano}] 
We will establish a relative Dolbeaut complex as \Cref{subsec:universal}, but without boundary divisors.  Let $\cX:=X\times \CC$ and $\Lie:=p_1^*L$, where $p_1:\cX\to X$ is the projection map.  Set $\tilde{\omega}:=zp_1^*\omega$, where $z$ is the coordinate for the second factor $\CC$ of $X\times \CC$. We define a morphism 
$$
\theta:\Lie\to \Lie\otimes \Omega_{\cX/\CC}^1
$$
defined by the composition 
$$
\Lie\stackrel{\wedge\tilde{\omega}}{\to} \Lie\otimes \Omega_{\cX}^1\to   \Lie\otimes \Omega_{\cX/\CC}^1
$$ 
where the second map is induced by the quotient $\Omega_{\cX}^1\to   \Omega_{\cX/\CC}^1$. Then $\theta\wedge\theta=0$ and we have a holomorphic (relative Dolbeaut) complex
$$
{\rm Dol}(\Lie,\theta):=		\Lie\stackrel{\theta}{\to}   \Lie\otimes \Omega_{\cX/\CC}^1\stackrel{\theta}{\to} \cdots   \stackrel{\theta}{\to}   \Lie\otimes \Omega_{\cX/\CC}^n. 
$$ 
Let us consider a fine resolution for this complex. Let $\mathcal A^k(\cX,\Lie)$ be the sheaf of $\Lie$-valued smooth $k$-forms on $\cX$.  By our construction, we have 
\begin{align}\label{eq:split2}
	\mathcal A^k(\cX,\Lie)=p_2^* dz\wedge  	p_1^{*}\cA^{k-1}(X,L)
	\bigoplus p_2^* d\bar{z}\wedge  	p_1^{*}\cA^{k-1}(X,L)\bigoplus  p_1^{*}\cA^k(X,L)
\end{align} 
where we denote by
$$
p_1^{*}\cA^k(X,L):=p_1^{-1}\cA^k(X,L)\otimes_{p^{-1}\cC^\infty_X}\cC^\infty\cX. 
$$
We denote by 
$$\cA^k(\cX/\CC,\Lie):=	\mathcal A^k(\cX,\Lie)/ p_2^* d\bar{z}\wedge  	p_1^{*}\cA^{k-1}(X,L)\bigoplus  p_1^{*}\cA^k(X,L).$$
Let   
$$
D''_{\omega}:\Lie\to \Lie\otimes  \mathcal A^k(\cX/\CC,\Lie)
$$
be the composition
$$
\Lie\stackrel{	 \db_{\Lie}+\tilde{\omega}}{\to} \Lie\otimes  \mathcal A^k(\cX,\Lie) \to \Lie\otimes  \mathcal A^k(\cX/\CC,\Lie). 
$$ 
Note that $D_\omega''^2=0$. Hence we have another complex
$$
\cA^0(\cX/\CC,\Lie)\stackrel{D_\omega''}{\to}    \cA^1(\cX/\CC,\Lie)\stackrel{D_\omega''}{\to} \cdots   \stackrel{D_\omega''}{\to}    \cA^{2n}(\cX/\CC,\Lie). 
$$ 
\begin{claim}
	The natural morphism ${\rm Dol}(\Lie,\theta)\to (\cA^\bullet(\cX/\CC,\Lie),D_\omega'')$ is a quasi-isomorphism.  
\end{claim}
\begin{proof}
	The proof is exactly the same as \Cref{prop:qis}. We omit it. 
\end{proof}
Note that we have	 \begin{align} 
{\rm Dol}(\cX,\theta)|_{X\times\{z\}}=	L^\bullet_{zw}:= 	  L\stackrel{\wedge z\omega}{\to} L\otimes \Omega_X^1\stackrel{\wedge z\omega}{\to}\cdots  \stackrel{\wedge z\omega}{\to}L\otimes \Omega_X^n, 
\end{align} 	 
We simply write $L^\bullet$ for $L^\bullet_{0}$. 
	 By the definition, there exists $\omega\in H^0(X,\Omega_X^1)$ such that ${\rm codim}_XZ(w)=w(X)=:m$.  
	 Fix any $i<m$.  We shall prove that $\mathbb H^i(X,L^\bullet)=0$ for a general $L\in {\rm Pic}^0(X)$. We assume by contradiction that, $\mathbb H^i(X,L^\bullet)\neq 0$  for any $L\in {\rm Pic}^0(X)$. It follows that, for a general $L\in {\rm Pic}^0(X)$, the dimension of $\mathbb H^i(X,L'^\bullet)$ is constant for line bundles $L'$ in a neighborhood of $L$ in $\ {\rm Pic}^0(X)$.   Therefore,  by the same arguments as in the proof of \Cref{thm:GL},  for any $\eta\in A^{0,1}(X)$ such that $\db \eta=0$, and any $u\in A^{p,q}(X,L)$ with $p+q=i$,  
	 $\eta\wedge u$ is $\db_L$-exact. Hence
	 the cup product
	 $$
	 H^{p,q}(X,L)\otimes H^{0,1}(X)\to H^{p,q+1}(X,L)
	 $$ 
	 vanishes identically.  If we take the conjugate, then  
\begin{align}\label{eq:cup}
	  H^{q,p}(X,L^\star)\otimes H^{0}(X,\Omega_{X}^1)\to H^{q+1,p}(X,L^\star),
\end{align} 
	 vanishes identically, where $L^\star$ is the dual of $L$. 
In what follows, we replace $L$ by $L^\star$ for notational simplicity. 

For any $e\in H^{p,q}(X,L)$, we choose a harmonic form $u_0\in \mathcal H^{p,q}(X,L)$ representing $e$. Then by \eqref{eq:cup}, we have
$$
u_0\wedge \omega=\db_L v
$$
for some $v\in A^{p+1,q-1}$. Let $\d_L$ be the $(1,0)$-part of the Chern  connection $\nabla_h$ for $(L,h)$. Note that $\nabla_h^2=0$.  Then we have
$$
\d_L(u_0\wedge \omega)=0=\db_L(u_0\wedge \omega).
$$
By the $\d\dbar$-lemma, there exists some $\gamma_1\in A^{p,q-1}(X,L)$ such that
$$
u_0\wedge \omega=\db_L\d'_L \gamma_1.
$$
Let $u_1:=-\d'_L \gamma_1$ and we consider $u_1\wedge \omega$. It is also $\d'_L$-exact, and both $\d_L'$ and $\db_L$ closed. Therefore, by  the  $\d\dbar$-lemma again, there exists some $\gamma_2\in A^{p,q-1}(X,L)$ such that
$$
u_1\wedge \omega=\db_L\d'_L \gamma_2 
$$
and we set $u_2:=-\d'_L \gamma_2$.  We iterate this algorithm, and obtain $u_1,\ldots,u_k,\ldots,$ such that for any $i\geq 0$, 
\begin{align}\label{eq:inductive}
	 u_i\wedge\omega=-\db_L u_{i+1}.
\end{align}   Based on the same arguments as in \Cref{thm:convergence}, we can assume that there exists some $\ep>0$ such that 
$$
u:=\sum_{i\geq 0}z^iu_i\in A^{k}(\cX_\ep,\Lie),
$$
where $\cX_\ep:=X\times \DD_\ep$.   By \eqref{eq:inductive}, we have
$$
(\db_\Lie+\tilde{\omega})\sigma=0. 
$$
Let $\sigma\in \Gamma(\cX_\ep, \cA^k(\cX/\CC,\Lie))$ be the image of $u$ under the quotient $\cA^k(\cX,\Lie)\to \cA^k(\cX/\CC,\Lie)$. Then we have
$$
D_\omega''(\sigma)=0. 
$$
Note that we have
\begin{equation*}
	\begin{tikzcd}
		\Gamma(\mathbb D_\ep, \mathbb R^i p_{2,*} ({\rm Dol}(\cX,\theta)))\arrow[d,"\simeq"] \arrow[r] & \mathbb H^i(X, L^\bullet )\arrow[d,equal] \\
		\bH^i(\cX_\ep, {\rm Dol}(\cX,\theta) )\arrow[r]\arrow[d,"\simeq"]&  \mathbb H^i(X, L^\bullet )\arrow[d,"\simeq","\varphi"']	 \\
		H^i( ( A^\bullet(\cX_\ep/\DD_\ep,\Lie),D_\omega''))\arrow[r,"\psi"] &  H^i(A^\bullet(X,L),\db_L)
	\end{tikzcd}
\end{equation*}
	 where the vertical arrows are isomorphic.  The above arguments imply that $\psi$ is surjective. It follows that
	 $$
	 \mathbb R^i p_{2,*} ({\rm Dol}(\cX,\theta)(\DD_\ep)\to  \mathbb H^i(X, L^\bullet )
	 $$
	 is surjective. By the cohomolgy and base change theorem, if we shrink $\ep$, 
$z\mapsto \dim \mathbb H^i(X, L^\bullet_{z\omega})$ is constant for $z\in \DD_\ep$. Therefore, by the assumption, $\mathbb H^i(X, L^\bullet_{z\omega})\neq 0$.  Note that ${\rm codim}_X Z(z\omega)={\rm codim}_X Z(\omega)>i$ if $z\neq 0$. By \Cref{lem:vanish}, we obtain that $\mathbb H^i(X, L^\bullet_{z\omega})= 0$, which yields a contradiction. The theorem is thus proved.  
	 \end{proof}




   
\providecommand{\bysame}{\leavevmode ---\ }
\providecommand{\og}{``}
\providecommand{\fg}{''}
\providecommand{\smfandname}{\&}
\providecommand{\smfedsname}{\'eds.}
\providecommand{\smfedname}{\'ed.}
\providecommand{\smfmastersthesisname}{M\'emoire}
\providecommand{\smfphdthesisname}{Th\`ese}

\end{document}